\RequirePackage{fix-cm}
\documentclass[smallextended]{svjour3}       % onecolumn (second format)
\smartqed  % flush right qed marks, e.g. at end of proof
\usepackage{graphicx}
\usepackage{amsmath}
\usepackage{amssymb}
\usepackage{amsfonts}
\usepackage{mathrsfs}
\usepackage{longtable}

\newcommand\relphantom[1]{\mathrel{\phantom{#1}}}
\begin{document}

\title{Galerkin finite element approximation for semilinear stochastic time-tempered fractional wave equations with multiplicative white noise and fractional Gaussian noise
\thanks{This work was supported by the National Natural Science Foundation of China under Grants No. 41875084, No. 11801452, No. 11671182, No. 11571153, and the Fundamental Research Funds for the Central Universities under Grant No. lzujbky-2018-ot03.}
}

\titlerunning{Galerkin approximation for stochastic PDEs}        % if too long for running head

\author{Yajing Li         \and
        Yejuan Wang %$^{*}$       
        \and      %etc.
        Weihua Deng
}


\institute{Y.J. Li \at
              College of Science, Northwest A \& F University, Yangling 712100, Shaanxi, P.R.
China
 \\
%              Fax: +123-45-678910\\
              \email{hliyajing@163.com}           %  \\
           \and
           Y.J. Wang %$^{*}$ 
           \at
              School of Mathematics and Statistics, Gansu Key Laboratory of Applied Mathematics and Complex Systems, Lanzhou University, Lanzhou 730000, P.R. China
              \\
              \email{wangyj@lzu.edu.cn % (corresponding author)
              }
            \and
           W.H. Deng \at
              School of Mathematics and Statistics, Gansu Key Laboratory of Applied Mathematics and Complex Systems, Lanzhou University, Lanzhou 730000, P.R. China
              \\
              \email{dengwh@lzu.edu.cn}
}
\date{Received: date / Accepted: date}
\maketitle

\begin{abstract}
To model wave propagation in inhomogeneous media with frequency-dependent power-law attenuation,
it is needed to use the fractional powers of symmetric coercive elliptic operators in space and the Caputo tempered fractional derivative in time. The model studied in this paper is semilinear stochastic space-time fractional wave equations driven by infinite dimensional multiplicative white noise and fractional Gaussian noise, because of the potential fluctuations of the external sources. The purpose of this work is to discuss the Galerkin finite element approximation for the semilinear stochastic fractional wave equation. We first provide a complete solution theory, e.g., existence, uniqueness, and regularity. Then the space-time multiplicative white noise and fractional Gaussian noise are discretized, which results in a regularized stochastic fractional wave equation while introducing a modeling error in the mean-square sense. We further present a complete regularity theory for the regularized equation. A standard finite element approximation is used for the spatial operator, and the mean-square priori estimates for the modeling error and for the approximation error to the solution of the regularized problem are established. 
%Finally, numerical experiments are performed to confirm the theoretical analyses.
%In this paper we discuss the Galerkin finite element approximations for semilinear stochastic fractional wave equations with fractional diffusion and Caputo tempered fractional time derivative driven by infinite dimensional multiplicative white noise and fractional noise.
%We first provide a complete solution theory, e.g., existence, uniqueness, and regularity, for semilinear stochastic fractional wave equations with Lipschitz nonlinear terms.
%Further, we discrete the space-time multiplicative white noise and fractional noise, which introduces a modeling error in the mean-square sense and result in a regularized stochastic fractional wave equation; then the complete regularity theory of the regularized equation is presented.
%For the discretization in space, a standard finite element approximation is used and the mean-square priori estimates for the modeling error and for the approximation error to the solution of the regularized problem are established. Finally,  numerical experiments are performed to confirm the theoretical analysis.

\end{abstract}

%Galerkin finite element method, semilinear stochastic time tempered fractional wave
%equation, fractional Laplacian, multiplicative white noise, multiplicative fractional noise

\section{Introduction}
\label{intro}
%Your text comes here. Separate text sections with

The classical wave equation well models the wave propagation in ideal medium. However, the wave propagation in complex inhomogeneous media generally has frequency-dependent attenuation, being observed in a wide range of areas including acoustics, viscous damping in the seismic isolation of buildings, structural vibration, and seismic wave propagation \cite{Chen04,Li:16,Meerschaert15,Szabo94}. The striking power-law feature of the attenuated wave propagation implies that the Laplacian in classical equation should be replaced by fractional powers of symmetric coercive elliptic operators in space, while the time-tempered derivative should be substituted for second time derivative. Because of the finite time/space scale, the tempered power-law distribution in some sense becomes more reasonable choice compared with the pure power-law one \cite{Bruno,Cartea,Meerschaert}. There are already some discussions on the numerical methods or correct ways of specifying the boundary conditions for tempered fractional differential equations; see, e.g., \cite{Baeumer,Deng-Zhang,Gajda,Zayernouri,Zhang:16} and the references therein or \cite{Deng-Zhang-Book,Deng2018}. As for the fractional wave equations, there are also some progresses not only on their numerical methods  \cite{Cuesta:06,Delic,Fan:2017,Gao,Sun:06,Wei} but also on their fundamental solutions and  properties \cite{Boyadjiev,Ferreira,Kian,Schneider}.

%Based on the CTRW model with truncated power-law waiting time and/or jump length
%distribution(s), anomalous diffusion is described and presents a very
%slow transition from anomalous to normal diffusion.  Tempered fractional differential equations play a significant role in the modeling of anomalous diffusion process. They arise naturally in a wide variety of applications such as physical,
%biological, and chemical processes \cite{Bruno,Cartea,Meerschaert}. Numerical methods have been studied for the tempered fractional differential equations; see, e.g., \cite{Baeumer,Deng-Zhang,Gajda,Zayernouri,Zhang:16} and the references therein.
%The wave propagation with frequency-dependent attenuation has been observed in a wide range of areas, such as acoustics, viscous dampers in seismic isolation of buildings, structural vibration, seismic wave propagation (see \cite{Chen04,Meerschaert,Szabo94}, for instance). The traditional wave equation models wave propagation in an ideal medium. to characterize wave propagation in inhomogeneous media with frequency-dependent power-law attenuation, the space-time fractional wave equation appears. For the deterministic time fractional wave equations, the fundamental solutions and their properties have been considered in, for example, \cite{Boyadjiev,Ferreira,Kian,Schneider} and the references therein; different kinds of numerical methods and approximation schemes have been developed, e.g., the finite difference method \cite{Cuesta:06,Delic,Gao,Sun:06,Wei}, the finite element method \cite{Fan:2017}.

Random effects arise naturally in practically physical systems; the ones considered in this paper are on the fluctuations of the external sources, and the fluctuations include both infinite dimensional multiplicative white noise and fractional Gaussian noise, which drive the semilinear space-time fractional wave equations. The multiplicative noise can capture the effects of geometrical confinements \cite{Lau:07}. The fractional Gaussian noise is the formal derivative of the fractional Brownian motion (fBm) $B^H$, being a centered Gaussian process with a special covariance function determined by Hurst parameter $H \in (0, 1)$. For $H = \frac{1}{2}$, $B^{\frac{1}{2}}$ is the standard Brownian motion, the formal time derivative of which is white noise. For $H \neq \frac{1}{2}$, $B^H$ behaves in a way completely different from the standard Brownian motion; especially, neither is a semi-martingale nor a Markov process. In addition, the fBm with Hurst parameter $H \in (\frac{1}{2}, 1)$ enjoys the property of a long range memory, which roughly implies that the decay of stochastic dependence with respect to the past is only sub-exponentially slow. This long-range dependence property of the fBm makes it a realistic choice of noise for problems with long memory in the applied sciences.

%The model studied in this paper is a semilinear stochastic space-time fractional wave equations driven by infinite dimensional multiplicative white noise and fractional noise
%
%%we also consider a multiplicative noise, which captures the effects of geometrical confinement \cite{Lau:07}.
%
%
%In probability theory, a fractional Brownian motion (fBm) $B^H$ is a centered Gaussian process with a special covariance function determined by Hurst parameter $H \in (0, 1)$. For $H = \frac{1}{2}$,
%$B^{\frac{1}{2}}$ is the standard Brownian motion where the generalized temporal derivative is white
%noise. For $H \neq \frac{1}{2}$, $B^H$ behaves in a completely different way than the standard Brownian motion; in particular, neither is a semi-martingale nor a Markov process. In addition, the fBm with
%Hurst parameter $H \in (\frac{1}{2}, 1)$ enjoys the property of a long range memory, which
%roughly implies that the decay of stochastic dependence with respect to the past is
%only sub-exponentially slow. This long-range dependence property of the fBm makes
%it a realistic choice of noise for problems with long memory in the applied sciences.
%
%Random effects arise naturally in many practically physical systems and are often the source of instability.

With the above introduction of the fractional wave equation and the external noises, now we propose the model, which is a space-time fractional wave equation driven by three nonlinear external source terms: a deterministic term and two stochastic terms, being respectively white noise and fractional Gaussian noise. Specifically, the model is a semilinear stochastic time tempered fractional wave equation with $\frac{3}{2}<\alpha<2$, $\frac{1}{2}<\beta<1$, $\frac{1}{2}< H<1$, and $\nu>0$:
\begin{equation} \label{eq0.1}
\left\{ \begin{array}
 {l@{\quad} l}
\displaystyle _0^c\partial_t^{\alpha,\nu}u(t,x)+(-\Delta)^{\beta} u(t,x)=f(t,u(t,x))+g(t,u(t,x))\frac{\partial^2\mathbb{W}(t,x)}{\partial t\partial x}\\
\\
\displaystyle~~~~~~~~~~~~~~~~~~~~~~~~~~~~~~~
+h(t,u(t,x))\frac{\partial^2\mathbb{W}^H(t,x)}{\partial t\partial x} \quad \rm{in~}
(0,T]\times\mathcal{D},\\
\\
 u(t,x)=0 \quad  \rm{on~} (0,T]\times\partial\mathcal{D},\\
\\
u(0,x)=a(x),~ \partial_tu(t,x)|_{t=0}=b(x)\quad  \rm{in~} \mathcal{D},\\
\end{array}\right.
\end{equation}
where $\mathcal{D}\subset \mathbb{R}^d$, $d=1,2,3$,  is a bounded convex polygonal domain with a boundary $\partial\mathcal{D}$;  $_0^c\partial_t^{\alpha,\nu}$ denotes the left-sided
Caputo tempered fractional derivative of order $\alpha$
with respect to $t$; $(-\Delta)^{\beta}$ is the
fractional Laplacian, the definition of which is based on the spectral decomposition of the Dirichlet Laplacian, as adopted in \cite{Nochetto15}; $\frac{\partial^2\mathbb{W}(t,x)}{\partial t\partial x}$ and $\frac{\partial^2\mathbb{W}^H(t,x)}{\partial t\partial x}$, respectively, represent the infinite dimensional white noise and fractional Gaussian noise defined on a complete filtered probability space $(\Omega, \mathcal{F}, \{\mathcal{F}_t \}_{t\geq 0}, \mathbb{P})$;
and the initial data $a$ and $b$ are $\mathcal{F}_0$-measurable random variables. Assumptions on the smoothness of the nonlinearites $f$, $g$, and $h$ will be given below.

Numerical approximations of stochastic wave equations with classical derivatives have been considered in recent literatures; see, e.g., \cite{Anton:2016,Cohen:2013,Du,Kovacs:2010,Wang:2014}. Stochastic solutions to wave equations with additive or multipicative fractional Gaussian noise have been studied in, for example, \cite{Balan:2012,Balan:2010,Caithamer,Sardanyons,Stojanovic,Tang:2010} and the references therein. There has, however, been  little mention of numerical approximations for semilinear stochastic fractional wave equations with fractional Gaussian noise even for the linear case. Very recently, we investigated the Galerkin finite element approximations for linear stochastic space-time fractional wave equations with an infinite dimensional additive noise \cite{Li:16}. The purpose of this paper is to consider the Galerkin finite element approximations for semilinear stochastic fractional wave equations with fractional Laplacian in space and Caputo tempered fractional time derivative driven by infinite dimensional multiplicative white noise and fractional Gaussian noise. The novelty and the difficulties of this work are in three aspects: (i) The nonlinear terms and nonlinear multiplicative noises (In comparison with our results recently published in \cite{Li:16}, the analysis of the nonlinear parts requires different mathematical machineries in order to derive error estimates); (ii) The multiplicative fractional Gaussian noise term (Since the fractional Brownian motion, neither is a semi-martingale nor a Markov process, and does not have the property of independent increments, some new ideas for dealing with multiplicative fractional Gaussian noise are developed here);  % to circumvent the difficulty caused by fractional Gaussian noise.
(iii) The complete solution theory (Here we develop a complete solution theory, e.g., existence, uniqueness, and regularity, for semilinear stochastic fractional wave equations with multiplicative white noise and fractional Gaussian noise in order to conduct new and very complicated error estimates).

This paper is organized as follows. In section \ref{sec:2}, we first
introduce some basic definitions, notations, and necessary preliminaries. We then, in section \ref{sec:3}, prove the existence, uniqueness, and regularity for the semilinear stochastic fractional wave equation. In Section \ref{sec:4}, we discretize the space-time multiplicative white noise and fractional Gaussian noise, which result in a regularized semilinear stochastic fractional wave equation while introducing a modeling error in the mean-square sense. The convergence order of the modeling error and the regularity of the regularized equation are well established. Section \ref{sec:5} is devoted to providing the finite element scheme for the regularized semilinear stochastic fractional wave equation, and the corresponding very detailed mean-square error estimates are presented. We conclude the paper with some discussions in the last section.
%In Section \ref{sec:6}, the numerical experiments are performed to confirm the convergence orders of the modeling error and the finite element approximations to the regularized equation.

\section{Preliminaries}
\label{sec:2}

In this section, we recall some basic definitions, notations, and necessary preliminaries; and collect useful facts on the Mittag-Leffler function, the Brownian motion, and the fractional Brownian motion.

\subsection{Fractional Laplacian and Caputo tempered fractional derivative} \label{subsec:2.2}
The operator $-\Delta$ $: L^2(\mathcal{D})\rightarrow L^2(\mathcal{D})$, with the domain
Dom$(-\Delta)=\{u \in H^1_0(\mathcal{D}),~ \Delta u \in L^2(
\mathcal{D})\}$, is positive, unbounded, and closed; and its inverse is
compact. Hence the spectrum of the operator $-\Delta$ is discrete, real, positive, and
accumulates at infinity. Moreover, the eigenfunctions $\{\varphi_k\}_{k\in \mathbb{N}} \subset
H^1_0 (\mathcal{D})$ satisfying
\begin{equation}\label{1.1}
\left\{
  \begin{array}{ll}
    -\Delta \varphi_k(x)=\lambda_k\varphi_k(x) & \hbox{in $\mathcal{D}$,} \\
\\
    \varphi_k(x)=0 & \hbox{on $\partial\mathcal{D}$,~ $k \in \mathbb{N}$,}
  \end{array}
\right.
\end{equation}
form an orthonormal basis of $L^2(\mathcal{D})$. Consequently, $\{\varphi_k\}_{k\in \mathbb{N}}$ is an orthogonal basis of $H^1_0(\mathcal{D})$ and $\|\nabla_x\varphi_k\|_{L^2(\mathcal{D})} =\sqrt{\lambda_k}$.

For any $s\in \mathbb{R}$, we denote by $\mathbb{H}^s(\mathcal{D})\subset L^2(\mathcal{D})$
the Hilbert space induced by the norm
\[\|u\|^2_{\mathbb{H}^s}=\sum^\infty_{k=1}\lambda^s_k (u,\varphi_k)^{2}.\]
In particular, $\mathbb{H}^0(\mathcal{D})=L^2(\mathcal{D})$ with $\|\cdot\|$ denoting the norm in $\mathbb{H}^0(\mathcal{D})$ and $(\cdot,\cdot)$ denoting the inner
product of $\mathbb{H}^0(\mathcal{D})$, $\mathbb{H}^1(\mathcal{D})=H_0^1(\mathcal{D})$,
and $\mathbb{H}^2(\mathcal{D})=H^2(\mathcal{D})\cap H_0^1(\mathcal{D})$. Then for any
$u\in\mathbb{H}^{2s}$, we have
\[(-\Delta)^s u =\sum^\infty_{k=1}\lambda^s_k (u,\varphi_k)\varphi_k,\]
which is the fractional Laplacian, also adopted in \cite{Nochetto15}.

Then we give some concepts of fractional calculus; for more details, one can refer to \cite{Cartea}, \cite[p. 91]{Kilbas}, and \cite[p. 78]{Podlubny}.
\begin{definition}
The left fractional integral of order $\alpha>0$ for a function $u$ is defined as
\[_0I_t^\alpha u(t)=\frac{1}{\Gamma(\alpha)}\int_0^t(t-s)^{\alpha-1}u(s)ds,\quad t>0,\]
where $\Gamma(\cdot)$ is the Gamma function.
\end{definition}
\begin{definition}\label{def2.2}
The left Caputo fractional derivative of order $\alpha>0$ for a function $u$ is defined as
\[^c_0\partial_t^\alpha u(t)=\frac{1}{\Gamma(n-\alpha)}\int_0^t(t-s)^{n-\alpha-1}\frac{\partial^nu(s)}{\partial s^n}ds,\quad t>0,~~0\leq n-1<\alpha<n,\]
where the function $u(t)$ has absolutely continuous derivatives up to order $n-1$.
\end{definition}
\begin{definition}\label{def2.3}
For $\alpha>0$, $\nu>0$, the left tempered fractional integral of order $\alpha$ for a function $u$ is defined as
\[_0I_t^{\alpha,\nu} u(t):=e^{-\nu t} {_0I}_t^{\alpha}[e^{\nu t}u(t)] =\frac{1}{\Gamma(\alpha)}\int_0^t(t-s)^{\alpha-1}e^{-\nu(t-s)}u(s)ds,\quad t>0.\]
\end{definition}
\begin{definition}\label{def2.4}
For $\alpha>0$, $\nu>0$, the left Caputo tempered fractional derivative of order $\alpha$ for a function $u$ is defined as
\begin{align*}
&^c_0\partial_t^{\alpha,\nu} u(t):=e^{-\nu t} {^c_0\partial}_t^\alpha [e^{\nu t}u(t)]\\
& \relphantom{=}{}=e^{-\nu t}\frac{1}{\Gamma(n-\alpha)}\int_0^t\frac{e^{\nu s}}{(t-s)^{\alpha-n+1}}\left(\frac{\partial}{\partial s}+\nu\right)^nu(s)ds,
\\
& \relphantom{====}{}
\quad t>0,~~0\leq n-1<\alpha<n,\end{align*}
where the function $u(t)$ has absolutely continuous derivatives up to order $n-1$, and
\[\left(\frac{\partial}{\partial s}+\nu\right)^n=\left(\frac{\partial}{\partial s}+\nu\right)
\left(\frac{\partial}{\partial s}+\nu\right)\cdots\left(\frac{\partial}{\partial s}+\nu\right).\]
\end{definition}
If $u$ is an abstract function belonging to $\mathbb{H}^s$ $(s\geq 0)$, then the integrals which appear in the above definitions are taken in Bochner's sense. A measurable function
$u:[0,\infty)\rightarrow \mathbb{H}^s$ is Bochner-integrable if $\|u\|_{\mathbb{H}^s}$ is Lebesgue-integrable.

\subsection{Mittag-Leffler function}\label{subsec:2.1}
Throughout this paper, we shall frequently use the Mittag-Leffler function $E_{\alpha,\beta}(z)$
defined as follows:
\begin{align}\label{eq.2.1}
E_{\alpha,\beta}(z)=\sum\limits_{k=0}^{\infty}\frac{z^k}{\Gamma(k\alpha+\beta
)},\quad z\in \mathbb{C},
\end{align}
%The Mittag-Leffler function $E_{\alpha,\beta}(z)$
which is a two-parameter family of
entire functions in $z$ of order $\alpha^{-1}$ and type 1 \cite[p. 42]{Kilbas}, and generalizes the exponential function in the sense that $E_{1,1}(z)=e^z$.
 For later use, we collect some results in the next lemma; see \cite{Kilbas,Podlubny}.
\begin{lemma}\label{le2.1}
Let $0<\alpha<2$ and $\beta\in\mathbb{R}$ be arbitrary. We suppose that $\mu$
is an arbitrary real number
 such that $\frac{\pi\alpha}{2}<\mu< \min(\pi,\pi\alpha)$.
Then there exists a constant $C=C(\alpha,\beta,\mu)>0$ such that
\begin{equation}\label{eq.2.3}
|E_{\alpha,\beta}(z)|\leq \frac{C}{1+|z|}, \quad\quad \mu\leq|\arg(z)|\leq \pi.
\end{equation}
Moreover, for $\lambda>0$, $\alpha>0$, and positive integer $m\in \mathbb{N}$, we have
\begin{equation}\label{eq.2.4}
\frac{d^m}{d t{^m}} E_{\alpha,1}(-\lambda^\beta t^\alpha)=
-\lambda^\beta t^{\alpha-m}E_{\alpha,\alpha -m+1}(-\lambda^\beta t^\alpha),
\quad t > 0
\end{equation}
and
\begin{equation}\label{eq.2.5}
\frac{d}{d t}\left(tE_{\alpha,2}(-\lambda^\beta t^\alpha)\right)=
E_{\alpha,1}(-\lambda^\beta t^\alpha), \quad t \geq 0.
\end{equation}
\end{lemma}

\subsection{Infinite dimensional white noise and fractional Gaussian noise}\label{subsec:2.3}
Let $(\Omega, \mathcal{F}, \{\mathcal{F}_t \}_{t\geq 0}, \mathbb{P})$ be a
complete filtered probability space satisfying that $\mathcal{F}_0$ contains
all $\mathbb{P}$-null sets of $\mathcal{F}$.
We define $L^2(\Omega; \mathbb{H}^{s}(\mathcal{D}))$ as the separable Hilbert space
of all strongly measurable, square-integrable
 random variables $\omega$, with values in $\mathbb{H}^{s}(\mathcal{D})$
 such that
\[\|\omega\|^2_{L^2(\Omega; \mathbb{H}^{s}(\mathcal{D}))}=\mathbb{E}
\|\omega\|_{\mathbb{H}^{s}}^2, \]
where $\mathbb{E}$ denotes the expectation. In the sequel, $C$ denotes an arbitrary positive constant, which may be different from line to line and even in the same line.
\begin{definition}
The two-sided one-dimensional fBm with Hurst index $H\in(0,1)$ is a Gaussian process $\xi^H=\{\xi^H(t),t\in\mathbb{R}\}$ on $(\Omega,\mathcal{F},\mathbb{P})$, having the properties
\begin{enumerate}
  \item [$(i)$] $\xi^H(0)=0$,
  \item [$(ii)$] $\mathbb{E}\xi^H(t)=0$, ~$t\in \mathbb{R}$,
  \item [$(iii)$] $\mathbb{E}[\xi^H(t)\xi^H(s)]=\frac{1}{2}\left(
      |t|^{2H}+|s|^{2H}-|t-s|^{2H}\right)$, ~$t$, $s\in\mathbb{R}$.
\end{enumerate}
\end{definition}
\begin{remark}
For $H=\frac{1}{2}$, we set $\xi^\frac{1}{2}(t)=\xi(t)$, where $\xi$ is a standard Brownian motion;  in this case the increments of the process are independent. On the contrary, for $H\neq \frac{1}{2}$ the increments are not independent.
\end{remark}
%Let $C^{\lambda}([0,T];L^2(\Omega; \mathbb{H}^{s}(\mathcal{D})))$ be the Banach space of H\"{o}lder continuous functions with exponent $\lambda > 0$ having values in $L^2(\Omega; \mathbb{H}^{s}(\mathcal{D}))$. A norm on this space is given by
%\[\|\omega\|^2_{C^{\lambda}}=\sup_{t\in[0,T]}\mathbb{E}\|\omega(t)\|_{\mathbb{H}^{s}}^2+
%\sup_{0\leq s<t \leq T}\frac{\mathbb{E}\|\omega(t)-\omega(s)\|_{\mathbb{H}^{s}}^2}{|t-s|^{2\lambda}}.\]
%$C([0, T]; L^2(\Omega; \mathbb{H}^{s}(\mathcal{D})))$ denotes the space of continuous functions on $[0, T]$ with values in $L^2(\Omega; \mathbb{H}^{s}(\mathcal{D}))$ with finite supremum norm.

Let $\mathbb{U}$ be a separable Hilbert space endowed with a Hilbert basis
$\{e_k\}_{k\geq1}$. We then consider $\frac{\partial^2{\mathbb{W}}(t,x)}{\partial t \partial x}$ and $\frac{\partial^2{\mathbb{W}^H}(t,x)}{\partial t \partial x}$, respectively, the $\mathbb{U}$-valued white noise and fractional Gaussian noise defined on $(\Omega, \mathcal{F},
\{\mathcal{F}_t \}_{t\geq 0}, {\mathbb{P}})$ such that
\begin{equation}\label{eq1.3}
 \frac{\partial^2{\mathbb{W}}(t,x)}{\partial t \partial x} =\sum\limits_{k=1}^{\infty}\varsigma_k(t)\dot{\xi}_k(t)e_k(x)
\end{equation}
and
\begin{equation}\label{eq1.3*}
 \frac{\partial^2{\mathbb{W}^H}(t,x)}{\partial t \partial x} =\sum\limits_{k=1}^{\infty}\varrho_k(t)\dot{\xi}_k^H(t)e_k(x),
\end{equation}
where $\varsigma_k(t)$ and $\varrho_k(t)$ are  continuous functions, rapidly decaying with the increase of $k$ to ensure
the convergence of the series, $\{\xi_k\}_{k=1}
^{\infty}$ and $\{\xi_k^H\}_{k=1}^{\infty}$, respectively, are the sequences of mutually independent one-dimensional standard Brownian motions and fractional Brownian motions with Hurst index $H\in\left(\frac{1}{2},1\right)$;
$\dot{\xi}_k(t)=\frac{d\xi_k(t)}{dt}$, $k=1,2,\cdots$, is the white noise, the
formal derivative of the Brownian motion $\xi_k(t)$, $k=1,2,\cdots$; and $\dot{\xi}_k^H(t)=\frac{d\xi_k^H(t)}{dt}$, $k=1,2,\cdots$, is the fractional Gaussian noise, the formal derivative of the fractional Brownian motion $\xi_k^H(t)$, $k=1,2,\cdots$.

We define $\mathbb{H}:=L^2(\mathcal{D})$ and denote the space of bounded linear operators from $\mathbb{U}$ to $\mathbb{H}$ by $\mathcal{L}(\mathbb{U},\mathbb{H})$. Let
\[\mathcal{L}_2^0(\mathbb{U}, \mathbb{H}) = \left\{R \in \mathcal{L}(\mathbb{U}, \mathbb{H}) :
\sum_{k=1}^{\infty}\|R \cdot e_k\|^2 < \infty\right\} \]
be the set of Hilbert-Schmidt operators from $\mathbb{U}$ to $\mathbb{H}$, and endow this set with the inner product $(R, S)_{\mathcal{L}_2^0}=\sum_{k}(Re_k, Se_k)$,
so that $\mathcal{L}_2^0(\mathbb{U}, \mathbb{H})$ can be considered as a Hilbert space with the norm $\|R\|_{\mathcal{L}_2^0}=\left(\sum\limits_{k=1}^{\infty}
\|R\cdot e_k\|^2\right)^{\frac{1}{2}}$.

The following notations will be used throughout the paper.
\begin{remark}\label{rem2.2}
For $g, h\in L^2(0,T;\mathcal{L}_2^0(\mathbb{U},\mathbb{H}))$,
$\frac{\partial^2{\mathbb{W}}(t,x)}{\partial t \partial x} =\sum\limits_{k=1}^{\infty}\varsigma_k(t)\dot{\xi}_k(t)e_k(x)$, and $\frac{\partial^2{\mathbb{W}^H}(t,x)}{\partial t \partial x} =\sum\limits_{k=1}^{\infty}\varrho_k(t)\dot{\xi}_k^H(t)e_k(x)$, we introduce the
notations $g(t,u)\frac{\partial^2{\mathbb{W}}(t,x)}{\partial t \partial x}$ and $h(t,u)\frac{\partial^2{\mathbb{W}^H}(t,x)}{\partial t \partial x}$ to formally represent 
\begin{equation}\label{eq.2.8}
\begin{split}
g(t,u)\frac{\partial^2{\mathbb{W}}(t,x)}{\partial t \partial x} &=
\sum\limits_{k=1}^{\infty}g(t,u)\cdot e_k\varsigma_k(t)\dot{\xi}_k(t)
\\
&=\sum\limits_{j,k=1}^{\infty}(g(t,u)\cdot e_k,\varphi_j)\varphi_j\varsigma_k(t)\dot{\xi}_k(t)\\
&=\sum\limits_{j,k=1}^{\infty}g^{j,k}(t)\varphi_j\varsigma_k(t)\dot{\xi}_k(t)
\end{split}\end{equation}
and
\begin{equation}\label{eq.2.9}
\begin{split}
h(t,u)\frac{\partial^2{\mathbb{W}^H}(t,x)}{\partial t \partial x} &=\sum\limits_{k=1}^{\infty}h(t,u)\cdot e_k\varrho_k(t)\dot{\xi}_k^H(t)
\\
&=\sum\limits_{j,k=1}^{\infty}(h(t,u)\cdot e_k,\varphi_j)\varphi_j\varrho_k(t)\dot{\xi}_k^H(t)\\
&=\sum\limits_{j,k=1}^{\infty}h^{j,k}(t)\varphi_j\varrho_k(t)\dot{\xi}_k^H(t),
\end{split}\end{equation}
where we have used the decomposition
\[g(t,u)\cdot e_k=\sum_{j=1}^\infty g^{j,k}(t)\varphi_j,\quad g^{j,k}(t):=(g(t,u)\cdot e_k,\varphi_j),\]
and
\[h(t,u)\cdot e_k=\sum_{j=1}^\infty h^{j,k}(t)\varphi_j,\quad h^{j,k}(t):=(h(t,u)\cdot e_k,\varphi_j),\]
which make sense since, from assumptions, $g(t,u)\cdot e_k$ and $h(t,u)\cdot e_k$ belong to $\mathbb{H}$ and $\{\varphi_j\}_{j\in\mathbb{N}^+}$ is a Hilbert basis of $\mathbb{H}$.
\end{remark}
The following proposition plays an important role in the proof of this paper (see, for instance, \cite{Kloeden,Mishura}).
\begin{proposition}\label{pro1}  For $H > \frac{1}{2}$ and $f, g \in L^2
(\Omega\times[0,T];\mathbb{R})$, we have
\begin{equation*}
\mathbb{E}\int^T_0f(s)d\xi^H(s)=0,
\end{equation*}
%and
\begin{equation*}
\begin{split}
& \mathbb{E}\left[\int^T_0f(s)d\xi^H(s)\int^T_0g(s)d\xi^H(s)\right]
\\
& =H(2H-1)
\int^T_0\int^T_0\mathbb{E}[f(s)g(r)]|s-r|^{2H-2}drds,
\end{split}
\end{equation*}
and
\begin{equation*}
\mathbb{E}\left[\int^T_0f(s)d\xi(s)\int^T_0g(s)d\xi(s)\right]=\int^T_0\mathbb{E}[f(s)g(s)]ds.
\end{equation*}
\end{proposition}

As a simple consequence of Proposition \ref{pro1}, we have
\begin{lemma}\label{le2.10}
Let $H>\frac{1}{2}$ and $\phi\in L^2(\Omega\times[0,T];\mathbb{R})$. Then for any $t_1$, $t_2\in[0,T]$ with $t_2>t_1$,
\[\mathbb{E}\bigg{|}\int_{t_1}^{t_2}\phi(s)d\xi^H(s)\bigg{|}^2\leq 2H(t_2-t_1)^{2H-1}
\int_{t_1}^{t_2}\mathbb{E}|\phi(s)|^2ds.\]
\end{lemma}
\begin{proof}
By Proposition \ref{pro1}, we find that
\begin{align*}
\mathbb{E}\bigg{|}\int_{t_1}^{t_2}\phi(s)d\xi^H(s)\bigg{|}^2
&\leq H(2H-1)\int_{t_1}^{t_2}
\int_{t_1}^{t_2}\mathbb{E}(|\phi(s)||\phi(r)|)|s-r|^{2H-2}drds\\
&\leq H(2H-1)\int_{t_1}^{t_2}
\int_{t_1}^{t_2}\mathbb{E}(|\phi(s)|^2)|s-r|^{2H-2}drds\\
&\leq H(2H-1)\int_{t_1}^{t_2}
\mathbb{E}(|\phi(s)|^2)\left(\int_{t_1}^{s}(s-r)^{2H-2}dr \right.
\\
&
~~~~
+
\left.\int^{t_2}_{s}(r-s)^{2H-2}dr\right)ds\\
&\leq 2H(t_2-t_1)^{2H-1}\int_{t_1}^{t_2}
\mathbb{E}|\phi(s)|^2ds,
\end{align*}
which completes the proof.
\end{proof}
We are now in a position to recall the Gr\"onwall-Bellman inequalities, which will be used later.
\begin{lemma}\label{lem4} Let $u(t)$ and $n(t)$ be real valued continuous functions
for $t \geq 0$, and $n(t)$ nonnegative for $t \geq 0$.
If $m \geq 0$ is a constant, $u(t)$ is nonnegative and satisfies the integral inequality
\begin{equation*}
\displaystyle u(t) \leq m +\int_{0}^tn(s)u(s)ds, ~~ t \geq 0,
\end{equation*}
then
\begin{equation*}
u(t) \leq m \exp{\bigg(\int_{0}^tn(s)ds\bigg)} {\rm~~for~~} t \geq 0.
\end{equation*}
If the negative part of the real valued function $m(t)$ is integrable on every closed and
bounded subinterval of $[0, \infty)$ and $u(t)$ satisfies the integral inequality
\begin{equation*}
u(t) \leq m(t) +\int_{0}^tn(s)u(s)ds, ~~t \geq 0,
\end{equation*}
then
\begin{equation*}
u(t) \leq m(t) +\int_{0}^tm(s)n(s) \exp{\bigg(\int_{s}^tn(\tau)d\tau\bigg)ds}
 {\rm~~for~~} t \geq 0.
\end{equation*}
If, in addition, the function $m(t)$ is nondecreasing, then
\begin{equation*}
u(t) \leq m(t)\exp{\bigg(\int_{0}^tn(s)ds\bigg)} {\rm~~for~~} t \geq 0.
\end{equation*}
\end{lemma}
As for the generalization of Gr\"onwall's lemma for singular kernels \cite{Henry}, there is
\begin{lemma}\label{lem2.11}
%Let $u:[0,T]\rightarrow[0,\infty)$ be a real function and $w(\cdot)$ is a nonnegative, locally integrable function on $[0,T]$ and the constants $a>0$ and $0<b<1$ such that
Suppose $b \geq 0$, $\gamma > 0$, and $a(t)$ is a nonnegative function locally integrable on $0 \leq t<T$ (some $T \leq +\infty$), and let $u(t)$ be nonnegative and locally integrable on $0 \leq t < T$ with
\[u(t) \leq a(t) + b \int_0^t(t - s)^{\gamma-1}u(s)ds\]
on this interval. Then
\begin{align*}
u(t) \leq a(t) + \int_0^t\left[\sum_{n=1}^{\infty}
\frac{(b\Gamma(\gamma))^n}{\Gamma(n\gamma)} (t - s)^{n\gamma-1}a(s)\right]ds, \qquad 0 \leq t < T.
\end{align*}
\end{lemma}

%\begin{lemma}
%Let $1\leq \delta<\infty$ and $0\leq s<t<\infty$. Then $t^{\delta}-s^{\delta}\leq \frac{1}{2-2^{\delta-1}}(t-s)^{\delta}$.
%\end{lemma}
%\begin{proof}
%For $1\leq \delta<\infty$, we have
%\[t^\delta=(t-s+s)^{\delta}\leq 2^{\delta-1}((t-s)^{\delta}+s^{\delta}),\]
%and
%\[s^{\delta}=(t-(t-s))^{\delta}\leq t^{\delta}-(t-s)^{\delta}.\]
%Then we have
%\[t^{\delta}\leq 2^{\delta-1}((t-s)^{\delta}+t^{\delta}-(t-s)^{\delta})
%=(t)\]
%\end{proof}

\section{Regularity of the solution}
\label{sec:3}

Now, we discuss the existence, uniqueness, and regularity of mild solutions to \eqref{eq0.1}.
Let $v(t,x)=e^{\nu t}u(t,x)$. Then by Definitions \ref{def2.3}-\ref{def2.4}, we rewrite \eqref{eq0.1}  as
\begin{equation} \label{eq1.4'}
\left\{ \begin{array}
 {l@{\quad} l}
\displaystyle _0^c\partial_t^{\alpha}v+(-\Delta)^{\beta} v=e^{\nu t}\left[f(t,u)
+g(t,u)\frac{\partial^2\mathbb{W}(t,x)}{\partial t\partial x}
+h(t,u)\frac{\partial^2\mathbb{W}^H(t,x)}{\partial t\partial x}\right]
\\
 ~~~~~~~~~~~~~~~~~~~~~~~~
\rm{in~} (0,T]\times\mathcal{D},\\
\\
 v(t,x)=0 \quad \rm{on~} (0,T]\times\partial\mathcal{D},\\
\\
v(0,x)=a(x),~ \partial_tv(0,x)=\nu a(x)+b(x)\quad \rm{in~} \mathcal{D},
\end{array}\right.
\end{equation}
where $_0^c\partial_t^{\alpha}$ denotes the left-sided Caputo fractional derivative of order $\alpha$ with respect to $t$. We assume that $a\in L^2(\Omega;\mathbb{H}^{2\widetilde{\gamma}})$ and $b\in L^2(\Omega;\mathbb{H}^{2\widetilde{\gamma}
-\frac{2\beta}{\alpha}})$ with $\widetilde{\gamma}=\max\{\gamma,\frac{2\beta}{\alpha}\}$ for some regularity parameter $\gamma>0$.

In order to ensure the existence and uniqueness of problem \eqref{eq0.1}, we list the following conditions:
\begin{itemize}
\item [$(\textbf{A}_1)$] There exists a positive constant $l$ such that the functions
$f:\mathbb{R}\times \mathbb{H}\rightarrow \mathbb{H}$, $g:\mathbb{R}\times \mathbb{H}\rightarrow \mathcal{L}_2^0(\mathbb{U},\mathbb{H})$, and $h:\mathbb{R}\times \mathbb{H}\rightarrow \mathcal{L}_2^0(\mathbb{U},\mathbb{H})$ satisfy
\begin{align*}
&\|f(t_1,u_1)-f(t_2,u_2)\|+\|g(t_1,u_1)-g(t_2,u_2)\|_{\mathcal{L}_2^0}
+\|h(t_1,u_1)-h(t_2,u_2)\|_{\mathcal{L}_2^0}\\
&\relphantom{=}{}\leq l(|t_1-t_2|+\|u_1-u_2\|)
 \quad \mbox{if~~} 0\leq \gamma\leq \frac{\beta}{\alpha},
\end{align*}
for all $t_1$, $t_2\in\mathbb{R}$ and $u_1$, $u_2\in \mathbb{H}$;
\begin{align*}
&\|(-\Delta)^{\gamma-\frac{\beta}{\alpha}}(f(t_1,u_1)-f(t_2,u_2))\|
+\|(-\Delta)^{\gamma-\frac{\beta}{\alpha}}(g(t_1,u_1)-g(t_2,u_2))\|_{\mathcal{L}_2^0}\\
&%\relphantom{=}{}
+\|(-\Delta)^{\gamma-\frac{\beta}{\alpha}}(h(t_1,u_1)-h(t_2,u_2))\|_{\mathcal{L}_2^0}
\\
&
\leq l(|t_1-t_2|+\|(-\Delta)^{\gamma-\frac{\beta}{\alpha}}(u_1-u_2)\|)
\quad \mbox{if}~~\gamma>\frac{\beta}{\alpha},
\end{align*}
for all $t_1$, $t_2\in \mathbb{R}$ and $u_1$, $u_2\in \mathbb{H}^{2\gamma-\frac{2\beta}{\alpha}}$;
\begin{align*}&\|f(t,u)\|
+\|g(t,u)\|_{\mathcal{L}_2^0}
+\|h(t,u)\|_{\mathcal{L}_2^0}\leq l(1+\|u\|) \quad \mbox{if~~} 0\leq \gamma\leq \frac{\beta}{\alpha},
\end{align*}
for all $t\in\mathbb{R}$ and $u\in\mathbb{H}$;
\begin{align*}&\|(-\Delta)^{\gamma-\frac{\beta}{\alpha}}f(t,u)\|
+\|(-\Delta)^{\gamma-\frac{\beta}{\alpha}}g(t,u)\|_{\mathcal{L}_2^0}
+\|(-\Delta)^{\gamma-\frac{\beta}{\alpha}}h(t,u)\|_{\mathcal{L}_2^0}\\
&\relphantom{=}{}\leq l\left(1+\|(-\Delta)^{\gamma-\frac{\beta}{\alpha}}u\|\right)
\quad \mbox{if~~}   \gamma> \frac{\beta}{\alpha},
\end{align*}
for all $t\in\mathbb{R}$ and $u\in\mathbb{H}^{2\gamma-\frac{2\beta}{\alpha}}$.
\item [$(\textbf{A}_2)$] $\{\varsigma_k(t)\}$, $\{\varrho_k(t)\}$, and their derivatives
are uniformly bounded by
\[|\varsigma_k(t)| \leq \mu_k,~ |\varrho_k(t)| \leq \widetilde{\mu}_k,~ |\varsigma^\prime_k(t)|\leq \gamma_k,~
|\varrho^\prime_k(t)|\leq \widetilde{\gamma}_k \quad \forall~
t \in [0, T],\]
and the series $(\{\mu_k\}$, $\{\widetilde{\mu}_k\}$, $\{\gamma_k\}$, $\{\widetilde{\gamma}_k \})$
is rapidly decaying with the increase of $k$.
\end{itemize}
\begin{remark}
It is worth mentioning that if $\alpha=2$ and $\beta=1$, then the condition $(\bf{A}_1)$ will reduce to the corresponding condition (9) in \cite{Anton:2016}.
\end{remark}

First, we give a representation of the mild solution to problem \eqref{eq0.1} using the
Dirichlet eigenpairs $\{(\lambda_k, \varphi_k )\}_{k=1}^{\infty}$.
\begin{lemma}\label{le1} The solution $u$ to problem \eqref{eq0.1} with $\frac{3}{2}<\alpha<2$,
$\frac{1}{2}<\beta\leq 1$, $\nu>0$, and
$\frac{1}{2}< H<1$ is given by
\begin{equation}\label{eq2.1}
\begin{split}
u(t,x)&=\int_{\mathcal{D}}\mathcal{T}^\nu_{\alpha,\beta}(t,x,y)a(y)dy
+\int_{\mathcal{D}}\mathcal{R}^\nu_{\alpha,\beta}(t,x,y)b(y)dy\\
&\relphantom{=}{}+\int_{0}^t\int_{\mathcal{D}}\mathcal{S}^\nu_{\alpha,\beta}(t-s,x,y)f(s,u(s,y))dyds\\
&\relphantom{=}{}+\int_{0}^t\int_{\mathcal{D}}\mathcal{S}^\nu_{\alpha,\beta}(t-s,x,y)g(s,u(s,y))
d{\mathbb{W}}(s,y)\\
&\relphantom{=}{}+\int_{0}^t\int_{\mathcal{D}}\mathcal{S}^\nu_{\alpha,\beta}(t-s,x,y)h(s,u(s,y))
d{\mathbb{W}}^H(s,y).
\end{split}
\end{equation}
Here
\begin{equation}\label{eq2.2'}
\mathcal{T}^\nu_{\alpha,\beta}(t,x,y)
=e^{-\nu t}\sum\limits_{k=1}^{\infty}\left(E_{\alpha,1}(-\lambda_k^\beta t^\alpha)+\nu
t E_{\alpha,2}(-\lambda_k^\beta t^\alpha)\right)\varphi_k(x)\varphi_k(y)
\end{equation}
is the fundamental solution of
\[\left\{
  \begin{array}{ll}
    ^c_0\partial_t^{\alpha,\nu} v(t, x)+ (-\Delta)^\beta v(t, x) = 0 & \hbox{{\rm in}~~$(0,T]\times\mathcal{D}$,} \\
    v(t,x)=0 & \hbox{{\rm on}~~$(0,T]\times\partial\mathcal{D}$,} \\
    v(0,x)=\phi(x),~\partial_tv(t, x)|_{t=0} = 0 & \hbox{{\rm in}~~$\mathcal{D}$,}
  \end{array}
\right.\]
so that $v(t, x) =\int_{\mathcal{D}} \mathcal{T}^\nu_{\alpha,\beta}(t,x,y)\phi(y)dy$, and
\begin{equation}\label{eq2.2''}
\mathcal{R}^\nu_{\alpha,\beta}(t,x,y) =t e^{-\nu t}\sum\limits_{k=1}^{\infty} E_{\alpha,2}(-\lambda_k^\beta t^\alpha)\varphi_k(x)\varphi_k(y)
\end{equation}
is the fundamental solution of
\[\left\{
  \begin{array}{ll}
    ^c_0\partial_t^{\alpha,\nu} v(t, x)+ (-\Delta)^\beta v(t, x) = 0 & \hbox{{\rm in}~~$(0,T]\times\mathcal{D}$,} \\
    v(t,x)=0 & \hbox{{\rm on}~~$(0,T]\times\partial\mathcal{D}$,} \\
    v(0,x)=0,~\partial_tv(0, x) = \psi(x) & \hbox{{\rm in}~~$\mathcal{D}$,}
  \end{array}
\right.\]
so that $v(t, x) =\int_{\mathcal{D}} \mathcal{R}^{\nu}_{\alpha,\beta}(t,x,y)\psi(y)dy$. For  \eqref{eq0.1} but with the initial data $v(0,x)=\partial_tv(0,x) \equiv 0$, we shall use the operator defined by
\begin{equation}\label{eq2.2}
\mathcal{S}^{\nu}_{\alpha,\beta}(t,x,y)=t^{\alpha-1}e^{-\nu t}\sum\limits_{k=1}^{\infty}E_{\alpha,\alpha}(-\lambda_k^\beta t^\alpha)\varphi_k(x) \varphi_k(y)
\end{equation}
and
\begin{align*}
v(t, x) =&\int_{0}^t\int_{\mathcal{D}}\mathcal{S}^\nu_{\alpha,\beta}(t-s,x,y)f(s,u(s,y))dyds
\\
&+\int_{0}^t\int_{\mathcal{D}}\mathcal{S}^\nu_{\alpha,\beta}(t-s,x,y)g(s,u(s,y))d{\mathbb{W}}(s,y)\\
&+\int_{0}^t\int_{\mathcal{D}}\mathcal{S}^\nu_{\alpha,\beta}(t-s,x,y)h(s,u(s,y))d{\mathbb{W}}^H(s,y).
\end{align*}
\end{lemma}
The proof of Lemma \ref{le1} is given in \ref{A}.

Then we have the following stability estimates for the homogeneous problem of \eqref{eq0.1}.
\begin{lemma}\label{le2.2}
Let $u$ be the solution of \eqref{eq0.1} with $f=g=h=0$. Then for all $t>0$,
%\begin{equation*}%\label{eq2.12}
\begin{align*}
&\displaystyle \| u(t)\|_{\mathbb{H}^p} \\
&
\displaystyle \leq  \left\{\begin{array}{l}C(1+\nu t)
e^{-\nu t}t^{-\frac{\alpha (p-q)}{2\beta}}\|a\|_{\mathbb{H}^q}+C e^{-\nu t}
t^{1-\frac{\alpha(p-r)}{2\beta}}\|b\|_{\mathbb{H}^r},\quad 0 \leq q, r\leq p\leq 2\beta,\\
\\
C(1+\nu t)e^{-\nu t}t^{-\alpha}\|a\|_{\mathbb{H}^q}+C e^{-\nu t}t^{1-\alpha}\|b\|_
{\mathbb{H}^r}, \quad\quad\qquad\qquad~~~~~~~~ q, r>p,
\end{array}\right.
\end{align*}
%\end{equation*}
and
\begin{equation*}\begin{split}
\displaystyle \|_0^c\partial_t^{\alpha,\nu} u(t)\|_{\mathbb{H}^p} \leq C(1+\nu t)
e^{-\nu t}t^{-\alpha-\frac{\alpha(p-q)}{2\beta}}\|a\|_{\mathbb{H}^q}+C e^{-\nu t}t^{1-\alpha-\frac{\alpha(p-r)}{2\beta}}\|b\|_{\mathbb{H}^r},
\end{split}\end{equation*}
where $0\leq p \leq q$ and $r\leq p+2\beta$.
\end{lemma}
\begin{proof} Note that $\{\varphi_k\}_{k=1}^{\infty}$ is an orthonormal basis in $L^2(\mathcal{D})$.
By Lemma \ref{le2.1} and \eqref{eq2.1}, we obtain
\begin{flalign*}
&\|u(t)\|_{\mathbb{H}^p}^2
\\
&=\sum_{k=1}^{\infty}\lambda_k^p\left(
\int_{\mathcal{D}}\mathcal{T}_{\alpha,\beta}^{\nu}(t,\cdot,y)a(y)dy+
\int_{\mathcal{D}}\mathcal{R}_{\alpha,\beta}^{\nu}(t,\cdot,y)b(y)dy,\varphi_k\right)^2
&
\end{flalign*}
\begin{flalign*}
&=\sum_{k=1}^{\infty}\lambda_k^{p}\bigg(\int_{\mathcal{D}}
e^{-\nu t}\sum_{l=1}^{\infty}\left(E_{\alpha,1}(-\lambda_l^\beta t^\alpha)+\nu
t E_{\alpha,2}(-\lambda_l^\beta t^\alpha)\right)\varphi_l(\cdot)\varphi_l(y)a(y)dy\\
&\relphantom{=}{}
+\int_{\mathcal{D}}te^{-\nu t} \sum_{l=1}^{\infty}
E_{\alpha,2}(-\lambda_l^\beta t^\alpha)\varphi_l(\cdot)\varphi_l(y)b(y)dy,\varphi_k\bigg)^2
&
\end{flalign*}
\begin{flalign*}
&\leq \sum_{k=1}^{\infty}\lambda_k^{p}\bigg{|}\int_{\mathcal{D}}e^{-\nu t}
\left(E_{\alpha,1}(-\lambda_k^\beta t^\alpha)+\nu
t E_{\alpha,2}(-\lambda_k^\beta t^\alpha)\right)\varphi_k(y)a(y)dy\\
&\relphantom{=}{}+\int_{\mathcal{D}}te^{-\nu t}
E_{\alpha,2}(-\lambda_k^\beta t^\alpha)\varphi_k(y)b(y)dy\bigg{|}^2
&
\end{flalign*}
\begin{flalign*}
&\leq C(1+\nu^2 t^2)e^{-2 \nu t}t^{-\frac{\alpha (p-q)}{\beta}}
\sum_{k=1}^{\infty}\frac{(\lambda_k^{\beta}t^{\alpha})^{\frac{p-q}{\beta}}}{(1+\lambda_k^{\beta}
t^{\alpha})^2}\lambda_k^{q}(a,\varphi_k)^2\\
&\relphantom{=}{}
+Ct^2 e^{-2\nu t}t^{-\frac{\alpha (p-r)}{\beta}}
\sum_{k=1}^{\infty}\frac{(\lambda_k^{\beta}t^{\alpha})^{\frac{p-r}{\beta}}}{(1+\lambda_k^{\beta}
t^{\alpha})^2}\lambda_k^{r}(b,\varphi_k)^2
&
\end{flalign*}
\begin{flalign*}
&\leq C(1+\nu^2 t^2)e^{-2 \nu t}t^{-\frac{\alpha (p-q)}{\beta}}\|a\|_{\mathbb{H}^q}^2+C e^{-2\nu t}t^{2-\frac{\alpha (p-r)}{\beta}}\|b\|_{\mathbb{H}^r}^2, &
\end{flalign*}
where we have used $\frac{(\lambda_k^{\beta}t^{\alpha})
^{\frac{p-q}{\beta}}}{(1+\lambda_k^\beta t^\alpha)^2} \leq C$ and $\frac{(\lambda_k^{\beta}t^{\alpha})
^{\frac{p-r}{\beta}}}{(1+\lambda_k^\beta t^\alpha)^2} \leq C$
for $0\leq q, r \leq p \leq  2\beta$.

Now we consider the case $q,~r>p$. Since $0<\lambda_1 \leq \lambda_2 \leq \lambda_3
\leq \cdots$ and $\lambda_k \rightarrow \infty$ as $k\rightarrow \infty$, we
obtain from Lemma \ref{le2.1} and \eqref{eq2.1} that
\begin{align*}
\|u(t)\|_{\mathbb{H}^p}^2&=\sum_{k=1}^{\infty}\lambda_k^p\left(
\int_{\mathcal{D}}\mathcal{T}_{\alpha,\beta}^{\nu}(t,\cdot,y)a(y)dy+
\int_{\mathcal{D}}\mathcal{R}_{\alpha,\beta}^{\nu}(t,\cdot,y)b(y)dy,\varphi_k\right)^2\\
&\leq C(1+\nu^2 t^2)e^{-2\nu t} \sum\limits_{k=1}^{\infty}
\frac{1}{\lambda_k^{q-p}(1+\lambda_k^\beta t^\alpha)^2}
\lambda_k^{q}(a, \varphi_k)^2\\
&\relphantom{=}{}+C t^2e^{-2\nu t} \sum\limits_{k=1}^{\infty}
\frac{1}{\lambda_k^{r-p}(1+\lambda_k^\beta t^\alpha)^2}
\lambda_k^{r}(b, \varphi_k)^2\\
&\leq C(1+\nu^2 t^2)e^{-2 \nu t}t^{-2\alpha} \|a\|_{\mathbb{H}^q}^2+C e^{-2\nu t}t^{2-2\alpha} \|b\|_{\mathbb{H}^r}^2 .
\end{align*}
%Thus, the conclusion follows immediately by the triangle inequality.

On the other hand, it follows from Lemma \ref{le2.1} and \eqref{eq2.1} that
\begin{flalign*}
&\|_0^c\partial_t^{\alpha,\nu}u(t)\|_{\mathbb{H}^p}^2
=\left\|-(-\Delta)^{\beta}u(t)\right\|_{\mathbb{H}^p}^2
=\|u(t)\|_{\mathbb{H}^{p+2\beta}}^2
&
\end{flalign*}
\begin{flalign*}
&\relphantom{=}{}=\sum_{k=1}^{\infty}\lambda_k^{p+2\beta}
\left(\int_{\mathcal{D}}\mathcal{T}^{\nu}_{\alpha,\beta}(t,\cdot,y)a(y)dy
+\int_{\mathcal{D}}\mathcal{R}^{\nu}_{\alpha,\beta}(t,\cdot,y)b(y)dy,\varphi_k
\right)^2 &
\end{flalign*}
\begin{flalign*}
&\relphantom{=}{}=\sum_{k=1}^{\infty}\lambda_k^{p+2\beta}
\bigg(\int_{\mathcal{D}}e^{-\nu t}\sum_{l=1}^{\infty}\left(E_{\alpha,1}(-\lambda_l^\beta t^{\alpha})+\nu tE_{\alpha,2}(-\lambda_l^\beta t^{\alpha})\right)\varphi_l(\cdot)\varphi_l(y) a(y)dy\\
&\relphantom{==}{}+\int_{\mathcal{D}}te^{-\nu t}\sum_{l=1}^{\infty}E_{\alpha,2}(-\lambda_l^\beta t^{\alpha})\varphi_l(\cdot)\varphi_l(y) b(y)dy,\varphi_k\bigg)^2
&
\end{flalign*}
\begin{flalign*}
&\relphantom{=}{}\leq\sum_{k=1}^{\infty}\lambda_k^{p+2\beta}
\bigg|\int_{\mathcal{D}}e^{-\nu t}\left(E_{\alpha,1}(-\lambda_k^\beta t^{\alpha})+\nu tE_{\alpha,2}(-\lambda_k^\beta t^{\alpha})\right)\varphi_k(y) a(y)dy\\
&\relphantom{==}{}+\int_{\mathcal{D}}te^{-\nu t}E_{\alpha,2}(-\lambda_k^\beta t^{\alpha})\varphi_k(y) b(y)dy\bigg|^2
&
\end{flalign*}
\begin{flalign*}
&\relphantom{=}{}\leq C(1+\nu^2 t^2) e^{-2\nu t}t^{-2\alpha-\frac{\alpha(p-q)}{\beta}}
\sum_{k=1}^{\infty}\frac{(\lambda_k^\beta t^{\alpha})^{\frac{2\beta+p-q}{\beta}}}{(1+\lambda_k^\beta t^{\alpha})^2}\lambda_k^{q}(a,\varphi_k)^2\\
&\relphantom{==}{}+Ct^2e^{-2\nu t} t^{-2\alpha-\frac{\alpha(p-r)}{\beta}} \sum_{k=1}^{\infty}\frac{(\lambda_k^\beta t^{\alpha})^{\frac{2\beta+p-r}{\beta}}}{(1+\lambda_k^\beta t^{\alpha})^2}\lambda_k^{r}(b,\varphi_k)^2
&
\end{flalign*}
\begin{flalign*}
&\relphantom{=}{}\leq C(1+\nu^2 t^2) e^{-2\nu t}t^{-2\alpha-\frac{\alpha(p-q)}{\beta}}\|a\|_{\mathbb{H}^q}^2
+C e^{-2 \nu t}t^{2-2\alpha-\frac{\alpha(p-r)}{\beta}}\|b\|_{\mathbb{H}^r}^2, &
\end{flalign*}
where we have used $\frac{(\lambda_k^\beta t^{\alpha})^{\frac{2\beta+p-q}{\beta}}}{(1+\lambda_k^\beta t^{\alpha})^2}\leq C$ and $\frac{(\lambda_k^\beta t^{\alpha})^{\frac{2\beta+p-r}{\beta}}}{(1+\lambda_k^\beta t^{\alpha})^2}\leq C$ for $0\leq p\leq q, r\leq p+2\beta$.
The proof has been completed.
\end{proof}
\begin{remark}\label{re3.4}
It follows from the proof of Lemma \ref{le2.2} that
\begin{align*}
\|u(t)\|_{\mathbb{H}^p}\leq C(1+\nu t)e^{-\nu t}t^{-\frac{\alpha(p-q)}{2\beta}}\|a\|_{\mathbb{H}^q}
+Ce^{-\nu t}
t^{1-\frac{\alpha(p-r)}{2\beta}}\|b\|_{\mathbb{H}^r}
\end{align*}
also holds true for all $t>0$, $0\leq q, r\leq p$ and $p-q$, $p-r\leq 2\beta$.
\end{remark}
The following theorem shows the existence, uniqueness, decay, and regularity properties of mild solutions.
\begin{theorem}\label{thm3.5}
Let $(a,b)$ be $\mathcal{F}_0$-adapted random variable and  $\|a\|_{L^2(\Omega;\mathbb{H}^{2\widetilde{\gamma}}({\mathcal{D}}))}$ $+$
$\|b\|_{L^2(\Omega;\mathbb{H}^{2\widetilde{\gamma}-\frac{2\beta}{\alpha}}({\mathcal{D}}))}$
$<\infty$ with $\widetilde{\gamma}=\max\{\gamma,\frac{\beta}{\alpha}\}$, and that the functions $f$, $g$ and $h$ satisfy $(\bf{A}_1)$ for some $\gamma\geq 0$. Let $(\bf{A}_2)$ holds, $\frac{3}{2}<\alpha<2$, $\frac{1}{2}<\beta\leq 1$, $\nu>0$, and $\frac{1}{2}<H<1$. Then problem \eqref{eq0.1} has a unique mild solution given by
\begin{equation*}
\begin{split}
u(t,x)&=\int_{\mathcal{D}}\mathcal{T}^\nu_{\alpha,\beta}(t,x,y)a(y)dy
+\int_{\mathcal{D}}\mathcal{R}^\nu_{\alpha,\beta}(t,x,y)b(y)dy\\
&\relphantom{=}{}+\int_{0}^t\int_{\mathcal{D}}\mathcal{S}^\nu_{\alpha,\beta}(t-s,x,y)f(s,u(s,y))dyds\\
&\relphantom{=}{}+\int_{0}^t\int_{\mathcal{D}}\mathcal{S}^\nu_{\alpha,\beta}(t-s,x,y)g(s,u(s,y))
d{\mathbb{W}}(s,y)\\
&\relphantom{=}{}+\int_{0}^t\int_{\mathcal{D}}\mathcal{S}^\nu_{\alpha,\beta}(t-s,x,y)h(s,u(s,y))
d{\mathbb{W}}^H(s,y) \quad \mathbb{P}\mbox{-a.s.}
\end{split}
\end{equation*}
for each $t\in[0,T]$, $u\in C([0,T];L^2(\Omega;\mathbb{H}^{2\widetilde{\gamma}}(\mathcal{D}))$.
Moreover, if $\gamma>\frac{\beta}{\alpha}$, then for any $\delta\in \left(0,2\beta-\frac{3\beta}{\alpha}\right)$,
\begin{align*}
&t^{\frac{\alpha\delta}{2\beta}}u\in C([0,T];L^2(\Omega;\mathbb{H}^{2\widetilde{\gamma}+\delta}({\mathcal{D}})))
\mbox{~~with~value~zero~at~~} t=0,
\end{align*}
and for $0\leq \theta_1\leq \theta_2\leq T$,
\begin{align*}
& \mathbb{E}\|u(\theta_2)-u(\theta_1)\|^2
\\
& \leq C|\theta_2-\theta_1|^{2\alpha-2}\left(\mathbb{E}\|a\|^2_{\mathbb{H}^{2\widetilde{\gamma}}}
+\mathbb{E}\|b\|^2_{\mathbb{H}^{2\widetilde{\gamma}-\frac{2\beta}{\alpha}}}
+\sup_{s\in[0,T]}\mathbb{E}\left(1+\|u(s)\|^2_{\mathbb{H}^{2\gamma}}\right)\right).
\end{align*}
\end{theorem}
\begin{proof}
We give the proof for the case of $\gamma>\frac{\beta}{\alpha}$; it can be similarily proved for  the other case of $0\leq \gamma\leq \frac{\beta}{\alpha}$.

We define a nonlinear mapping $\mathcal{P}$ as follows
\begin{align}\label{eq.3.7*}
\mathcal{P}u(t)&=\int_{\mathcal{D}}\mathcal{T}^\nu_{\alpha,\beta}(t,x,y)a(y)dy
+\int_{\mathcal{D}}\mathcal{R}^\nu_{\alpha,\beta}(t,x,y)b(y)dy\nonumber\\
&\relphantom{=}{}+\int_{0}^t\int_{\mathcal{D}}\mathcal{S}^\nu_{\alpha,\beta}(t-s,x,y)f(s,u(s,y))dyds\nonumber\\
&\relphantom{=}{}+\int_{0}^t\int_{\mathcal{D}}\mathcal{S}^\nu_{\alpha,\beta}(t-s,x,y)g(s,u(s,y))
d{\mathbb{W}}(s,y)\nonumber\\
&\relphantom{=}{}+\int_{0}^t\int_{\mathcal{D}}\mathcal{S}^\nu_{\alpha,\beta}(t-s,x,y)h(s,u(s,y))
d{\mathbb{W}}^H(s,y),
\end{align}
and denote by $C([0,T];L^2(\Omega;\mathbb{H}^{2\gamma}(\mathcal{D})))_{\theta}$ the space $C([0,T];L^2(\Omega;\mathbb{H}^{2\gamma}(\mathcal{D})))$ equipped with the following weighted norm: $\|u\|_{\theta}^2:=\max_{0\leq t\leq T}\mathbb{E}\|e^{-\theta t}u(t)\|^2_{\mathbb{H}^{2\gamma}}$ $\forall u\in C([0,T];L^2(\Omega;\mathbb{H}^{2\gamma}(\mathcal{D})))_{\theta}$, which is equivalent to the standard norm of $C([0,T];L^2(\Omega;\mathbb{H}^{2\gamma}(\mathcal{D})))$ for any fixed parameter $\theta>0$. Then we show that the map $\mathcal{P}:$ $C([0,T];L^2(\Omega;\mathbb{H}^{2\gamma}(\mathcal{D})))_{\theta}$ $\rightarrow$ $C([0,T];L^2(\Omega;\mathbb{H}^{2\gamma}(\mathcal{D})))_{\theta}$ has a unique fixed point.

Step 1. $\mathcal{P}$ maps $C([0,T];L^2(\Omega;\mathbb{H}^{2\gamma}(\mathcal{D})))_{\theta}$ into itself.

Let us consider $\sigma>0$ small enough. Then we have
\begin{flalign}\label{eq.3.8}
&\mathbb{E}\|e^{-\theta(t+\sigma)}\mathcal{P}u(t+\sigma)-e^{-\theta t}\mathcal{P}u(t)\|^2_{\mathbb{H}^{2\gamma}}\nonumber\\
&\leq C\mathbb{E}\bigg{\|}e^{-\theta(t+\sigma)}\int_{\mathcal{D}}
\mathcal{T}^\nu_{\alpha,\beta}(t+\sigma,\cdot,y)a(y)dy-e^{-\theta t}\int_{\mathcal{D}}
\mathcal{T}^\nu_{\alpha,\beta}(t,\cdot,y)a(y)dy\bigg{\|}_{\mathbb{H}^{2\gamma}}^2\nonumber\\
&\relphantom{=}{}+C\mathbb{E}\bigg{\|}e^{-\theta(t+\sigma)}\int_{\mathcal{D}}
\mathcal{R}^\nu_{\alpha,\beta}(t+\sigma,\cdot,y)b(y)dy-e^{-\theta t}\int_{\mathcal{D}}
\mathcal{R}^\nu_{\alpha,\beta}(t,\cdot,y)b(y)dy\bigg{\|}_{\mathbb{H}^{2\gamma}}^2\nonumber\\
&\relphantom{=}{}+C\mathbb{E}\bigg{\|}e^{-\theta(t+\sigma)}\int_0^{t+\sigma}\int_{\mathcal{D}}
\mathcal{S}^\nu_{\alpha,\beta}(t+\sigma-s,\cdot,y)f(s,u(s,y))dyds\nonumber\\
&\relphantom{==}{}-e^{-\theta t}\int_0^t\int_{\mathcal{D}}
\mathcal{S}^\nu_{\alpha,\beta}(t-s,\cdot,y)f(s,u(s,y))dyds\bigg{\|}_{\mathbb{H}^{2\gamma}}^2\nonumber\\
&\relphantom{=}{}+C\mathbb{E}\bigg{\|}e^{-\theta(t+\sigma)}\int_0^{t+\sigma}\int_{\mathcal{D}}
\mathcal{S}^\nu_{\alpha,\beta}(t+\sigma-s,\cdot,y)g(s,u(s,y))d\mathbb{W}(s,y)\nonumber\\
&\relphantom{==}{}-e^{-\theta t}\int_0^t\int_{\mathcal{D}}
\mathcal{S}^\nu_{\alpha,\beta}(t-s,\cdot,y)g(s,u(s,y))d\mathbb{W}(s,y)\bigg{\|}_{\mathbb{H}^{2\gamma}}^2\nonumber\\
&\relphantom{=}{}+C\mathbb{E}\bigg{\|}e^{-\theta(t+\sigma)}\int_0^{t+\sigma}\int_{\mathcal{D}}
\mathcal{S}^\nu_{\alpha,\beta}(t+\sigma-s,\cdot,y)h(s,u(s,y))d\mathbb{W}^H(s,y)\nonumber\\
&\relphantom{==}{}-e^{-\theta t}\int_0^t\int_{\mathcal{D}}
\mathcal{S}^\nu_{\alpha,\beta}(t-s,\cdot,y)h(s,u(s,y))d\mathbb{W}^H(s,y)\bigg{\|}_{\mathbb{H}^{2\gamma}}^2\nonumber\\
&\relphantom{}{}:=I_1+I_2+I_3+I_4+I_5.
\end{flalign}
Noticing that $\{\varphi_k\}_{k=1}^{\infty}$ is an orthonormal basis in $L^2(\mathcal{D})$, by Lemma \ref{le2.1} and Lagrange's mean value theorem, we get
\begin{flalign}\label{eq.3.9}
I_1&=C\mathbb{E}\sum_{k=1}^{\infty}\lambda_k^{2\gamma}\bigg(\int_{\mathcal{D}}
e^{-(\theta+\nu)(t+\sigma)}\sum_{l=1}^{\infty}\left(E_{\alpha,1}(-\lambda_l^\beta
(t+\sigma)^{\alpha}) \right.
\nonumber\\
&
\left. \relphantom{=}{}+\nu(t+\sigma)E_{\alpha,2}(-\lambda_l^\beta(t+\sigma)^{\alpha})\right)
\varphi_l(\cdot)\varphi_l(y)\nonumber\\
&\relphantom{=}{}\cdot a(y)dy-\int_{\mathcal{D}}e^{-(\theta+\nu)t}
\sum_{l=1}^{\infty}\left(E_{\alpha,1}(-\lambda_l^\beta
t^{\alpha})
\right.
\nonumber\\
&
\left. \relphantom{=}{}
+\nu t E_{\alpha,2}(-\lambda_l^\beta t^{\alpha})\right)\varphi_l(\cdot)\varphi_l(y)a(y)dy,\varphi_k\bigg)^2  &
\end{flalign}
\begin{flalign*}
&=C\mathbb{E}\sum_{k=1}^{\infty}\lambda_k^{2\gamma}\bigg|\int_{\mathcal{D}}
e^{-(\theta+\nu)(t+\sigma)}\left(E_{\alpha,1}(-\lambda_k^\beta
(t+\sigma)^{\alpha})
\right.
\nonumber\\
&
\left. \relphantom{=}{}
+\nu(t+\sigma)E_{\alpha,2}(-\lambda_k^\beta(t+\sigma)^{\alpha})\right)
\varphi_k(y)a(y)dy
\nonumber\\
&\relphantom{=}{}-\int_{\mathcal{D}}e^{-(\theta+\nu)t}
\left(E_{\alpha,1}(-\lambda_k^\beta
t^{\alpha})+\nu t E_{\alpha,2}(-\lambda_k^\beta t^{\alpha})\right)\varphi_k(y)a(y)dy\bigg|^2\nonumber
&
\end{flalign*}
\begin{flalign*}
&\leq C\mathbb{E}\sum_{k=1}^{\infty}\lambda_k^{2\gamma}\bigg|\int_{\mathcal{D}}
\left(e^{-(\theta+\nu)(t+\sigma)}E_{\alpha,1}(-\lambda_k^\beta
(t+\sigma)^{\alpha})-e^{-(\theta+\nu)t}E_{\alpha,1}(-\lambda_k^\beta
t^{\alpha})\right)
%\right.
\nonumber\\
&
%\left.
\relphantom{=}{}
\cdot\varphi_k(y)a(y)dy\bigg|^2
%\nonumber\\
%&\relphantom{=}{}
+C\nu^2\mathbb{E}\sum_{k=1}^{\infty}\lambda_k^{2\gamma}\bigg|\int_{\mathcal{D}}
\left(e^{-(\theta+\nu)(t+\sigma)}(t+\sigma)
\right.
\nonumber\\
&
\left.
\relphantom{=}{}
\cdot E_{\alpha,2}(-\lambda_k^\beta(t+\sigma)^{\alpha})
-e^{-(\theta+\nu)t}t E_{\alpha,2}(-\lambda_k^\beta t^{\alpha})\right)\varphi_k(y)a(y)dy\bigg|^2\nonumber\\
&
\end{flalign*}
\begin{flalign*}
&\leq C\left|e^{-(\theta+\nu)(t+\sigma)}-e^{-(\theta+\nu)t}\right|^2\mathbb{E}\sum_{k=1}^{\infty}
|E_{\alpha,1}(-\lambda_k^\beta(t+\sigma)^{\alpha})|^2\lambda_k^{2\gamma}(a,\varphi_k)^2\nonumber\\
&\relphantom{=}{}+Ce^{-2(\theta+\nu)t}\mathbb{E}\sum_{k=1}^{\infty}\left|E_{\alpha,1}(-\lambda_k^\beta
(t+\sigma)^{\alpha})-E_{\alpha,1}(-\lambda_k^\beta
t^{\alpha})\right|^2\lambda_k^{2\gamma}(a,\varphi_k)^2\nonumber\\
&\relphantom{=}{}
+C\nu^2\left|e^{-(\theta+\nu)(t+\sigma)}-e^{-(\theta+\nu)t}\right|^2\mathbb{E}\sum_{k=1}^{\infty}
|(t+\sigma)E_{\alpha,2}(-\lambda_k^\beta(t+\sigma)^{\alpha})|^2
\nonumber\\
&\relphantom{=}{}
\cdot\lambda_k^{2\gamma}(a,\varphi_k)^2
+C\nu^2e^{-2(\theta+\nu)t}\mathbb{E}\sum_{k=1}^{\infty}\left|(t+\sigma)E_{\alpha,2}(-\lambda_k^\beta(t+\sigma)^{\alpha})
\right.
\nonumber\\
&
\left.
\relphantom{=}{}
-tE_{\alpha,2}(-\lambda_k^\beta t^{\alpha})\right|^2\lambda_k^{2\gamma}(a,\varphi_k)^2\nonumber
&
\end{flalign*}
\begin{flalign*}
&\leq C\left|e^{-(\theta+\nu)(t+\sigma)}-e^{-(\theta+\nu)t}\right|^2\mathbb{E}\|a\|_{\mathbb{H}^{2\gamma}}^2
\nonumber\\
&\relphantom{=}{}
+C\sigma^2\mathbb{E}\sum_{k=1}^{\infty}\left|-\lambda_k^{\beta}\widehat{\theta}^{\alpha-1}E_{\alpha,\alpha}(-\lambda_k^\beta\widehat{\theta}^{\alpha})\right|^2
\lambda_k^{2\gamma}(a,\varphi_k)^2\nonumber\\
&\relphantom{=}{}
+C\nu^2\left|e^{-(\theta+\nu)(t+\sigma)}-e^{-(\theta+\nu)t}\right|^2 T^2\mathbb{E}\|a\|_{\mathbb{H}^{2\gamma}}^2
\nonumber\\
&\relphantom{=}{}
+C\nu^2\sigma^2\mathbb{E}\sum_{k=1}^{\infty}\left|E_{\alpha,1}(-\lambda_k^\beta\widetilde{\theta}^\alpha)\right|^2
\lambda_k^{2\gamma}(a,\varphi_k)^2\nonumber
&
\end{flalign*}
\begin{flalign*}
&\leq C\left|e^{-(\theta+\nu)(t+\sigma)}-e^{-(\theta+\nu)t}\right|^2\mathbb{E}\|a\|_{\mathbb{H}^{2\gamma}}^2
\nonumber\\
&\relphantom{=}{}
+\frac{C}{t^2}\sigma^2\mathbb{E}\sum_{k=1}^{\infty}\left|\frac{\lambda_k^\beta\widehat{\theta}^\alpha}{1+\lambda_k^\beta\widehat{\theta}^\alpha}\right|^2
\lambda_k^{2\gamma}(a,\varphi_k)^2
+C\sigma^2\mathbb{E}\|a\|_{\mathbb{H}^{2\gamma}}^2\nonumber
&
\end{flalign*}
\begin{flalign*}
&\leq C\left|e^{-(\theta+\nu)(t+\sigma)}-e^{-(\theta+\nu)t}\right|^2\mathbb{E}\|a\|_{\mathbb{H}^{2\widetilde{\gamma}}}^2
+\frac{C}{t^2}\sigma^2\mathbb{E}\|a\|_{\mathbb{H}^2_{2\widetilde{\gamma}}}
\nonumber\\
&\relphantom{=}{}
+
C\sigma^2\mathbb{E}\|a\|_{\mathbb{H}^{2\widetilde{\gamma}}}^2
\rightarrow 0 \quad \mbox{as}~~\sigma\rightarrow 0, \nonumber &
\end{flalign*}
where $\widehat{\theta}$, $\widetilde{\theta}\in (t,t+\sigma)$, and we have used $\frac{\lambda_k^\beta \widehat{\theta}^\alpha}{1+\lambda_k^\beta \widehat{\theta}^\alpha}\leq C$ and $\mathbb{E}\|a\|_{\mathbb{H}^{2\gamma}}^2\leq C \mathbb{E}\|a\|^2_{\mathbb{H}^{2\widetilde{\gamma}}}.\\$
In a similar way, we deduce that
\begin{flalign}\label{eq.3.10}
I_2&=C\mathbb{E}\sum_{k=1}^{\infty}\lambda_k^{2\gamma}\bigg(\int_{\mathcal{D}}(t+\sigma)
e^{-(\theta+\nu)(t+\sigma)}\sum_{l=1}^{\infty}E_{\alpha,2}(-\lambda_l^\beta(t+\sigma)^\alpha)
\varphi_l(\cdot)\varphi_l(y)
\nonumber\\
&\relphantom{=}{}
\cdot b(y)dy-\int_{\mathcal{D}}te^{-(\theta+\nu)t}\sum_{l=1}^{\infty}
E_{\alpha,2}(-\lambda_l^\beta t^\alpha)\varphi_l(\cdot)\varphi_l(y)b(y)dy,\varphi_k\bigg)^2\nonumber &
\end{flalign}
\begin{flalign}
&=C\mathbb{E}\sum_{k=1}^{\infty}\lambda_k^{2\gamma}\bigg{|}\int_{\mathcal{D}}(t+\sigma)e^{-(\theta+\nu)(t+\sigma)}
E_{\alpha,2}(-\lambda_k^{\beta}(t+\sigma)^\alpha)\varphi_k(y)b(y)dy\nonumber\\
&\relphantom{=}{}-\int_{\mathcal{D}}
te^{-(\theta+\nu)t}E_{\alpha,2}(-\lambda_k^{\beta}t^\alpha)\varphi_k(y)b(y)dy\bigg{|}^2\nonumber &
\end{flalign}
\begin{flalign}
&\leq C\left|e^{-(\theta+\nu)(t+\sigma)}-e^{-(\theta+\nu)t}\right|^2\mathbb{E}\sum_{k=1}^{\infty}
\left|(t+\sigma)E_{\alpha,2}(-\lambda_k^\beta(t+\sigma)^\alpha)\right|^2
\nonumber\\
&\relphantom{=}{}
\cdot\lambda_k^{2\gamma}
(b,\varphi_k)^2+Ce^{-2(\theta+\nu)t}\mathbb{E}\sum_{k=1}^{\infty}\left|(t+\sigma)
E_{\alpha,2}(-\lambda_k^\beta (t+\sigma)^\alpha) \right.
\nonumber\\
&\relphantom{=}{}
\left.
-t
E_{\alpha,2}(-\lambda_k^\beta t^\alpha)\right|^2\lambda_k^{2\gamma}(b,\varphi_k)^2  &
\end{flalign}
\begin{flalign}
& \leq C\left|e^{-(\theta+\nu)(t+\sigma)}-e^{-(\theta+\nu)t}\right|^2\mathbb{E}
\sum_{k=1}^{\infty}\frac{(\lambda_k^\beta(t+\sigma)^{\alpha})^{\frac{2}{\alpha}}}
{(1+\lambda_k^\beta(t+\sigma)^\alpha)^2}\lambda_k^{2\gamma-\frac{2\beta}{\alpha}}
(b,\varphi_k)^2\nonumber\\
&\relphantom{=}{}+C\sigma^2\mathbb{E}\sum_{k=1}^{\infty}
\left|E_{\alpha,1}(-\lambda_k^\beta\widetilde{\theta}^\alpha)\right|^2
\lambda_k^{2\gamma}(b,\varphi_k)^2\nonumber &
\end{flalign}
\begin{flalign}
& \leq C\left|e^{-(\theta+\nu)(t+\sigma)}-e^{-(\theta+\nu)t}\right|^2
\mathbb{E}\|b\|^2_{\mathbb{H}^{2\gamma-\frac{2\beta}{\alpha}}}
\nonumber\\
&\relphantom{=}{}
+\frac{C}{\widetilde{\theta}^2}
\sigma^2\mathbb{E}\sum_{k=1}^{\infty}\frac{(\lambda_k^\beta\widetilde{\theta}^\alpha)^{\frac{2}{\alpha}}}
{(1+\lambda_k^\beta\widetilde{\theta}^\alpha)^2}\lambda_k^{2\gamma-\frac{2\beta}{\alpha}}(b,\varphi_k)^2\nonumber &
\end{flalign}
\begin{flalign}
&\leq C\left|e^{-(\theta+\nu)(t+\sigma)}-e^{-(\theta+\nu)t}\right|^2
\mathbb{E}\|b\|^2_{\mathbb{H}^{2\widetilde{\gamma}-\frac{2\beta}{\alpha}}}
\nonumber\\
&\relphantom{=}{}
+\frac{C}{\widetilde{\theta}^{2}}
\sigma^2\mathbb{E}\|b\|_{\mathbb{H}^{2\widetilde{\gamma}-\frac{2\beta}{\alpha}}}^2
\rightarrow 0 ~~\mbox{as}~~\sigma\rightarrow 0, \nonumber &
\end{flalign}
where $\widetilde{\theta}\in (t,t+\sigma)$, we have used $\frac{(\lambda_k^\beta(t+\sigma)^\alpha)^{\frac{2}{\alpha}}}{(1+\lambda_k^\beta(t+\sigma)^\alpha)^2}
\leq C$, $\frac{(\lambda_k^\beta\widetilde{\theta}^\alpha)^{\frac{2}{\alpha}}}
{(1+\lambda_k^\beta\widetilde{\theta}^\alpha)^2}\leq C$, and $\mathbb{E}\|b\|^2_{\mathbb{H}^{2\gamma-\frac{2\beta}{\alpha}}}\leq C\mathbb{E}\|b\|^2_{\mathbb{H}^{2\widetilde{\gamma}-\frac{2\beta}{\alpha}}}$.

By $(\bf{A}_1)$, Lemma \ref{le2.1}, H\"{o}lder's inequality, and $u\in C([0,T];L^2(\Omega;\mathbb{H}^{2\gamma}(\mathcal{D})))$, we have
\begin{flalign}\label{eq.3.10*}
I_3&=C\mathbb{E}\sum_{k=1}^{\infty}\lambda_k^{2\gamma}\bigg(\int_0^{t+\sigma}
\int_{\mathcal{D}}(t+\sigma-s)^{\alpha-1}e^{-\theta(t+\sigma)}e^{-\nu(t+\sigma-s)}
\nonumber\\
&\relphantom{=}{}
\cdot\sum_{l=1}^{\infty}E_{\alpha,\alpha}(-\lambda_l^\beta(t+\sigma-s)^\alpha)\varphi_l(\cdot)
\varphi_l(y)f(s,u(s,y))dyds
\nonumber\\
&\relphantom{=}{}
-\int_0^t\int_{\mathcal{D}}(t-s)^{\alpha-1}e^{-\theta t}e^{-\nu(t-s)}\sum_{l=1}^{\infty}E_{\alpha,\alpha}(-\lambda_l^\beta(t-s)^\alpha)
\varphi_l(\cdot)\varphi_l(y)\nonumber\\
&\relphantom{=}{}\cdot f(s,u(s,y))dyds,\varphi_k\bigg)^2
&
\end{flalign}
\begin{flalign}
&=C\mathbb{E}\sum_{k=1}^{\infty}\lambda_k^{2\gamma}\bigg|\int_0^{t+\sigma}\int_{\mathcal{D}}
(t+\sigma-s)^{\alpha-1}e^{-\theta(t+\sigma)}e^{-\nu (t+\sigma-s)}
\nonumber\\
&\relphantom{=}{}
\cdot E_{\alpha,\alpha}(-\lambda_k^\beta(t+\sigma-s)^\alpha)\varphi_k(y) f(s,u(s,y))dyds
\nonumber\\
&\relphantom{=}{}
-\int_0^t\int_{\mathcal{D}}(t-s)^{\alpha-1}e^{-\theta t}e^{-\nu(t-s)}E_{\alpha,\alpha}(-\lambda_k^\beta(t-s)^{\alpha})\varphi_k(y)
\nonumber\\
&\relphantom{=}{}\cdot
f(s,u(s,y))dyds\bigg|^2\nonumber
&
\end{flalign}
\begin{flalign}
&\leq C\mathbb{E}\sum_{k=1}^{\infty}\bigg|\int_0^t\Big((t+\sigma-s)^{\alpha-1}e^{-\theta(t+\sigma)}
e^{-\nu(t+\sigma-s)}E_{\alpha,\alpha}(-\lambda_k^\beta(t+\sigma-s)^\alpha)\nonumber\\
&\relphantom{=}{} -(t-s)^{\alpha-1}e^{-\theta t}e^{-\nu(t-s)} E_{\alpha,\alpha}(-\lambda_k^\beta(t-s)^\alpha)\Big)\lambda_k^\gamma(f(s,u(s)),\varphi_k)ds\bigg|^2\nonumber\\
&\relphantom{=}{}+C\mathbb{E}\sum_{k=1}^{\infty}\bigg|\int_t^{t+\sigma}(t+\sigma-s)^{\alpha-1}
e^{-\theta(t+\sigma)}e^{-\nu(t+\sigma-s)}
\nonumber\\
&\relphantom{=}{}
\cdot
E_{\alpha,\alpha}(-\lambda_k^\beta(t+\sigma-s)^\alpha)
\lambda_k^\gamma(f(s,u(s)),\varphi_k)ds\bigg|^2\nonumber
&
\end{flalign}
\begin{flalign}
&\leq CT\mathbb{E}\sum_{k=1}^{\infty}\int_0^t\Big|(t+\sigma-s)^{\alpha-1}e^{-\theta(t+\sigma)}
e^{-\nu(t+\sigma-s)}E_{\alpha,\alpha}(-\lambda_k^\beta(t+\sigma-s)^\alpha)\nonumber\\
&\relphantom{=}{} -(t-s)^{\alpha-1}
e^{-\theta t}e^{-\nu(t-s)}E_{\alpha,\alpha}(-\lambda_k^\beta(t-s)^\alpha)\Big|^2\lambda_k^{2\gamma}
(f(s,u(s)),\varphi_k)^2ds\nonumber\\
&\relphantom{=}{}
+C\sigma\int_t^{t+\sigma}(t+\sigma-s)^{2\alpha-4}\sum_{k=1}^{\infty}
\frac{(\lambda_k^\beta(t+\sigma-s)^{\alpha})^{\frac{2}{\alpha}}}{(1+\lambda_k^\beta(t+\sigma-s)^{\alpha})^2}
\lambda_k^{2\gamma-\frac{2\beta}{\alpha}}
\nonumber\\
&\relphantom{=}{}
\cdot(f(s,u(s)),\varphi_k)^2ds\nonumber
&
\end{flalign}
\begin{flalign}
&\leq C\mathbb{E}\sum_{k=1}^{\infty}\int_0^t\Big|(t+\sigma-s)^{\alpha-1}e^{-\theta(t+\sigma)}
e^{-\nu(t+\sigma-s)}E_{\alpha,\alpha}(-\lambda_k^\beta(t+\sigma-s)^\alpha)\nonumber\\
&\relphantom{=}{} -(t-s)^{\alpha-1}
e^{-\theta t}e^{-\nu(t-s)}E_{\alpha,\alpha}(-\lambda_k^\beta(t-s)^\alpha)\Big|^2\lambda_k^{2\gamma}
(f(s,u(s)),\varphi_k)^2ds\nonumber\\
&\relphantom{=}{} +C\sigma\int_t^{t+\sigma}(t+\sigma-s)^{2\alpha-4}
\left(1+\mathbb{E}\|u(s)\|^2_{\mathbb{H}^{2\gamma-\frac{2\beta}{\alpha}}}\right)ds\nonumber &
\end{flalign}
\begin{flalign}
&\leq C\mathbb{E}\sum_{k=1}^{\infty}\int_0^t \Big|(t+\sigma-s)^{\alpha-1}e^{-\theta(t+\sigma)}
e^{-\nu(t+\sigma-s)}E_{\alpha,\alpha}(-\lambda_k^\beta(t+\sigma-s)^\alpha)\nonumber\\
&\relphantom{=}{} -(t-s)^{\alpha-1}
e^{-\theta t}e^{-\nu(t-s)}E_{\alpha,\alpha}(-\lambda_k^\beta(t-s)^\alpha)\Big|^2\lambda_k^{2\gamma}
(f(s,u(s)),\varphi_k)^2ds
\nonumber\\
&\relphantom{=}{}
+C\sigma^{2\alpha-2}, \nonumber &
\end{flalign}
where we have used $\frac{(\lambda_k^\beta(t+\sigma-s)^{\alpha})^{\frac{2}{\alpha}}}{(1+\lambda_k^\beta(t+\sigma-s)^{\alpha})^2}
\leq C$ and $\mathbb{E}\|u(s)\|^2_{\mathbb{H}^{2\gamma-\frac{2\sigma}{\alpha}}}
\leq C\mathbb{E}\|u(s)\|_{\mathbb{H}^{2\gamma}}^2$. Furthermore,
\begin{flalign}\label{eq.3.11}
&\lambda_k^{\frac{\beta}{\alpha}}\Big|(t+\sigma-s)^{\alpha-1}e^{-\theta(t+\sigma)}
e^{-\nu(t+\sigma-s)}E_{\alpha,\alpha}(-\lambda_k^\beta(t+\sigma-s)^\alpha)\nonumber\\
&\relphantom{=}{}-(t-s)^{\alpha-1}
e^{-\theta t}e^{-\nu(t-s)}E_{\alpha,\alpha}(-\lambda_k^\beta(t-s)^{\alpha})\Big|\nonumber
&
\end{flalign}
\begin{flalign}
&\relphantom{}{}\leq e^{-\theta(t+\sigma)}e^{-\nu(t+\sigma-s)}\lambda_k^{\frac{\beta}{\alpha}}\left|
(t+\sigma-s)^{\alpha-1}E_{\alpha,\alpha}(-\lambda_{k}^\beta(t+\sigma-s)^\alpha)\right.
\nonumber\\
&\relphantom{=}{}
\left.
-(t-s)^{\alpha-1}
E_{\alpha,\alpha}(-\lambda_k^\beta(t-s)^\alpha)\right|\nonumber\\
&\relphantom{=}{}+(t-s)^{\alpha-1}\lambda_k^{\frac{\beta}{\alpha}}
\left|E_{\alpha,\alpha}(-\lambda_k^\beta(t-s)^{\alpha})\right|
\left|e^{-\theta(t+\sigma)}e^{-\nu(t+\sigma-s)}-e^{-\theta t}e^{-\nu(t-s)}\right|\nonumber
&
\end{flalign}
\begin{flalign}
&\relphantom{}{}\leq \lambda_k^{\frac{\beta}{\alpha}}\left|(t+\sigma-s)^{\alpha-1}E_{\alpha,\alpha}(-\lambda_k^\beta(t+\sigma-s)^\alpha)
\right.
\nonumber\\
&\relphantom{=}{}
\left.
-(t-s)^{\alpha-1}E_{\alpha,\alpha}(-\lambda_k^\beta(t-s)^\alpha)\right|\nonumber\\
&\relphantom{=}{}
+C(t-s)^{\alpha-2}\frac{(\lambda_k^\beta(t-s)^{\alpha})^{\frac{1}{\alpha}}}{1+\lambda_k^\beta(t-s)^{\alpha}}
\left|e^{-\theta(t+\sigma)}e^{-\nu(t+\sigma-s)}-e^{-\theta t}e^{-\nu(t-s)}\right| %\nonumber
&
\end{flalign}
\begin{flalign}
&\relphantom{}{}\leq \left|\int_t^{t+\sigma}(\tau-s)^{\alpha-2}\lambda_k^{\frac{\beta}{\alpha}}
E_{\alpha,\alpha-1}(-\lambda_k^\beta(\tau-s)^{\alpha})d\tau\right|\nonumber\\
&\relphantom{=}{}+
C(t-s)^{\alpha-2}\left|e^{-\theta(t+\sigma)}e^{-\nu(t+\sigma-s)}-e^{-\theta t}e^{-\nu(t-s)}\right|\nonumber
&
\end{flalign}
\begin{flalign}
&\relphantom{}{}\leq C\int_t^{t+\sigma}(\tau-s)^{\alpha-3}\frac{(\lambda_k^\beta(\tau-s)^\alpha)^{\frac{1}{\alpha}}}
{1+\lambda_k^\beta(\tau-s)^{\alpha}}d\tau
\nonumber\\
&\relphantom{=}{}
+C(t-s)^{\alpha-2}
\left|e^{-\theta(t+\sigma)}e^{-\nu(t+\sigma-s)}-e^{-\theta t}e^{-\nu(t-s)}\right|\nonumber
&
\end{flalign}
\begin{flalign}
&\relphantom{}{}\leq \int_t^{t+\sigma}(\tau-s)^{\alpha-3}d\tau+C(t-s)^{\alpha-2}
\left|e^{-\theta(t+\sigma)}e^{-\nu(t+\sigma-s)}-e^{-\theta t}e^{-\nu(t-s)}\right|, \nonumber &
\end{flalign}
where we have used $\frac{(\lambda_k^\beta(\tau-s)^\alpha)^{\frac{1}{\alpha}}}
{1+\lambda_k^\beta(\tau-s)^{\alpha}}\leq C$.

Combining \eqref{eq.3.10*} and \eqref{eq.3.11}, in view of $(\bf{A}_1)$, and $u\in C([0,T];L^2(\Omega;\mathbb{H}^{2\gamma}(\mathcal{D})))$, we conclude from Lebesgue's dominated convergence theorem that
\begin{flalign}\label{eq.3.12}
I_3&\leq C\mathbb{E}\sum_{k=1}^{\infty}\int_0^t\left(\int_t^{t+\sigma}(\tau-s)^{\alpha-3}d\tau\right)^2
\lambda_k^{2\gamma-\frac{2\beta}{\alpha}}(f(s,u(s)),\varphi_k)^2ds\nonumber\\
&\relphantom{=}{}+C\mathbb{E}\sum_{k=1}^{\infty}\int_0^t(t-s)^{2\alpha-4}
\left|e^{-\theta(t+\sigma)}e^{-\nu(t+\sigma-s)}-e^{-\theta t}e^{-\nu (t-s)}\right|^2\nonumber\\
&\relphantom{=}{}\cdot \lambda_k^{2\gamma-\frac{2\beta}{\alpha}}(f(s,u(s)),\varphi_k)^2ds+C\sigma^{2\alpha-2}\nonumber&
\end{flalign}
\begin{flalign}
&\leq C\int_0^t\left((t+\sigma-s)^{\alpha-2}-(t-s)^{\alpha-2}\right)^2
\left(1+\mathbb{E}\|u(s)\|_{\mathbb{H}^{2\gamma-\frac{2\beta}{\alpha}}}^2\right)ds\nonumber\\
&\relphantom{=}{}+C\int_0^t(t-s)^{2\alpha-4}\left|e^{-\theta(t+\sigma)}e^{-\nu(t+\sigma-s)}
-e^{-\theta t}e^{-\nu(t-s)}\right|^2
\nonumber\\
&\relphantom{=}{}
\cdot\left(1+\mathbb{E}\|u(s)\|_{\mathbb{H}^{2\gamma-\frac{2\beta}{\alpha}}}^2\right)ds
+C\sigma^{2\alpha-2}
&
\end{flalign}
\begin{flalign}
&\leq C\sigma^{2\alpha-3}+C\int_0^t\left((t+\sigma-s)^{\alpha-2}-(t-s)^{\alpha-2}\right)^2ds
\nonumber\\
&\relphantom{=}{}
+ C\sigma^{2\alpha-2} \rightarrow 0 ~~\mbox{as}~~\sigma\rightarrow 0, \nonumber &
\end{flalign}
where we have used $\mathbb{E}\|u(s)\|^2_{\mathbb{H}^{2\gamma-\frac{2\sigma}{\alpha}}}
\leq C\mathbb{E}\|u(s)\|_{\mathbb{H}^{2\gamma}}^2.$

For $I_4$, arguing as in the proof of \eqref{eq.3.10*}-\eqref{eq.3.12} and noticing that $\{\varphi_k\}_{k=1}^{\infty}$ is an orthonormal basis in $L^2(\mathcal{D})$ and $\{\xi_l\}_{l=1}^{\infty}$ is a family of mutually independent one-dimensional standard Brownian motions, we deduce from \eqref{eq.2.8}, $(\bf{A}_1)$-$(\bf{A}_2)$, Proposition \ref{pro1},  and $u\in C([0,T];L^2(\Omega;\mathbb{H}^{2\gamma}(\mathcal{D})))$ that
\begin{flalign}\label{eq.3.13}
&I_4
\nonumber\\
&=C\mathbb{E}\sum_{k=1}^{\infty}\lambda_k^{2\gamma}\bigg(\int_0^{t+\sigma}
\int_{\mathcal{D}}(t+\sigma-s)^{\alpha-1}e^{-\theta(t+\sigma)}e^{-\nu(t+\sigma-s)}
\nonumber\\
&\relphantom{=}{}
\cdot\sum_{l=1}^{\infty}E_{\alpha,\alpha}(-\lambda_l^\beta(t+\sigma-s)^\alpha)\varphi_l(\cdot)
\varphi_l(y)g(s,u(s,y))d\mathbb{W}(s,y)
%\nonumber
\\
&\relphantom{=}{}
-\int_0^t\int_{\mathcal{D}}(t-s)^{\alpha-1}e^{-\theta t}e^{-\nu(t-s)}\sum_{l=1}^{\infty}E_{\alpha,\alpha}(-\lambda_l^\beta(t-s)^\alpha)
\varphi_l(\cdot)\varphi_l(y)\nonumber\\
&\relphantom{=}{} \cdot g(s,u(s,y))d\mathbb{W}(s,y),\varphi_k\bigg)^2\nonumber &
\end{flalign}
\begin{flalign}
&=C\mathbb{E}\sum_{k=1}^{\infty}\lambda_k^{2\gamma}\bigg|\int_0^{t+\sigma}
\int_{\mathcal{D}}(t+\sigma-s)^{\alpha-1}e^{-\theta(t+\sigma)}e^{-\nu(t+\sigma-s)}
\nonumber\\
&\relphantom{=}{}
\cdot E_{\alpha,\alpha}(-\lambda_k^\beta(t+\sigma-s)^\alpha)
\varphi_k(y)g(s,u(s,y)) d\mathbb{W}(s,y)
\nonumber\\
&\relphantom{=}{}
-\int_0^t\int_{\mathcal{D}}(t-s)^{\alpha-1}e^{-\theta t}e^{-\nu(t-s)}E_{\alpha,\alpha}(-\lambda_k^\beta(t-s)^\alpha)
\nonumber\\
&\relphantom{=}{}
\cdot\varphi_k(y)g(s,u(s,y))d\mathbb{W}(s,y)\bigg|^2 \nonumber &
\end{flalign}
\begin{flalign}%\label{eq.3.12}
\nonumber\\
&\leq C\mathbb{E}\sum_{k=1}^{\infty}\bigg|\int_0^{t}
\int_{\mathcal{D}}\Big((t+\sigma-s)^{\alpha-1}e^{-\theta(t+\sigma)}e^{-\nu(t+\sigma-s)}
E_{\alpha,\alpha}(-\lambda_k^\beta(t+\sigma-s)^\alpha)
\nonumber\\
&\relphantom{=}{}
-(t-s)^{\alpha-1}e^{-\theta t} e^{-\nu(t-s)}
E_{\alpha,\alpha}(-\lambda_k^\beta(t-s)^\alpha)\Big)\lambda_k^{\gamma}
\varphi_k(y)
\nonumber\\
&\relphantom{=}{}
\cdot\sum_{j,l=1}^{\infty}(g(s,u(s))\cdot e_l,\varphi_j)\varphi_j(y)\varsigma_l(s)dyd\xi_l(s)\bigg|^2\nonumber\\
&\relphantom{=}{}+C\mathbb{E}\sum_{k=1}^{\infty}\bigg|\int_t^{t+\sigma}
\int_{\mathcal{D}}(t+\sigma-s)^{\alpha-1}e^{-\theta(t+\sigma)}e^{-\nu(t+\sigma-s)}
\nonumber\\
&\relphantom{=}{} \cdot E_{\alpha,\alpha}(-\lambda_k^\beta(t+\sigma-s)^\alpha)\lambda_k^\gamma
\varphi_k(y) \sum_{j,l=1}^{\infty}(g(s,u(s))\cdot e_l,\varphi_j)\varphi_j(y)\varsigma_l(s)dyd\xi_l(s)\bigg|^2\nonumber &
\end{flalign}
\begin{flalign*}
&=C\mathbb{E}\sum_{k=1}^{\infty}\bigg|\int_0^{t}
\Big((t+\sigma-s)^{\alpha-1}e^{-\theta(t+\sigma)}e^{-\nu(t+\sigma-s)}
E_{\alpha,\alpha}(-\lambda_k^\beta(t+\sigma-s)^\alpha)
\nonumber\\
&\relphantom{=}{}
-(t-s)^{\alpha-1}e^{-\theta t} e^{-\nu(t-s)}
E_{\alpha,\alpha}(-\lambda_k^\beta(t-s)^\alpha)\Big)\lambda_k^{\gamma}
\sum_{l=1}^{\infty}(g(s,u(s))\cdot e_l,\varphi_k)
\nonumber\\
&\relphantom{=}{}
\cdot\varsigma_l(s)d\xi_l(s)\bigg|^2+C\mathbb{E}\sum_{k=1}^{\infty}\bigg|\int_t^{t+\sigma}
(t+\sigma-s)^{\alpha-1}e^{-\theta(t+\sigma)}e^{-\nu(t+\sigma-s)}
\nonumber\\
&\relphantom{=}{}\cdot E_{\alpha,\alpha}(-\lambda_k^\beta(t+\sigma-s)^\alpha)\lambda_k^\gamma
\sum_{l=1}^{\infty}(g(s,u(s))\cdot e_l,\varphi_k)\varsigma_l(s)d\xi_l(s)\bigg|^2 &
\end{flalign*}

\begin{flalign*}%\label{eq.3.12}
\nonumber\\
&=C\mathbb{E}\sum_{k,l=1}^{\infty}\int_0^{t}\Big|
\Big((t+\sigma-s)^{\alpha-1}e^{-\theta(t+\sigma)}e^{-\nu(t+\sigma-s)}
E_{\alpha,\alpha}(-\lambda_k^\beta(t+\sigma-s)^\alpha)\nonumber\\
&\relphantom{=}{} -(t-s)^{\alpha-1}e^{-\theta t}e^{-\nu(t-s)}
E_{\alpha,\alpha}(-\lambda_k^\beta(t-s)^\alpha)\Big)\lambda_k^{\gamma}
(g(s,u(s))\cdot e_l,\varphi_k)\varsigma_l(s)\Big|^2 ds\nonumber\\
&\relphantom{=}{}+C\mathbb{E}\sum_{k,l=1}^{\infty}\int_t^{t+\sigma}
\Big|(t+\sigma-s)^{\alpha-1}e^{-\theta(t+\sigma)}e^{-\nu(t+\sigma-s)}
E_{\alpha,\alpha}(-\lambda_k^\beta(t+\sigma-s)^\alpha)\nonumber\\
&\relphantom{=}{}
\cdot\lambda_k^\gamma(g(s,u(s))\cdot e_l,\varphi_k)\varsigma_l(s)\Big|^2 ds\nonumber &
\end{flalign*}
\begin{flalign*}
&\leq C\mathbb{E}\sum_{k,l=1}^{\infty}\int_0^t\bigg(\left(\int_t^{t+\sigma}(\tau-s)^{\alpha-3}d\tau\right)^2+
(t-s)^{2\alpha-4}\left|e^{-\theta(t+\sigma)}e^{-\nu(t+\sigma-s)}\right.\nonumber\\
&\relphantom{=}{} \left. -e^{-\theta t}e^{-\nu(t-s)}\right|^2\bigg) \left|\lambda_k^{\gamma-\frac{\beta}{\alpha}}(g(s,u (s))\cdot e_l,\varphi_k)\varsigma_l(s)\right|^2ds\nonumber\\
&\relphantom{=}{}+ C\mathbb{E}\sum_{k,l=1}^{\infty}\int_t^{t+\sigma}(t+\sigma-s)^{2\alpha-4}
\frac{(\lambda_k^\beta(t+\sigma-s)^\alpha)^{\frac{2}{\alpha}}}{(1+\lambda_k^\beta(t+\sigma-s)^\alpha)^2} \nonumber\\
&\relphantom{=}{}
\cdot\left|\lambda_k^{\gamma-\frac{\beta}{\alpha}}(g(s,u(s))\cdot e_l,\varphi_k)\varsigma_l(s)\right|^2ds\nonumber
&
\end{flalign*}
\begin{flalign*}
&\leq C\int_0^t\left(\left((t+\sigma-s)^{\alpha-2}-(t-s)^{\alpha-2}\right)^2+(t-s)^{2\alpha-4}
\left|e^{-\theta(t+\sigma)}e^{-\nu(t+\sigma-s)}
\right.\right.
\nonumber\\
&\relphantom{=}{}\left.\left. -e^{-\theta t}e^{-\nu(t-s)}\right|^2\right) \left(1+\mathbb{E}\|u(s)\|_{\mathbb{H}^{2\gamma-\frac{2\beta}{\alpha}}}^2\right)ds
+C\sigma^{2\alpha-3}
\nonumber
&
\end{flalign*}
\begin{flalign*}
&\leq C\int_0^t\left((t+\sigma-s)^{\alpha-2}-(t-s)^{\alpha-2}\right)^2ds\nonumber\\
&\relphantom{=}{}+C\int_0^t(t-s)^{2\alpha-4}
\left|e^{-\theta(t+\sigma)}e^{-\nu(t+\sigma-s)}-e^{-\theta t}e^{-\nu(t-s)}\right|^2ds
\nonumber\\
&\relphantom{=}{}
+C\sigma^{2\alpha-3}\rightarrow 0~~\mbox{as}~~\sigma\rightarrow 0. &
\end{flalign*}
Similar to the arguments in \eqref{eq.3.13}, in view of (2.9), $(\bf{A}_1)$-$(\bf{A}_2)$, Lemma \ref{le2.10}, $u\in$ $C([0,T];$ $L^2(\Omega;$ $\mathbb{H}^{2\gamma}(\mathcal{D})))$, and the mutual independence of the family of one-dimensional fractional Brownian motions $\{\xi_l^H\}_{l=1}^{\infty}$, we obtain that
\begin{flalign}\label{eq.3.14}
&I_5
\nonumber\\
&=C\mathbb{E}\sum_{k=1}^{\infty}\lambda_k^{2\gamma}\bigg(\int_0^{t+\sigma}
\int_{\mathcal{D}}(t+\sigma-s)^{\alpha-1}e^{-\theta(t+\sigma)}e^{-\nu(t+\sigma-s)}
\nonumber\\
&\relphantom{=}{}
\cdot\sum_{l=1}^{\infty}E_{\alpha,\alpha}(-\lambda_l^\beta(t+\sigma-s)^\alpha)\varphi_l(\cdot)
\varphi_l(y) h(s,u(s,y))d\mathbb{W}^H(s,y)
\nonumber\\
&\relphantom{=}{}
-\int_0^t\int_{\mathcal{D}}(t-s)^{\alpha-1}e^{-\theta t}e^{-\nu(t-s)}\sum_{l=1}^{\infty}E_{\alpha,\alpha}(-\lambda_l^\beta(t-s)^\alpha)
\varphi_l(\cdot)\varphi_l(y)\nonumber\\
&\relphantom{=}{} \cdot h(s,u(s,y))d\mathbb{W}^H(s,y),\varphi_k\bigg)^2\nonumber &
\end{flalign}
\begin{flalign}
&=C\mathbb{E}\sum_{k=1}^{\infty}\lambda_k^{2\gamma}\bigg|\int_0^{t+\sigma}
\int_{\mathcal{D}}(t+\sigma-s)^{\alpha-1}e^{-\theta(t+\sigma)}e^{-\nu(t+\sigma-s)}
\nonumber\\
&\relphantom{=}{} \cdot E_{\alpha,\alpha}(-\lambda_k^\beta(t+\sigma-s)^\alpha)
\varphi_k(y)h(s,u(s,y)) d\mathbb{W}^H(s,y)
\nonumber\\
&\relphantom{=}{}
-\int_0^t\int_{\mathcal{D}}(t-s)^{\alpha-1}e^{-\theta t}e^{-\nu(t-s)}E_{\alpha,\alpha}(-\lambda_k^\beta(t-s)^\alpha)
\varphi_k(y)
\nonumber\\
&\relphantom{=}{}
\cdot h(s,u(s,y))d\mathbb{W}^H(s,y)\bigg|^2 &
\end{flalign}
\begin{flalign}
&\leq C\mathbb{E}\sum_{k=1}^{\infty}\bigg|\int_0^{t}
\int_{\mathcal{D}}\Big((t+\sigma-s)^{\alpha-1}e^{-\theta(t+\sigma)}e^{-\nu(t+\sigma-s)}
E_{\alpha,\alpha}(-\lambda_k^\beta(t+\sigma-s)^\alpha)\nonumber\\
&\relphantom{=}{} -(t-s)^{\alpha-1}e^{-\theta t} e^{-\nu(t-s)}
E_{\alpha,\alpha}(-\lambda_k^\beta(t-s)^\alpha)\Big)\lambda_k^{\gamma}
\varphi_k(y)
\nonumber\\
&\relphantom{=}{}
\cdot\sum_{j,l=1}^{\infty}(h(s,u(s,y))\cdot e_l,\varphi_j)\varphi_j(y)\varrho_l(s)dyd\xi_l^H(s)\bigg|^2\nonumber\\
&\relphantom{=}{}+C\mathbb{E}\sum_{k=1}^{\infty}\bigg|\int_t^{t+\sigma}
\int_{\mathcal{D}}(t+\sigma-s)^{\alpha-1}e^{-\theta(t+\sigma)}e^{-\nu(t+\sigma-s)}
\varphi_k(y)\nonumber\\
&\relphantom{=}{}\cdot E_{\alpha,\alpha}(-\lambda_k^\beta(t+\sigma-s)^\alpha)\lambda_k^\gamma \sum_{j,l=1}^{\infty}(h(s,u(s,y))\cdot e_l,\varphi_j)\varphi_j(y)\varrho_l(s)dyd\xi_l^H(s)\bigg|^2\nonumber &
\end{flalign}
\begin{flalign}
&\leq CHT^{2H-1}\mathbb{E}\sum_{k,l=1}^{\infty}\int_0^{t}\Big|
\Big((t+\sigma-s)^{\alpha-1}e^{-\theta(t+\sigma)}e^{-\nu(t+\sigma-s)}
\nonumber\\
&\relphantom{=}{}
\cdot E_{\alpha,\alpha}(-\lambda_k^\beta(t+\sigma-s)^\alpha)-(t-s)^{\alpha-1}e^{-\theta t}e^{-\nu(t-s)}
E_{\alpha,\alpha}(-\lambda_k^\beta(t-s)^\alpha)\Big)
\nonumber\\
&\relphantom{=}{}\cdot\lambda_k^{\gamma}
(h(s,u(s))\cdot e_l,\varphi_k)\varrho_l(s)\Big|^2ds+CH\sigma^{2H-1}\mathbb{E}\sum_{k,l=1}^{\infty}\int_t^{t+\sigma}
\Big|(t+\sigma-s)^{\alpha-1}\nonumber\\
&\relphantom{=}{} \cdot e^{-\theta(t+\sigma)}e^{-\nu(t+\sigma-s)}
E_{\alpha,\alpha}(-\lambda_k^\beta(t+\sigma-s)^\alpha)\lambda_k^\gamma
(h(s,u(s))\cdot e_l,\varphi_k)\varrho_l(s)\Big|^2 ds\nonumber &
\end{flalign}
\begin{flalign}
&\leq C\int_0^t\bigg(\left(\int_t^{t+\sigma}(\tau-s)^{\alpha-3}d\tau\right)^2+
(t-s)^{2\alpha-4}\left|e^{-\theta(t+\sigma)}e^{-\nu(t+\sigma-s)}\right.\nonumber\\
&\relphantom{=}{} \left. -e^{-\theta t}e^{-\nu(t-s)}\right|^2\bigg)\left(1+\mathbb{E}\|u(s)\|_{\mathbb{H}^{2\gamma-\frac{2\beta}{\alpha}}}^2\right)ds
\nonumber\\
&\relphantom{=}{}+C\sigma^{2H-1}\int_t^{t+\sigma}(t+\sigma-s)^{2\alpha-4}
\left(1+\mathbb{E}\|u(s)\|_{\mathbb{H}^{2\gamma-\frac{2\beta}{\alpha}}}^2\right)ds\nonumber &
\end{flalign}
\begin{flalign}
&\leq C\int_0^t\left((t+\sigma-s)^{\alpha-2}-(t-s)^{\alpha-2}\right)^2ds\nonumber\\
&\relphantom{=}{}+C\int_0^t(t-s)^{2\alpha-4}
\left|e^{-\theta(t+\sigma)}e^{-\nu(t+\sigma-s)}-e^{-\theta t}e^{-\nu(t-s)}\right|^2ds\nonumber\\
&\relphantom{=}{} +C\sigma^{2H+2\alpha-4}\rightarrow 0~~\mbox{as}~~\sigma\rightarrow 0. \nonumber &
\end{flalign}
From \eqref{eq.3.8}-\eqref{eq.3.10} and \eqref{eq.3.12}-\eqref{eq.3.14}, it follows that
$\mathbb{E}\|e^{-\theta(t+\sigma)}\mathcal{P}u(t+\sigma)-e^{-\theta t}\mathcal{P}u(t)\|_{\mathbb{H}^{2\gamma}}^2$ tends to zero as $\sigma\rightarrow 0$. Consequently, $\mathcal{P}u$ belongs to $C([0,T];L^2(\Omega;\mathbb{H}^{2\gamma}(\mathcal{D})))_{\theta}$.

Step 2. $\mathcal{P}:$ $C([0,T];L^2(\Omega;\mathbb{H}^{2\gamma}(\mathcal{D})))_{\theta}$ $\rightarrow$ $C([0,T];L^2(\Omega;\mathbb{H}^{2\gamma}(\mathcal{D})))_{\theta}$ has a unique fixed point.

For any $u_1$, $u_2\in C([0,T];L^2(\Omega;\mathbb{H}^{2\gamma}(\mathcal{D})))_{\theta}$, by \eqref{eq.3.7*} we have
\begin{align}\label{2.1}
\displaystyle &\mathbb{E}\|e^{-\theta t}(\mathcal{P}u_1(t)-\mathcal{P}u_2(t))\|_{\mathbb{H}^{2\gamma}}^2\nonumber\\
&\relphantom{}{} \leq
C\mathbb{E}\bigg{\|}e^{-\theta t}\int_{0}^t\int_{\mathcal{D}}\mathcal{S}^\nu_{\alpha,\beta}(t-s,\cdot,y)
(f(s,u_1(s,y))-f(s,u_2(s,y)))\nonumber\\
&\relphantom{=}{}\cdot dyds\bigg{\|}_{\mathbb{H}^{2\gamma}}^2+C\mathbb{E}\bigg{\|}e^{-\theta t}\int_{0}^t\int_{\mathcal{D}}\mathcal{S}^\nu_{\alpha,\beta}(t-s,\cdot,y)
(g(s,u_1(s,y))\nonumber\\
&\relphantom{=}{}-g(s,u_2(s,y)))d{\mathbb{W}}(s,y)\bigg{\|}_{\mathbb{H}^{2\gamma}}^2+C\mathbb{E}\bigg{\|}e^{-\theta t}\int_{0}^t\int_{\mathcal{D}}\mathcal{S}^\nu_{\alpha,\beta}(t-s,\cdot,y)
\nonumber\\
&\relphantom{=}{}
\cdot(h(s,u_1(s,y))-h(s,u_2(s,y)))d{\mathbb{W}^H}(s,y)\bigg{\|}_{\mathbb{H}^{2\gamma}}^2\nonumber\\
&\relphantom{}{}:=\mathcal{Z}_1+\mathcal{Z}_2+\mathcal{Z}_3.
\end{align}
Since $\{\varphi_k\}_{k=1}^{\infty}$ is an orthonormal basis in $L^2({\mathcal{D}})$, by \eqref{eq2.2}, H\"{o}lder's inequality, Lemma \ref{le2.1}, and $(\bf{A}_1)$, we have
\begin{flalign}\label{2.2}
\mathcal{Z}_1&= C\mathbb{E}\bigg{\|}e^{-\theta t}\int_{0}^t\int_{\mathcal{D}}(t-s)^{\alpha-1}e^{-\nu(t-s)}
\sum_{l=1}^{\infty}E_{\alpha,\alpha}\left(-\lambda_l^{\beta}(t-s)^\alpha\right)
\varphi_l(\cdot)\varphi_l(y)\nonumber\\
&\relphantom{}{}\cdot(f(s,u_1(s,y))-f(s,u_2(s,y)))dyds\bigg{\|}_{\mathbb{H}^{2\gamma}}^2\nonumber &
\end{flalign}
\begin{flalign}
&=C\mathbb{E}\sum_{k=1}^{\infty}\lambda_k^{2\gamma}\bigg(\int_0^t\int_{\mathcal{D}}(t-s)^{\alpha-1}
e^{-\theta t}e^{-\nu(t-s)}\sum_{l=1}^{\infty}
E_{\alpha,\alpha}(-\lambda_l^\beta(t-s)^{\alpha})\nonumber\\
&\relphantom{}{} \cdot\varphi_l(\cdot)\varphi_l(y)(f(s,u_1(s,y))-f(s,u_2(s,y)))dyds,\varphi_k\bigg)^2 &
\end{flalign}
\begin{flalign}
&=C\mathbb{E}\sum_{k=1}^{\infty} \lambda_k^{2\gamma} \bigg{|}\int_{0}^t\int_{\mathcal{D}}(t-s)^{\alpha-1}e^{-\theta t}e^{-\nu(t-s)}
E_{\alpha,\alpha}\left(-\lambda_k^{\beta}(t-s)^\alpha\right)
\varphi_k(y)\nonumber\\
&\relphantom{}{}\cdot (f(s,u_1(s,y))-f(s,u_2(s,y)))dyds\bigg{|}^2\nonumber
&
\end{flalign}
\begin{flalign}
&\leq C\mathbb{E}\int_{0}^t(t-s)^{2\alpha-2}e^{-2\theta t}
e^{-2\nu(t-s)}\sum_{k=1}^{\infty}
\bigg{|}E_{\alpha,\alpha}\left(-\lambda_k^{\beta}(t-s)^\alpha\right)\bigg{|}^2
\lambda_k^{2\gamma}\nonumber\\
&\relphantom{}{}\cdot
((f(s,u_1(s))-f(s,u_2(s))),\varphi_k)^2ds\nonumber&
\end{flalign}
\begin{flalign}
&\leq C\int_{0}^t(t-s)^{2\alpha-4}e^{-2\theta t} e^{-2\nu(t-s)}\mathbb{E}\sum_{k=1}^{\infty}
\frac{(\lambda_k^\beta(t-s)^{\alpha})^{\frac{2}{\alpha}}}{(1+\lambda_k^\beta(t-s)^\alpha)^2}
\lambda_k^{2\gamma-\frac{2\beta}{\alpha}}\nonumber\\
&\relphantom{}{}\cdot((f(s,u_1(s))-f(s,u_2(s))),
\varphi_k)^2ds\nonumber &
\end{flalign}
\begin{flalign}
&= C\int_{0}^t(t-s)^{2\alpha-4}e^{-2\theta t} e^{-2\nu(t-s)}\mathbb{E}
{\|}f(s,u_1(s))-f(s,u_2(s)){\|}_{\mathbb{H}^{2\gamma-\frac{2\beta}{\alpha}}}^2ds\nonumber
&
\end{flalign}
\begin{flalign}
&\leq C\int_{0}^t(t-s)^{2\alpha-4}e^{-2\theta t} e^{-2\nu(t-s)}\mathbb{E}
{\|}u_1(s)-u_2(s){\|}_{\mathbb{H}^{2\gamma}}^2ds, \nonumber &
\end{flalign}
where we have used $\frac{(\lambda_k^\beta(t-s)^{\alpha})^{\frac{2}{\alpha}}}{(1+\lambda_k^\beta(t-s)^\alpha)^2}
\leq C$ and $\mathbb{E}\|u_1(s)-u_2(s)\|^2_{\mathbb{H}^{2\gamma-\frac{2\beta}{\alpha}}}
\leq C\mathbb{E}\|u_1(s)-u_2(s)\|_{\mathbb{H}^{2\gamma}}^2$.

For $\mathcal{Z}_2$, noticing that $\{\xi_l\}_{l=1}^\infty$ is a family of independent one-dimensional standard Brownian motions and $\{\varphi_k\}_{k=1}^{\infty}$ is an orthonormal basis in $L^2(\mathcal{D})$, in view of Lemma \ref{le2.1}, Proposition \ref{pro1}, \eqref{eq.2.8}, \eqref{eq2.2}, and $(\bf{A}_1)$-$(\bf{A}_2)$, we get that
\begin{flalign}\label{2.3}
\mathcal{Z}_2&= C\mathbb{E}\bigg{\|}e^{-\theta t} \int_{0}^t\int_{\mathcal{D}}(t-s)^{\alpha-1}e^{-\nu(t-s)}
\sum_{l=1}^{\infty}E_{\alpha,\alpha}\left(-\lambda_l^{\beta}(t-s)^\alpha\right)
\varphi_l(\cdot)\varphi_l(y)\nonumber\\
&\relphantom{=}{}\cdot(g(s,u_1(s,y))-g(s,u_2(s,y)))d\mathbb{W}(s,y)\bigg{\|}_{\mathbb{H}^{2\gamma}}^2\nonumber &
\end{flalign}
\begin{flalign}
&=C\mathbb{E}\sum_{k=1}^{\infty}\lambda_k^{2\gamma}\bigg(\int_0^t\int_{\mathcal{D}}
(t-s)^{\alpha-1}e^{-\theta t} e^{-\nu(t-s)} \sum_{l=1}^{\infty}E_{\alpha,\alpha}\left(-\lambda_l^{\beta}(t-s)^\alpha\right)
\nonumber\\
&\relphantom{=}{}\cdot\varphi_l(\cdot)\varphi_l(y)(g(s,u_1(s,y))-g(s,u_2(s,y)))d\mathbb{W}(s,y),\varphi_k\bigg)^2\nonumber &
\end{flalign}
\begin{flalign}
&=C\mathbb{E}\sum_{k=1}^{\infty} \lambda_k^{2\gamma} \bigg{|}\int_{0}^t\int_{\mathcal{D}}(t-s)^{\alpha-1}e^{-\theta t}e^{-\nu(t-s)}
E_{\alpha,\alpha}\left(-\lambda_k^{\beta}(t-s)^\alpha\right)
\varphi_k(y)\nonumber\\
&\relphantom{=}{}\cdot
\sum\limits_{j,l=1}^{\infty}((g(s,u_1(s))-g(s,u_2(s)))\cdot e_l,\varphi_j)\varphi_j(y)\varsigma_l(s)dyd\xi_l(s)\bigg{|}^2\nonumber &
\end{flalign}
\begin{flalign}
&=C\mathbb{E}\sum_{k=1}^{\infty} \bigg{|}\int_{0}^t(t-s)^{\alpha-1}e^{-\theta t}  e^{-\nu(t-s)}
E_{\alpha,\alpha}\left(-\lambda_k^{\beta}(t-s)^\alpha\right)\lambda_k^\gamma \\
&\relphantom{=}{}\cdot
\sum\limits_{l=1}^{\infty}((g(s,u_1(s))-g(s,u_2(s)))\cdot e_l,\varphi_k)\varsigma_l(s)d\xi_l(s)\bigg{|}^2\nonumber &
\end{flalign}
\begin{flalign}
&\leq C\mathbb{E}\sum_{k,l=1}^{\infty}\int_{0}^t(t-s)^{2\alpha-4} e^{-2\theta t}e^{-2\nu(t-s)}\frac{(\lambda_k^\beta(t-s)^\alpha)^{\frac{2}{\alpha}}}{(1+\lambda_k^\beta(t-s)^\alpha)^2}
\lambda_k^{2\gamma-\frac{2\beta}{\alpha}}\nonumber\\
&\relphantom{=}{}\cdot
((g(s,u_1(s))-g(s,u_2(s)))\cdot e_l,\varphi_k)^2|\varsigma_l(s)|^2ds\nonumber &
\end{flalign}
\begin{flalign}
&\leq C\int_{0}^t(t-s)^{2\alpha-4}e^{-2\theta t} e^{-2\nu(t-s)}
\mathbb{E}\|(-\Delta)^{\gamma-\frac{\beta}{\alpha}}(g(s,u_1(s))
\nonumber\\
&\relphantom{=}{}
-g(s,u_2(s)))\|_{\mathcal{L}_2^0}^2ds\nonumber &
\end{flalign}
\begin{flalign}
&\leq C\int_{0}^t(t-s)^{2\alpha-4}e^{-2\theta t} e^{-2\nu(t-s)}\mathbb{E}{\|}u_1(s)-u_2(s){\|}_{\mathbb{H}^{2\gamma}}^2ds. \nonumber &
\end{flalign}
Arguing as in the proof of \eqref{2.3} and noticing that $\{\xi_l^H\}_{l=1}^\infty$ is a family of independent one-dimensional fractional Brownian motions, we deduce from Lemmas \ref{le2.1} and \ref{le2.10}, \eqref{eq.2.9}, \eqref{eq2.2}, $(\bf{A}_1)$, and $(\bf{A}_2)$ that
\begin{flalign}\label{2.4}
\mathcal{Z}_3&= C\mathbb{E}\bigg{\|}e^{-\theta t} \int_{0}^t\int_{\mathcal{D}}(t-s)^{\alpha-1}e^{-\nu(t-s)}
\sum_{l=1}^{\infty}E_{\alpha,\alpha}\left(-\lambda_l^{\beta}(t-s)^\alpha\right)
\varphi_l(\cdot)\varphi_l(y)\nonumber\\
&\relphantom{=}{}\cdot (h(s,u_1(s,y))-h(s,u_2(s,y)))d\mathbb{W}^H(s,y)\bigg{\|}_{\mathbb{H}^{2\gamma}}^2\nonumber &
\end{flalign}
\begin{flalign}
&=C\mathbb{E}\sum_{k=1}^{\infty} \lambda_k^{2\gamma}
\bigg(\int_0^t\int_{\mathcal{D}}(t-s)^{\alpha-1}e^{-\theta t}e^{-\nu(t-s)}\sum_{l=1}^{\infty}E_{\alpha,\alpha}\left(-\lambda_l^{\beta}(t-s)^\alpha\right)
\nonumber\\
&\relphantom{=}{}\cdot \varphi_l(\cdot)\varphi_l(y) (h(s,u_1(s,y))-h(s,u_2(s,y)))d\mathbb{W}^H(s,y),\varphi_k\bigg)^2\nonumber &
\end{flalign}
\begin{flalign}
&=C\mathbb{E}\sum_{k=1}^{\infty} \lambda_k^{2\gamma} \bigg{|}\int_{0}^t\int_{\mathcal{D}}(t-s)^{\alpha-1} e^{-\theta t} e^{-\nu(t-s)}
E_{\alpha,\alpha}\left(-\lambda_k^{\beta}(t-s)^\alpha\right)
\varphi_k(y)\nonumber\\
&\relphantom{=}{}\cdot
\sum\limits_{j,l=1}^{\infty}((h(s,u_1(s))-h(s,u_2(s)))\cdot e_l,\varphi_j)\varphi_j(y)\varrho_l(s)dyd\xi_l^H(s)\bigg{|}^2\nonumber &
\end{flalign}
\begin{flalign}
&=C\mathbb{E}\sum_{k=1}^{\infty} \bigg{|}\int_{0}^t(t-s)^{\alpha-1}e^{-\theta t} e^{-\nu(t-s)}
E_{\alpha,\alpha}\left(-\lambda_k^{\beta}(t-s)^\alpha\right)\lambda_k^\gamma\nonumber\\
&\relphantom{=}{}\cdot
\sum\limits_{l=1}^{\infty}((h(s,u_1(s))-h(s,u_2(s)))\cdot e_l,\varphi_k)\varrho_l(s)d\xi_l^H(s)\bigg{|}^2
&
\end{flalign}
\begin{flalign}
&\leq CHT^{2H-1}\mathbb{E}\sum_{k,l=1}^{\infty}\int_{0}^t(t-s)^{2\alpha-4}e^{-2\theta t} e^{-2\nu(t-s)}\frac{(\lambda_k^{\beta}(t-s)^\alpha)^{\frac{2}{\alpha}}}{(1+\lambda_k^{\beta}(t-s)^\alpha)^2}
\nonumber\\
&\relphantom{=}{}\cdot \lambda_k^{2\gamma-\frac{2\beta}{\alpha}}
((h(s,u_1(s))-h(s,u_2(s)))\cdot e_l,\varphi_k)^2|\varrho_l(s)|^2ds\nonumber &
\end{flalign}
\begin{flalign}
&\leq C\int_{0}^t(t-s)^{2\alpha-4}e^{-2\theta t} e^{-2\nu(t-s)}
\nonumber\\
&\relphantom{=}{}\cdot
\mathbb{E}
\|(-\Delta)^{\gamma-\frac{\beta}{\alpha}}(h(s,u_1(s))-h(s,u_2(s)))\|_{\mathcal{L}_2^0}^2ds\nonumber&
\end{flalign}
\begin{flalign}
&\leq C\int_{0}^t(t-s)^{2\alpha-4}e^{-2\theta t} e^{-2\nu(t-s)}\mathbb{E}{\|}u_1(s)-u_2(s){\|}_{\mathbb{H}^{2\gamma}}^2ds. \nonumber  &
\end{flalign}

Then it follows from \eqref{2.1}-\eqref{2.4} that
\begin{flalign*}
\displaystyle &\mathbb{E}\|e^{-\theta t} (\mathcal{P}u_1(t)-\mathcal{P}u_2(t))\|_{\mathbb{H}^{2\gamma}}^2\nonumber\\
&\relphantom{}{} \leq C
\int_{0}^t(t-s)^{2\alpha-4}e^{-2(\theta+\nu)(t-s)}\mathbb{E}{\|}e^{-\theta s} (u_1(s)-u_2(s)){\|}_{\mathbb{H}^{2\gamma}}^2ds\nonumber\\
&\relphantom{}{}\leq C\int_{0}^t(t-s)^{2\alpha-4}e^{-2(\theta+\nu)(t-s)}\max_{s\in[0,T]}
\mathbb{E}{\|}e^{-\theta s} (u_1(s)-u_2(s)){\|}_{\mathbb{H}^{2\gamma}}^2ds\nonumber\\
&\relphantom{}{}=C\theta^{-(2\alpha-3)}\int_0^1(1-\tau)^{2\alpha-4}(\theta t)^{2\alpha-3}e^{-2\theta t(1-\tau)}e^{-2\nu t(1-\tau)}d\tau\|u_1-u_2\|_{\theta}^2\nonumber\\
&\relphantom{}{}\leq C\max_{\theta>0,T\geq t\geq 0,\tau\in[0,1]}
\left([(1-\tau)\theta t]^{\alpha-\frac{3}{2}}e^{-2\theta t(1-\tau)}\right)
\frac{t^{\alpha-\frac{3}{2}}}{\theta^{\alpha-\frac{3}{2}}}
\int_0^1(1-\tau)^{\alpha-\frac{5}{2}}d\tau
\nonumber\\
&\relphantom{=}{}
\cdot\|u_1-u_2\|_\theta^2\nonumber\\
&\relphantom{}{}=C\frac{t^{\alpha-\frac{3}{2}}}{\theta^{\alpha-\frac{3}{2}}}
\|u_1-u_2\|_\theta^2\quad \forall u_1,\, u_2\in C([0,T]; L^2(\Omega;\mathbb{H}^{2\gamma}(\mathcal{D}))).
\end{flalign*}

Step 3. We show that for any $\delta\in \left(0,2\beta-\frac{3\beta}{\alpha}\right)$,
\[t^{\frac{\alpha\delta}{2\beta}}u\in C([0,T];L^2(\Omega;\mathbb{H}^{2\gamma+\delta}(\mathcal{D})))\]
with value zero at $t=0$.

Thanks to Lemma \ref{le2.2} and Remark \ref{re3.4}, it follows from \eqref{eq2.1} that
\begin{align}\label{2.5}
& t^{\frac{\alpha\delta}{\beta}}\mathbb{E}\|u(t)\|_{\mathbb{H}^{2\gamma+\delta}}^2
\nonumber\\
&\leq Ct^{\frac{\alpha\delta}{\beta}}\mathbb{E}\bigg{\|} \int_{\mathcal{D}}\mathcal{T}^\nu_{\alpha,\beta}(t,\cdot,y)a(y)dy\bigg{\|}_{
\mathbb{H}^{2\gamma+\delta}}^2+Ct^{\frac{\alpha\delta}{\beta}}\mathbb{E}\bigg{\|} \int_{\mathcal{D}}\mathcal{R}^\nu_{\alpha,\beta}(t,\cdot,y) b(y)dy\bigg{\|}_{\mathbb{H}^{2\gamma+\delta}}^2\nonumber\\
&\relphantom{=}{}+ Ct^{\frac{\alpha\delta}{\beta}}\mathbb{E}\bigg{\|} \int_{0}^t\int_{\mathcal{D}}\mathcal{S}^\nu_{\alpha,\beta}(t-s,\cdot,y)
f(s,u(s,y))dyds\bigg{\|}_{\mathbb{H}^{2\gamma+\delta}}^2\nonumber\\
&\relphantom{=}{}+Ct^{\frac{\alpha\delta}{\beta}}\mathbb{E}\bigg{\|} \int_{0}^t\int_{\mathcal{D}}\mathcal{S}^\nu_{\alpha,\beta}(t-s,\cdot,y)
g(s,u(s,y))d{\mathbb{W}}(s,y)\bigg{\|}_{\mathbb{H}^{2\gamma+\delta}}^2\nonumber\\
&\relphantom{=}{}+Ct^{\frac{\alpha\delta}{\beta}}\mathbb{E}\bigg{\|} \int_{0}^t\int_{\mathcal{D}}\mathcal{S}^\nu_{\alpha,\beta}(t-s,\cdot,y)
h(s,u(s,y))d{\mathbb{W}}^H(s,y)\bigg{\|}_{\mathbb{H}^{2\gamma+\delta}}^2\nonumber\\
&\leq C(1+\nu^2t^2)e^{-2\nu t}t^{-\frac{\alpha(2\gamma-2\widetilde{\gamma})}{\beta}}\mathbb{E}\|a\|_{\mathbb{H}^{2\widetilde{\gamma}}}^2
+Ce^{-2\nu t}t^{-\frac{\alpha(2\gamma-2\widetilde{\gamma})}{\beta}}\mathbb{E}\|b\|_{\mathbb{H}^{2\widetilde{\gamma}-\frac{2\beta}{\alpha}}}^2\nonumber\\
&\relphantom{=}{}+M_1+M_2+M_3.
\end{align}
Arguing as in the proof of \eqref{2.2}, \eqref{2.3}, and \eqref{2.4}, noticing that $\{\varphi_l\}_{l=1}^{\infty}$ is an orthonormal basis in $L^2(\mathcal{D})$, $\{\xi_l\}_{l=1}^{\infty}$ is a family of independent one-dimensional standard Brownian motions, and $\{\xi_l^H\}_{l=1}^{\infty}$ is a family of independent one-dimensional fractional Brownian motions, we deduce from Lemmas \ref{le2.1} and \ref{le2.10}, Proposition \ref{pro1}, \eqref{eq.2.8}-\eqref{eq.2.9}, \eqref{eq2.2}, $(\bf{A}_1)$-$(\bf{A}_2)$, and H\"{o}lder's inequality that
\begin{flalign}\label{2.6}
M_1&=Ct^{\frac{\alpha\delta}{\beta}}\mathbb{E}\bigg{\|} \int_{0}^t\int_{\mathcal{D}}(t-s)^{\alpha-1}e^{-\nu (t-s)}\sum\limits_{l=1}^{\infty}E_{\alpha,\alpha}(-\lambda_l^\beta (t-s)^\alpha)\varphi_l(\cdot)\varphi_l(y)
\nonumber\\
&\relphantom{=}{}
\cdot f(s,u(s,y))dyds\bigg{\|}_{\mathbb{H}^{2\gamma+\delta}}^2\nonumber
&
\end{flalign}
\begin{flalign}
&=Ct^{\frac{\alpha\delta}{\beta}}
\mathbb{E}\sum\limits_{k=1}^{\infty}\lambda_k^{2\gamma+\delta}\bigg(\int_{0}^t\int_{\mathcal{D}}(t-s)^{\alpha-1}e^{-\nu (t-s)}\sum_{l=1}^{\infty}E_{\alpha,\alpha}(-\lambda_l^\beta (t-s)^\alpha)\nonumber\\
&\relphantom{=}{}\cdot \varphi_l(\cdot)\varphi_l(y) f(s,u(s,y))dyds,\varphi_k\bigg)^2 &
\end{flalign}
\begin{flalign}
&=Ct^{\frac{\alpha\delta}{\beta}}
\mathbb{E}\sum\limits_{k=1}^{\infty}\lambda_k^{2\gamma+\delta}\bigg|\int_{0}^t\int_{\mathcal{D}}(t-s)^{\alpha-1}e^{-\nu (t-s)}E_{\alpha,\alpha}(-\lambda_k^\beta (t-s)^\alpha)
\nonumber\\
&\relphantom{=}{}\cdot
\varphi_k(y)f(s,u(s,y))dyds\bigg|^2\nonumber
&
\end{flalign}
\begin{flalign}
&\leq Ct^{\frac{\alpha\delta}{\beta}} \mathbb{E}\int_0^t(t-s)^{2\alpha-2}e^{-2\nu(t-s)}\sum\limits_{k=1}^{\infty}\lambda_k^{\delta}
\left|E_{\alpha,\alpha}(-\lambda_k^\beta (t-s)^\alpha)\right|^2
\nonumber\\
&\relphantom{=}{}\cdot
\lambda_k^{2\gamma}(f(s,u(s)),\varphi_k)^2ds\nonumber
&
\end{flalign}
\begin{flalign}
&\leq Ct^{\frac{\alpha\delta}{\beta}} \mathbb{E}\int_0^t(t-s)^{2\alpha-\frac{\alpha\delta}{\beta}-4}e^{-2\nu(t-s)}
\sum\limits_{k=1}^{\infty}\frac{(\lambda_k^\beta(t-s)^{\alpha})^{\frac{\delta}{\beta}+\frac{2}{\alpha}}}
{(1+\lambda_k^{\beta}(t-s)^\alpha)^2}
\lambda_k^{2\gamma-\frac{2\beta}{\alpha}}
\nonumber\\
&\relphantom{=}{}\cdot
(f(s,u(s)),\varphi_k)^2ds\nonumber
&
\end{flalign}
\begin{flalign}
&\leq Ct^{\frac{\alpha\delta}{\beta}}\int_0^t(t-s)^{2\alpha-\frac{\alpha\delta}{\beta}-4}
\left(1+\mathbb{E}\|u(s)\|_{\mathbb{H}^{2\gamma}}^2\right)ds, \nonumber &
\end{flalign}

\begin{flalign}\label{2.8}
M_2&= Ct^{\frac{\alpha\delta}{\beta}}\mathbb{E}\bigg{\|} \int_{0}^t\int_{\mathcal{D}}(t-s)^{\alpha-1}e^{-\nu (t-s)}\sum\limits_{l=1}^{\infty}E_{\alpha,\alpha}(-\lambda_l^\beta (t-s)^\alpha)\varphi_l(\cdot)\varphi_l(y)\nonumber\\
&\relphantom{=}{}\cdot g(s,u(s,y))d\mathbb{W}(s,y)\bigg{\|}_{\mathbb{H}^{2\gamma+\delta}}^2\nonumber &
\end{flalign}
\begin{flalign}
&=Ct^{\frac{\alpha\delta}{\beta}}
\mathbb{E}\sum\limits_{k=1}^{\infty}\lambda_k^{2\gamma+\delta}\bigg(\int_{0}^t\int_{\mathcal{D}}(t-s)^{\alpha-1}e^{-\nu (t-s)}\sum_{l=1}^{\infty}E_{\alpha,\alpha}(-\lambda_l^\beta (t-s)^\alpha)\nonumber\\
&\relphantom{=}{}\cdot \varphi_l(\cdot)\varphi_l(y) g(s,u(s,y))d\mathbb{W}(s,y),\varphi_k\bigg)^2\nonumber &
\end{flalign}
\begin{flalign}
&=Ct^{\frac{\alpha\delta}{\beta}}
\mathbb{E}\sum\limits_{k=1}^{\infty}\lambda_k^{2\gamma+\delta}\bigg|\int_{0}^t\int_{\mathcal{D}}(t-s)^{\alpha-1}e^{-\nu (t-s)}E_{\alpha,\alpha}(-\lambda_k^\beta (t-s)^\alpha)\varphi_k(y)\nonumber  \\
&\relphantom{=}{}\cdot \sum_{j,l=1}^{\infty}(g(s,u(s,y))\cdot e_l,\varphi_j)\varphi_j(y)\varsigma_l(s)dyd\xi_l(s)\bigg|^2\nonumber &
\end{flalign}
\begin{flalign}
&= Ct^{\frac{\alpha\delta}{\beta}} \mathbb{E}\sum\limits_{k=1}^{\infty}\bigg|\int_0^t(t-s)^{\alpha-1}e^{-\nu(t-s)}
E_{\alpha,\alpha}(-\lambda_k^\beta (t-s)^\alpha)\lambda_k^{\gamma+\frac{\delta}{2}}\nonumber\\
&\relphantom{=}{}\cdot\sum_{l=1}^{\infty}(g(s,u(s))\cdot e_l,\varphi_k)\varsigma_l(s)d\xi_l(s)\bigg|^2 &
\end{flalign}
\begin{flalign}
&\leq Ct^{\frac{\alpha\delta}{\beta}} \mathbb{E}\sum\limits_{k,l=1}^{\infty}\int_0^t(t-s)^{2\alpha-\frac{\alpha\delta}{\beta}-4}
e^{-2\nu(t-s)}\frac{(\lambda_k^\beta(t-s)^{\alpha})^{\frac{\delta}{\beta}+\frac{2}{\alpha}}}
{(1+\lambda_k^{\beta}(t-s)^\alpha)^2}
\lambda_k^{2\gamma-\frac{2\beta}{\alpha}}\nonumber \\
&\relphantom{=}{}\cdot(g(s,u(s))\cdot e_l,\varphi_k)^2|\varsigma_l(s)|^2ds\nonumber &
\end{flalign}
\begin{flalign}
&\leq Ct^{\frac{\alpha\delta}{\beta}}\int_0^t(t-s)^{2\alpha-\frac{\alpha\delta}{\beta}-4}
\mathbb{E}\|(-\Delta)^{\gamma-\frac{\beta}{\alpha}}g(s,u(s))\|_{\mathcal{L}_2^0}^2ds\nonumber &
\end{flalign}
\begin{flalign}
&\leq Ct^{\frac{\alpha\delta}{\beta}}\int_0^t(t-s)^{2\alpha-\frac{\alpha\delta}{\beta}-4}
\left(1+\mathbb{E}\|u(s)\|_{\mathbb{H}^{2\gamma}}^2\right)ds, \nonumber &
\end{flalign}
and
\begin{flalign}\label{2.11}
M_3&= Ct^{\frac{\alpha\delta}{\beta}}\mathbb{E}\bigg{\|} \int_{0}^t\int_{\mathcal{D}}(t-s)^{\alpha-1}e^{-\nu (t-s)}\sum\limits_{l=1}^{\infty}E_{\alpha,\alpha}(-\lambda_l^\beta (t-s)^\alpha)\varphi_l(\cdot)\varphi_l(y)\nonumber\\
&\relphantom{=}{}\cdot h(s,u(s,y))d\mathbb{W}^H(s,y)\bigg{\|}_{\mathbb{H}^{2\gamma+\delta}}^2\nonumber &
\end{flalign}
\begin{flalign}
&=Ct^{\frac{\alpha\delta}{\beta}}
\mathbb{E}\sum\limits_{k=1}^{\infty}\lambda_k^{2\gamma+\delta}\bigg(\int_{0}^t\int_{\mathcal{D}}(t-s)^{\alpha-1}e^{-\nu (t-s)}\sum_{l=1}^{\infty}E_{\alpha,\alpha}(-\lambda_l^\beta (t-s)^\alpha)\nonumber\\
&\relphantom{=}{}\cdot \varphi_l(\cdot)\varphi_l(y) h(s,u(s,y))d\mathbb{W}^H(s,y),\varphi_k\bigg)^2\nonumber
&
\end{flalign}
\begin{flalign}
&=Ct^{\frac{\alpha\delta}{\beta}}
\mathbb{E}\sum\limits_{k=1}^{\infty}\lambda_k^{2\gamma+\delta}\bigg|\int_{0}^t\int_{\mathcal{D}}
(t-s)^{\alpha-1}e^{-\nu (t-s)}
E_{\alpha,\alpha}(-\lambda_k^\beta (t-s)^\alpha)
\varphi_k(y)\nonumber\\
&\relphantom{=}{}\cdot\sum_{j,l=1}^{\infty}(h(s,u(s,y))\cdot e_l,\varphi_j)
\varphi_j(y)\varrho_l(s)dyd\xi_l^H(s)\bigg|^2\nonumber
&
\end{flalign}
\begin{flalign}
&= Ct^{\frac{\alpha\delta}{\beta}} \mathbb{E}\sum\limits_{k=1}^{\infty}\bigg|\int_0^t(t-s)^{\alpha-1}e^{-\nu(t-s)}
E_{\alpha,\alpha}(-\lambda_k^\beta (t-s)^\alpha)\lambda_k^{\gamma+\frac{\delta}{2}}\nonumber\\
&\relphantom{=}{}\cdot\sum_{l=1}^{\infty}(h(s,u(s))\cdot e_l,\varphi_k)\varrho_l(s)d\xi_l^H(s)\bigg|^2
&
\end{flalign}
\begin{flalign}
&\leq CHT^{2H-1}t^{\frac{\alpha\delta}{\beta}} \mathbb{E}\sum\limits_{k,l=1}^{\infty}\int_0^t(t-s)^{2\alpha-\frac{\alpha\delta}{\beta}-4}e^{-2\nu(t-s)}
\frac{(\lambda_k^\beta(t-s)^{\alpha})^{\frac{\delta}{\beta}+\frac{2}{\alpha}}}
{(1+\lambda_k^{\beta}(t-s)^\alpha)^2}
\nonumber\\
&\relphantom{=}{}\cdot\lambda_k^{2\gamma-\frac{2\beta}{\alpha}} (h(s,u(s))\cdot e_l,\varphi_k)^2|\varrho_l(s)|^2ds\nonumber &
\end{flalign}
\begin{flalign}
&\leq Ct^{\frac{\alpha\delta}{\beta}}\int_0^t(t-s)^{2\alpha-\frac{\alpha\delta}{\beta}-4}
\mathbb{E}\|(-\Delta)^{\gamma-\frac{\beta}{\alpha}}h(s,u(s))\|_{\mathcal{L}_2^0}^2ds\nonumber &
\end{flalign}
\begin{flalign}
&\leq Ct^{\frac{\alpha\delta}{\beta}}\int_0^t(t-s)^{2\alpha-\frac{\alpha\delta}{\beta}-4}
\left(1+\mathbb{E}\|u(s)\|_{\mathbb{H}^{2\gamma}}^2\right)ds, \nonumber &
\end{flalign}
where we have used $\frac{(\lambda_k^\beta(t-s)^{\alpha})^{\frac{\delta}{\beta}+\frac{2}{\alpha}}}
{(1+\lambda_k^{\beta}(t-s)^\alpha)^2}\leq C$ and $\mathbb{E}\|u(s)\|_{\mathbb{H}^{2\gamma-\frac{\beta}{\alpha}}}\leq C\mathbb{E}\|u(s)\|_{\mathbb{H}^{2\gamma}}^2$.

Noticing that $u\in C([0,T];L^2(\Omega;\mathbb{H}^{2\gamma}(\mathcal{D})))$, by the similar arguments in Step 1, we obtain from \eqref{2.6}-\eqref{2.11} that for all $t\in [0,T]$,
\begin{align}
t^{\frac{\alpha\delta}{\beta}}\mathbb{E}\|u(t)\|_{\mathbb{H}^{2\gamma+\delta}}^2
&\leq C(1+\nu^2t^2)e^{-2\nu t}
t^{-\frac{\alpha(2\gamma-2\widetilde{\gamma})}{\beta}}\mathbb{E}\|a\|_{\mathbb{H}^{2\widetilde{\gamma}}}^2
\nonumber\\
&\relphantom{=}{}
+Ce^{-2\nu t}t^{-\frac{\alpha(2\gamma-2\widetilde{\gamma})}{\beta}}
\mathbb{E}\|b\|_{\mathbb{H}^{2\widetilde{\gamma}-\frac{2\beta}{\alpha}}}^2\nonumber\\
&\relphantom{=}{}+Ct^{\frac{\alpha\delta}{\beta}}\int_0^t(t-s)^{2\alpha-\frac{\alpha\delta}{\beta}-4}
\left(1+\mathbb{E}\|u(s)\|_{\mathbb{H}^{2\gamma}}^2\right)ds\nonumber\\
&\leq Ct^{\frac{-\alpha(2\gamma-2\widetilde{\gamma})}{\beta}}+Ct^{2\alpha-3},
\end{align}
and $t^{\frac{\alpha\delta}{2\beta}}u\in C([0,T];L^2(\Omega;\mathbb{H}^{2\gamma+\delta}(\mathcal{D})))$.

It follows from Remark \ref{re3.4} that the operators
\begin{equation}\label{eq.3.25}
t^{\frac{\alpha\delta}{\beta}}\mathbb{E}\bigg{\|}\int_{\mathcal{D}}\mathcal{T}_{\alpha,\beta}^{\nu}
(t,\cdot,y)a(y)dy\bigg{\|}_{\mathbb{H}^{2\gamma+\delta}}^2\leq C\mathbb{E}\|a\|_{\mathbb{H}^{2\widetilde{\gamma}}}^2,
\end{equation}
\begin{equation}\label{eq.3.26}
t^{\frac{\alpha\delta}{\beta}}\mathbb{E}\bigg{\|}\int_{\mathcal{D}}\mathcal{R}_{\alpha,\beta}^{\nu}
(t,\cdot,y)b(y)dy\bigg{\|}_{\mathbb{H}^{2\gamma+\delta}}^2\leq C\mathbb{E}\|b\|_{\mathbb{H}^{2\widetilde{\gamma}-\frac{2\beta}{\alpha}}}^2,
\end{equation}
and
\begin{align}\label{eq.3.27}
& t^{\frac{\alpha\delta}{\beta}}\mathbb{E}\bigg{\|}\int_{\mathcal{D}}\mathcal{T}_{\alpha,\beta}^{\nu}
(t,\cdot,y)a(y)dy\bigg{\|}_{\mathbb{H}^{2\gamma+\delta}}^2
\\
& \leq Ct^{\frac{\alpha\delta}{\beta}}t^{-\frac{\alpha(2\gamma-2\widetilde{\gamma})}{\beta}}\mathbb{E}\|a\|_{\mathbb{H}^{2\widetilde{\gamma}+\delta}}^2 \rightarrow 0 ~~\mbox{as}~~t\rightarrow 0, \nonumber
\end{align}
\begin{align}\label{eq.3.28}
& t^{\frac{\alpha\delta}{\beta}}\mathbb{E}\bigg{\|}\int_{\mathcal{D}}\mathcal{T}_{\alpha,\beta}^{\nu}
(t,\cdot,y)b(y)dy\bigg{\|}_{\mathbb{H}^{2\gamma+\delta}}^2
\\
& \leq Ct^{\frac{\alpha\delta}{\beta}}t^{-\frac{\alpha(2\gamma-2\widetilde{\gamma})}{\beta}}
\mathbb{E}\|b\|_{\mathbb{H}^{2\widetilde{\gamma}+\delta-\frac{2\beta}{\alpha}}}^2
\rightarrow 0 ~~\mbox{as}~~t\rightarrow 0.
\end{align}
Since $\mathbb{H}^{2\widetilde{\gamma}+\delta}$ and $\mathbb{H}^{2\widetilde{\gamma}+\delta-\frac{2\beta}{\alpha}}$ are, respectively, dense subsets of $\mathbb{H}^{2\widetilde{\gamma}}$ and $\mathbb{H}^{2\widetilde{\gamma}-\frac{2\beta}{\alpha}}$, we find that $t^{\frac{\alpha\delta}{2\beta}}u(t)$ vanishes at $t=0$.

Step 4. We prove that for $0\leq \theta_1\leq \theta_2\leq T$,

\begin{align*}
& \mathbb{E}\|u(\theta_2)-u(\theta_1)\|^2
\\
& \leq C|\theta_2-\theta_1|^{2\alpha-2}\left(\mathbb{E}\|a\|_{\mathbb{H}^{2\widetilde{\gamma}}}^2
+\mathbb{E}\|b\|_{\mathbb{H}^{2\widetilde{\gamma}-\frac{2\beta}{\alpha}}}^2
+\sup_{s\in[0,T]}\mathbb{E}\left(1+\|u(s)\|_{\mathbb{H}^{2\gamma}}^2
\right)\right).
\end{align*}

Arguing as in the proof of Step 1, for $0\leq \theta_1\leq \theta_2\leq T$ we have 
\begin{flalign}\label{eq.3.29}
& \mathbb{E}\|u(\theta_2)-u(\theta_1)\|^2 \nonumber &
\end{flalign}
\begin{flalign}
&\leq C\mathbb{E}\bigg{\|}\int_{\mathcal{D}}\mathcal{T}_{\alpha,\beta}^{\nu}(\theta_2,\cdot,y)a(y)dy-
\int_{\mathcal{D}}\mathcal{T}_{\alpha,\beta}^{\nu}(\theta_1,\cdot,y)a(y)dy\bigg{\|}^2
\nonumber &
\end{flalign}
\begin{flalign}
&\relphantom{=}{}+C\mathbb{E}\bigg{\|}\int_{\mathcal{D}}\mathcal{R}_{\alpha,\beta}^{\nu}(\theta_2,\cdot,y)b(y)dy-
\int_{\mathcal{D}}\mathcal{R}_{\alpha,\beta}^{\nu}(\theta_1,\cdot,y)b(y)dy\bigg{\|}^2\nonumber &
\end{flalign}
\begin{flalign}
&\relphantom{=}{}+C\mathbb{E}\bigg{\|}\int_{0}^{\theta_2}\int_{\mathcal{D}}\mathcal{S}^\nu_{\alpha,\beta}(\theta_2-s,\cdot,y)
f(s,u(s,y))dyds  \nonumber\\
&\relphantom{=}{}-\int_{0}^{\theta_1}\int_{\mathcal{D}}\mathcal{S}^\nu_{\alpha,\beta}(\theta_1-s,\cdot,y)
f(s,u(s,y))dyds\bigg{\|}^2
\nonumber &
\end{flalign}
\begin{flalign}
&\relphantom{=}{}+C\mathbb{E}\bigg{\|}\int_{0}^{\theta_2}\int_{\mathcal{D}}\mathcal{S}^\nu_{\alpha,\beta}(\theta_2-s,\cdot,y)
g(s,u(s,y))d{\mathbb{W}}(s,y) \nonumber \\
&\relphantom{=}{}-\int_{0}^{\theta_1}\int_{\mathcal{D}}\mathcal{S}^\nu_{\alpha,\beta}(\theta_1-s,\cdot,y)
g(s,u(s,y))d{\mathbb{W}}(s,y)\bigg{\|}^2\nonumber &
\end{flalign}
\begin{flalign}
&\relphantom{=}{}+C\mathbb{E}\bigg{\|}\int_{0}^{\theta_2}\int_{\mathcal{D}}\mathcal{S}^\nu_{\alpha,\beta}(\theta_2-s,\cdot,y)
h(s,u(s,y))d{\mathbb{W}^H}(s,y)\nonumber\\
&\relphantom{=}{}-\int_{0}^{\theta_1}\int_{\mathcal{D}}\mathcal{S}^\nu_{\alpha,\beta}(\theta_1-s,\cdot,y)
h(s,u(s,y))d{\mathbb{W}^H}(s,y)\bigg{\|}^2\nonumber\\
&:=\mathcal{E}_1+\mathcal{E}_2+\mathcal{E}_3+\mathcal{E}_4+\mathcal{E}_5. &
\end{flalign}
Since $\{\varphi_k\}_{k=1}^{\infty}$ is an orthonormal basis in $L^2(\mathcal{D})$, by Lemma \ref{le2.1}, and Lagrange's mean value theorem, we obtain
\begin{flalign}\label{eq.3.30}
&\mathcal{E}_1
 \nonumber\\
& =C\mathbb{E}\sum_{k=1}^{\infty}\bigg(\int_{\mathcal{D}}
e^{-\nu\theta_2}\sum_{l=1}^{\infty}\left(E_{\alpha,1}(-\lambda_l^\beta\theta_2^\alpha)
+\nu\theta_2E_{\alpha,2}(-\lambda_l^\beta\theta_2^\alpha)\right)  \nonumber\\
& \relphantom{=}{} \varphi_l(\cdot)\varphi_l(y)a(y)dy
-\int_{\mathcal{D}}e^{-\nu\theta_1}\sum_{l=1}^{\infty}\left(E_{\alpha,1}(-\lambda_l^\beta\theta_1^\alpha)
+\nu\theta_1E_{\alpha,2}(-\lambda_l^\beta\theta_1^\alpha)\right) \nonumber\\
& \relphantom{=}{} \varphi_l(\cdot)\varphi_l(y)a(y)dy,\varphi_k
\bigg)^2 &
\end{flalign}

\begin{flalign*}
&\leq C\mathbb{E}\sum_{k=1}^{\infty}\bigg|\int_{\mathcal{D}}\left(e^{-\nu\theta_2}E_{\alpha,1}(-\lambda_k^\beta\theta_2^\alpha)
-e^{-\nu\theta_1}E_{\alpha,1}(-\lambda_k^\beta\theta_1^\alpha)\right)\varphi_k(y)a(y)dy\bigg|^2\nonumber\\
&\relphantom{=}{}+C\nu^2\mathbb{E}\sum_{k=1}^{\infty}\bigg|\int_{\mathcal{D}}\left(
e^{-\nu\theta_2}\theta_2E_{\alpha,2}(-\lambda_k^\beta\theta_2^\alpha)-e^{-\nu\theta_1}\theta_1E_{\alpha,2}(-\lambda_k^\beta\theta_1^\alpha)\right)
\nonumber\\
&\relphantom{=}{}
\cdot\varphi_k(y)a(y)dy\bigg|^2\nonumber &
\end{flalign*}

\begin{flalign*}
&\leq C|e^{-\nu\theta_2}-e^{-\nu\theta_1}|^2\mathbb{E}\sum_{k=1}^{\infty}
\left|E_{\alpha,1}(-\lambda_k^\beta\theta_2^\alpha)\right|^2(a,\varphi_k)^2\nonumber\\
&\relphantom{=}{}
+Ce^{-2\nu\theta_1}\mathbb{E}\sum_{k=1}^{\infty}\left|E_{\alpha,1}(-\lambda_k^\beta\theta_2^\alpha)
-E_{\alpha,1}(-\lambda_k^\beta\theta_1^\alpha)\right|^2(a,\varphi_k)^2\nonumber\\
&\relphantom{=}{}
+C\nu^2|e^{-\nu\theta_2}-e^{-\nu\theta_1}|^2\mathbb{E}\sum_{k=1}^{\infty}\left|\theta_2E_{\alpha,2}(-\lambda_k^\beta\theta_2^\alpha)\right|^2
(a,\varphi_k)^2\nonumber\\
&\relphantom{=}{}
+C\nu^2e^{-2\nu\theta_1}\mathbb{E}\sum_{k=1}^{\infty}\left|\theta_2E_{\alpha,2}(-\lambda_k^\beta\theta_2^\alpha)-\theta_1E_{\alpha,2}(-\lambda_k^\beta\theta_1^\alpha)\right|^2
(a,\varphi_k)^2\nonumber &
\end{flalign*}

\begin{flalign*}
&\leq C|\theta_2-\theta_1|^2\mathbb{E}\|a\|^2+C|\theta_2-\theta_1|^2\mathbb{E}\sum_{k=1}^{\infty}
\left|-\lambda_k^\beta\widehat{\theta}^{\alpha-1}E_{\alpha,\alpha}(-\lambda_k^\beta\widehat{\theta}^\alpha)\right|^2
(a,\varphi_k)^2\nonumber\\
&\relphantom{=}{}+C\nu^2|\theta_2-\theta_1|^2T^2\mathbb{E}\|a\|^2
+C\nu^2|\theta_2-\theta_1|^2\mathbb{E}\sum_{k=1}^{\infty}\left|E_{\alpha,1}(-\lambda_k^\beta\widetilde{\theta}^\alpha)\right|^2
(a,\varphi_k)^2
&
\end{flalign*}

\begin{flalign*}
&\leq C|\theta_2-\theta_1|^2\mathbb{E}\|a\|_{\mathbb{H}^{2\widetilde{\gamma}}}^2
+C|\theta_2-\theta_1|^2\mathbb{E}\sum_{k=1}^{\infty}\left|\frac{(\lambda_k^\beta\widehat{\theta}^\alpha)^{\frac{\alpha-1}{\alpha}}}{1+\lambda_k^\beta\widehat{\theta}^\alpha}\right|^2
\lambda_k^{\frac{2\beta}{\alpha}}(a,\varphi_k)^2
\\
&
\relphantom{=}{}+C|\theta_2-\theta_1|^2\mathbb{E}\|a\|^2\\
&\leq C|\theta_2-\theta_1|^2\mathbb{E}\|a\|_{\mathbb{H}^{2\widetilde{\gamma}}}^2, &
\end{flalign*}
where $\widehat{\theta}$, $\widetilde{\theta}\in (\theta_1,\theta_2)$, and we have used $\frac{(\lambda_k^\beta\widehat{\theta}^\alpha)^{\frac{\alpha-1}{\alpha}}}{1+\lambda_k^\beta\widehat{\theta}^\alpha}
\leq C$, $\mathbb{E}\|a\|^2\leq C\mathbb{E}\|a\|_{\mathbb{H}^{2\widetilde{\gamma}}}^2$, and $\mathbb{E}\|a\|_{\mathbb{H}^{\frac{2\beta}{\alpha}}}^2\leq C\mathbb{E}\|a\|_{\mathbb{H}^{2\widetilde{\gamma}}}^2$.

Similar to \eqref{eq.3.30}, we have
\begin{flalign}\label{eq.3.31}
& \mathcal{E}_2 \nonumber
\\
&=C\mathbb{E}\sum_{k=1}^{\infty}\bigg(\int_{\mathcal{D}}
\theta_2e^{-\nu\theta_2}\sum_{l=1}^{\infty}E_{\alpha,2}(-\lambda_l^\beta\theta_2^\alpha)
\varphi_l(\cdot)\varphi_l(y)b(y)dy-\int_{\mathcal{D}}\theta_1e^{-\nu\theta_1}\nonumber\\
&\relphantom{=}{}\cdot
\sum_{l=1}^{\infty}E_{\alpha,2}(-\lambda_l^\beta\theta_1^\alpha)
\varphi_l(\cdot)\varphi_l(y)b(y)dy,\varphi_k\bigg)^2 &
\end{flalign}

\begin{flalign*}
&\leq C\mathbb{E}\sum_{k=1}^{\infty}\bigg|\int_{\mathcal{D}}\theta_2e^{-\nu\theta_2}E_{\alpha,2}(-\lambda_k^\beta\theta_2^\alpha)\varphi_k(y)b(y)dy
\\
&\relphantom{=}{}
-\int_{\mathcal{D}}\theta_1e^{-\nu\theta_1}E_{\alpha,2}(-\lambda_k^\beta\theta_1^\alpha)\varphi_k(y)b(y)dy\bigg|^2\nonumber &
\end{flalign*}

\begin{flalign*}
&\leq C|e^{-\nu\theta_2}-e^{-\nu\theta_1}|^2\mathbb{E}\sum_{k=1}^{\infty}
\left|\theta_2E_{\alpha,2}(-\lambda_k^\beta\theta_2^\alpha)\right|^2(b,\varphi_k)^2 \\
&\relphantom{=}{}
+Ce^{-2\nu\theta_1}\mathbb{E}\sum_{k=1}^{\infty}\left|\theta_2E_{\alpha,2}(-\lambda_k^\beta\theta_2^\alpha)-\theta_1E_{\alpha,1}(-\lambda_k^\beta\theta_1^\alpha)\right|^2
(b,\varphi_k)^2
&
\end{flalign*}
\begin{flalign*}
&\leq C|e^{-\nu\theta_2}-e^{-\nu\theta_1}|^2\mathbb{E}\sum_{k=1}^{\infty}\frac{(\lambda_k^\beta\theta_2^\alpha)^{\frac{2}{\alpha}}}{(1+\lambda_k^\beta\theta_2^\alpha)^2}
\lambda_k^{-\frac{2\beta}{\alpha}}(b,\varphi_k)^2
\\
&\relphantom{=}{}
+C|\theta_2-\theta_1|^2\mathbb{E}\sum_{k=1}^{\infty}\left|E_{\alpha,1}(-\lambda_k^\beta\widetilde{\theta}^\alpha)\right|^2
(b,\varphi_k)^2\nonumber\\
&\leq C|\theta_2-\theta_1|^2\mathbb{E}\|b\|^2
\leq C|\theta_2-\theta_1|^2\mathbb{E}\|b\|_{\mathbb{H}^{2\widetilde{\gamma}-\frac{2\beta}{\alpha}}}^2, &
\end{flalign*}
where $\widetilde{\theta}\in(\theta_1,\theta_2)$, and we have used $\frac{(\lambda_k^\beta\theta_2^\alpha)^{\frac{2}{\alpha}}}{(1+\lambda_k^\beta\theta_2^\alpha)^2}
\leq C$, $\frac{1}{(1+\lambda_k^\beta\widetilde{\theta}^\alpha)^2}\leq C$, $\mathbb{E}\|b\|^2\leq C\mathbb{E}\|b\|_{\mathbb{H}^{2\widetilde{\gamma}-\frac{2\beta}{\alpha}}}^2$.

Note that $a^\theta-b^\theta\leq (a-b)^{\theta}$ for $a>b>0$ and $0<\theta<1$. Thanks to Lemma \ref{le2.1} and Eq.(1.83) in \cite{Podlubny}, by $(\bf{A}_1)$, and H\"{o}lder's inequality, there is
\begin{flalign}\label{eq.3.33}
& \mathcal{E}_3 \nonumber
\\
&=C\mathbb{E}\sum_{k=1}^{\infty}\bigg(\int_0^{\theta_2}\int_{\mathcal{D}}
(\theta_2-s)^{\alpha-1}e^{-\nu(\theta_2-s)}\sum_{l=1}^{\infty}E_{\alpha,\alpha}(-\lambda_l^\beta(\theta_2-s)^\alpha)
\varphi_l(\cdot)\varphi_l(y)\nonumber\\
&\relphantom{=}{} \cdot f(s,u(s,y))dyds-\int_0^{\theta_1}\int_{\mathcal{D}}
(\theta_1-s)^{\alpha-1}e^{-\nu(\theta_1-s)} \\
& \relphantom{=}{} \cdot\sum_{l=1}^{\infty}E_{\alpha,\alpha}(-\lambda_l^\beta(\theta_1-s)^\alpha)
\varphi_l(\cdot)\varphi_l(y)f(s,u(s,y))dyds,\varphi_k\bigg)^2 \nonumber &
\end{flalign}
\begin{flalign*}
&=C\mathbb{E}\sum_{k=1}^{\infty}\bigg|\int_0^{\theta_2}\int_{\mathcal{D}}
(\theta_2-s)^{\alpha-1}e^{-\nu(\theta_2-s)}E_{\alpha,\alpha}(-\lambda_k^\beta(\theta_2-s)^\alpha)
\nonumber\\
&\relphantom{=}{}\cdot \varphi_k(y)f(s,u(s,y))dyds-\int_0^{\theta_1}\int_{\mathcal{D}}
(\theta_1-s)^{\alpha-1}e^{-\nu(\theta_1-s)}
\nonumber\\
&\relphantom{=}{}\cdot
E_{\alpha,\alpha}(-\lambda_k^\beta(\theta_1-s)^\alpha)
\varphi_k(y)f(s,u(s,y))dyds\bigg|^2\nonumber &
\end{flalign*}
\begin{flalign*}
&\leq CT\mathbb{E}\sum_{k=1}^{\infty}\int_0^{\theta_1}\Big|(\theta_2-s)^{\alpha-1}e^{-\nu(\theta_2-s)}
E_{\alpha,\alpha}(-\lambda_k^\beta(\theta_2-s)^\alpha)
\\
&
\relphantom{=}{}-(\theta_1-s)^{\alpha-1}e^{-\nu(\theta_1-s)} E_{\alpha,\alpha}(-\lambda_k^\beta(\theta_1-s)^\alpha)\Big|^2(f(s,u(s)),\varphi_k)^2ds\\
&\relphantom{=}{} +C|\theta_2-\theta_1|\mathbb{E}\sum_{k=1}^{\infty}\int_{\theta_1}^{\theta_2}
\bigg|(\theta_2-s)^{\alpha-1}e^{-\nu(\theta_2-s)}
E_{\alpha,\alpha}(-\lambda_k^\beta(\theta_2-s)^{\alpha})\bigg|^2 
\\
&\relphantom{=}{} \cdot
(f(s,u(s)),\varphi_k)^2ds\nonumber
&
\end{flalign*}
\begin{flalign*}
&\leq C\mathbb{E}\sum_{k=1}^{\infty}\int_0^{\theta_1}e^{-2\nu(\theta_2-s)}\Big|(\theta_2-s)^{\alpha-1}
E_{\alpha,\alpha}(-\lambda_k^\beta(\theta_2-s)^\alpha)-(\theta_1-s)^{\alpha-1}
\nonumber\\
&\relphantom{=}{}\cdot E_{\alpha,\alpha}(-\lambda_k^\beta(\theta_1-s)^\alpha)\Big|^2(f(s,u(s)),\varphi_k)^2ds+C\mathbb{E}\sum_{k=1}^{\infty}\int_0^{\theta_1}(\theta_1-s)^{2\alpha-2}
\\
&\relphantom{=}{}\cdot \left|E_{\alpha,\alpha}(-\lambda_k^\beta(\theta_1-s)^\alpha)\right|^2 \left|e^{-\nu(\theta_2-s)}
-e^{-\nu(\theta_1-s)}\right|^2(f(s,u(s)),\varphi_k)^2ds\\
&\relphantom{=}{}+C|\theta_2-\theta_1|\int_{\theta_1}^{\theta_2}
(\theta_2-s)^{2\alpha-2} \mathbb{E}\sum_{k=1}^{\infty}\frac{\lambda_k^{-(2\gamma-\frac{2\beta}{\alpha})}}{(1+\lambda_k^\beta
(\theta_2-s)^{\alpha})^2}\lambda_k^{2\gamma-\frac{2\beta}{\alpha}}\\
&\relphantom{=}{}\cdot(f(s,u(s)),\varphi_k)^2ds &
\end{flalign*}
\begin{flalign*}
&\leq C\mathbb{E}\sum_{k=1}^{\infty}\int_0^{\theta_1}\bigg|\int_{\theta_1}^{\theta_2}
(\tau-s)^{\alpha-2}E_{\alpha,\alpha-1}(-\lambda_k^\beta(\tau-s)^\alpha)d\tau\bigg|^2
(f(s,u(s)),\varphi_k)^2ds\nonumber\\
&\relphantom{=}{}+C\mathbb{E}\sum_{k=1}^{\infty}\int_0^{\theta_1}(\theta_1-s)^{2\alpha-2}|\theta_2-\theta_1|^2
\frac{\lambda_k^{-(2\gamma-\frac{2\beta}{\alpha})}}{(1+\lambda_k^\beta(\theta_1-s)^\alpha)^2}
\lambda_k^{2\gamma-\frac{2\beta}{\alpha}}\nonumber\\
&\relphantom{=}{}\cdot (f(s,u(s)),\varphi_k)^2ds+C|\theta_2-\theta_1|\int_{\theta_1}^{\theta_2}
(\theta_2-s)^{2\alpha-2}\mathbb{E}
\|f(s,u(s))\|_{\mathbb{H}^{2\gamma-\frac{2\beta}{\alpha}}}^2ds 
 &
\end{flalign*}
\begin{flalign*}
&\leq C|\theta_2-\theta_1|^{2\alpha-2}\int_0^{\theta_1}
\mathbb{E}\|f(s,u(s))\|_{\mathbb{H}^{2\gamma-\frac{2\beta}{\alpha}}}^2ds\nonumber\\
&\relphantom{=}{}+C|\theta_2-\theta_1|^2\int_0^{\theta_1}
(\theta_1-s)^{2\alpha-2}\mathbb{E}
\|f(s,u(s))\|_{\mathbb{H}^{2\gamma-\frac{2\beta}{\alpha}}}^2ds\nonumber\\
&\relphantom{=}{}+C|\theta_2-\theta_1|\int_{\theta_1}^{\theta_2}
(\theta_2-s)^{2\alpha-2}\mathbb{E}
\|f(s,u(s))\|_{\mathbb{H}^{2\gamma-\frac{2\beta}{\alpha}}}^2ds\nonumber\\
&\leq C\left(|\theta_2-\theta_1|^{2\alpha-2}+|\theta_2-\theta_1|^2
+|\theta_2-\theta_1|^{2\alpha}\right)\sup_{s\in[0,T]}\mathbb{E}\left(1+\|u(s)\|_{\mathbb{H}^{2\gamma}}^2
\right), &
\end{flalign*}
where we have used $\frac{\lambda_k^{-(2\gamma-\frac{2\beta}{\alpha})}}{(1+\lambda_k^\beta(\theta_2-s)^\alpha)^2}\leq C$ for $\gamma>\frac{\beta}{\alpha}$ and $\mathbb{E}\|u(s)\|_{\mathbb{H}^{2\gamma-\frac{2\beta}{\alpha}}}^2\leq C \mathbb{E}\|u(s)\|_{\mathbb{H}^{2\gamma}}^2$.

Arguing as in the proof of \eqref{eq.3.33} and noticing that $\{\varphi_k\}_{k=1}^{\infty}$ is an orthonormal basis in $L^2({\mathcal{D}})$ and $\{\xi_l\}_{l=1}^\infty$ is a family of mutually independent one-dimensional standard Brownian motions, in view of \eqref{eq.2.8}, $(\bf{A}_1)$-$(\bf{A}_2)$, and Proposition \ref{pro1}, we deduce that
\begin{flalign}\label{eq.3.34}
&\mathcal{E}_4 \nonumber
\\
&=C\mathbb{E}\sum_{k=1}^{\infty}\bigg(
\int_{0}^{\theta_2}\int_{\mathcal{D}}(\theta_2-s)^{\alpha-1}e^{-\nu(\theta_2-s)}
\sum_{l=1}^{\infty}E_{\alpha,\alpha}\left(-\lambda_l^{\beta}(\theta_2-s)^\alpha\right)\varphi_l(\cdot)\varphi_l(y)\nonumber\\
&\relphantom{=}{} \cdot g(s,u(s,y)) d\mathbb{W}(s,y)-\int_{0}^{\theta_1}\int_{\mathcal{D}}(\theta_1-s)^{\alpha-1}e^{-\nu(\theta_1-s)}
\\
&
\relphantom{=}{} \cdot \sum_{l=1}^{\infty}E_{\alpha,\alpha}\left(-\lambda_l^{\beta}(\theta_1-s)^\alpha\right)\varphi_l(\cdot)\varphi_l(y)g(s,u(s,y)) d\mathbb{W}(s,y),\varphi_k\bigg)^2 \nonumber&
\end{flalign}
\begin{flalign}
&=C\mathbb{E}\sum_{k=1}^{\infty}\bigg|
\int_{0}^{\theta_2}\int_{\mathcal{D}}(\theta_2-s)^{\alpha-1}e^{-\nu(\theta_2-s)}
E_{\alpha,\alpha}\left(-\lambda_k^{\beta}(\theta_2-s)^\alpha\right)\varphi_k(y)\nonumber\\
&\relphantom{=}{} \cdot g(s,u(s,y)) d\mathbb{W}(s,y)-\int_{0}^{\theta_1}\int_{\mathcal{D}}(\theta_1-s)^{\alpha-1}e^{-\nu(\theta_1-s)}
\nonumber\\
&\relphantom{=}{} \cdot
E_{\alpha,\alpha}\left(-\lambda_k^{\beta}(\theta_1-s)^\alpha\right)\varphi_k(y)g(s,u(s,y)) d\mathbb{W}(s,y)\bigg|^2\nonumber&
\end{flalign}
\begin{flalign*}
&\leq C\mathbb{E}\sum_{k=1}^{\infty}\bigg|
\int_{0}^{\theta_1}\int_{\mathcal{D}}\bigg((\theta_2-s)^{\alpha-1}e^{-\nu(\theta_2-s)}
E_{\alpha,\alpha}\left(-\lambda_k^{\beta}(\theta_2-s)^\alpha\right)\nonumber\\
&\relphantom{=}{} -(\theta_1-s)^{\alpha-1}e^{-\nu(\theta_1-s)} E_{\alpha,\alpha}\left(-\lambda_k^{\beta}(\theta_1-s)^\alpha\right)\bigg) \varphi_k(y)\nonumber\\
&\relphantom{=}{}
\cdot\sum_{j,l=1}^{\infty}(g(s,u(s))\cdot e_l,\varphi_j)\varphi_j(y)\varsigma_l(s)dyd\xi_l(s)\bigg|^2 & 
\\
&\relphantom{=}{}+C\mathbb{E}\sum_{k=1}^{\infty}\bigg|\int_{\theta_1}^{\theta_2}\int_{\mathcal{D}}
(\theta_2-s)^{\alpha-1}e^{-\nu(\theta_2-s)}
E_{\alpha,\alpha}\left(-\lambda_k^{\beta}(\theta_2-s)^\alpha\right)\varphi_k(y)\nonumber\\
&\relphantom{=}{}\cdot
\sum_{j,l=1}^{\infty}(g(s,u(s))\cdot e_l,\varphi_j)\varphi_j(y)\varsigma_l(s)dyd\xi_l(s)\bigg|^2\nonumber &
\end{flalign*}
\begin{flalign*}
&=C\mathbb{E}\sum_{k=1}^{\infty}\bigg|
\int_{0}^{\theta_1}\bigg((\theta_2-s)^{\alpha-1}e^{-\nu(\theta_2-s)}
E_{\alpha,\alpha}\left(-\lambda_k^{\beta}(\theta_2-s)^\alpha\right)\nonumber\\
&\relphantom{=}{} -(\theta_1-s)^{\alpha-1}e^{-\nu(\theta_1-s)}
E_{\alpha,\alpha}\left(-\lambda_k^{\beta}(\theta_1-s)^\alpha\right)\bigg) \\
&\relphantom{=}{}\cdot\sum_{l=1}^{\infty}(g(s,u(s))\cdot e_l,\varphi_k)\varsigma_l(s)d\xi_l(s)\bigg|^2
+C\mathbb{E}\sum_{k=1}^{\infty}\bigg|\int_{\theta_1}^{\theta_2}
(\theta_2-s)^{\alpha-1}
\\
&\relphantom{=}{}
\cdot e^{-\nu(\theta_2-s)}
E_{\alpha,\alpha}\left(-\lambda_k^{\beta}(\theta_2-s)^\alpha\right)\sum_{l=1}^{\infty}(g(s,u(s))\cdot e_l,\varphi_k)\varsigma_l(s)d\xi_l(s)\bigg|^2 &
\end{flalign*}
\begin{flalign*}
&=C\mathbb{E}\sum_{k,l=1}^{\infty}
\int_{0}^{\theta_1}\bigg|\bigg((\theta_2-s)^{\alpha-1}e^{-\nu(\theta_2-s)}
E_{\alpha,\alpha}\left(-\lambda_k^{\beta}(\theta_2-s)^\alpha\right)\nonumber\\
&\relphantom{=}{}-(\theta_1-s)^{\alpha-1}e^{-\nu(\theta_1-s)}\cdot
E_{\alpha,\alpha}\left(-\lambda_k^{\beta}(\theta_1-s)^\alpha\right)\bigg)(g(s,u(s))\cdot e_l,\varphi_k)\varsigma_l(s)\bigg|^2 ds \\
&\relphantom{=}{}+C\mathbb{E}\sum_{k,l=1}^{\infty}\int_{\theta_1}^{\theta_2}\bigg|
(\theta_2-s)^{\alpha-1}e^{-\nu(\theta_2-s)}
E_{\alpha,\alpha}\left(-\lambda_k^{\beta}(\theta_2-s)^\alpha\right)
\\
&\relphantom{=}{}
\cdot(g(s,u(s))\cdot e_l,\varphi_k)\varsigma_l(s)\bigg|^2 ds &
\end{flalign*}
\begin{flalign*}
&\leq C\mathbb{E}\sum_{k,l=1}^{\infty}
\int_{0}^{\theta_1}\bigg(\bigg|\int_{\theta_1}^{\theta_2}(\tau-s)^{\alpha-2}
E_{\alpha,\alpha-1}\left(-\lambda_k^{\beta}(\tau-s)^\alpha\right)d\tau\bigg|^2 \\
&\relphantom{=}{} +(\theta_1-s)^{2\alpha-2}|\theta_2-\theta_1|^2\frac{1}{(1+\lambda_k^\beta(\theta_1-s)^\alpha)^2}\bigg)
\lambda_k^{-(2\gamma-\frac{2\beta}{\alpha})}
\\
&\relphantom{=}{}
 \cdot\left|\lambda_k^{\gamma-\frac{\beta}{\alpha}}(g(s,u(s))\cdot e_l,\varphi_k)\varsigma_l(s)\right|^2 ds\nonumber\\
&\relphantom{=}{}+C\mathbb{E}\sum_{k,l=1}^{\infty}\int_{\theta_1}^{\theta_2}
(\theta_2-s)^{2\alpha-2}\frac{\lambda_k^{-(2\gamma-\frac{2\beta}{\alpha})}}{(1+\lambda_k^\beta(\theta_2-s)^\alpha)^2}
\\
&\relphantom{=}{}
\cdot\left|\lambda_k^{\gamma-\frac{\beta}{\alpha}}(g(s,u(s))\cdot e_l,\varphi_k)\varsigma_l(s)\right|^2 ds &
\end{flalign*}
\begin{flalign*}
&\leq C|\theta_2-\theta_1|^{2\alpha-2}
\int_{0}^{\theta_1}\mathbb{E}\|(-\Delta)^{\gamma-\frac{\beta}{\alpha}}g(s,u(s))\|_{\mathcal{L}_2^0}^2ds\nonumber\\
&\relphantom{=}{}
+C|\theta_2-\theta_1|^2\int_{0}^{\theta_1}(\theta_1-s)^{2\alpha-2} \mathbb{E}\|(-\Delta)^{\gamma-\frac{\beta}{\alpha}}g(s,u(s))\|_{\mathcal{L}_2^0}^2ds\nonumber\\
&\relphantom{=}{}+C\int_{\theta_1}^{\theta_2}\bigg|
(\theta_2-s)^{2\alpha-2}\mathbb{E}\|(-\Delta)^{\gamma-\frac{\beta}{\alpha}}g(s,u(s))\|_{\mathcal{L}_2^0}^2ds\nonumber\\
&\leq C\left(|\theta_2-\theta_1|^{2\alpha-2}+|\theta_2-\theta_1|^2+|\theta_2-\theta_1|^{2\alpha-1}\right)
\\
&
\relphantom{=}{}\cdot\sup_{s\in[0,T]}\mathbb{E}\left(1+\|u(s)\|_{\mathbb{H}^{2\gamma}}^2\right). &
\end{flalign*}
Analogous to the arguments in \eqref{eq.3.34}, in view of \eqref{eq.2.9}, $(\bf{A}_1)$-$(\bf{A}_2)$, Lemma \ref{le2.10}, and the mutual independence of the family of one-dimensional fractional Brownian motions $\{\xi_l^H\}_{l=1}^\infty$, there is
\begin{flalign}\label{eq.3.35}
&\mathcal{E}_5 \nonumber
\\
&=C\mathbb{E}\sum_{k=1}^{\infty}\bigg(
\int_{0}^{\theta_2}\int_{\mathcal{D}}(\theta_2-s)^{\alpha-1}e^{-\nu(\theta_2-s)}
\sum_{l=1}^{\infty}E_{\alpha,\alpha}\left(-\lambda_l^{\beta}(\theta_2-s)^\alpha\right)\nonumber\\
&\relphantom{=}{} \cdot \varphi_l(\cdot)
\varphi_l(y) h(s,u(s,y)) d\mathbb{W}^H(s,y)-\int_{0}^{\theta_1}\int_{\mathcal{D}}(\theta_1-s)^{\alpha-1}e^{-\nu(\theta_1-s)}
\nonumber\\
&\relphantom{=}{}\cdot \sum_{l=1}^{\infty}E_{\alpha,\alpha}\left(-\lambda_l^{\beta}(\theta_1-s)^\alpha\right)\cdot\varphi_l(\cdot)\varphi_l(y)h(s,u(s,y)) d\mathbb{W}^H(s,y),\varphi_k\bigg)^2 &
\end{flalign}
\begin{flalign*}
&= C\mathbb{E}\sum_{k=1}^{\infty}\bigg|
\int_{0}^{\theta_1}\int_{\mathcal{D}}\bigg((\theta_2-s)^{\alpha-1}e^{-\nu(\theta_2-s)}
E_{\alpha,\alpha}\left(-\lambda_k^{\beta}(\theta_2-s)^\alpha\right)\\
&\relphantom{=}{}-(\theta_1-s)^{\alpha-1}e^{-\nu(\theta_1-s)} E_{\alpha,\alpha}\left(-\lambda_k^{\beta}(\theta_1-s)^\alpha\right)\bigg) \varphi_k(y) \\
&\relphantom{=}{}\cdot\sum_{j,l=1}^{\infty}(h(s,u(s))\cdot e_l,\varphi_j)\varphi_j(y)\varrho_l(s)dyd\xi_l^H(s)\bigg|^2
\\
&\relphantom{=}{}+C\mathbb{E}\sum_{k=1}^{\infty}\bigg|\int_{\theta_1}^{\theta_2}\int_{\mathcal{D}}
(\theta_2-s)^{\alpha-1}e^{-\nu(\theta_2-s)}
E_{\alpha,\alpha}\left(-\lambda_k^{\beta}(\theta_2-s)^\alpha\right)\varphi_k(y)\nonumber\\
&\relphantom{=}{}\cdot
\sum_{j,l=1}^{\infty}(h(s,u(s))\cdot e_l,\varphi_j)\varphi_j(y)\varrho_l(s)dyd\xi_l^H(s)\bigg|^2&
\end{flalign*}
\begin{flalign*}
 &=C\mathbb{E}\sum_{k=1}^{\infty}\bigg|
\int_{0}^{\theta_1}\bigg((\theta_2-s)^{\alpha-1}e^{-\nu(\theta_2-s)}
E_{\alpha,\alpha}\left(-\lambda_k^{\beta}(\theta_2-s)^\alpha\right)\\
&\relphantom{=}{} -(\theta_1-s)^{\alpha-1}e^{-\nu(\theta_1-s)}
E_{\alpha,\alpha}\left(-\lambda_k^{\beta}(\theta_1-s)^\alpha\right)\bigg) \sum_{l=1}^{\infty}(h(s,u(s))\cdot e_l,\varphi_k)\\
&\relphantom{=}{}\cdot \varrho_l(s)d\xi_l^H(s)\bigg|^2+C\mathbb{E}\sum_{k=1}^{\infty}\bigg|\int_{\theta_1}^{\theta_2}
(\theta_2-s)^{\alpha-1}e^{-\nu(\theta_2-s)}
E_{\alpha,\alpha}\left(-\lambda_k^{\beta}(\theta_2-s)^\alpha\right)\nonumber\\
&\relphantom{=}{}\cdot\sum_{l=1}^{\infty}(h(s,u(s))\cdot e_l,\varphi_k)\varrho_l(s)d\xi_l^H(s)\bigg|^2
&
\end{flalign*}
\begin{flalign*}
&\leq CHT^{2H-1}\mathbb{E}\sum_{k,l=1}^{\infty}
\int_{0}^{\theta_1}\bigg|\bigg((\theta_2-s)^{\alpha-1}e^{-\nu(\theta_2-s)}
E_{\alpha,\alpha}\left(-\lambda_k^{\beta}(\theta_2-s)^\alpha\right)\nonumber\\
&\relphantom{=}{}-(\theta_1-s)^{\alpha-1}e^{-\nu(\theta_1-s)}
E_{\alpha,\alpha}\left(-\lambda_k^{\beta}(\theta_1-s)^\alpha\right)\bigg) (h(s,u(s))\cdot e_l,\varphi_k)\\
&\relphantom{=}{}\cdot \varrho_l(s)\bigg|^2 ds+CH|\theta_2-\theta_1|^{2H-1}
\mathbb{E}\sum_{k,l=1}^{\infty}\int_{\theta_1}^{\theta_2}\bigg|
(\theta_2-s)^{\alpha-1}e^{-\nu(\theta_2-s)}
\\
&\relphantom{=}{} \cdot E_{\alpha,\alpha} \left(-\lambda_k^{\beta}(\theta_2-s)^\alpha\right)\cdot(h(s,u(s))\cdot e_l,\varphi_k)\varrho_l(s)\bigg|^2 ds
&
\end{flalign*}
\begin{flalign*}
&\leq C\mathbb{E}\sum_{k,l=1}^{\infty}
\int_{0}^{\theta_1}\bigg(\bigg|\int_{\theta_1}^{\theta_2}(\tau-s)^{\alpha-2}
E_{\alpha,\alpha-1}\left(-\lambda_k^{\beta}(\tau-s)^\alpha\right)d\tau\bigg|^2
+(\theta_1-s)^{2\alpha-2}  \\
&\relphantom{=}{}\cdot |\theta_2-\theta_1|^2 \frac{1}{(1+\lambda_k^\beta(\theta_1-s)^\alpha)^2}\bigg) \lambda_k^{-(2\gamma-\frac{2\beta}{\alpha})}\lambda_k^{2\gamma-\frac{2\beta}{\alpha}}
|(h(s,u(s))\cdot e_l,\varphi_k)\\
&\relphantom{=}{}\cdot \varrho_l(s)|^2 ds+C|\theta_2-\theta_1|^{2H-1}\mathbb{E}\sum_{k,l=1}^{\infty}\int_{\theta_1}^{\theta_2}
(\theta_2-s)^{2\alpha-2}\frac{\lambda_k^{-(2\gamma-\frac{2\beta}{\alpha})}}{(1+\lambda_k^\beta(\theta_2-s)^\alpha)^2}\nonumber\\
&\relphantom{=}{} \cdot \left|\lambda_k^{\gamma-\frac{\beta}{\alpha}}(h(s,u(s))\cdot e_l,\varphi_k)\varrho_l(s)\right|^2 ds&
\end{flalign*}
\begin{flalign*}
&\leq C|\theta_2-\theta_1|^{2\alpha-2}
\int_{0}^{\theta_1}\mathbb{E}\|(-\Delta)^{\gamma-\frac{\beta}{\alpha}}h(s,u(s))\|_{\mathcal{L}_2^0}^2ds\\
&\relphantom{=}{}
+C|\theta_2-\theta_1|^2\int_{0}^{\theta_1}(\theta_1-s)^{2\alpha-2} \mathbb{E}\|(-\Delta)^{\gamma-\frac{\beta}{\alpha}}h(s,u(s))\|_{\mathcal{L}_2^0}^2ds\\
&\relphantom{=}{}+C|\theta_2-\theta_1|^{2H-1}\int_{\theta_1}^{\theta_2}\bigg|
(\theta_2-s)^{2\alpha-2}\mathbb{E}\|(-\Delta)^{\gamma-\frac{\beta}{\alpha}}h(s,u(s))\|_{\mathcal{L}_2^0}^2ds&
\end{flalign*}
\begin{flalign*}
&\leq C\left(|\theta_2-\theta_1|^{2\alpha-2}+|\theta_2-\theta_1|^2+|\theta_2-\theta_1|^{2H+2\alpha-2}\right) &
\\
&\relphantom{=}{}\cdot\sup_{s\in[0,T]}\mathbb{E}\left(1+\|u(s)\|_{\mathbb{H}^{2\gamma}}^2\right).
\end{flalign*}
Hence, \eqref{eq.3.29}-\eqref{eq.3.35} imply that
\begin{align}\label{eq.3.36}
\mathbb{E}\|u(\theta_2)-u(\theta_1)\|^2\leq & C|\theta_2-\theta_1|^{2\alpha-2}\Big(\mathbb{E}\|a\|_{\mathbb{H}^{2\widetilde{\gamma}}}^2
+\mathbb{E}\|b\|_{\mathbb{H}^{2\widetilde{\gamma}-\frac{2\beta}{\alpha}}}^2
\\
&
+\sup_{s\in[0,T]}\mathbb{E}\left(1+\|u(s)\|_{\mathbb{H}^{2\gamma}}^2\right)\Big), \nonumber
\end{align}
which completes the proof of this theorem.
\end{proof}
\begin{theorem} Assume that the conditions of Theorem \ref{thm3.5} hold. Then we have
$$\sup_{0\leq t\leq T}\|\partial_t u(t)\|_{\mathbb{H}^{2\widetilde{\gamma}-\frac{2\beta}{\alpha}}}\leq C.$$
\end{theorem}
\begin{proof}
Thanks to Lemma \ref{le2.1}, \eqref{eq2.1}, and Eq. (1.83) in \cite{Podlubny}, we find that
\begin{align}\label{eq.3.37}
&\partial_tu(t,x) \nonumber
\\
&=\int_{\mathcal{D}}(-\nu e^{-\nu t})\sum\limits_{k=1}^{\infty}
\left(E_{\alpha,1}(-\lambda_k^\beta t^\alpha)+\nu tE_{\alpha,2}(-\lambda_k^\beta t^\alpha)\right)
\varphi_k(x)\varphi_k(y)a(y)dy\nonumber\\
&\relphantom{=}{}+\int_{\mathcal{D}}e^{-\nu t}
\sum\limits_{k=1}^{\infty}\left(-\lambda_k^\beta t^{\alpha-1}E_{\alpha,\alpha}(-\lambda_k^\beta t^\alpha)+\nu E_{\alpha,1}(-\lambda_k^\beta t^\alpha)\right)\nonumber\\
& \relphantom{=}{} \cdot \varphi_k(x)\varphi_k(y)a(y)dy \nonumber\\
&\relphantom{=}{}+\int_{\mathcal{D}}e^{-\nu t}\sum\limits_{k=1}^{\infty}\left(-\nu tE_{\alpha,2}(-\lambda_k^\beta t^\alpha)+E_{\alpha,1}(-\lambda_k^\beta t^\alpha)\right)\varphi_k(x) \varphi_k(y)b(y)dy\nonumber\\
&\relphantom{=}{}+\int^t_0\int_{\mathcal{D}}e^{-\nu(t-s)}\sum\limits_{k=1}^{\infty}
\Big(-\nu (t-s)^{\alpha-1}E_{\alpha,\alpha}(-\lambda_k^\beta (t-s)^\alpha) \nonumber \\
&\relphantom{=}{}+(t-s)^{\alpha-2} E_{\alpha,\alpha-1}(-\lambda_k^\beta(t-s)^\alpha)\Big)
\varphi_k(x)\varphi_k(y)f(s,u(s,y))dyds%\nonumber\\
&
\end{align}

\begin{flalign*}
&\relphantom{=}{}+\int^t_0\int_{\mathcal{D}}e^{-\nu(t-s)}\sum\limits_{k=1}^{\infty}
\Big(-\nu (t-s)^{\alpha-1}E_{\alpha,\alpha}(-\lambda_k^\beta (t-s)^\alpha)\nonumber\\
&\relphantom{=}{}
+(t-s)^{\alpha-2} E_{\alpha,\alpha-1}(-\lambda_k^\beta(t-s)^\alpha)\Big)
\varphi_k(x)\varphi_k(y)g(s,u(s,y))d\mathbb{W}(s,y)\nonumber\\
&\relphantom{=}{}+\int^t_0\int_{\mathcal{D}}e^{-\nu(t-s)}\sum\limits_{k=1}^{\infty}
\Big(-\nu (t-s)^{\alpha-1}E_{\alpha,\alpha}(-\lambda_k^\beta (t-s)^\alpha)\nonumber\\
&\relphantom{=}{}
+(t-s)^{\alpha-2} E_{\alpha,\alpha-1}(-\lambda_k^\beta(t-s)^\alpha)\Big)
\varphi_k(x)\varphi_k(y)h(s,u(s,y))d\mathbb{W}^H(s,y)\nonumber\\
&=\sum_{i=1}^6\mathcal{X}_i. 
\end{flalign*}
Observe that $\{\varphi_k\}_{k=1}^{\infty}$ is an orthonormal basis in $L^2({\mathcal{D}})$.
By Lemma \ref{le2.1}, we have
\begin{flalign*}\label{eq.3.38}
&\mathbb{E}\|\mathcal{X}_1\|_{\mathbb{H}^{2\widetilde{\gamma}-\frac{2\beta}{\alpha}}}^2\nonumber\\ 
&%\relphantom{=}{}
= \mathbb{E}\sum\limits_{k=1}^{\infty}\lambda_k^{2\widetilde{\gamma}-\frac{2\beta}{\alpha}}
\bigg(\int_{\mathcal{D}}(-\nu e^{-\nu t})\sum\limits_{l=1}^{\infty}
\left(E_{\alpha,1}(-\lambda_l^\beta t^\alpha)+\nu tE_{\alpha,2}(-\lambda_l^\beta t^\alpha)\right)
\nonumber\\
&\relphantom{=}{}\cdot\varphi_l(\cdot)\varphi_l(y)a(y)dy,\varphi_k\bigg)^2
\nonumber\\
&%\relphantom{=}{}
=\mathbb{E}\sum\limits_{k=1}^{\infty}\lambda_k^{2\widetilde{\gamma}-\frac{2\beta}{\alpha}}
\bigg|\int_{\mathcal{D}}(-\nu e^{-\nu t})
\left(E_{\alpha,1}(-\lambda_k^\beta t^\alpha)+\nu tE_{\alpha,2}(-\lambda_k^\beta t^\alpha)\right)
\varphi_k(y)a(y)dy\bigg|^2 \nonumber&  
\end{flalign*}
\begin{flalign}
&%\relphantom{=}{}
\leq C e^{-2\nu t}\mathbb{E}\sum\limits_{k=1}^{\infty}
\left|E_{\alpha,1}(-\lambda_k^\beta t^\alpha)\right|^2\lambda_k^{2\widetilde{\gamma}-\frac{2\beta}{\alpha}}
(a,\varphi_k)^2 \\
&\relphantom{=}{}
+Ce^{-2\nu t}\mathbb{E}\sum\limits_{k=1}^{\infty}t^2
\left|E_{\alpha,2}(-\lambda_k^\beta t^\alpha)\right|^2\lambda_k^{2\widetilde{\gamma}-\frac{2\beta}{\alpha}}
(a,\varphi_k)^2 \nonumber \\
&%\relphantom{=}{}
\leq Ce^{-2\nu t}\mathbb{E}\sum\limits_{k=1}^{\infty}
\frac{1}{(1+\lambda_k^\beta t^\alpha)^2\lambda_k^{\frac{2\beta}{\alpha}}}\lambda_k^{2\widetilde{\gamma}}
(a,\varphi_k)^2\leq C\mathbb{E}\|a\|_{\mathbb{H}^{2\widetilde{\gamma}}}^2, \nonumber & 
\end{flalign}
\begin{flalign}\label{eq.3.39}
&\mathbb{E}\|\mathcal{X}_2\|_{\mathbb{H}^{2\widetilde{\gamma}-\frac{2\beta}{\alpha}}}^2\nonumber\\  
&%\relphantom{=}{}
= \mathbb{E}\sum\limits_{k=1}^{\infty}\lambda_k^{2\widetilde{\gamma}-\frac{2\beta}{\alpha}}
\bigg(\int_{\mathcal{D}}e^{-\nu t}
\sum\limits_{l=1}^{\infty}\left(-\lambda_l^\beta t^{\alpha-1}E_{\alpha,\alpha}(-\lambda_l^\beta t^\alpha)+\nu E_{\alpha,1}(-\lambda_l^\beta t^\alpha)\right)\nonumber
\\
&\relphantom{=}{}\cdot\varphi_l(\cdot)\varphi_l(y)a(y)dy,\varphi_k\bigg)^2 \nonumber
\\
&%\relphantom{=}{}
=\mathbb{E}\sum\limits_{k=1}^{\infty}\lambda_k^{2\widetilde{\gamma}-\frac{2\beta}{\alpha}}
\bigg{|}\int_{\mathcal{D}}e^{-\nu t}
\left(-\lambda_k^\beta t^{\alpha-1}E_{\alpha,\alpha}(-\lambda_k^\beta t^\alpha)+\nu E_{\alpha,1}(-\lambda_k^\beta t^\alpha)\right)
 \nonumber
\\
&
\relphantom{=}{}
\cdot\varphi_k(y)a(y)dy\bigg{|}^2 &
\end{flalign}
\begin{flalign}
&%\relphantom{=}{}
\leq Ce^{-2\nu t}\mathbb{E}\sum\limits_{k=1}^{\infty}\left|\lambda_k^\beta t^{\alpha-1}E_{\alpha,\alpha}(-\lambda_k^\beta t^\alpha)\right|^2\lambda_k^{2\widetilde{\gamma}-\frac{2\beta}{\alpha}}
(a,\varphi_k)^2\nonumber\\
&\relphantom{=}{}
+Ce^{-2\nu t}\mathbb{E}\sum\limits_{k=1}^{\infty}
\left|E_{\alpha,1}(-\lambda_k^\beta t^\alpha)\right|^2\lambda_k^{2\widetilde{\gamma}-\frac{2\beta}{\alpha}}
(a,\varphi_k)^2\nonumber\\
&%\relphantom{=}{}
\leq Ce^{-2\nu t}\mathbb{E}\sum\limits_{k=1}^{\infty}
\left|\frac{(\lambda_k^\beta t^\alpha)^{\frac{\alpha-1}{\alpha}}}{1+\lambda_k^\beta t^\alpha}\right|^2\lambda_k^{2\widetilde{\gamma}}(a,\varphi_k)^2
\nonumber\\
&\relphantom{=}{}
+Ce^{-2\nu t}\mathbb{E}\sum\limits_{k=1}^{\infty}
\frac{1}{\lambda_k^{\frac{2\beta}{\alpha}}(1+\lambda_k^\beta t^\alpha)^2}\lambda_k^{2\widetilde{\gamma}}
(a,\varphi_k)^2\nonumber\\
&%\relphantom{=}{}
\leq C\mathbb{E}\|a\|_{\mathbb{H}^{2\widetilde{\gamma}}}^2, \nonumber &
\end{flalign}
and
\begin{flalign}\label{eq.3.40}
&\mathbb{E}\|\mathcal{X}_3\|_{\mathbb{H}^{2\widetilde{\gamma}-\frac{2\beta}{\alpha}}}^2 \nonumber\\
&%\relphantom{=}{}
= \mathbb{E}\sum\limits_{k=1}^{\infty}\lambda_k^{2\widetilde{\gamma}-\frac{2\beta}{\alpha}}
\bigg(\int_{\mathcal{D}}e^{-\nu t}
\sum\limits_{l=1}^{\infty}\left(-\nu t E_{\alpha,2}(-\lambda_l^\beta t^\alpha)+ E_{\alpha,1}(-\lambda_l^\beta t^\alpha)\right)
\nonumber\\
&
\relphantom{=}{} \cdot\varphi_l(\cdot)\varphi_l(y)b(y)dy,\varphi_k\bigg)^2\nonumber\\
& %\relphantom{=}{}
= \mathbb{E}\sum\limits_{k=1}^{\infty}\lambda_k^{2\widetilde{\gamma}-\frac{2\beta}{\alpha}}
\bigg{|}\int_{\mathcal{D}}e^{-\nu t}\left(-\nu tE_{\alpha,2}(-\lambda_k^\beta t^\alpha)+E_{\alpha,1}(-\lambda_k^\beta t^\alpha)\right) 
\nonumber\\
&
\relphantom{=}{} \cdot\varphi_k(y)b(y)dy\bigg{|}^2 &
\end{flalign}
\begin{flalign*}
&%\relphantom{=}{}
\leq Ce^{-2\nu t}\mathbb{E}\sum\limits_{k=1}^{\infty}t^2
\left|E_{\alpha,2}(-\lambda_k^\beta t^\alpha)\right|^2\lambda_k^{2\widetilde{\gamma}-\frac{2\beta}{\alpha}}(b,\varphi_k)^2\nonumber\\
&\relphantom{=}{}
+Ce^{-2\nu t}\mathbb{E}\sum\limits_{k=1}^{\infty}
\left|E_{\alpha,1}(-\lambda_k^\beta t^\alpha)\right|^2\lambda_k^{2\widetilde{\gamma}-\frac{2\beta}{\alpha}}(b,\varphi_k)^2\nonumber\\
&%\relphantom{=}{}
\leq Ce^{-2\nu t}\mathbb{E}\sum\limits_{k=1}^{\infty}
\frac{1}{(1+\lambda_k^\beta t^\alpha)^2}\lambda_k^{2\widetilde{\gamma}-\frac{2\beta}{\alpha}}
(b,\varphi_k)^2
\leq C\mathbb{E}\|b\|_{\mathbb{H}^{2\widetilde{\gamma}-\frac{2\beta}{\alpha}}}^2, &
\end{flalign*}
where we have used $\frac{1}{(1+\lambda_k^\beta t^\alpha)^2\lambda_k^{\frac{2\beta}{\alpha}}}\leq C$, $\frac{(\lambda_k^\beta t^\alpha)^{\frac{\alpha-1}{\alpha}}}{(1+\lambda_k^\beta t^\alpha)^2}\leq C$ and
$\frac{1}{(1+\lambda_k^\beta t^\alpha)^2}\leq C$.\\
Noticing that  $\{\varphi_k\}_{k=1}^{\infty}$ is an orthonormal basis in $L^2({\mathcal{D}})$,  $\{\xi_l\}_{l=1}^\infty$ and $\{\xi_l^H\}_{l=1}^\infty$, respectively, are the sequences of mutually independent one-dimensional standard Brownian motions and fractional Brownian motions, we deduce from \eqref{eq.2.8}-\eqref{eq.2.9}, $(\bf{A}_1)$-$(\bf{A}_2)$, Lemma \ref{le2.1}, Proposition \ref{pro1}, Lemma \ref{le2.10}, and H\"{o}lder's inequality that
\begin{align}\label{eq.3.41}
&\mathbb{E}\|\mathcal{X}_4\|_{\mathbb{H}^{2\widetilde{\gamma}-\frac{2\beta}{\alpha}}}^2  \nonumber\\
&%\relphantom{=}{}
=\mathbb{E}\sum\limits_{k=1}^{\infty}\lambda_k^{2\widetilde{\gamma}-\frac{2\beta}{\alpha}}
\bigg(\int^t_0\int_{\mathcal{D}}e^{-\nu(t-s)}\sum\limits_{l=1}^{\infty}
\Big(-\nu (t-s)^{\alpha-1}E_{\alpha,\alpha}(-\lambda_l^\beta (t-s)^\alpha)\nonumber\\
&\relphantom{=}{}+(t-s)^{\alpha-2} E_{\alpha,\alpha-1}(-\lambda_l^\beta(t-s)^\alpha)\Big)
\varphi_l(\cdot)\varphi_l(y)f(s,u(s,y))dyds,\varphi_k\bigg)^2\nonumber\\
&%\relphantom{=}{}
=\mathbb{E}\sum\limits_{k=1}^{\infty}\lambda_k^{2\widetilde{\gamma}-\frac{2\beta}{\alpha}}
\bigg{|}
\int^t_0\int_{\mathcal{D}}e^{-\nu(t-s)}
\Big(-\nu (t-s)^{\alpha-1}E_{\alpha,\alpha}(-\lambda_k^\beta (t-s)^\alpha)\nonumber\\
&\relphantom{=}{}+(t-s)^{\alpha-2} E_{\alpha,\alpha-1}(-\lambda_k^\beta(t-s)^\alpha)\Big)
\varphi_k(y)f(s,u(s,y))dyds\bigg{|}^2&
\end{align}
\begin{flalign*}
&\relphantom{=}{}\leq CT \mathbb{E}\sum\limits_{k=1}^{\infty}\int^t_0\bigg|e^{-\nu(t-s)}
\Big(-\nu (t-s)^{\alpha-1}E_{\alpha,\alpha}(-\lambda_k^\beta (t-s)^\alpha)\nonumber\\
&\relphantom{==}{}+(t-s)^{\alpha-2} E_{\alpha,\alpha-1}(-\lambda_k^\beta(t-s)^\alpha)\Big)\bigg|^2
\lambda_k^{2\widetilde{\gamma}-\frac{2\beta}{\alpha}}
(f(s,u(s)),\varphi_k)^2ds \nonumber \\
&\relphantom{=}{}\leq C \mathbb{E}\sum\limits_{k=1}^{\infty}
\int^t_0 \Big( (t-s)^{2\alpha-2}+(t-s)^{2\alpha-4}\Big)\frac{1}{(1+\lambda_k^\beta(t-s)^\alpha)^2} \nonumber \\
& \relphantom{==}{}
\cdot\lambda_k^{2\widetilde{\gamma}-\frac{2\beta}{\alpha}}
(f(s,u(s)),\varphi_k)^2ds
&
\end{flalign*}
\begin{flalign*}
&\relphantom{=}{}\leq C \int^t_0 \Big( (t-s)^{2\alpha-2}+(t-s)^{2\alpha-4}\Big)
\left(1+\mathbb{E}\|u(s)\|_{\mathbb{H}^{2\widetilde{\gamma}-\frac{2\beta}{\alpha}}}^2\right)ds\nonumber\\
&\relphantom{=}{}\leq C \int^t_0 \Big( (t-s)^{2\alpha-2}+(t-s)^{2\alpha-4}\Big)
\left(1+\mathbb{E}\|u(s)\|_{\mathbb{H}^{2\widetilde{\gamma}}}^2\right)ds \\
%&
%\end{align*}
%\begin{flalign*}
&\relphantom{=}{}
\leq C\sup_{s\in[0,T]}\mathbb{E}\left(1+\|u(s)\|_{\mathbb{H}^{2\widetilde{\gamma}}}^2\right), &
\end{flalign*}
\begin{flalign}\label{eq.3.42}
&\mathbb{E}\|\mathcal{X}_5\|_{\mathbb{H}^{2\widetilde{\gamma}-\frac{2\beta}{\alpha}}}^2 \nonumber\\
&%\relphantom{=}{}
=\mathbb{E}\sum\limits_{k=1}^{\infty}\lambda_k^{2\widetilde{\gamma}-\frac{2\beta}{\alpha}}
\bigg(\int^t_0\int_{\mathcal{D}}e^{-\nu(t-s)}\sum\limits_{l=1}^{\infty}
\Big(-\nu (t-s)^{\alpha-1}E_{\alpha,\alpha}(-\lambda_l^\beta (t-s)^\alpha)\nonumber\\
&\relphantom{=}{}+(t-s)^{\alpha-2} E_{\alpha,\alpha-1}(-\lambda_l^\beta(t-s)^\alpha)\Big)
\varphi_l(\cdot)\varphi_l(y)g(s,u(s,y))d\mathbb{W}(s,y),\varphi_k\bigg)^2\nonumber\\
&%\relphantom{=}{}
=\mathbb{E}\sum\limits_{k=1}^{\infty}\lambda_k^{2\widetilde{\gamma}-\frac{2\beta}{\alpha}}
\bigg{|}\int^t_0\int_{\mathcal{D}}e^{-\nu(t-s)}
\Big(-\nu (t-s)^{\alpha-1}E_{\alpha,\alpha}(-\lambda_k^\beta (t-s)^\alpha)\nonumber\\
&\relphantom{=}{}+(t-s)^{\alpha-2} E_{\alpha,\alpha-1}(-\lambda_k^\beta(t-s)^\alpha)\Big)
\varphi_k(y)g(s,u(s,y))d\mathbb{W}(s,y)\bigg{|}^2 &
\end{flalign}
\begin{flalign*}
&=\mathbb{E}\sum\limits_{k=1}^{\infty}\bigg{|}
\int^t_0\int_{\mathcal{D}}e^{-\nu(t-s)}
\Big(-\nu (t-s)^{\alpha-1}E_{\alpha,\alpha}(-\lambda_k^\beta (t-s)^\alpha)+(t-s)^{\alpha-2} \nonumber\\ &\relphantom{=}{}\cdot E_{\alpha,\alpha-1}(-\lambda_k^\beta(t-s)^\alpha)\Big)
\lambda_k^{\widetilde{\gamma}-\frac{\beta}{\alpha}}
\varphi_k(y)\sum_{j,l=1}^{\infty}(g(s,u(s))\cdot e_l,\varphi_j)\varphi_j(y)\varsigma_l(s)
\\ &\relphantom{=}{}
\cdot dyd\xi_l(s)\bigg{|}^2 \nonumber\\
&=\mathbb{E}\sum\limits_{k=1}^{\infty}\bigg{|}
\int^t_0e^{-\nu(t-s)}\Big(-\nu (t-s)^{\alpha-1}E_{\alpha,\alpha}(-\lambda_k^\beta (t-s)^\alpha)+(t-s)^{\alpha-2} \nonumber\\ &\relphantom{=}{}\cdot E_{\alpha,\alpha-1}(-\lambda_k^\beta(t-s)^\alpha)\Big)
\lambda_k^{\widetilde{\gamma}-\frac{\beta}{\alpha}}
\sum_{l=1}^{\infty}(g(s,u(s))\cdot e_l,\varphi_k)\varsigma_l(s)d\xi_l(s)\bigg{|}^2
&
\end{flalign*}
\begin{flalign*}
&\leq C\mathbb{E}\sum\limits_{k,l=1}^{\infty}\int^t_0\Big( (t-s)^{2\alpha-2}+(t-s)^{2\alpha-4}\Big)\frac{1}{(1+\lambda_k^\beta(t-s)^\alpha)^2}\\
&\relphantom{=}{}\cdot
\bigg{|}\lambda_k^{\widetilde{\gamma}-\frac{\beta}{\alpha}}
(g(s,u(s))\cdot e_l,\varphi_k)\varsigma_l(s)\bigg{|}^2ds \nonumber\\
&\leq C
\int^t_0\Big((t-s)^{2\alpha-2}+(t-s)^{2\alpha-4}\Big)\mathbb{E}\|(-\Delta)^{\widetilde{\gamma}-\frac{\beta}{\alpha}}g(s,u(s))\|_{\mathcal{L}_2^0}^2ds 
&
\end{flalign*}
\begin{flalign*}
&\leq C
\int^t_0\Big((t-s)^{2\alpha-2}+(t-s)^{2\alpha-4}\Big)\left(1+\mathbb{E}\|u(s)\|_{\mathbb{H}^{2\widetilde{\gamma}-\frac{2\beta}{\alpha}}}^2\right)ds \nonumber\\
&\leq C
\int^t_0\Big((t-s)^{2\alpha-2}+(t-s)^{2\alpha-4}\Big)
\left(1+\mathbb{E}\|u(s)\|_{\mathbb{H}^{2\widetilde{\gamma}}}^2\right)
ds\nonumber\\
&\leq C\sup_{s\in[0,T]}\mathbb{E}\left(1+\|u(s)\|_{\mathbb{H}^{2\widetilde{\gamma}}}^2\right), &
\end{flalign*}
and
\begin{flalign} \label{eq.3.43}
&\mathbb{E}\|\mathcal{X}_6\|_{\mathbb{H}^{2\widetilde{\gamma}-\frac{2\beta}{\alpha}}}^2 \nonumber\\
&%\relphantom{=}{}
=\mathbb{E}\sum\limits_{k=1}^{\infty}\lambda_k^{2\widetilde{\gamma}-\frac{2\beta}{\alpha}}
\bigg(\int^t_0\int_{\mathcal{D}}e^{-\nu(t-s)}\sum\limits_{l=1}^{\infty}
\Big(-\nu (t-s)^{\alpha-1}E_{\alpha,\alpha}(-\lambda_l^\beta (t-s)^\alpha)\nonumber\\
&\relphantom{=}{}+(t-s)^{\alpha-2} E_{\alpha,\alpha-1}(-\lambda_l^\beta(t-s)^\alpha)\Big)
\varphi_l(\cdot)\varphi_l(y)h(s,u(s,y))d\mathbb{W}^H(s,y),\varphi_k\bigg)^2\nonumber\\
&%\relphantom{=}{}
=\mathbb{E}\sum\limits_{k=1}^{\infty}\lambda_k^{2\widetilde{\gamma}-\frac{2\beta}{\alpha}}
\bigg{|}\int^t_0\int_{\mathcal{D}}e^{-\nu(t-s)}
\Big(-\nu(t-s)^{\alpha-1}E_{\alpha,\alpha}(-\lambda_k^\beta (t-s)^\alpha)\nonumber\\
&\relphantom{=}{}+(t-s)^{\alpha-2} E_{\alpha,\alpha-1}(-\lambda_k^\beta(t-s)^\alpha)\Big)
\varphi_k(y)h(s,u(s,y))d\mathbb{W}^H(s,y)\bigg{|}^2 
&%\relphantom{=}{}
\end{flalign}

\begin{flalign*}
&
=\mathbb{E}\sum\limits_{k=1}^{\infty}\bigg{|}
\int^t_0\int_{\mathcal{D}}e^{-\nu(t-s)}\Big(-\nu (t-s)^{\alpha-1}E_{\alpha,\alpha}(-\lambda_k^\beta (t-s)^\alpha)+(t-s)^{\alpha-2} \nonumber\\ &\relphantom{=}{}\cdot E_{\alpha,\alpha-1}(-\lambda_k^\beta(t-s)^\alpha)\Big)
\lambda_k^{\widetilde{\gamma}-\frac{\beta}{\alpha}}
\varphi_k(y)\sum_{j,l=1}^{\infty}(h(s,u(s))\cdot e_l,\varphi_j)\varphi_j(y)\varrho_l(s)
\\ &
\relphantom{=}{}\cdot
dyd\xi^H_l(s)\bigg{|}^2 \nonumber\\
&%\relphantom{=}{}
\leq CHT^{2H-1}\mathbb{E}\sum\limits_{k,l=1}^{\infty}
\int^t_0\bigg{|}e^{-\nu(t-s)}\Big(-\nu (t-s)^{\alpha-1}E_{\alpha,\alpha}(-\lambda_k^\beta (t-s)^\alpha) \\ 
&\relphantom{=}{} +(t-s)^{\alpha-2} E_{\alpha,\alpha-1}(-\lambda_k^\beta(t-s)^\alpha)\Big)
\lambda_k^{\widetilde{\gamma}-\frac{\beta}{\alpha}}
(h(s,u(s))\cdot e_l,\varphi_k)\varrho_l(s)\bigg{|}^2ds &
\end{flalign*}
\begin{flalign*}
& %\relphantom{=}{}
\leq C\mathbb{E}\sum\limits_{k,l=1}^{\infty}
\int^t_0\Big((t-s)^{2\alpha-2}+(t-s)^{2\alpha-4}\Big)\frac{1}{(1+\lambda_k^\beta(t-s)^\alpha)^2}\nonumber\\
&\relphantom{=}{}\cdot\left|\lambda_k^{\widetilde{\gamma}-\frac{\beta}{\alpha}}(h(s,u(s))\cdot e_l,\varphi_k)\varrho_l(s)\right|^2ds\nonumber\\
&%\relphantom{=}{}
\leq
C\int^t_0\Big((t-s)^{2\alpha-2}+(t-s)^{2\alpha-4}\Big)
\mathbb{E}\|(-\Delta)^{\widetilde{\gamma}-\frac{\beta}{\alpha}}h(s,u(s))\|_{\mathcal{L}_2^0}^2ds &
\end{flalign*}
\begin{flalign*}
&
%\relphantom{=}{}
\leq C
\int^t_0\Big((t-s)^{2\alpha-2}+(t-s)^{2\alpha-4}\Big)\left(1+\mathbb{E}\|u(s)\|_{\mathbb{H}^{2\widetilde{\gamma}-\frac{2\beta}{\alpha}}}^2\right)ds \\
&%\relphantom{=}{}
\leq C
\int^t_0\Big((t-s)^{2\alpha-2}+(t-s)^{2\alpha-4}\Big)
\left(1+\mathbb{E}\|u(s)\|_{\mathbb{H}^{2\widetilde{\gamma}}}^2\right)ds \\
&%\relphantom{=}{}
\leq C\sup_{s\in[0,T]}\mathbb{E}\left(1+\|u(s)\|_{\mathbb{H}^{2\widetilde{\gamma}}}^2\right), &
\end{flalign*}
where we have used $\frac{1}{(1+\lambda_k^\beta(t-s)^\alpha)^2}\leq C$ and $\mathbb{E}\|u(s)\|_{\mathbb{H}^{2\widetilde{\gamma}-\frac{2\beta}{\alpha}}}^2 \leq C \mathbb{E}\|u(s)\|_{\mathbb{H}^{2\widetilde{\gamma}}}^2$. Then it follows from \eqref{eq.3.37}-\eqref{eq.3.43} that
\begin{align}\label{eq.3.44}
&\sup_{0\leq t\leq T}\mathbb{E}\|\partial_t u(t)\|_{\mathbb{H}^{2\widetilde{\gamma}-\frac{2\beta}{\alpha}}}^2
\\
&\leq C\left(\mathbb{E}\|a\|_{\mathbb{H}^{2\widetilde{\gamma}}}^2+
\mathbb{E}\|b\|_{\mathbb{H}^{2\widetilde{\gamma}-\frac{2\beta}{\alpha}}}^2
+\sup_{s\in[0,T]}\mathbb{E}\left(1+\|u(s)\|_{\mathbb{H}^{2\widetilde{\gamma}}}^2\right)\right), \nonumber
\end{align}
which completes the proof of this theorem.
\end{proof}
\section{Regularity and approximation of white noise and fractional Gaussian noise}
\label{sec:4}

~~~~~Now we define a partition of $[0, T]$ by intervals $[t_i, t_{i+1}]$ for $i = 1,
2, \ldots, N$, where $t_i = (i-1)\tau $, $\tau = T/N$. A sequence of noise which
 approximates the space-time white noise is defined as
\begin{equation*}
\displaystyle \frac{\partial^2{\mathbb{W}_n}(t,x)}{\partial t \partial x}=\sum
\limits_{k=1}^{\infty}\varsigma^n_k(t)e_k(x)\left(\sum\limits_{i=1}^{N}
\frac{1}{\sqrt{\tau}}{\xi}_{ki}\chi_i(t)\right),
\end{equation*}
and another sequence of noise which approximates the space-time fractional noise with Hurst parameter $H\in(1/2,1)$ is defined as
\begin{equation*}
\displaystyle \frac{\partial^2{\mathbb{W}^H_n}(t,x)}{\partial t \partial x}=\sum
\limits_{k=1}^{\infty}\varrho^n_k(t)e_k(x)\left(\sum\limits_{i=1}^{N}
\frac{1}{\tau^{1-H}}{\xi}^H_{ki}\chi_i(t)\right),
\end{equation*}
where $\chi_i(t)$ is the
characteristic function for the $i$th time subinterval,
\[\xi_{ki} =\frac{1}{\sqrt{\tau}}\int_{t_i}^{t_{i+1}}d\xi_k(t)=\frac{1}{\sqrt{\tau}}
(\xi_k(t_{i+1})-\xi_k(t_i))\sim  \mathcal{N}(0, 1),\]
\[\xi_{ki}^H =\frac{1}{\tau^{H}}\int_{t_i}^{t_{i+1}}d\xi^H_k(t)=\frac{1}{\tau^{H}}
(\xi^H_k(t_{i+1})-\xi_k^H(t_i))\sim  \mathcal{N}(0, 1),\]
$\varsigma^n_k (t)$ and $\varrho^n_k (t)$, respectively, are the approximations of $\varsigma_k(t)$
and $\varrho_k(t)$ in the space direction.
 Then $\frac{\partial^2{\mathbb{W}_n}(t,x)}{\partial t \partial x}$ and  $\frac{\partial^2{\mathbb{W}^H_n}(t,x)}{\partial t \partial x}$ are, respectively,
 substituted for $\frac{\partial^2\mathbb{W}(t,x)}{\partial t \partial x}$ and  $\frac{\partial^2\mathbb{W}^H(t,x)}{\partial t \partial x}$ in
  \eqref{eq0.1} to obtain the equation %:
\begin{equation}\label{eq2.14}
\left\{ \begin{array}
 {l@{\quad} l}
\displaystyle _0^c \partial_t^{\alpha,\nu}u_n(t,x)+(-\Delta)^{\beta} u_n(t,x)=f(t,u_n(t,x))+g(t,u_n(t,x))\frac{\partial^2\mathbb{W}_n(t,x)}{\partial t \partial x}\\
\\
\displaystyle~~~~~~~~~~~~~~~~~~~~~~~~~~~~~~~~~~~~+h(t,u_n(t,x))\frac{\partial^2\mathbb{W}^H_n(t,x)}{\partial t \partial x}
\quad \rm{in~} (0,T]\times\mathcal{D},\\
 \\
u_n(t,x)=0 \quad \rm{on~} (0,T]\times\partial\mathcal{D},\\
 \\
u_n(0,x)=a(x),~\partial_tu_n(0,x)=b(x)\quad  \rm{in~} \mathcal{D}. \end{array}\right.
\end{equation}

As a simple consequence of Lemma \ref{le1}, we get an integral formulation of \eqref{eq2.14}.
\begin{lemma}\label{lem4.1} The solution $u_n$ to problem \eqref{eq2.14} with $\frac{3}{2}<\alpha<2$, $\frac{1}{2}<\beta\leq 1$, $\nu>0$, and $\frac{1}{2}<H<1$ is given by
\begin{equation}\label{eq2.15}\begin{split}
\displaystyle u_n(t,x)&=\int_{\mathcal{D}}\mathcal{T}^\nu_{\alpha,\beta}(t,x,y)a(y)dy
+\int_{\mathcal{D}}\mathcal{R}^\nu_{\alpha,\beta}(t,x,y)b(y)dy\\
&\relphantom{=}{}+\int_{0}^t\int_{\mathcal{D}}\mathcal{S}^\nu_{\alpha,\beta}(t-s,x,y)f(s,u_n(s,y))dyds\\
&\relphantom{=}{}+\int_{0}^t\int_{\mathcal{D}}\mathcal{S}^\nu_{\alpha,\beta}(t-s,x,y)g(s,u_n(s,y))
d{\mathbb{W}_n}(s,y)\\
&\relphantom{=}{}+\int_{0}^t\int_{\mathcal{D}}\mathcal{S}^\nu_{\alpha,\beta}(t-s,x,y)h(s,u_n(s,y))
d{\mathbb{W}^H_n}(s,y),
\end{split}\end{equation}
where $\mathcal{T}^\nu_{\alpha,\beta}(t,x,y)$, $\mathcal{R}^\nu_{\alpha,\beta}(t,x,y)$, and $\mathcal{S}^\nu_{\alpha,\beta}(t-s,x,y)$ are given in Lemma \ref{le1}.
\end{lemma}
The following theorem shows the regularity of the solution of \eqref{eq2.15}, which will be used in the error analysis.
\begin{theorem}\label{thm4.2}
Assume $\{\varsigma_k^n(t)\}$ and $\{\varrho_k^n(t)\}$ are uniformly bounded by $|\varsigma_k^n|\leq \mu_k^n$ and $|\varrho_k^n|\leq \widetilde{\mu}_k^n$ for all $t\in[0,T]$, and the series $(\{\mu_k^n\},\{\widetilde{\mu}_k^n\})$ is rapidly decaying with the increase of $k$. Further assume that the functions $f$, $g$, and $h$ satisfy $(\textbf{A}_1)$ for some $\gamma\geq 0$.
Let $(\textbf{A}_2)$ holds, $\frac{3}{2}<\alpha<2$, $\frac{1}{2}<\beta\leq 1$, $\nu>0$, and $\frac{1}{2}<H<1$, and let $u_n$ be the solution to \eqref{eq2.15}, where the $\mathcal{F}_0$-adapted random initial values satisfy $a\in L^2(\Omega;\mathbb{H}^{2\widetilde{\gamma}}(\mathcal{D}))$, $b\in L^2(\Omega;\mathbb{H}^{2\widetilde{\gamma}-\frac{2\beta}{\alpha}}(\mathcal{D}))$ with $\widetilde{\gamma}=\max(\gamma,\frac{\beta}{\alpha})$.
Then it holds that
\begin{align*}
\sup_{0\leq t\leq T}\mathbb{E}\|u_n(t)\|_{\mathbb{H}^{2\widetilde{\gamma}}}^2 &\leq \left(C+C(\mu_1^n)^2+C(\widetilde{\mu}_1^n)^2\right)\mathbb{E}\|a\|_{\mathbb{H}^{2\widetilde{\gamma}}}^2
\nonumber\\
&\relphantom{=}{}
+\left(C+C(\mu_1^n)^2+C(\widetilde{\mu}_1^n)^2\right)\mathbb{E}\|b\|_{\mathbb{H}^{2\widetilde{\gamma}-\frac{2\beta}{\alpha}}}^2
\nonumber\\
&\relphantom{=}{}+C+C(\mu_1^n)^2+C(\widetilde{\mu}_1^n)^2
+C\left(1+(\mu_1^n)^2+(\widetilde{\mu}_1^n)^2\right)^2,
\end{align*}
\begin{align*}
&\sup_{0\leq t\leq T}\mathbb{E}\|\partial_t u_n(t)\|_{\mathbb{H}^{2\widetilde{\gamma}-\frac{2\beta}{\alpha}}}^2\leq C\mathbb{E}\|a\|_{\mathbb{H}^{2\widetilde{\gamma}}}^2
+C\mathbb{E}\|b\|_{\mathbb{H}^{2\widetilde{\gamma}-\frac{2\beta}{\alpha}}}^2\\
&\relphantom{=}{}~~~~~~~~~~ +\left(C+C(\mu_1^n)^2+C(\widetilde{\mu}_1^n)^2\right)\left(1+\sup_{0\leq  t\leq T}
\mathbb{E}\|u_n(t)\|_{\mathbb{H}^{2\widetilde{\gamma}}}^2\right),
\end{align*}
\begin{align*}
&\sup_{0\leq t\leq T}
\mathbb{E}\|{_0^c\partial}_t^{\alpha,\nu} u_n(t)\|_{\mathbb{H}^{2\widetilde{\gamma}-2\beta}}^2
\\
&~~~~\leq \left(C+C(\mu_1^n)^2\tau^{-1}+C(\widetilde{\mu}_1^n)^2\tau^{2H-2}\right)
\left(1+\sup_{0\leq t\leq T}
\mathbb{E}\|u_n(t)\|_{\mathbb{H}^{2\widetilde{\gamma}}}^2\right),
\end{align*}
and for $0\leq \theta_1\leq\theta_2\leq T$,
\begin{align*}
&\mathbb{E}\|u_n(\theta_2)-u_n(\theta_1)\|_{\mathbb{H}^{2\widetilde{\gamma}-\frac{2\beta}{\alpha}}}^2
\\
&
\leq C|\theta_2-\theta_1|^2
\bigg(\mathbb{E}\|a\|_{\mathbb{H}^{2\widetilde{\gamma}}}^2
+\mathbb{E}\|b\|_{\mathbb{H}^{2\widetilde{\gamma}-\frac{2\beta}{\alpha}}}^2+1+
(\mu_1^n)^2+\sup_{0\leq s\leq T}
\mathbb{E}\|u_n(s)\|_{\mathbb{H}^{2\widetilde{\gamma}}}^2\nonumber\\
&\relphantom{=}{}
+(\widetilde{\mu}_1^n)^2+ (\mu_1^n)^2\sup_{0\leq s\leq T}
\mathbb{E}\|u_n(s)\|_{\mathbb{H}^{2\widetilde{\gamma}}}^2+(\widetilde{\mu}_1^n)^2 \sup_{0\leq s\leq T}
\mathbb{E}\|u_n(s)\|_{\mathbb{H}^{2\widetilde{\gamma}}}^2 \bigg).
\end{align*}
\end{theorem}
The proof of Theorem \ref{thm4.2} is given in  \ref{A.2}.

%--------------------------------------------------------------------------------------------
%-------------------------------------------------------------------------------------------
In order to prove that the solution $u_n$ of \eqref{eq2.14} indeed approximates
$u$, the solution of \eqref{eq0.1}, first we need the
assumptions on $\{\varsigma_k(t)\}$, $\{\varsigma^n_k (t)\}$, $\{\varrho_k(t)\}$, and $\{\varrho^n_k (t)\}$.
\begin{itemize}
  \item [$(\bf{A}_3)$] Assume that $\{\varsigma_k(t)\}$, $\{\varrho_k(t)\}$ and their derivatives are uniformly bounded by
\[|\varsigma_k(t)| \leq \mu_k,~ |\varrho_k(t)| \leq \widetilde{\mu}_k, ~|\varsigma^\prime_k(t)|\leq \gamma_k,~ |\varrho^\prime_k(t)|\leq \widetilde{\gamma}_k \quad \forall
t \in [0, T],\]
and that the coefficients $\{\varsigma^n_k (t)\}$ and $\{\varrho^n_k (t)\}$ are constructed such that
\[|\varsigma_k(t)- \varsigma^n_k (t)| \leq \eta^n_k , ~|\varrho_k(t)- \varrho^n_k (t)| \leq \widetilde{\eta}^n_k ,~|\varsigma^n_k (t)|\leq \mu^n_k ,~ |\varrho^n_k (t)|\leq \widetilde{\mu}^n_k,\] \[|(\varsigma^{n}_k)^\prime(t)| \leq \gamma^n_k,
~|(\varrho^{n}_k)^\prime(t)| \leq \widetilde{\gamma}^n_k \quad \forall t \in [0, T]\]
with positive sequences $\{\eta^n_k\}$ and $\{\widetilde{\eta}^n_k\}$ being arbitrarily chosen, and $\{\mu^n_k\}$, $\{\widetilde{\mu}^n_k\}$, $\{\gamma^n_k\}$, and
$\{\widetilde{\gamma}^n_k\}$ being related to $\{\eta^n_k \mu_k\}$, $\{\widetilde{\eta}^n_k \widetilde{\mu}_k\}$, $\{\gamma_k\}$, and $\{\widetilde{\gamma}_k\}$.
The series $(\{\mu_k\}$, $\{\widetilde{\mu}_k\}$, $\{\gamma_k\}$, $\{\widetilde{\gamma}_k\}$, $\{\gamma^n_k\}$, $\{\widetilde{\gamma}^n_k\}$, $\{\mu^n_k\}$, and $\{\widetilde{\mu}^n_k\})$ are rapidly decaying with the increase of $k$, and the series $(\{\eta^n_k\}$, $\{\widetilde{\eta}^n_k\})$ is required
to rapidly decay to ensure the convergence of the series in Theorem \ref{thm2.1}.
\end{itemize}
\begin{theorem}\label{thm2.1} Assume that the functions $f$, $g$ and $h$ satisfy $(\textbf{A}_1)$ for some $\gamma\geq 0$, and that $(\textbf{A}_3)$ holds, $\frac{3}{2}<\alpha<2$, $\frac{1}{2}<\beta\leq 1$, $\nu>0$ and $\frac{1}{2}<H<1$.
Let $u_n$ and $u$ be the solutions of \eqref{eq2.15} and \eqref{eq0.1}, respectively, where the $\mathcal{F}_0$-adapted random initial values satisfy $a\in L^2(\Omega;\mathbb{H}^{2\widetilde{\gamma}}(\mathcal{D}))$,  $b\in L^2(\Omega;\mathbb{H}^{2\widetilde{\gamma}-\frac{2\beta}{\alpha}}(\mathcal{D}))$ with $\widetilde{\gamma}=\max(\gamma,\frac{\beta}{\alpha})$. Then for some constant $C>0$ independent of $\tau$,
\begin{align*}
\mathbb{E}\|u(t)-u_n(t)\|_{\mathbb{H}^{2\widetilde{\gamma}-\frac{2\beta}{\alpha}}}^2&\leq C\tau^2+C\sum_{l=1}^{\infty}(\eta_l^n)^2+C\sum_{l=1}^{\infty}(\widetilde{\eta}_l^n)^2, \quad t>0,
\end{align*}
provided that the infinite series are all convergent.
\end{theorem}
The detailed proof of Theorem \ref{thm2.1} is given in the \ref{B.2}.
%=========================================================================================
%=========================================================================================
\section{Galerkin finite element approximation}\label{sec:5}
~~~~In this section, we provide the Galerkin FEM scheme and derive the corresponding
 error estimates.
%--------------------------------------------------------------------------------------
\subsection{Spatially Galerkin FEM and its properties}\label{sec:5.1}
Let $\mathcal{T}_{\bar{h}}$ be a shape regular and quasi-uniform triangulation of the convex polygonal  domain
$\mathcal{D}$. Let $\mathcal{S}_{\bar{h}}\subset \mathbb{H}^\beta(\mathcal{D})$ be the space of  piecewise polynomial functions with respect to $\tau_{\bar{h}}$, which are zero on the boundary of $\mathcal{D}$.

On the space $\mathcal{S}_{\bar{h}}$ we define the orthogonal $L^2$-projection $P_{\bar{h}}: \mathbb{H}^{0}(\mathcal{D}) \rightarrow \mathcal{S}_{\bar{h}}$ and the generalized Ritz projection $R_{\bar{h}} : \mathbb{H}^{\beta}(\mathcal{D}) \rightarrow \mathcal{S}_{\bar{h}}$, respectively, by
\begin{equation*}
(P_{\bar{h}}\psi,\chi)=(\psi,\chi)\quad\forall \chi\in \mathcal{S}_{\bar{h}},
\end{equation*}
\begin{equation*}
\left((-\Delta)^{\frac{\beta}{2}} R_{\bar{h}}\psi, (-\Delta)^{\frac{\beta}{2}}\chi\right)=\left((-\Delta)^{\frac{\beta}{2}}\psi,(-\Delta)^{\frac
{\beta}{2}}\chi\right) \quad \forall~ \chi \in \mathcal{S}_{\bar{h}}.
\end{equation*}
The projection $R_{\bar{h}}$ of $\psi$ is unique, since $\psi\in \mathbb{H}^{\beta}(\mathcal{D})$ and it equals to zero on the boundary.

In the next lemma, we establish the error estimates for $P_{\bar{h}}\psi$ and $R_{\bar{h}}\psi$; see  \cite{Li:16} for details.
\begin{lemma}\label{le3.0} The operators $P_{\bar{h}}$ and $R_{\bar{h}}$ satisfy
\begin{equation}\label{3.1}
\|P_{\bar{h}}\psi-\psi\|+{\bar{h}}^{\beta}\|(-\Delta)^{\frac{\beta}{2}}(P_{\bar{h}}\psi-\psi)\| \leq C{\bar{h}}^q\|\psi\|_{\mathbb{H}^q} \quad {\rm for}\quad \psi\in \mathbb{H}^q,~ q \in [\beta, r],
\end{equation}
and
\begin{equation}\label{3.2}
\|R_{\bar{h}}\psi- \psi\|+{\bar{h}}^{\beta}\|(-\Delta)^{\frac{\beta}{2}}(R_{\bar{h}}\psi-\psi)\| \leq C{\bar{h}}^q\|\psi\|_{\mathbb{H}^q} \quad{\rm for}\quad \psi\in \mathbb{H}^q,~ q \in [\beta, r].
\end{equation}
\end{lemma}
\begin{remark}
The number $r$ refers to the order of accuracy of the family $\{S_{\bar{h}}\}$. In the case $r=2$, $S_{\bar{h}}$ is a piecewise linear finite element subspace. For the case $r>2$, $S_{\bar{h}}$ often consists of piecewise polynomials of degree at most $r-1$ on a triangulation $\tau_{\bar{h}}$. For instance, $r=4$ in the case of piecewise cubic polynomial subspaces.
\end{remark}

The discrete fractional Laplacian $(-\Delta_{\bar{h}})^\beta : \mathcal{S}_{\bar{h}} \rightarrow \mathcal{S}_{\bar{h}}$ is then defined by
\begin{equation}\label{eq3.0}
((-\Delta_{\bar{h}})^\beta \psi, \chi) = ((-\Delta)^\frac{\beta}{2}\psi, (-\Delta)^\frac{\beta}{2} \chi) \quad \forall \psi, \chi \in \mathcal{S}_{\bar{h}},
\end{equation}
and thus we can write the spatial FEM approximation of \eqref{eq2.14} as
\begin{equation}\label{eq3.1}\begin{split}
^c_0\partial_t^{\alpha,\nu} u^{\bar{h}}_n(t,x)+(-\Delta_{\bar{h}})^\beta u^{\bar{h}}_n(t,x)=&P_{\bar{h}}\bigg(f(t,u_n^{\bar{h}}(t,x))+
g(t,u_n^{\bar{h}}(t,x))\frac{\partial^2 \mathbb{W}_n(t,x)}{\partial t \partial x}\\
&+h(t,u_n^{\bar{h}}(t,x))\frac{\partial^2 \mathbb{W}_n^H(t,x)}{\partial t \partial x}\bigg), \quad 0<t\leq T,
\end{split}\end{equation}
with $u^{\bar{h}}_n(0)=a_{\bar{h}}$ and $\partial_t{u^{\bar{h}}_n}(0)=b_{\bar{h}},$
where $a_{\bar{h}}=P_{\bar{h}} a$, $b_{\bar{h}}=P_{\bar{h}} b$.

Now we give a representation of the solution of \eqref{eq3.1} using the eigenvalues and eigenfunctions $\{\lambda_k^{{\bar{h}},\beta}\}_{k=1}^{M}$ and $\{\varphi_k^{\bar{h}}\}_{k=1}^{M}$ of the discrete fractional Laplacian $(-\Delta_{\bar{h}})^\beta$. Since we know that the operator $(-\Delta_{\bar{h}})^\beta$ is symmetrical, $\{\varphi_k^{\bar{h}}\}_{k=1}^{M}$ is orthogonal. Take $\{\varphi_k^{\bar{h}}\}_{k=1}^{M}$ as the orthonormal bases in $\mathcal{S}_{\bar{h}}$ and define the discrete analogues of \eqref{eq2.2'}-\eqref{eq2.2} by
\begin{equation}\label{eq3.2}
\widehat{\mathcal{T}}_{\alpha,\beta}^{\nu}(t,x,y)= e^{-\nu t}\sum\limits_{k=1}
^{M}\left(E_{\alpha,1}\left(-\lambda_k^{{\bar{h}},\beta}t^{\alpha}\right)+\nu tE_{\alpha,2}
\left(-\lambda_k^{{\bar{h}},\beta}t^{\alpha}\right)\right)\varphi_k^{\bar{h}}(x)\varphi_k^{\bar{h}}(y),
\end{equation}
\begin{equation}\label{eq3.6}
\widehat{\mathcal{R}}_{\alpha,\beta}^{\nu}(t,x,y)= te^{-\nu t}\sum\limits_{k=1}^{M}
E_{\alpha,2}\left(-\lambda_k^{{\bar{h}},\beta} t^\alpha\right)\varphi_k^{\bar{h}}(x)\varphi_k^{\bar{h}}(y),
\end{equation}
and
\begin{equation}\label{eq3.3}
\widehat{\mathcal{S}}_{\alpha,\beta}^{\nu}(t,x,y)= t^{\alpha-1}e^{-\nu t}
\sum\limits_{k=1}^{M}E_{\alpha,\alpha}\left(-\lambda_k^{{\bar{h}},\beta} t^\alpha\right)
\varphi_k^{\bar{h}}(x)\varphi_k^{\bar{h}}(y).
\end{equation}
Then the solution $u^{\bar{h}}_n$ of the discrete problem \eqref{eq3.1} can be expressed by
\begin{equation}\label{eq3.4}
\begin{split}
\displaystyle u_n^{\bar{h}}(t, x) &=\int_{\mathcal{D}}\widehat{\mathcal{T}}^\nu_{\alpha,\beta}(t,x,y)a_{\bar{h}}(y)dy
+\int_{\mathcal{D}}\widehat{\mathcal{R}}^\nu_{\alpha,\beta}(t,x,y)b_{\bar{h}}(y)dy\\
&\relphantom{=}{}+\int_{0}^t\int_{\mathcal{D}}\widehat{\mathcal{S}}^\nu_{\alpha,\beta}(t-s,x,y)P_{\bar{h}}
f(s,u_n^{\bar{h}}(s,y))dyds\\
&\relphantom{=}{}+\int_{0}^t\int_{\mathcal{D}}\widehat{\mathcal{S}}^\nu_{\alpha,\beta}(t-s,x,y)
P_{\bar{h}}g(s,u_n^{\bar{h}}(s,y))d{\mathbb{W}}(s,y)\\
&\relphantom{=}{}+\int_{0}^t\int_{\mathcal{D}}\widehat{\mathcal{S}}^\nu_{\alpha,\beta}(t-s,x,y)
P_{\bar{h}}h(s,u_n^{\bar{h}}(s,y))d{\mathbb{W}}^H(s,y).
\end{split}
\end{equation}

Also, on the finite element space $\mathcal{S}_{\bar{h}}$, we introduce the discrete norm $|\!|\!|\cdot|\!|\!|_{p}$ for any $p\in \mathbb{R}$ defined by
\begin{equation}\label{eq3.7}
|\!|\!|\psi|\!|\!|_{p}^2=\sum\limits^M_{k=1}(\lambda_k^{\bar{h},\beta})^{\frac{p}{\beta}}
(\psi,\varphi^{\bar{h}}_k)^2, \quad \psi\in \mathcal{S}_{\bar{h}}.
\end{equation}
It is clear that the norm $|\!|\!|\cdot|\!|\!|_{p}$ is well defined for all real
$p$. From the definition of the discrete fractional Laplacian$(-\Delta_{\bar{h}})^{\beta}$
we have $|\!|\!|\psi|\!|\!|_{p}=\|\psi\|_{\mathbb{H}^p}$ for $p=0, \beta$ and for all $\psi\in \mathcal{S}_{\bar{h}}$. Therefore there is no confusion in using $\|\psi\|_{\mathbb{H}^p}$ instead of $|\!|\!|\psi|\!|\!|_{p}$ for $p=0, \beta$ and for all $\psi\in \mathcal{S}_{\bar{h}}$.
Further, we need the following inverse inequality; see \cite[Lemma 3.2]{Li:16} for details.
\begin{lemma}\label{le5.2} For any $l >s$, there exists a constant $C$ independent of $\bar{h}$ such that
\begin{equation}\label{eq3.5}
\displaystyle |\!|\!|\chi|\!|\!|_{l}\leq C\bar{h}^{s-l}|\!|\!|\chi|\!|\!|_{s} \quad \forall
\chi\in \mathcal{S}_{\bar{h}}.
\end{equation}
\end{lemma}

As the discrete analogues of Lemma \ref{le2.2}, we have following estimates.
\begin{lemma}\label{le3.1}
Let $\widehat{\mathcal{T}}_{\alpha,\beta}^{\nu}(t,x,y)$ be defined by \eqref{eq3.2} and $a_{\bar{h}} \in \mathcal{S}_{\bar{h}}$. Then, for all $t>0$,
\begin{equation*}
\displaystyle \bigg|\!\bigg|\!\bigg|\int_{\mathcal{D}}\widehat{\mathcal{T}}_{\alpha,\beta}^{\nu}(t,x,y)
a_{\bar{h}}(y)dy\bigg|\!\bigg|\!\bigg|_{p} \leq \left\{\begin{array}{l}
C (1+\nu t)e^{-\nu t}t^{\frac{\alpha (q-p)}{2\beta}}|\!|\!|a_{\bar{h}}|\!|\!|_{q},\quad 0\leq q \leq p \leq 2\beta,\\
\\
C(1+\nu t)e^{-\nu t}t^{-\alpha}|\!|\!|a_{\bar{h}}|\!|\!|_{q}, \quad\quad\quad~~ q>p.
\end{array}\right.
\end{equation*}
\end{lemma}
\begin{lemma}\label{le3.4}
Let $\widehat{\mathcal{R}}_{\alpha,\beta}^{\nu}(t,x,y)$ be defined by \eqref{eq3.6} and
$b_{\bar{h}} \in \mathcal{S}_{\bar{h}}$. Then, for all $t>0$,
\begin{equation*}
\displaystyle \bigg|\!\bigg|\!\bigg|\int_{\mathcal{D}}\widehat{\mathcal{R}}_{\alpha,\beta}^{\nu}(t,x,y)
b_{\bar{h}}(y)dy\bigg|\!\bigg|\!\bigg|_{p} \leq \left\{\begin{array}{l}
Ce^{-\nu t}t^{1-\frac{\alpha (p-q)}{2\beta}}|\!|\!|b_{\bar{h}}|\!|\!|_{q},
\quad 0 \leq q\leq p\leq 2\beta,\\ \\
Ce^{-\nu t}t^{1-\alpha}|\!|\!|b_{\bar{h}}|\!|\!|_{q}, \quad\quad\quad~~ q>p.
\end{array}\right.
\end{equation*}
\end{lemma}

By slightly modifying the proof of Lemma 3.5 in \cite{Li:16}, we have
\begin{lemma}\label{le3.2}
Let $\widehat{\mathcal{S}}_{\alpha,\beta}^{\nu}(t,x,y)$ be defined by \eqref{eq3.3}
and $\psi\in \mathcal{S}_{\bar{h}}$. Then, for all $t > 0$,
\begin{equation*}
\displaystyle \bigg|\!\bigg|\!\bigg|\int_{\mathcal{D}}\widehat{\mathcal{S}}_{\alpha,\beta}^{\nu}(t,x,y)
\psi(y)dy\bigg|\!\bigg|\!\bigg|_{p}
 \leq\left\{\begin{array}{l}
 Ce^{-\nu t}t^{-1+\alpha+\frac{\alpha(q-p)}{2\beta}}|\!|\!|\psi|\!|\!|_{q}, ~~~  p-2\beta \leq q \leq p,\\
\\
\displaystyle Ce^{-\nu t}t^{-1} |\!|\!|\psi|\!|\!|_{q},~~~~\quad\quad\quad  q>p.
\end{array}
\right.
\end{equation*}
\end{lemma}
%===================================================================================
\subsection{Mean-square convergence analysis}
~~~~In this subsection, we derive a error estimate for the problem \eqref{eq2.14}. First, we
need the following Lipschitz assumption:
\begin{enumerate}
  \item [$(\bf{A}_4)$] There exists a positive constant $l'$ such that for any $\gamma\geq 0$, %$$
  \begin{align*}
 & \|f(t,u_1)-f(t,u_2)\|+\|g(t,u_1)-g(t,u_2)\|_{\mathcal{L}_2^0}+\|h(t,u_1)-h(t,u_2)\|_{\mathcal{L}_2^0}
  \\
 & \leq l'\|u_1-u_2\|
  \end{align*}
  %$$
      for all $t\in\mathbb{R}$ and $u_1, u_2\in \mathbb{H}$.
\end{enumerate}
\begin{theorem}\label{thm5.6}
Assume that the functions $f$, $g$ and $h$ satisfy $(\textbf{A}_1)$ for some $\gamma\geq 0$, $(\textbf{A}_3)$-$(\textbf{A}_4)$ hold, $\frac{3}{2}<\alpha<2$, $\frac{1}{2}<\beta\leq 1$, $\nu>0$, $\frac{1}{2}<H<1$, and that the $\mathcal{F}_0$-adapted random initial values satify $a\in L^2(\Omega;\mathbb{H}^{2\widetilde{\gamma}}(\mathcal{D}))$, $b\in L^2(\Omega;\mathbb{H}^{2\widetilde{\gamma}-\frac{2\beta}{\alpha}}(\mathcal{D}))$, with $\widetilde{\gamma}=\max(\gamma,\frac{\beta}{\alpha})$. Let $u_n$ and $u_n^h$ be the solutions of \eqref{eq2.14} and \eqref{eq3.1}, respectively.
Then, with $\ell_{\bar{h}}=|\ln {\bar{h}}|$,
\begin{equation*}\begin{split}
\displaystyle&\mathbb{E}\|u_n(t)-u_n^{\bar{h}}(t)\|^2+{\bar{h}}^{2\beta}\mathbb{E}
\|(-\Delta)^{\frac{\beta}{2}}(u_n^{\bar{h}}(t)-u_n(t))\|^2
\\
& \leq C\ell_{\bar{h}}
 {\bar{h}}^{4\widetilde{\gamma}}+C {\bar{h}}^{4\widetilde{\gamma}} \quad \forall~
t \in [0, T],
\end{split}
\end{equation*}
where $C$ is a positive constant independent of $\tau$ and ${\bar{h}}$.
\end{theorem}
\begin{proof}
We split the error $u_n^{\bar{h}}-u_n$ into
$$u_n^{\bar{h}}-u_n=(u_n^{\bar{h}}-P_{\bar{h}}u_n)+(P_{\bar{h}}u_n-u_n):=\Pi_1+\Pi_2.$$
By \eqref{eq.4.9} and \eqref{3.1} we have
\begin{equation}\label{3.9}
\begin{split}
\mathbb{E}\|\Pi_2(t)\|^2+{\bar{h}}^{2\beta}\mathbb{E}\|(-\Delta)^{\frac{\beta}{2}}\Pi_2(t)\|^2
&\leq C\bar{h}^{4\beta}\mathbb{E}\|u_n(t)\|_{\mathbb{H}^{2\beta}}^2.
\end{split}
\end{equation}
Moreover, we consider the equation
\begin{align}\label{eq.5.12}
_0^c\partial_t^{\alpha,\nu} v(t,x) +(-\Delta_{\bar{h}})^{\beta}
v(t,x)=&P_{\bar{h}}\bigg(f(t,u_n(t,x))+g(t,u_n(t,x))\frac{\partial^2 \mathbb{W}_n(t,x)}{\partial t \partial x}\nonumber\\
&+h(t,u_n(t,x))\frac{\partial^2 \mathbb{W}_n^H(t,x)}{\partial t \partial x}\bigg)
\end{align}
with $0<t\leq T$, $v(0)=a_{\bar{h}}=P_{\bar{h}}a$, and $\partial_t{v^{\bar{h}}_n}(0)=b_{\bar{h}}(x)=P_{\bar{h}}b$.
Let $\Pi_1^1=v-P_{\bar{h}}u_n$, $\Pi_1^2=u_n^{\bar{h}}-v$. Then
$\Pi_1=\Pi_1^1+\Pi_1^2$.

It follows from \eqref{eq2.14} and \eqref{eq.5.12} that
\[_0^c\partial_t^{\alpha,\nu} \Pi_1^1 +(-\Delta_{\bar{h}})^{\beta}
\Pi_1^1 = (-\Delta_{\bar{h}})^{\beta}(R_{\bar{h}}u_n-P_{\bar{h}}u_n)
\quad{\rm with}\quad\Pi_1^1(0)=\partial_t\Pi_1^1(0)=0,\]
where we have used the identity $(-\Delta_{\bar{h}})^{\beta}R_{\bar{h}}=P_{\bar{h}}(-\Delta)^{\beta}$.
By \eqref{eq3.4}, we obtain that
\begin{equation*}\begin{split}
\Pi_1^1(t,x)&=\int_{0}^t\int_{\mathcal{D}}\widehat{\mathcal{S}}_{\alpha,\beta}^
{\nu}(t-s,x,y)(-\Delta_{\bar{h}})^{\beta}(R_{\bar{h}}u_n(s,y)-P_{\bar{h}}u_n(s,y))dyds.
\end{split}\end{equation*}
For any $0<\varepsilon<2\beta$, by H\"{o}lder's inequality, and using Lemmas \ref{le3.0}, \ref{le5.2} and \ref{le3.2}, we deduce that, for $p=0, \beta$,
\begin{flalign}\label{eq.5.13}
&\displaystyle \mathbb{E}\|\Pi_1^1(t)\|_{\mathbb{H}^p}^2 
\nonumber\\
&\leq C\mathbb{E}
\bigg(\int^t_0\bigg{\|}\int_{\mathcal{D}}\widehat{\mathcal{S}}_{\alpha,\beta}^{\nu}
(t-s,x,y)(-\Delta_{\bar{h}})^{\beta}(R_{\bar{h}}u_n(s,y)-P_{\bar{h}}u_n(s,y))dy
\bigg{\|}_{\mathbb{H}^p}ds\bigg)^2\nonumber\\
\displaystyle & \leq C \mathbb{E}\left(\int^t_0(t-s)^{\frac{\alpha\varepsilon}
{2\beta}-1}e^{-\nu(t-s)}|\!|\!|(-\Delta_{\bar{h}})^{\beta}(R_{\bar{h}}u_n-P_{\bar{h}}u_n)(s)
|\!|\!|_{\varepsilon-2\beta+p}ds\right)^2\nonumber\\
\displaystyle & =C \mathbb{E}\left(\int^t_0(t-s)^{\frac{\alpha\varepsilon}
{2\beta}-1}e^{-\nu(t-s)}|\!|\!|(R_{\bar{h}}u_n-P_{\bar{h}}u_n)(s)|\!|\!|_{\varepsilon+p}
ds\right)^2\nonumber\\
\displaystyle & \leq C{\bar{h}}^{-2\varepsilon}\mathbb{E}\left(\int^t_0(t-s)^{\frac{\alpha
\varepsilon}{2\beta}-1}e^{-\nu(t-s)}\|(R_{\bar{h}}u_n-P_{\bar{h}}u_n)(s)\|_{\mathbb{H}^p}ds\right)^2\nonumber
&
\end{flalign}
\begin{flalign}
&
 \displaystyle\leq C {\bar{h}}^{4\widetilde{\gamma}-2p-2\varepsilon}\mathbb{E}\left(\int^t_0(t-s)^
{\frac{\alpha\varepsilon}{2\beta}-1}e^{-\nu(t-s)}\|u_n(s)\|_{\mathbb{H}^{2\widetilde{\gamma}}}ds
\right)^2\nonumber\\
\displaystyle & \leq C {\bar{h}}^{4\widetilde{\gamma}-2p-2\varepsilon}\int^t_0(t-s)^{\frac{\alpha
\varepsilon}{2\beta}-1}e^{-2\nu(t-s)}ds\int^t_0(t-s)^{\frac{\alpha\varepsilon}
{2\beta}-1}\mathbb{E}\|u_n(s)\|_{\mathbb{H}^{2\widetilde{\gamma}}}^2ds\nonumber\\
\displaystyle & \leq C {\bar{h}}^{4\widetilde{\gamma}-2p-2\varepsilon}(2\nu)^{-\frac{\alpha
\varepsilon}{2\beta}}\Gamma\left(\frac{\alpha\varepsilon}{2\beta}\right)\int^t_0
(t-s)^{\frac{\alpha\varepsilon}{2\beta}-1}\mathbb{E}\|u_n(s)\|_{\mathbb{H}^{2\widetilde{\gamma}}}^2ds\nonumber\\
& \leq C {\bar{h}}^{4\widetilde{\gamma}-2p-2\varepsilon}(2\nu)^{-\frac{\alpha
\varepsilon}{2\beta}}\Gamma\left(\frac{\alpha\varepsilon}{2\beta}\right)
t^{\frac{\alpha\varepsilon}{2\beta}}\nonumber\\
& \leq C \varepsilon^{-1} {\bar{h}}^{4\widetilde{\gamma}-2p-2\varepsilon}
\leq C\ell_{\bar{h}}\bar{h}^{4\widetilde{\gamma}-2p}
\quad \forall t\in[0,T]. &
\end{flalign}
The last two inequalities follow from the fact $\Gamma(\frac{\alpha\varepsilon}
{2\beta})\sim \frac{2\beta}{\alpha\varepsilon}$ as $\varepsilon\rightarrow 0^+$ and by choosing $\varepsilon=1/\ell_{\bar{h}}$.

Noticing that $\Pi_1^2$ satisfies the following equation
\begin{align*}
&_0^c\partial_t^{\alpha,\nu} \Pi_1^2(t,x) +(-\Delta_{\bar{h}})^{\beta}
\Pi_1^2(t,x)=P_{\bar{h}}f(t,u_n^{\bar{h}}(t,x))-P_{\bar{h}}f(t,u_n(t,x))\\
&\relphantom{=}{}+P_{\bar{h}}g(t,u_n^{\bar{h}}(t,x))\frac{\partial^2 \mathbb{W}_n(t,x)}{\partial t \partial x}-P_{\bar{h}}g(t,u_n(t,x))\frac{\partial^2 \mathbb{W}_n(t,x)}{\partial t \partial x}\\
&\relphantom{=}{}+P_{\bar{h}}h(t,u_n^{\bar{h}}(t,x))\frac{\partial^2 \mathbb{W}_n^H(t,x)}{\partial t \partial x}
-P_{\bar{h}}h(t,u_n(t,x))\frac{\partial^2 \mathbb{W}_n^H(t,x)}{\partial t \partial x}
\end{align*}
with $\Pi_1^2(0)=\partial_t\Pi_1^2(0)=0$, by \eqref{eq3.4} we have
\begin{align*}
&\Pi_1^2(t,x)
\\
&=\int_{0}^t\int_{\mathcal{D}}\widehat{\mathcal{S}}_{\alpha,\beta}^
{\nu}(t-s,x,y)(P_{\bar{h}}f(s,u_n^{\bar{h}}(s,y))-P_{\bar{h}}f(s,u_n(s,y)))dyds\\
&\relphantom{=}{}+\int_{0}^t\int_{\mathcal{D}}\widehat{\mathcal{S}}_{\alpha,\beta}^
{\nu}(t-s,x,y)(P_{\bar{h}}g(s,u_n^{\bar{h}}(s,y))d\mathbb{W}_n(s,y)-P_{\bar{h}}g(s,u_n(s,y))
\\
&\relphantom{=}{}
\times d\mathbb{W}_n(s,y))+\int_{0}^t\int_{\mathcal{D}}\widehat{\mathcal{S}}_{\alpha,\beta}^
{\nu}(t-s,x,y)(P_{\bar{h}}h(s,u_n^{\bar{h}}(s,y))d\mathbb{W}_n^H(s,y)
\\
&\relphantom{=}{}-P_{\bar{h}}h(s,u_n(s,y))d\mathbb{W}_n^H(s,y)).
\end{align*}
For $p=0,\beta$, we observe that
\begin{flalign}\label{4.10}
&\mathbb{E}|\!|\Pi_1^2(t)|\!|_{\mathbb{H}^p}^2
\nonumber\\
&
\leq C\mathbb{E}
\bigg{\|}\int^t_0\int_{\mathcal{D}}\widehat{\mathcal{S}}_{\alpha,\beta}^{\nu}
(t-s,x,y)(P_{\bar{h}}f(s,u_n^{\bar{h}}(s,y))-P_{\bar{h}}f(s,u_n(s,y)))dy
ds\bigg{\|}_{\mathbb{H}^p}^2\nonumber\\
&\relphantom{=}{}+C\mathbb{E}
\bigg{\|}\int^t_0\int_{\mathcal{D}}\widehat{\mathcal{S}}_{\alpha,\beta}^{\nu}
(t-s,x,y)(P_{\bar{h}}g(s,u_n^{\bar{h}}(s,y))d\mathbb{W}_n(s,y)
\nonumber\\
&\relphantom{=}{}
-P_{\bar{h}}g(s,u_n(s,y))d\mathbb{W}_n(s,y))
\bigg{\|}_{\mathbb{H}^p}^2+C\mathbb{E}
\bigg{\|}\int^t_0\int_{\mathcal{D}}\widehat{\mathcal{S}}_{\alpha,\beta}^{\nu}
(t-s,x,y)
\nonumber\\
&\relphantom{=}{}
\times (P_{\bar{h}}h(s,u_n^{\bar{h}}(s,y))d\mathbb{W}_n^H(s,y)-P_{\bar{h}}h(s,u_n(s,y))d\mathbb{W}_n^H(s,y))
\bigg{\|}_{\mathbb{H}^p}^2\nonumber\\
&\relphantom{=}{}:=\Lambda_1+\Lambda_2+\Lambda_3.
\end{flalign}

By using the fact that $\{\varphi_k^{\bar{h}}\}_{k=1}^{M}$ is an orthonormal basis in $S_{\bar{h}}$, the assumption on $f$ given in $(\bf{A}_4)$, Lemma \ref{le2.1}, H\"{o}lder's inequality, \eqref{eq3.3}, \eqref{3.9}  and \eqref{eq.5.13}, we deduce that
\begin{flalign}\label{4.7}
&\Lambda_1\nonumber\\
&=C\mathbb{E}\sum_{k=1}^{M}(\lambda_{k}^{\bar{h},\beta})^{\frac{p}{\beta}}
\bigg(\int^t_0\int_{\mathcal{D}}(t-s)^{\alpha-1}e^{-\nu(t-s)}
\sum_{l=1}^{M}E_{\alpha,\alpha}(-\lambda_{l}^{\bar{h},\beta}(t-s)^\alpha)\nonumber\\
&\relphantom{=}{}\times\varphi_l^{\bar{h}}(\cdot)
\varphi_l^{\bar{h}}(y)(P_{\bar{h}}f(s,u_n^{\bar{h}}(s,y))-P_{\bar{h}}f(s,u_n(s,y)))dyds,\varphi_k^{\bar{h}}\bigg)^2\nonumber\\
&=C\mathbb{E}\sum_{k=1}^{M}(\lambda_{k}^{\bar{h},\beta})^{\frac{p}{\beta}}
\bigg|\int^t_0\int_{\mathcal{D}}(t-s)^{\alpha-1}e^{-\nu(t-s)}
E_{\alpha,\alpha}(-\lambda_{k}^{\bar{h},\beta}(t-s)^\alpha)
\varphi_k^{\bar{h}}(y)\nonumber\\
&\relphantom{=}{}\times(P_{\bar{h}}f(s,u_n^{\bar{h}}(s,y))-P_{\bar{h}}f(s,u_n(s,y)))dyds\bigg|^2\nonumber
&
\end{flalign}
\begin{flalign}
&\leq CT\mathbb{E}\sum_{k=1}^{M}\int^t_0
(t-s)^{2\alpha-2-\frac{\alpha p}{\beta}}
e^{-2\nu(t-s)}\frac{(\lambda_{k}^{\bar{h},\beta}(t-s)^\alpha)^{\frac{p}{\beta}}}{(1+\lambda_{k}^{\bar{h},\beta}(t-s)^\alpha)^2}\nonumber\\
&\relphantom{=}{}\times
(P_{\bar{h}}f(s,u_n^{\bar{h}}(s))-P_{\bar{h}}f(s,u_n(s)),\varphi_k^{\bar{h}})^2ds\nonumber\\
&\leq C\mathbb{E}\int^t_0
(t-s)^{2\alpha-2-\frac{\alpha p}{\beta}}e^{-2\nu(t-s)}\|P_{\bar{h}}f(s,u_n^{\bar{h}}(s))-P_{\bar{h}}f(s,u_n(s))\|^2ds\nonumber\\
&\leq C\int^t_0(t-s)^{2\alpha-2-\frac{\alpha p}{\beta}}e^{-2\nu(t-s)}\mathbb{E}\|u_n^{\bar{h}}(s)-u_n(s)\|^2ds\nonumber
&
\end{flalign}
\begin{flalign}
&\leq C\int^t_0(t-s)^{2\alpha-2-\frac{\alpha p}{\beta}}e^{-2\nu(t-s)}\mathbb{E}\|u_n^{\bar{h}}(s)-u_n(s)\|_{\mathbb{H}^p}^2ds\nonumber\\
&\leq C\int^t_0(t-s)^{2\alpha-2-\frac{\alpha p}{\beta}}e^{-2\nu(t-s)}(\mathbb{E}\|\Pi_1^2(s)\|_{\mathbb{H}^p}^2+\mathbb{E}\|\Pi_1^1(s)\|_{\mathbb{H}^p}^2
\nonumber\\
&\relphantom{=}{}+\mathbb{E}\|\Pi_2(s)\|_{\mathbb{H}^p}^2)ds\nonumber\\
&\leq C\int^t_0(t-s)^{2\alpha-2-\frac{\alpha p}{\beta}}e^{-2\nu(t-s)}\mathbb{E}\|\Pi_1^2(s)\|_{\mathbb{H}^p}^2ds +C\ell_{\bar{h}}\bar{h}^{4\widetilde{\gamma}-2p}+C\bar{h}^{4\widetilde{\gamma}-2p}, &
\end{flalign}
where we have used $\frac{(\lambda_{k}^{\bar{h},\beta}(t-s)^\alpha)^{\frac{p}{\beta}}}{(1+\lambda_{k}^{\bar{h},\beta}(t-s)^\alpha)^2}\leq C$ and $\mathbb{E}\|u_n^{\bar{h}}(s)-u_n(s)\|^2\leq C\mathbb{E}\|u_n^{\bar{h}}(s)-u_n(s)\|_{\mathbb{H}^p}^2$.

Without loss of generality, we assume that there exists a positive integer $N_t$ such that $t=t_{N_t+1}$. Since $\{\varphi_k^{\bar{h}}\}_{k=1}^{M}$ is an orthonormal basis in $S_{\bar{h}}$,
$\{\xi_l\}_{l=1}^\infty$ is a family of mutually independent one-dimensional standard Brownian motions with independent increments, then by Lemma \ref{le2.1}, $(\bf{A}_4)$, H\"{o}lder's inequality, the It\^{o} isometry, the boundedness assumption on $\varsigma_k^n(t)$, \eqref{eq3.3}, \eqref{3.9} and \eqref{eq.5.13}, we obtain that
\begin{flalign}\label{4.8}
&\Lambda_2 \nonumber
\\
&=C\mathbb{E}\sum_{k=1}^{M}(\lambda_{k}^{\bar{h},\beta})^{\frac{p}{\beta}}
\bigg(\int^t_0\int_{\mathcal{D}}(t-s)^{\alpha-1}e^{-\nu(t-s)}
\sum_{l=1}^{M}E_{\alpha,\alpha}(-\lambda_{l}^{\bar{h},\beta}(t-s)^\alpha)\nonumber\\
&\relphantom{=}{}\times\varphi_l^{\bar{h}}(\cdot)
\varphi_l^{\bar{h}}(y)(P_{\bar{h}}g(s,u_n^{\bar{h}}(s,y))-P_{\bar{h}}g(s,u_n(s,y)))d\mathbb{W}_n(s,y),\varphi_k^{\bar{h}}\bigg)^2\nonumber &
\end{flalign}
\begin{flalign}
&=C\mathbb{E}\sum_{k=1}^{M}(\lambda_{k}^{\bar{h},\beta})^{\frac{p}{\beta}}
\bigg|\int^t_0\int_{\mathcal{D}}(t-s)^{\alpha-1}e^{-\nu(t-s)}
E_{\alpha,\alpha}(-\lambda_{k}^{\bar{h},\beta}(t-s)^\alpha)
\varphi_k^{\bar{h}}(y)\nonumber\\
&\relphantom{=}{}\times(P_{\bar{h}}g(s,u_n^{\bar{h}}(s,y))d\mathbb{W}_n(s,y)-P_{\bar{h}}g(s,u_n(s,y))d\mathbb{W}_n(s,y))\bigg|^2\nonumber\\
&=C\mathbb{E}\sum_{k=1}^{M}(\lambda_{k}^{\bar{h},\beta})^{\frac{p}{\beta}}
\bigg|\int^t_0\int_{\mathcal{D}}(t-s)^{\alpha-1}e^{-\nu(t-s)}
E_{\alpha,\alpha}(-\lambda_{k}^{\bar{h},\beta}(t-s)^\alpha)
\varphi_k^{\bar{h}}(y)\nonumber\\
&\relphantom{=}{}\times\sum_{j=1}^{M}\sum_{l=1}^{\infty}(P_{\bar{h}}(g(s,u_n^{\bar{h}}(s))-g(s,u_n(s)))\cdot e_l,\varphi_j^{\bar{h}})\varphi_j^{\bar{h}}(y)\varsigma_l^n(s)
\nonumber\\
&\relphantom{=}{}\times
\left(\sum_{i=1}^{N_t}\frac{1}{\sqrt{\tau}}\xi_{li}\chi_i(s)\right)dyds\bigg|^2\nonumber&
\end{flalign}
\begin{flalign}
&=C\mathbb{E}\sum_{k=1}^{M}(\lambda_{k}^{\bar{h},\beta})^{\frac{p}{\beta}}
\bigg{|}\sum_{i=1}^{N_t}\int^{t_{i+1}}_{t_i}(t-s)^{\alpha-1}e^{-\nu(t-s)}
E_{\alpha,\alpha}(-\lambda_{k}^{\bar{h},\beta}(t-s)^\alpha)\nonumber\\
&\relphantom{=}{}\times
\sum_{l=1}^{\infty}(P_{\bar{h}}(g(s,u_n^{\bar{h}}(s))-g(s,u_n(s)))\cdot e_l,\varphi_k^{\bar{h}})\varsigma_l^n(s)\frac{\xi_{l}(t_{i+1})-\xi_{l}(t_{i})}{\tau}ds\bigg|^2\nonumber\\
&=\frac{C}{\tau^2}\mathbb{E}\sum_{k=1}^{M}(\lambda_{k}^{\bar{h},\beta})^{\frac{p}{\beta}}
\bigg{|}\sum_{i=1}^{N_t}\int^{t_{i+1}}_{t_i}
\int^{t_{i+1}}_{t_i}(t-s)^{\alpha-1}e^{-\nu(t-s)}
E_{\alpha,\alpha}(-\lambda_{k}^{\bar{h},\beta}(t-s)^\alpha)\nonumber\\
&\relphantom{=}{}\times
\sum_{l=1}^{\infty}(P_{\bar{h}}(g(s,u_n^{\bar{h}}(s))-g(s,u_n(s)))\cdot e_l,\varphi_k^{\bar{h}})\varsigma_l^n(s)dsd\xi_l(r)\bigg{|}^2\nonumber&
\end{flalign}
\begin{flalign}
&\leq\frac{C}{\tau^2}\mathbb{E}\sum_{k=1}^{M}\sum_{l=1}^{\infty}(\lambda_{k}^{\bar{h},\beta})^{\frac{p}{\beta}}
\nonumber\\
&\relphantom{=}{}\times
\sum_{i=1}^{N_t}\bigg{|}\int^{t_{i+1}}_{t_i}
\int^{t_{i+1}}_{t_i}(t-s)^{\alpha-1}e^{-\nu(t-s)}
E_{\alpha,\alpha}(-\lambda_{k}^{\bar{h},\beta}(t-s)^\alpha)\nonumber\\
&\relphantom{=}{}\times
(P_{\bar{h}}(g(s,u_n^{\bar{h}}(s))-g(s,u_n(s)))\cdot e_l,\varphi_k^{\bar{h}})\varsigma_l^n(s)dsd\xi_l(r)\bigg{|}^2\nonumber\\
&\leq\frac{C}{\tau}\mathbb{E}\sum_{k=1}^{M}\sum_{l=1}^{\infty}(\lambda_{k}^{\bar{h},\beta})^{\frac{p}{\beta}}
\sum_{i=1}^{N_t}\bigg{|}\int^{t_{i+1}}_{t_i}
(t-s)^{\alpha-1}e^{-\nu(t-s)}
E_{\alpha,\alpha}(-\lambda_{k}^{\bar{h},\beta}(t-s)^\alpha)\nonumber\\
&\relphantom{=}{}\times
(P_{\bar{h}}(g(s,u_n^{\bar{h}}(s))-g(s,u_n(s)))\cdot e_l,\varphi_k^{\bar{h}})\varsigma_l^n(s)ds\bigg{|}^2\nonumber&
\end{flalign}
\begin{flalign}
& \leq C\mathbb{E}\sum_{k=1}^{M}\sum_{l=1}^{\infty}(\lambda_{k}^{\bar{h},\beta})^{\frac{p}{\beta}}
\sum_{i=1}^{N_t}\int^{t_{i+1}}_{t_i}(t-s)^{2\alpha-2}e^{-2\nu(t-s)}
\left|E_{\alpha,\alpha}(-\lambda_{k}^{\bar{h},\beta}(t-s)^\alpha)\right|^2\nonumber\\
&\relphantom{=}{}\times
|(P_{\bar{h}}(g(s,u_n^{\bar{h}}(s))-g(s,u_n(s)))\cdot e_l,\varphi_k^{\bar{h}})|^2|\varsigma_l^n(s)|^2ds\nonumber\\
&\leq C(\mu_1^n)^2\mathbb{E}\sum_{k=1}^{M}\sum_{l=1}^{\infty}
\int^{t}_{0}(t-s)^{2\alpha-2-\frac{\alpha p}{\beta}}e^{-2\nu(t-s)}\frac{(\lambda_{k}^{\bar{h},\beta}(t-s)^\alpha)^{\frac{p}{\beta}}}{(1+\lambda_{k}^{\bar{h},\beta}(t-s)^\alpha)^2}\nonumber\\
&\relphantom{=}{}\times
(P_{\bar{h}}(g(s,u_n^{\bar{h}}(s))-g(s,u_n(s)))\cdot e_l,\varphi_k^{\bar{h}})^2ds\nonumber\\
&\leq C(\mu_1^n)^2\mathbb{E}
\int^{t}_{0}(t-s)^{2\alpha-2-\frac{\alpha p}{\beta}}e^{-2\nu(t-s)}
\|P_{\bar{h}}g(s,u_n^{\bar{h}}(s))-P_{\bar{h}}g(s,u_n(s))\|_{\mathcal{L}_2^0}^2ds\nonumber&
\end{flalign}
\begin{flalign}
&\leq C(\mu_1^n)^2\int^{t}_{0}(t-s)^{2\alpha-2-\frac{\alpha p}{\beta}}e^{-2\nu(t-s)}
\mathbb{E}\|u_n^{\bar{h}}(s)-u_n(s)\|^2ds\nonumber\\
&\leq C(\mu_1^n)^2\int^{t}_{0}(t-s)^{2\alpha-2-\frac{\alpha p}{\beta}}e^{-2\nu(t-s)}
\mathbb{E}\|u_n^{\bar{h}}(s)-u_n(s)\|^2_{\mathbb{H}^p}ds\nonumber\\
&\leq C(\mu_1^n)^2\int^{t}_{0}(t-s)^{2\alpha-2-\frac{\alpha p}{\beta}}e^{-2\nu(t-s)}
(\mathbb{E}\|\Pi_1^2(s)\|_{\mathbb{H}^p}^2+\mathbb{E}\|\Pi_1^1(s)\|_{\mathbb{H}^p}^2
\nonumber
\\
&\relphantom{=}{}
+\mathbb{E}\|\Pi_2(s)\|_{\mathbb{H}^p}^2)ds\nonumber\\
&\leq C(\mu_1^n)^2\int^t_0
(t-s)^{2\alpha-2-\frac{\alpha p}{\beta}}
e^{-2\nu(t-s)}\mathbb{E}\|\Pi_1^2(s)\|_{\mathbb{H}^p}^2ds +C(\mu_1^n)^2\ell_{\bar{h}}\bar{h}^{4\widetilde{\gamma}-2p} \nonumber
\\
&\relphantom{=}{}
+C(\mu_1^n)^2\bar{h}^{4\widetilde{\gamma}-2p}. &
\end{flalign}
Note that $\{\xi_l^H\}_{l=1}^\infty$ is a family of mutually independent one-dimensional fractional Brownian motions. Similar to the arguments in  \eqref{4.8}, in view of \eqref{eq.2.9}, $(\bf{A}_4)$, Lemmas \ref{le2.1} and \ref{le2.10}, H\"{o}lder's inequality, the boundness assumption on $\varrho_k^n(t)$, \eqref{eq3.3}, \eqref{3.9} and \eqref{eq.5.13}, we deduce that
\begin{flalign}\label{4.9}
&\Lambda_3 \nonumber\\
&=C\mathbb{E}\sum_{k=1}^{M}(\lambda_{k}^{\bar{h},\beta})^{\frac{p}{\beta}} \nonumber\\
&\relphantom{=}{}\times
\bigg(\int^t_0\int_{\mathcal{D}}(t-s)^{\alpha-1}e^{-\nu(t-s)}
\sum_{l=1}^{M}E_{\alpha,\alpha}(-\lambda_{l}^{\bar{h},\beta}(t-s)^\alpha)\varphi_l^{\bar{h}}(\cdot)
\varphi_l^{\bar{h}}(y)\nonumber\\
&\relphantom{=}{}\times(P_{\bar{h}}h(s,u_n^{\bar{h}}(s,y))d\mathbb{W}_n^H(s,y)-P_{\bar{h}}h(s,u_n(s,y))d\mathbb{W}_n^H(s,y)),\varphi_k^{\bar{h}}\bigg)^2\nonumber\\
&=C\mathbb{E}\sum_{k=1}^{M}(\lambda_{k}^{\bar{h},\beta})^{\frac{p}{\beta}}
\bigg|\int^t_0\int_{\mathcal{D}}(t-s)^{\alpha-1}e^{-\nu(t-s)}
E_{\alpha,\alpha}(-\lambda_{k}^{\bar{h},\beta}(t-s)^\alpha)
\varphi_k^{\bar{h}}(y)\nonumber\\
&\relphantom{=}{}\times(P_{\bar{h}}h(s,u_n^{\bar{h}}(s,y))d\mathbb{W}_n^H(s,y)-P_{\bar{h}}h(s,u_n(s,y))d\mathbb{W}_n^H(s,y))\bigg|^2\nonumber&
\end{flalign}
\begin{flalign}
&=C\mathbb{E}\sum_{k=1}^{M}(\lambda_{k}^{\bar{h},\beta})^{\frac{p}{\beta}}
\bigg|\int^t_0\int_{\mathcal{D}}(t-s)^{\alpha-1}e^{-\nu(t-s)}
E_{\alpha,\alpha}(-\lambda_{k}^{\bar{h},\beta}(t-s)^\alpha)
\varphi_k^{\bar{h}}(y)\nonumber\\
&\relphantom{=}{}\times\sum_{j=1}^{M}\sum_{l=1}^{\infty}(P_{\bar{h}}(h(s,u_n^{\bar{h}}(s))-h(s,u_n(s)))\cdot e_l,\varphi_j^{\bar{h}})\varphi_j^{\bar{h}}(y)\varrho_l^n(s)
\nonumber\\
&\relphantom{=}{}\times
\left(\sum_{i=1}^{N_t}\frac{1}{\tau^{1-H}}\xi_{li}^H\chi_i(s)\right)dyds\bigg|^2\nonumber\\
&=C\mathbb{E}\sum_{k=1}^{M}(\lambda_{k}^{\bar{h},\beta})^{\frac{p}{\beta}}
\bigg|\int^t_0(t-s)^{\alpha-1}e^{-\nu(t-s)}
E_{\alpha,\alpha}(-\lambda_{k}^{\bar{h},\beta}(t-s)^\alpha)\nonumber\\
&\relphantom{=}{}\times
\sum_{l=1}^{\infty}(P_{\bar{h}}(h(s,u_n^{\bar{h}}(s))-h(s,u_n(s)))\cdot e_l,\varphi_k^{\bar{h}})\varrho_l^n(s)
\nonumber\\
&\relphantom{=}{}\times
\frac{1}{\tau}\int_0^t\sum_{i=1}^{N_t}\chi_i(r)d\xi_{l}^H(r)\chi_i(s)ds\bigg|^2\nonumber&
\end{flalign}
\begin{flalign}
&=\frac{C}{\tau^2}\mathbb{E}\sum_{k=1}^{M}\sum_{l=1}^{\infty}(\lambda_{k}^{\bar{h},\beta})^{\frac{p}{\beta}}
\bigg|\int^t_0\int^t_0(t-s)^{\alpha-1}e^{-\nu(t-s)}
E_{\alpha,\alpha}(-\lambda_{k}^{\bar{h},\beta}(t-s)^\alpha)\nonumber\\
&\relphantom{=}{}\times
(P_{\bar{h}}(h(s,u_n^{\bar{h}}(s))-h(s,u_n(s)))\cdot e_l,\varphi_k^{\bar{h}})\varrho_l^n(s)\sum_{i=1}^{N_t}\chi_i(r)\chi_i(s)ds d\xi_{l}^H(r)\bigg|^2\nonumber\\
&\leq\frac{C}{\tau^2}t^{2H-1}\mathbb{E}\sum_{k=1}^{M}\sum_{l=1}^{\infty}(\lambda_{k}^{\bar{h},\beta})^{\frac{p}{\beta}}
\int^t_0\bigg|\int^t_0(t-s)^{\alpha-1}e^{-\nu(t-s)}
E_{\alpha,\alpha}(-\lambda_{k}^{\bar{h},\beta}(t-s)^\alpha)\nonumber\\
&\relphantom{=}{}\times
(P_{\bar{h}}(h(s,u_n^{\bar{h}}(s))-h(s,u_n(s)))\cdot e_l,\varphi_k^{\bar{h}})\varrho_l^n(s)\sum_{i=1}^{N_t}\chi_i(r)\chi_i(s)ds \bigg|^2 dr\nonumber&
\end{flalign}
\begin{flalign}
&\leq\frac{C}{\tau^2}\mathbb{E}\sum_{k=1}^{M}\sum_{l=1}^{\infty}(\lambda_{k}^{\bar{h},\beta})^{\frac{p}{\beta}}
\nonumber\\
&\relphantom{=}{}\times\int^t_0\sum_{i=1}^{N_t}\chi_i(r)^2\bigg|\int^t_0(t-s)^{\alpha-1}e^{-\nu(t-s)}
E_{\alpha,\alpha}(-\lambda_{k}^{\bar{h},\beta}(t-s)^\alpha)\nonumber\\
&\relphantom{=}{}\times
(P_{\bar{h}}(h(s,u_n^{\bar{h}}(s))-h(s,u_n(s)))\cdot e_l,\varphi_k^{\bar{h}})\varrho_l^n(s)\chi_i(s)ds \bigg|^2 dr\nonumber\\
&\leq \frac{C}{\tau}\mathbb{E}\sum_{k=1}^{M}\sum_{l=1}^{\infty}(\lambda_{k}^{\bar{h},\beta})^{\frac{p}{\beta}}
\nonumber\\
&\relphantom{=}{}\times\sum_{i=1}^{N_t}\int^{t_{i+1}}_{t_i}\int^{t_{i+1}}_{t_i}
(t-s)^{2\alpha-2}e^{-2\nu(t-s)}
\left|E_{\alpha,\alpha}(-\lambda_{k}^{\bar{h},\beta}(t-s)^\alpha)\right|^2\nonumber\\
&\relphantom{=}{}\times
|(P_{\bar{h}}(h(s,u_n^{\bar{h}}(s))-h(s,u_n(s)))\cdot e_l,\varphi_k^{\bar{h}})|^2|\varrho_l^n(s)|^2dsdr\nonumber&
\end{flalign}
\begin{flalign}
&\leq C(\widetilde{\mu}_1^n)^2\mathbb{E}\sum_{k=1}^{M}\sum_{l=1}^{\infty}
\int^{t}_{0}(t-s)^{2\alpha-2-\frac{\alpha p}{\beta}}e^{-2\nu(t-s)}\frac{(\lambda_{k}^{\bar{h},\beta}(t-s)^\alpha)^{\frac{p}{\beta}}}{(1+\lambda_{k}^{\bar{h},\beta}(t-s)^\alpha)^2}\nonumber\\
&\relphantom{=}{}\times
(P_{\bar{h}}(h(s,u_n^{\bar{h}}(s))-h(s,u_n(s)))\cdot e_l,\varphi_k^{\bar{h}})^2ds\nonumber\\
&\leq C(\widetilde{\mu}_1^n)^2\mathbb{E}\int^{t}_{0}(t-s)^{2\alpha-2-\frac{\alpha p}{\beta}}e^{-2\nu(t-s)}\mathbb{E}\|P_{\bar{h}}h(s,u_n^{\bar{h}}(s))-P_{\bar{h}}h(s,u_n(s))
\|_{\mathcal{L}_2^0}^2ds\nonumber\\
&\leq C(\widetilde{\mu}_1^n)^2\int^{t}_{0}(t-s)^{2\alpha-2-\frac{\alpha p}{\beta}}e^{-2\nu(t-s)}
\mathbb{E}\|u_n^{\bar{h}}(s)-u_n(s)\|^2ds\nonumber&
\end{flalign}
\begin{flalign}
&\leq C(\widetilde{\mu}_1^n)^2\int^{t}_{0}(t-s)^{2\alpha-2-\frac{\alpha p}{\beta}}e^{-2\nu(t-s)}
\mathbb{E}\|u_n^{\bar{h}}(s)-u_n(s)\|^2_{\mathbb{H}^p}ds\nonumber\\
&\leq C(\widetilde{\mu}_1^n)^2\int^{t}_{0}(t-s)^{2\alpha-2-\frac{\alpha p}{\beta}}e^{-2\nu(t-s)}
(\mathbb{E}\|\Pi_1^2(s)\|_{\mathbb{H}^p}^2+\mathbb{E}\|\Pi_1^1(s)\|_{\mathbb{H}^p}^2
\nonumber
\\
&\relphantom{=}{}
+\mathbb{E}\|\Pi_2(s)\|_{\mathbb{H}^p}^2)ds\nonumber\\
&\leq C(\widetilde{\mu}_1^n)^2\int^t_0(t-s)^{2\alpha-2-\frac{\alpha p}{\beta}}e^{-2\nu(t-s)}\mathbb{E}\|\Pi_1^2(s)\|_{\mathbb{H}^p}^2ds +C(\widetilde{\mu}_1^n)^2\ell_{\bar{h}}\bar{h}^{4\widetilde{\gamma}-2p} \nonumber
\\
&\relphantom{=}{}
+C(\widetilde{\mu}_1^n)^2\bar{h}^{4\widetilde{\gamma}-2p}. &
\end{flalign}
Combining \eqref{4.10}-\eqref{4.9} together, we obtain for $p=0,\beta$ that
\begin{align*}
&\mathbb{E}|\!|\Pi_1^2(t)|\!|_{\mathbb{H}^p}^2
\\
&\leq C
\int^t_0(t-s)^{2\alpha-2-\frac{\alpha p}{\beta}}e^{-2\nu(t-s)}
\mathbb{E}\|\Pi_1^2(s)\|_{\mathbb{H}^p}^2 ds
+C\ell_{\bar{h}}\bar{h}^{4\widetilde{\gamma}-2p}
+C\bar{h}^{4\widetilde{\gamma}-2p}.
\end{align*}
Lemma \ref{lem4} conduces us to
\begin{align*}
\mathbb{E}|\!|\Pi_1^2(t)|\!|^2&\leq C\ell_{\bar{h}}\bar{h}^{4\widetilde{\gamma}}
+C\bar{h}^{4\widetilde{\gamma}}
\quad \forall t\in[0,T],
\end{align*}
and by Lemma \ref{lem2.11}, we have
\begin{align*}
\mathbb{E}|\!|\Pi_1^2(t)|\!|_{\mathbb{H}^\beta}^2&\leq
C\ell_{\bar{h}}\bar{h}^{4\widetilde{\gamma}-2\beta}
+C\bar{h}^{4\widetilde{\gamma}-2\beta} \quad t\in[0,T].
\end{align*}
Therefore,
\begin{align*}
\mathbb{E}|\!|\Pi_1^2(t)|\!|_{\mathbb{H}^p}^2&\leq
C\ell_{\bar{h}}\bar{h}^{4\widetilde{\gamma}-2p}
+C\bar{h}^{4\widetilde{\gamma}-2p} \quad t\in[0,T],
\end{align*}
and consequently, the desired assertion follows by the triangle inequality.
\end{proof}
Furthermore, thanks to  Theorems \ref{thm2.1} and  \ref{thm5.6}, a space-time mean-square error
 estimate for problem  \eqref{eq0.1}  follows from the triangle inequality.
\begin{theorem}\label{4.2}
Suppose that the assumptions of Theorem \ref{thm5.6} hold. Let $u$ be the
solution of \eqref{eq0.1} with $a\in L^2(\Omega;\mathbb{H}^{2\widetilde{\gamma}}(\mathcal{D}))$, $b\in L^2(\Omega;\mathbb{H}^{2\widetilde{\gamma}-\frac{2\beta}{\alpha}}(\mathcal{D}))$, and let $u_n^{\bar{h}}$ be the
 solution of \eqref{eq3.1} with $a_{\bar{h}}=P_{\bar{h}}a$, $b_h=P_{\bar{h}}b$. Then, with $\ell_{\bar{h}}=|\ln {\bar{h}}|$,
\begin{align*}
&\mathbb{E}\|u(t)-u_n^{\bar{h}}(t)\|^2
\\
&
\leq C\tau^2+C\sum_{l=1}^{\infty}(\eta_l^n)^2+C\sum_{l=1}^{\infty}(\widetilde{\eta}_l^n)^2
+C\ell_{\bar{h}}\bar{h}^{4\widetilde{\gamma}}
+C\bar{h}^{4\widetilde{\gamma}} \quad \forall t\in[0,T],
\end{align*}
provided that the infinite series are convergent, where $C$ is a positive constant independent of $\tau$ and ${\bar{h}}$.
\end{theorem}

\section{Conclusions}\label{sec:6}

This paper first introduces the semilinear stochastic space-time fractional wave equations with external infinite dimensional multiplicative white noise and fractional Gaussian noise. It is modeling the wave propagation with frequency-dependent power-law attenuation under the influences of the fluctuations of the external noises, the striking features of which are multiscale and usually happening in inhomogeneous media.
Then we establish the theory of the existence, uniqueness, and regularity of the model and its regularized version. Finally, the theory of finite element approximations is proposed.

\appendix
\renewcommand\thesection{\appendixname~\Alph{section}}
\renewcommand\thetheorem{\Alph{section}.\arabic{theorem}}
\renewcommand\theequation{\Alph{section}.\arabic{equation}}
\section{Proof of Lemma \ref{le1}.}\label{A}
\begin{proof}
Let
\begin{equation}\label{a1}
\displaystyle v(t,x)= \sum\limits_{k=1}^{\infty}(v(t),\varphi_k)\varphi_k(x)
=\sum\limits_{k=1}^{\infty}v_k(t)\varphi_k(x)
\end{equation}
be the solution of \eqref{eq1.4'}.
Substituting \eqref{eq.2.8}-\eqref{eq.2.9} and \eqref{a1} into \eqref{eq1.4'}, and taking the scalar product in $\mathbb{H}$ of \eqref{eq1.4'} with $\varphi_k$, we get that for $1<\alpha<2$,
\begin{equation}\label{a2}
\displaystyle ^c_0\partial^\alpha_t v_k(t)+ \lambda_k^\beta v_k(t)=e^{\nu t}f_k(t)+e^{\nu t}
\sum_{j=1}^{\infty}g^{k,j}(t)\varsigma_j(t)\dot{\xi}_j(t)+e^{\nu t}\sum_{j=1}^{\infty}h^{k,j}(t)
\varrho_j(t)\dot{\xi}^H_j(t)
\end{equation}
with $v_k(0)=a_{k}$,~$\partial_tv_k(0)=\nu a_k+b_{k}$, where $a_k=(a,\varphi_k)$, $b_k=(b,\varphi_k)$, $f_k(t)=(f(t,u),\varphi_k)$, and $g^{k,j}(t)$ and $h^{k,j}(t)$ are given in Remark \ref{rem2.2}.

By Theorem 5.15 in \cite{Kilbas}, we have
\begin{equation}\label{a13}\begin{split}
\displaystyle v_k(t)&= a_{k} E_{\alpha,1}(-\lambda_k^\beta t^\alpha)+(\nu a_{k}+b_k)tE_{\alpha,2}(-\lambda_k^\beta t^\alpha)\\
&\relphantom{=}{}+\int^t_0 (t-s)^{\alpha-1}E_{\alpha,\alpha}(-\lambda_k^\beta (t-s)^\alpha)
e^{\nu s}f_k(s)ds\\
&\relphantom{=}{}+\int^t_0 (t-s)^{\alpha-1}E_{\alpha,\alpha}(-\lambda_k^\beta (t-s)^\alpha)
e^{\nu s}\sum_{j=1}^{\infty}g^{k,j}(s)\varsigma_j(s)\dot{\xi}_j(s)ds\\
&\relphantom{=}{}+\int^t_0 (t-s)^{\alpha-1}E_{\alpha,\alpha}(-\lambda_k^\beta (t-s)^\alpha)
e^{\nu s}\sum_{j=1}^{\infty}h^{k,j}(s)\varrho_j(s)\dot{\xi}^H_j(s)ds.
\end{split}\end{equation}

%\begin{flushleft}
%$\dot{\xi}^H_j(s)ds\varphi_k(x)$
%\end{flushleft}

Then, it follows from \eqref{eq.2.8}-\eqref{eq.2.9}, \eqref{a1}, and \eqref{a13} that
%\begin{flushleft}
%abc
%\end{flushleft}
\begin{flalign*}\label{a4}
%\begin{split}
& \displaystyle v(t,x)
 \nonumber\\
& \displaystyle = \sum\limits_{k=1}^{\infty}a_{k} E_{\alpha,1}(-\lambda_k^\beta t^\alpha)\varphi_k(x)
+\sum\limits_{k=1}^{\infty}(\nu a_{k}+b_k)t E_{\alpha,2}(-\lambda_k^\beta t^\alpha)\varphi_k(x) \\
&\relphantom{=}{}+\sum\limits_{k=1}^{\infty}\int^t_0 (t-s)^{\alpha-1}E_{\alpha,\alpha}(-\lambda_k^\beta (t-s)^\alpha)e^{\nu s}f_k(s)ds\varphi_k(x) \\
&\relphantom{=}{}+\sum\limits_{k=1}^{\infty}\int^t_0 (t-s)^{\alpha-1}E_{\alpha,\alpha}(-\lambda_k^\beta (t-s)^\alpha)e^{\nu s}\sum_{j=1}^{\infty}g^{k,j}(s)\varsigma_j(s)\dot{\xi}_j(s)ds\varphi_k(x)  \\
&\relphantom{=}{}+\sum\limits_{k=1}^{\infty}\int^t_0 (t-s)^{\alpha-1}E_{\alpha,\alpha}(-\lambda_k^\beta (t-s)^\alpha)e^{\nu s}\sum_{j=1}^{\infty}h^{k,j}(s)\varrho_j(s)\dot{\xi}^H_j(s)ds\varphi_k(x)  & %\\  \nonumber
%\end{split} &
\end{flalign*}

\begin{flalign*}
&=\sum\limits_{k=1}^{\infty} E_{\alpha,1}(-\lambda_k^\beta t^\alpha)(a, \varphi_k)\varphi_k(x)
+\sum\limits_{k=1}^{\infty}tE_{\alpha,2}(-\lambda_k^\beta t^\alpha)(\nu a+b, \varphi_k)\varphi_k(x)\\
&\relphantom{=}{}+\sum\limits_{k=1}^{\infty}\int^t_0 (t-s)^{\alpha-1}E_{\alpha,\alpha}(-\lambda_k^\beta (t-s)^\alpha)e^{\nu s}(f(s,u),\varphi_k)ds\varphi_k(x)\\
&\relphantom{=}{}+\sum\limits_{k=1}^{\infty}\int^t_0 (t-s)^{\alpha-1}E_{\alpha,\alpha}(-\lambda_k^\beta (t-s)^\alpha)e^{\nu s}\left(\sum_{l,j=1}^{\infty}g^{l,j}(s)\varsigma_j(s)\dot{\xi}_j(s)\varphi_l,\varphi_k\right)ds\varphi_k(x)\\
&\relphantom{=}{}+\sum\limits_{k=1}^{\infty}\int^t_0 (t-s)^{\alpha-1}E_{\alpha,\alpha}(-\lambda_k^\beta (t-s)^\alpha)e^{\nu s}
\left(\sum_{l,j=1}^{\infty}h^{l,j}(s)\varrho_j(s)\dot{\xi}^H_j(s)\varphi_l,\varphi_k\right)ds\varphi_k(x) & %\\
\end{flalign*}
\begin{flalign*}
&=\int_{\mathcal{D}}\sum\limits_{k=1}^{\infty} \left(E_{\alpha,1}(-\lambda_k^\beta t^\alpha)+\nu tE_{\alpha,2}(-\lambda_k^\beta t^\alpha)\right)\varphi_k(x)\varphi_k(y)a(y)dy\\
&\relphantom{=}{}+\int_{\mathcal{D}}\sum\limits_{k=1}^{\infty}tE_{\alpha,2}(-\lambda_k^\beta t^\alpha)\varphi_k(x) \varphi_k(y)b(y)dy\\
&\relphantom{=}{}+\int^t_0\int_{\mathcal{D}} (t-s)^{\alpha-1}e^{\nu s}\sum\limits_{k=1}^{\infty} E_{\alpha,\alpha}(-\lambda_k^\beta (t-s)^\alpha)\varphi_k(x)\varphi_k(y)f(s,u(s,y))dyds\\
&\relphantom{=}{}+\int^t_0\int_{\mathcal{D}} (t-s)^{\alpha-1}e^{\nu s}
\sum\limits_{k=1}^{\infty}E_{\alpha,\alpha}(-\lambda_k^\beta (t-s)^\alpha)\varphi_k(x)\varphi_k(y)g(s,u(s,y))d\mathbb{W}(s,y)\\
&\relphantom{=}{}+\int^t_0\int_{\mathcal{D}}(t-s)^{\alpha-1}e^{\nu s}\sum\limits_{k=1}^{\infty}
E_{\alpha,\alpha}(-\lambda_k^\beta (t-s)^\alpha)\varphi_k(x)\varphi_k(y)h(s,u(s,y))d\mathbb{W}^H(s,y). &
\end{flalign*}

The convergence of the above series can be ascertained by the fact that for any
$\gamma\in\{1,2,\alpha\}$, $E_{\alpha,\gamma}\left(-\lambda_k^\beta t^\alpha\right)\sim
\left[\lambda_k^\beta t^\alpha\Gamma(\gamma-\alpha)\right]^{-1}$ when $k\rightarrow +\infty$.\\
Noticing that $v(t,x)=e^{\nu t}u(t,x)$, hence we obtain
\begin{flalign*}
& u(t,x)
\\
&=\int_{\mathcal{D}}e^{-\nu t}\sum\limits_{k=1}^{\infty} \left(E_{\alpha,1}(-\lambda_k^\beta t^\alpha)+\nu tE_{\alpha,2}(-\lambda_k^\beta t^\alpha)\right)\varphi_k(x)\varphi_k(y)a(y)dy\\
&\relphantom{=}{}+\int_{\mathcal{D}}e^{-\nu t}\sum\limits_{k=1}^{\infty}tE_{\alpha,2}(-\lambda_k^\beta t^\alpha)\varphi_k(x) \varphi_k(y)b(y)dy\\
&\relphantom{=}{}+\int^t_0\int_{\mathcal{D}} (t-s)^{\alpha-1}e^{-\nu (t-s)}\sum\limits_{k=1}^{\infty} E_{\alpha,\alpha}(-\lambda_k^\beta (t-s)^\alpha)\varphi_k(x)\varphi_k(y)f(s,u(s,y))dyds\\
&\relphantom{=}{}+\int^t_0\int_{\mathcal{D}}(t-s)^{\alpha-1}e^{-\nu (t-s)}
\sum\limits_{k=1}^{\infty}E_{\alpha,\alpha}(-\lambda_k^\beta (t-s)^\alpha)\varphi_k(x)\varphi_k(y)g(s,u(s,y))d\mathbb{W}(s,y)\\
&\relphantom{=}{}+\int^t_0\int_{\mathcal{D}}(t-s)^{\alpha-1}e^{-\nu (t-s)}\sum\limits_{k=1}^{\infty}
E_{\alpha,\alpha}(-\lambda_k^\beta (t-s)^\alpha)\varphi_k(x)\varphi_k(y)h(s,u(s,y))d\mathbb{W}^H(s,y) &
\end{flalign*}
\begin{flalign*}
&=\int_{\mathcal{D}}\mathcal{T}^\nu_{\alpha,\beta}(t,x,y)a(y)dy
+\int_{\mathcal{D}}\mathcal{R}^\nu_{\alpha,\beta}(t,x,y)b(y)dy\\
&\relphantom{=}{}+\int_{0}^t\int_{\mathcal{D}}\mathcal{S}^\nu_{\alpha,\beta}(t-s,x,y)f(s,u(s,y))dyds\\
&\relphantom{=}{}+\int_{0}^t\int_{\mathcal{D}}\mathcal{S}^\nu_{\alpha,\beta}(t-s,x,y)g(s,u(s,y))d{\mathbb{W}}(s,y)\\
&\relphantom{=}{}+\int_{0}^t\int_{\mathcal{D}}\mathcal{S}^\nu_{\alpha,\beta}(t-s,x,y)h(s,u(s,y))d{\mathbb{W}}^H(s,y). &
\end{flalign*}

Conversely, assuming that $u$ satisfies \eqref{eq2.1}, then for $t=0$,
\begin{equation*}
\begin{split}
u(0,x)&=\int_{\mathcal{D}}\mathcal{T}^\nu_{\alpha,\beta}(0,x,y)a(y)dy
=\int_{\mathcal{D}}\sum\limits_{k=1}^{\infty} E_{\alpha,1}(0)\varphi_k(x)\varphi_k(y)a(y)dy\\
&=\int_{\mathcal{D}}\sum\limits_{k=1}^{\infty}\varphi_k(x)\varphi_k(y)a(y)dy=a(x);
\end{split}
\end{equation*}
by Lemma \ref{le2.1} and Eq. (1.10.7) in \cite{Kilbas}, we have
\begin{flalign*}
& \partial_tu(t,x)|_{t=0}
\\
&=\partial_t\bigg[\int_{\mathcal{D}}\mathcal{T}^\nu_{\alpha,\beta}(t,x,y)a(y)dy
+\int_{\mathcal{D}}\mathcal{R}^\nu_{\alpha,\beta}(t,x,y)b(y)dy\\
&\relphantom{=}{}+\int_{0}^t\int_{\mathcal{D}}\mathcal{S}^\nu_{\alpha,\beta}(t-s,x,y)f(s,u(s,y))dyds\\
&\relphantom{=}{}+\int_{0}^t\int_{\mathcal{D}}\mathcal{S}^\nu_{\alpha,\beta}(t-s,x,y)g(s,u(s,y))d{\mathbb{W}}(s,y)\\
&\relphantom{=}{}+\int_{0}^t\int_{\mathcal{D}}\mathcal{S}^\nu_{\alpha,\beta}(t-s,x,y)h(s,u(s,y))d{\mathbb{W}}^H(s,y)\bigg]_{t=0} &
\end{flalign*}
\begin{flalign*}
&=\bigg[\int_{\mathcal{D}}(-\nu e^{-\nu t})\sum\limits_{k=1}^{\infty}
\left(E_{\alpha,1}(-\lambda_k^\beta t^\alpha)+\nu tE_{\alpha,2}(-\lambda_k^\beta t^\alpha)\right)
\varphi_k(x)\varphi_k(y)a(y)dy\\
&\relphantom{=}{}+\int_{\mathcal{D}}e^{-\nu t}
\sum\limits_{k=1}^{\infty}\left(-\lambda_k^\beta t^{\alpha-1}E_{\alpha,\alpha}(-\lambda_k^\beta t^\alpha)+\nu E_{\alpha,1}(-\lambda_k^\beta t^\alpha)\right)\varphi_k(x)\varphi_k(y)a(y)dy\\
&\relphantom{=}{}+\int_{\mathcal{D}}e^{-\nu t}\sum\limits_{k=1}^{\infty}\left(-\nu tE_{\alpha,2}(-\lambda_k^\beta t^\alpha)+E_{\alpha,1}(-\lambda_k^\beta t^\alpha)\right)\varphi_k(x) \varphi_k(y)b(y)dy\\
&\relphantom{=}{}+\int^t_0\int_{\mathcal{D}}e^{-\nu(t-s)}\sum\limits_{k=1}^{\infty}
\Big(-\nu (t-s)^{\alpha-1}E_{\alpha,\alpha}(-\lambda_k^\beta (t-s)^\alpha)\\
&\relphantom{==}{}+(t-s)^{\alpha-2} E_{\alpha,\alpha-1}(-\lambda_k^\beta(t-s)^\alpha)\Big)
\varphi_k(x)\varphi_k(y)f(s,u(s,y))dyds\\
&\relphantom{=}{}+\int^t_0\int_{\mathcal{D}}e^{-\nu(t-s)}\sum\limits_{k=1}^{\infty}
\Big(-\nu (t-s)^{\alpha-1}E_{\alpha,\alpha}(-\lambda_k^\beta (t-s)^\alpha)\\
&\relphantom{==}{}
+(t-s)^{\alpha-2} E_{\alpha,\alpha-1}(-\lambda_k^\beta(t-s)^\alpha)\Big)
\varphi_k(x)\varphi_k(y)g(s,u(s,y))d\mathbb{W}(s,y)\\
&\relphantom{=}{}+\int^t_0\int_{\mathcal{D}}e^{-\nu(t-s)}\sum\limits_{k=1}^{\infty}
\Big(-\nu (t-s)^{\alpha-1}E_{\alpha,\alpha}(-\lambda_k^\beta (t-s)^\alpha)\\
&\relphantom{==}{}
+(t-s)^{\alpha-2} E_{\alpha,\alpha-1}(-\lambda_k^\beta(t-s)^\alpha)\Big)
\varphi_k(x)\varphi_k(y)h(s,u(s,y))d\mathbb{W}^H(s,y)\bigg]_{t=0}\\
&=b(x).&
\end{flalign*}
We further show that \eqref{eq2.1} is the solution of \eqref{eq0.1}, i.e., to prove that
 \eqref{eq2.1} satisfies
\begin{equation}\label{A.3}\begin{split}
&{_0I}_t^{1,\nu}[ {^c_0\partial}^{\alpha,\nu}_t u(t,x)]+{_0I}_t^{1,\nu}[(-\Delta)^{\beta}u(t,x)]\\
&\relphantom{=}{}={_0I}_t^{1,\nu}\left[f(t,u)+g(t,u)\frac{\partial^2 \mathbb{W}(t,x)}{\partial t \partial x}+h(t,u)\frac{\partial^2 \mathbb{W}^H(t,x)}{\partial t \partial x}\right].
\end{split}\end{equation}
By Definitions \ref{def2.3}-\ref{def2.4}, it follows from \eqref{A.3} that
\begin{equation}\label{eq.2.18}\begin{split}
&{_0I}_t^{1}[ {^c_0\partial}^{\alpha}_t e^{\nu t}u(t,x)]+ {_0I}_t^{1}[(-\Delta)^{\beta}e^{\nu t}u(t,x)]\\
&\relphantom{=}{}={_0I}_t^{1}\left[e^{\nu t}\left(f(t,u)+g(t,u)\frac{\partial^2 \mathbb{W}(t,x)}{\partial t \partial x}+h(t,u)\frac{\partial^2\mathbb{W}^H(t,x)}{\partial t \partial x}\right)\right].
\end{split}\end{equation}
Let $v(t,x)=e^{\nu t}u(t,x)$. Then \eqref{eq.2.18} can be rewritten as
\begin{equation}\label{eq.2.19}\begin{split}
&{_0I}_t^{1}[ {^c_0\partial}^{\alpha}_t v(t,x)]+{_0I}_t^{1}[(-\Delta)^{\beta}v(t,x)]\\
&\relphantom{=}{}={_0I}_t^{1}\left[e^{\nu t}\left(f(t,u)+g(t,u)\frac{\partial^2 \mathbb{W}(t,x)}{\partial t \partial x}+h(t,u)\frac{\partial^2\mathbb{W}^H(t,x)}{\partial t \partial x}\right)\right].
\end{split}\end{equation}
Evidently, $v(0,x)=a(x)$ and $\partial_tv(t,x)|_{t=0}=\nu a(x)+b(x)$.
From \eqref{eq2.1} we have
\begin{equation}\label{eq2.17}
\begin{split}
v(t,x)&=\int_{\mathcal{D}}e^{\nu t}\mathcal{T}^\nu_{\alpha,\beta}(t,x,y)a(y)dy
+\int_{\mathcal{D}}e^{\nu t}\mathcal{R}^\nu_{\alpha,\beta}(t,x,y)b(y)dy\\
&\relphantom{=}{}+\int_{0}^t\int_{\mathcal{D}}e^{\nu t}\mathcal{S}^\nu_{\alpha,\beta}(t-s,x,y)f(s,u(s,y))dyds\\
&\relphantom{=}{}+\int_{0}^t\int_{\mathcal{D}}e^{\nu t}\mathcal{S}^\nu_{\alpha,\beta}(t-s,x,y)g(s,u(s,y))d{\mathbb{W}}(s,y)\\
&\relphantom{=}{}+\int_{0}^t\int_{\mathcal{D}}e^{\nu t}\mathcal{S}^\nu_{\alpha,\beta}(t-s,x,y)h(s,u(s,y))d{\mathbb{W}}^H(s,y).
\end{split}
\end{equation}
Thus, it only remains to prove that \eqref{eq2.17} satisfies \eqref{eq.2.19}.
Notice that
\begin{align}
& _0I_t^{1}[{^c_0\partial}^{\alpha}_t v(t,x)] \\
&
={^c_0\partial}^{\alpha-1}_t v(t,x)-\frac{t^{2-\alpha}}{\Gamma(3-\alpha)}\partial_tv(t,x)|_{t=0} \nonumber
\\
&
={^c_0\partial}^{\alpha-1}_t v(t,x)-\frac{t^{2-\alpha}}{\Gamma(3-\alpha)}(\nu a(x)+b(x)). \nonumber
\end{align}

Then, performing ${^c_0\partial}^{\alpha-1}_t$ on both sides of \eqref{eq2.17}, and using Fubini's theorem, Definition \ref{def2.2}, Lemma \ref{le2.1}, and Eqs. (1.83) and (1.100) in \cite{Podlubny}, we get

\begin{flalign*}
\end{flalign*}

\begin{flalign}\label{A.4}
& ^c_0\partial^{\alpha-1}_t v(t,x) \nonumber\\
&= {^c_0\partial}^{\alpha-1}_t\bigg[\int_{\mathcal{D}}e^{\nu t}\mathcal{T}^\nu_{\alpha,\beta}(t,x,y)a(y)dy
+\int_{\mathcal{D}}e^{\nu t}\mathcal{R}^\nu_{\alpha,\beta}(t,x,y)b(y)dy\nonumber\\
&\relphantom{=}{}+\int_{0}^t\int_{\mathcal{D}}e^{\nu t}\mathcal{S}^\nu_{\alpha,\beta}(t-s,x,y)f(s,u(s,y))dyds\\
&\relphantom{=}{}+\int_{0}^t\int_{\mathcal{D}}e^{\nu t}\mathcal{S}^\nu_{\alpha,\beta}(t-s,x,y)g(s,u(s,y))d{\mathbb{W}}(s,y)\nonumber\\
&\relphantom{=}{}+\int_{0}^t\int_{\mathcal{D}}e^{\nu t}\mathcal{S}^\nu_{\alpha,\beta}(t-s,x,y)h(s,u(s,y))d{\mathbb{W}}^H(s,y)\bigg] \nonumber  &
\end{flalign}
\begin{flalign*}
&=\int_{\mathcal{D}}\sum\limits_{k=1}^{\infty} \Big(-\lambda_k^\beta tE_{\alpha,2}(-\lambda_k^\beta t^\alpha)+\nu t^{2-\alpha}E_{\alpha,3-\alpha}(-\lambda_k^\beta t^\alpha)\Big)\varphi_k(x)\varphi_k(y)a(y)dy\nonumber\\
&\relphantom{=}{}+\int_{\mathcal{D}}\sum\limits_{k=1}^{\infty}t^{2-\alpha}
E_{\alpha,3-\alpha}(-\lambda_k^\beta t^\alpha)\varphi_k(x) \varphi_k(y)b(y)dy\nonumber\\
&\relphantom{=}{}+\int^t_0\int_{\mathcal{D}} \sum\limits_{k=1}^{\infty} E_{\alpha,1}(-\lambda_k^\beta (t-s)^\alpha)e^{\nu s}\varphi_k(x)\varphi_k(y)f(s,u(s,y))dyds\nonumber\\
&\relphantom{=}{}+\int^t_0\int_{\mathcal{D}} \sum\limits_{k=1}^{\infty} E_{\alpha,1}(-\lambda_k^\beta (t-s)^\alpha)e^{\nu s}\varphi_k(x)\varphi_k(y)g(s,u(s,y))d{\mathbb{W}}(s,y)\nonumber\\
&\relphantom{=}{}+\int^t_0\int_{\mathcal{D}} \sum\limits_{k=1}^{\infty} E_{\alpha,1}(-\lambda_k^\beta (t-s)^\alpha)e^{\nu s}\varphi_k(x)\varphi_k(y)h(s,u(s,y))d{\mathbb{W}^H}(s,y). &
\end{flalign*}
Acting $_0I_t^1(-\Delta)^{\beta}$ on both sides of \eqref{eq2.17}, in view of Fubini's theorem and
 Eq. (1.100) in \cite{Podlubny}, we obtain

\begin{flalign}\label{A.5}
&_0I_t^1[(-\Delta)^{\beta}v(t,x)] \\
&=\int_0^t\int_{\mathcal{D}}\sum\limits_{k=1}^{\infty} \left(E_{\alpha,1}(-\lambda_k^\beta s^\alpha)+\nu sE_{\alpha,2}(-\lambda_k^\beta s^\alpha)\right)\lambda_k^\beta\varphi_k(x)\varphi_k(y)a(y)dyds\nonumber\\
&\relphantom{}{}+\int_0^t\int_{\mathcal{D}}\sum\limits_{k=1}^{\infty}sE_{\alpha,2}(-\lambda_k^\beta s^\alpha)\lambda_k^\beta\varphi_k(x) \varphi_k(y)b(y)dyds\nonumber\\
&\relphantom{}{}+\int_0^t\int^s_0\int_{\mathcal{D}} (s-r)^{\alpha-1}e^{\nu r}\sum\limits_{k=1}^{\infty} E_{\alpha,\alpha}(-\lambda_k^\beta (s-r)^\alpha)
\lambda_k^\beta\varphi_k(x)\varphi_k(y)f(r,u(r,y))dydrds\nonumber\\
&\relphantom{}{}+\int_0^t\int^s_0\int_{\mathcal{D}} (s-r)^{\alpha-1}e^{\nu r}
\sum\limits_{k=1}^{\infty}E_{\alpha,\alpha}(-\lambda_k^\beta (s-r)^\alpha)
\lambda_k^\beta\varphi_k(x)\varphi_k(y)g(r,u(r,y))d\mathbb{W}(r,y)ds\nonumber\\
&\relphantom{}{}+\int_0^t\int^s_0\int_{\mathcal{D}}(s-r)^{\alpha-1}e^{\nu r}\sum\limits_{k=1}^{\infty}
E_{\alpha,\alpha}(-\lambda_k^\beta (s-r)^\alpha)
\lambda_k^\beta\varphi_k(x)\varphi_k(y)h(r,u(r,y))d\mathbb{W}^H(r,y)ds \nonumber &
\end{flalign}
\begin{flalign}
&=\int_{\mathcal{D}}\sum\limits_{k=1}^{\infty} \left(tE_{\alpha,2}(-\lambda_k^\beta t^\alpha)+\nu t^2E_{\alpha,3}(-\lambda_k^\beta t^\alpha)\right)\lambda_k^\beta \varphi_k(x)\varphi_k(y)a(y)dy\nonumber\\
&\relphantom{}{}+\int_{\mathcal{D}}\sum\limits_{k=1}^{\infty}t^2E_{\alpha,3}(-\lambda_k^\beta t^\alpha)\lambda_k^\beta \varphi_k(x) \varphi_k(y)b(y)dy\nonumber\\
&\relphantom{}{}+\int_0^t\int_{\mathcal{D}} (t-s)^{\alpha}e^{\nu s}\sum\limits_{k=1}^{\infty} E_{\alpha,\alpha+1}(-\lambda_k^\beta (t-s)^\alpha)\lambda_k^\beta \varphi_k(x)\varphi_k(y)f(s,u(s,y))dyds\nonumber\\
&\relphantom{}{}+\int_0^t\int_{\mathcal{D}} (t-s)^{\alpha}e^{\nu s}\sum\limits_{k=1}^{\infty} E_{\alpha,\alpha+1}(-\lambda_k^\beta (t-s)^\alpha)\lambda_k^\beta \varphi_k(x)\varphi_k(y)g(s,u(s,y))d\mathbb{W}(s,y)\nonumber\\
&\relphantom{}{}+\int_0^t\int_{\mathcal{D}} (t-s)^{\alpha}e^{\nu s}\sum\limits_{k=1}^{\infty} E_{\alpha,\alpha+1}(-\lambda_k^\beta (t-s)^\alpha)\lambda_k^\beta \varphi_k(x)\varphi_k(y)h(s,u(s,y))d\mathbb{W}^H(s,y). \nonumber &
\end{flalign}
Substituting \eqref{A.4} and \eqref{A.5} into the left side of \eqref{eq.2.19}, in view of \eqref{eq.2.1} and \eqref{eq.2.8}-\eqref{eq.2.9}, we deduce that
\begin{flalign}\label{A.6}
&_0I_t^1[^c_0\partial^{\alpha}_t v(t,x)]+{_0I}_t^1[(-\Delta)^{\beta}v(t,x)]
=-\frac{t^{2-\alpha}}{\Gamma(3-\alpha)}(\nu a(x)+b(x))\nonumber\\
&\relphantom{=}{}+\int_{\mathcal{D}}\sum\limits_{k=1}^{\infty} \Big(-\lambda_k^\beta tE_{\alpha,2}(-\lambda_k^\beta t^\alpha)+\nu t^{2-\alpha}
E_{\alpha,3-\alpha}(-\lambda_k^\beta t^\alpha)\Big)\varphi_k(x)\varphi_k(y)a(y)dy\nonumber\\
&\relphantom{=}{}+\int_{\mathcal{D}}\sum\limits_{k=1}^{\infty}t^{2-\alpha}
E_{\alpha,3-\alpha}(-\lambda_k^\beta t^\alpha)\varphi_k(x) \varphi_k(y)b(y)dy\nonumber\\
&\relphantom{=}{}+\int^t_0\int_{\mathcal{D}} \sum\limits_{k=1}^{\infty} E_{\alpha,1}(-\lambda_k^\beta (t-s)^\alpha)e^{\nu s}\varphi_k(x)\varphi_k(y)f(s,u(s,y))dyds\nonumber\\
&\relphantom{=}{}+\int^t_0\int_{\mathcal{D}} \sum\limits_{k=1}^{\infty} E_{\alpha,1}(-\lambda_k^\beta (t-s)^\alpha)e^{\nu s}\varphi_k(x)\varphi_k(y)g(s,u(s,y))d{\mathbb{W}}(s,y)\nonumber\\
&\relphantom{=}{}+\int^t_0\int_{\mathcal{D}} \sum\limits_{k=1}^{\infty} E_{\alpha,1}(-\lambda_k^\beta (t-s)^\alpha)e^{\nu s}\varphi_k(x)\varphi_k(y)h(s,u(s,y))d{\mathbb{W}^H}(s,y)\nonumber &
\end{flalign}
\begin{flalign}
&\relphantom{=}{}+\int_{\mathcal{D}}\sum\limits_{k=1}^{\infty} \left(tE_{\alpha,2}(-\lambda_k^\beta t^\alpha)+\nu t^2E_{\alpha,3}(-\lambda_k^\beta t^\alpha)\right)\lambda_k^\beta \varphi_k(x)\varphi_k(y)a(y)dy\nonumber\\
&\relphantom{=}{}+\int_{\mathcal{D}}\sum\limits_{k=1}^{\infty}t^2E_{\alpha,3}(-\lambda_k^\beta t^\alpha)\lambda_k^\beta \varphi_k(x) \varphi_k(y)b(y)dy\nonumber\\
&\relphantom{=}{}+\int_0^t\int_{\mathcal{D}} (t-s)^{\alpha}e^{\nu s}\sum\limits_{k=1}^{\infty} E_{\alpha,\alpha+1}(-\lambda_k^\beta (t-s)^\alpha)\lambda_k^\beta \varphi_k(x)\varphi_k(y)f(s,u(s,y))dyds\nonumber\\
&\relphantom{=}{}+\int_0^t\int_{\mathcal{D}} (t-s)^{\alpha}e^{\nu s}\sum\limits_{k=1}^{\infty} E_{\alpha,\alpha+1}(-\lambda_k^\beta (t-s)^\alpha)\lambda_k^\beta \varphi_k(x)\varphi_k(y)g(s,u(s,y))d\mathbb{W}(s,y)\nonumber\\
&\relphantom{=}{}+\int_0^t\int_{\mathcal{D}} (t-s)^{\alpha}e^{\nu s}\sum\limits_{k=1}^{\infty} E_{\alpha,\alpha+1}(-\lambda_k^\beta (t-s)^\alpha)\lambda_k^\beta \varphi_k(x)\varphi_k(y)h(s,u(s,y))d\mathbb{W}^H(s,y)\nonumber &
\end{flalign}
\begin{flalign}
&%\relphantom{=}{}
=\int_0^t\int_{\mathcal{D}}e^{\nu s}\sum\limits_{k=1}^{\infty} \varphi_k(x)\varphi_k(y)f(s,u(s,y)) dyds\nonumber\\
&\relphantom{=}{}+\int_0^t\int_{\mathcal{D}}e^{\nu s}\sum\limits_{k=1}^{\infty} \varphi_k(x)\varphi_k(y)g(s,u(s,y))d\mathbb{W}(s,y)\nonumber\\
&\relphantom{=}{}+\int_0^t\int_{\mathcal{D}}e^{\nu s}\sum\limits_{k=1}^{\infty} \varphi_k(x)\varphi_k(y)h(s,u(s,y))d\mathbb{W}^H(s,y)\nonumber &
\end{flalign}
\begin{flalign}
&%\relphantom{=}{}
={_0I}_t^1[e^{\nu t}f(t,u)]+\int_0^te^{\nu s}\int_{\mathcal{D}}\sum\limits_{k=1}^{\infty} \varphi_k(x)\varphi_k(y)\sum_{j,l=1}^{\infty}g^{j,l}(t)\varphi_j(y)\varsigma_l(s)\dot{\xi}_l(s)dyds\nonumber\\
&\relphantom{=}{}+\int_0^te^{\nu s}\int_{\mathcal{D}}\sum\limits_{k=1}^{\infty} \varphi_k(x)\varphi_k(y)\sum_{j,l=1}^{\infty}h^{j,l}(t)\varphi_j(y)\varrho_l(s)\dot{\xi}^H_l(s)dyds\nonumber &
\end{flalign}
\begin{flalign}
&%\relphantom{=}{}
={_0I}_t^1[e^{\nu t}f(t,u)]+\int_0^te^{\nu s}\sum_{k,l=1}^{\infty}g^{k,l}(t)\varphi_k(x)\varsigma_l(s)\dot{\xi}_l(s)ds\nonumber\\
&\relphantom{=}{}
+\int_0^te^{\nu s}\sum_{k,l=1}^{\infty}h^{k,l}(t)\varphi_k(x)\varrho_l(s)\dot{\xi}^H_l(s)ds\nonumber\\
&%\relphantom{=}{}
={_0I}_t^1\left[e^{\nu t}\left(f(t,u)+g(t,u)\frac{\partial^2 \mathbb{W}(t,x)}{\partial t \partial x}+h(t,u)\frac{\partial^2 \mathbb{W}^H(t,x)}{\partial t \partial x}\right)
\right]. &
\end{flalign}
The lemma is now proved.
\end{proof}

\section{Proof of Theorem \ref{thm4.2}.}\label{A.2}
\begin{proof}
We divide the proof into four steps.

Step 1. Estimate of $\mathbb{E}\|u_n(t)\|_{\mathbb{H}^{2\widetilde{\gamma}}}^2$.

By Lemma \ref{le2.2} and Remark \ref{re3.4}, from \eqref{eq2.15} we obtain  that
\begin{align}\label{eq.4.4}
\mathbb{E}\|u_n(t)\|_{\mathbb{H}^{2\widetilde{\gamma}}}^2&\leq C\mathbb{E}\bigg{\|}\int_{\mathcal{D}}\mathcal{T}^\nu_{\alpha,\beta}(t,\cdot,y)a(y)dy\bigg{\|}_{\mathbb{H}^{2\widetilde{\gamma}}}^2
+C\mathbb{E}\bigg{\|}\int_{\mathcal{D}}\mathcal{R}^\nu_{\alpha,\beta}(t,\cdot,y)b(y)dy\bigg{\|}_{\mathbb{H}^{2\widetilde{\gamma}}}^2\nonumber\\
&\relphantom{=}{}+C\mathbb{E}\bigg{\|}\int_{0}^t\int_{\mathcal{D}}\mathcal{S}^\nu_{\alpha,\beta}(t-s,\cdot,y)
f(s,u_n(s,y))dyds\bigg{\|}_{\mathbb{H}^{2\widetilde{\gamma}}}^2\nonumber\\
&\relphantom{=}{}+C\mathbb{E}\bigg{\|}\int_{0}^t\int_{\mathcal{D}}\mathcal{S}^\nu_{\alpha,\beta}(t-s,\cdot,y)
g(s,u_n(s,y))d{\mathbb{W}_n}(s,y)\bigg{\|}_{\mathbb{H}^{2\widetilde{\gamma}}}^2\nonumber\\
&\relphantom{=}{}+C\mathbb{E}\bigg{\|}\int_{0}^t\int_{\mathcal{D}}\mathcal{S}^\nu_{\alpha,\beta}(t-s,\cdot,y)
h(s,u_n(s,y))d{\mathbb{W}^H_n}(s,y)\bigg{\|}_{\mathbb{H}^{2\widetilde{\gamma}}}^2\nonumber\\
&\leq C(1+\nu^2t^2)e^{-2\nu t}\mathbb{E}\|a\|_{\mathbb{H}^{2\widetilde{\gamma}}}^2+Ce^{-2\nu t}\mathbb{E}\|b\|_{\mathbb{H}^{2\widetilde{\gamma}-\frac{2\beta}{\alpha}}}^2
+\mathcal{J}_1+\mathcal{J}_2+\mathcal{J}_3.
\end{align}
To estimate $\mathcal{J}_1$, we use the fact that $\{\varphi_k\}_{k=1}^{\infty}$ is an orthonormal basis in $L^2(\mathcal{D})$, the assumption on $f$ given in $(\bf{A}_1)$, Lemma \ref{le2.1}, and  H\"{o}lder's inequality, which leads to
\begin{align}\label{eq.4.5}
\mathcal{J}_1 & = C\mathbb{E}\sum_{k=1}^{\infty}\lambda_k^{2\widetilde{\gamma}}\bigg{(}
\int_{0}^t\int_{\mathcal{D}}(t-s)^{\alpha-1}e^{-\nu(t-s)}
\sum_{l=1}^{\infty}E_{\alpha,\alpha}\left(-\lambda_l^\beta(t-s)^\alpha\right)
\varphi_l(\cdot)\varphi_l(y)\nonumber\\
&\relphantom{=}{}\cdot f(s,u_n(s,y))dyds,\varphi_k\bigg{)}^2\nonumber\\
&=C\mathbb{E}\sum_{k=1}^{\infty}\lambda_k^{2\widetilde{\gamma}}\bigg{|}\int_{0}^t
\int_{\mathcal{D}}(t-s)^{\alpha-1}e^{-\nu(t-s)}
E_{\alpha,\alpha}\left(-\lambda_k^\beta(t-s)^\alpha\right)\varphi_k(y)
f(s,u_n(s,y))dyds\bigg{|}^2\nonumber\\
&\leq CT\mathbb{E}\sum_{k=1}^{\infty}\int_{0}^t(t-s)^{2\alpha-4}e^{-2\nu(t-s)}
\frac{(\lambda_k^\beta(t-s)^\alpha)^{\frac{2}{\alpha}}}{(1+\lambda_k^\beta(t-s)^\alpha)^2}
\lambda_k^{2\widetilde{\gamma}-\frac{2\beta}{\alpha}}(f(s,u_n(s)),\varphi_k)^2ds\nonumber\\
&\leq C\int_{0}^t (t-s)^{2\alpha-4} e^{-2\nu(t-s)}\left(1+\mathbb{E}\|u_n(s)\|_{\mathbb{H}^{2\widetilde{\gamma}-
\frac{2\beta}{\alpha}}}^2\right)ds\nonumber\\
&\leq C+C\int_{0}^t (t-s)^{2\alpha-4}e^{-2\nu(t-s)} \mathbb{E}\|u_n(s)\|_{\mathbb{H}^{2\widetilde{\gamma}}}^2ds,
\end{align}
where we also use $\frac{(\lambda_k^\beta(t-s)^\alpha)^{\frac{2}{\alpha}}}{(1+\lambda_k^\beta(t-s)^\alpha)^2}\leq C$ and $\mathbb{E}\|u_n(s)\|_{\mathbb{H}^{2\widetilde{\gamma}-\frac{2\beta}{\alpha}}}^2\leq C\mathbb{E}\|u_n(s)\|_{\mathbb{H}^{2\widetilde{\gamma}}}^2$. %\\

Without loss of generality, we assume that there exists a positive integer $N_t$ such that $t=t_{N_t+1}$. Since $\{\varphi_k\}_{k=1}^{\infty}$ is an orthonormal basis in $L^2(\mathcal{D})$, the Brownian motion has independent increments and
$\{\xi_l\}_{l=1}^{\infty}$ is a family of mutally independent one-dimensional standard Brownain motions, then by \eqref{eq.2.8}, $(\bf{A}_1)$, H\"{o}lder's inequality, the It\^{o} isometry, and the boundedness assumption of $\varsigma_k^n(t)$, we obtain
\begin{flalign}\label{eq.4.6}
& \mathcal{J}_2 \nonumber \\
& =C\mathbb{E}\sum_{k=1}^{\infty}\lambda_k^{2\widetilde{\gamma}}\bigg{(}
\int_{0}^t\int_{\mathcal{D}}(t-s)^{\alpha-1}e^{-\nu(t-s)}
\sum_{l=1}^{\infty}E_{\alpha,\alpha}\left(-\lambda_l^\beta(t-s)^\alpha\right)
\varphi_l(\cdot)\varphi_l(y)\nonumber\\
&\relphantom{=}{}\cdot g(s,u_n(s,y))d\mathbb{W}_n(s,y),\varphi_k\bigg{)}^2\nonumber\\
& =C\mathbb{E}\sum_{k=1}^{\infty}\lambda_k^{2\widetilde{\gamma}}
\bigg{|}\int_{0}^t\int_{\mathcal{D}}(t-s)^{\alpha-1}e^{-\nu(t-s)}
E_{\alpha,\alpha}\left(-\lambda_k^\beta(t-s)^\alpha\right)\varphi_k(y)
g(s,u_n(s,y))
\nonumber\\
& \relphantom{=}{} \cdot d\mathbb{W}_n(s,y)\bigg{|}^2&
\end{flalign}

\begin{flalign*}
&=C\mathbb{E}\sum_{k=1}^{\infty}\lambda_k^{2\widetilde{\gamma}}
\bigg(\int_{0}^{t}\int_{\mathcal{D}}(t-s)^{\alpha-1}e^{-\nu(t-s)}E_{\alpha,\alpha}
\left(-\lambda_k^\beta(t-s)^\alpha\right)
\varphi_k(y)\\
&\relphantom{=}{}\cdot \sum_{j,l=1}^{\infty}(g(s,u_n(s))\cdot e_l,\varphi_j) \varphi_j(y)\varsigma_l^n(s)\left(\sum_{i=1}^{N_t}\frac{1}{\sqrt{\tau}}
\xi_{li}\chi_i(s)\right)dyds\bigg)^2\nonumber\\
&=C\mathbb{E}\sum_{k=1}^{\infty}\lambda_k^{2\widetilde{\gamma}}
\bigg(\sum_{i=1}^{N_t}\int_{t_i}^{t_{i+1}}
(t-s)^{\alpha-1}e^{-\nu(t-s)}E_{\alpha,\alpha}\left(-\lambda_k^\beta(t-s)^\alpha\right)
\\
&\relphantom{=}{} \cdot \sum_{l=1}^{\infty}(g(s,u_n(s))\cdot e_l,\varphi_k) \varsigma_l^n(s)\frac{\xi_l(t_{i+1})-\xi_l(t_i)}{\tau}ds\bigg)^2&
\end{flalign*}

\begin{flalign*}
& = \frac{C}{\tau^2}\mathbb{E}\sum_{k=1}^{\infty}\lambda_k^{2\widetilde{\gamma}}
\bigg(\sum_{i=1}^{N_t}
\int_{t_i}^{t_{i+1}}\int_{t_i}^{t_{i+1}}(t-s)^{\alpha-1}
e^{-\nu(t-s)}E_{\alpha,\alpha}\left(-\lambda_k^\beta(t-s)^\alpha\right)\nonumber\\
&\relphantom{=}{}\cdot\sum_{l=1}^{\infty}(g(s,u_n(s))\cdot e_l,\varphi_k)\varsigma_l^n(s)dsd\xi_l(r)\bigg)^2\nonumber\\
&=\frac{C}{\tau^2}\mathbb{E}\sum_{k,l=1}^{\infty}\lambda_k^{2\widetilde{\gamma}} \sum_{i=1}^{N_t}\bigg(
\int_{t_i}^{t_{i+1}}\int_{t_i}^{t_{i+1}}(t-s)^{\alpha-1}
e^{-\nu(t-s)}E_{\alpha,\alpha}\left(-\lambda_k^\beta(t-s)^\alpha\right)\nonumber\\
&\relphantom{=}{}\cdot(g(s,u_n(s))\cdot e_l,\varphi_k)\varsigma_l^n(s)dsd\xi_l(r)\bigg)^2&
\end{flalign*}

\begin{flalign*}
&\leq \frac{C}{\tau}\mathbb{E}\sum_{k,l=1}^{\infty}\lambda_k^{2\widetilde{\gamma}} \sum_{i=1}^{N_t}\bigg|
\int_{t_i}^{t_{i+1}}(t-s)^{\alpha-1}
e^{-\nu(t-s)}E_{\alpha,\alpha}\left(-\lambda_k^\beta(t-s)^\alpha\right)\nonumber\\
&\relphantom{=}{}\cdot(g(s,u_n(s))\cdot e_l,\varphi_k)\varsigma_l^n(s)ds\bigg|^2\nonumber\\
&\leq C\mathbb{E}\sum_{k,l=1}^{\infty}\lambda_k^{2\widetilde{\gamma}} \sum_{i=1}^{N_t}
\int_{t_i}^{t_{i+1}}(t-s)^{2\alpha-2}
e^{-2\nu(t-s)}\left|E_{\alpha,\alpha}\left(-\lambda_k^\beta(t-s)^\alpha\right)\right|^2\nonumber\\
&\relphantom{=}{}\cdot|(g(s,u_n(s))\cdot e_l,\varphi_k)|^2|\varsigma_l^n(s)|^2ds&
\end{flalign*}

\begin{flalign*}
&\leq  C(\mu_1^n)^2\mathbb{E}\sum_{k,l=1}^{\infty}
\int_{0}^t (t-s)^{2\alpha-4} e^{-2\nu(t-s)}\frac{(\lambda_k^\beta(t-s)^\alpha)^{\frac{2}{\alpha}}}{(1+\lambda_k^\beta(t-s)^\alpha)^2}
\lambda_k^{2\widetilde{\gamma}-\frac{2\beta}{\alpha}}(g(s,u_n(s))\cdot e_l,\varphi_k)^2ds \nonumber\\
&\leq  C(\mu_1^n)^2\mathbb{E}\int_{0}^t (t-s)^{2\alpha-4} e^{-2\nu(t-s)}
\|(-\Delta)^{\widetilde{\gamma}-\frac{\beta}{\alpha}}g(s,u_n(s))\|_{\mathcal{L}_2^0}^2ds \nonumber\\
&\leq  C(\mu_1^n)^2+C(\mu_1^n)^2\int_{0}^t (t-s)^{2\alpha-4} e^{-2\nu(t-s)}
\mathbb{E}\|u_n(s)\|_{\mathbb{H}^{2\widetilde{\gamma}}}^2ds.
\end{flalign*}
Note that $\{\xi_l^H\}_{l=1}^{\infty}$ is a family of mutally independent one-dimensional fractional Brownain motions. By slightly modifying the arguments in \eqref{eq.4.6}, in view of \eqref{eq.2.9}, $(\bf{A}_1)$, Lemma \ref{le2.10}, H\"{o}lder's inequality, and the boundedness assumption of $\varrho_k^n(t)$, we deduce that
\begin{flalign}\label{eq.4.7}
&  \mathcal{J}_3 \nonumber\\
& = C\mathbb{E}\sum_{k=1}^{\infty}\lambda_k^{2\widetilde{\gamma}}
\bigg{(}\int_{0}^t\int_{\mathcal{D}}(t-s)^{\alpha-1}e^{-\nu(t-s)}
\sum_{l=1}^{\infty}E_{\alpha,\alpha}\left(-\lambda_l^\beta(t-s)^\alpha\right)\nonumber\\
&\relphantom{=}{}\cdot 
\varphi_l(\cdot)\varphi_l(y) h(s,u_n(s,y))d\mathbb{W}^H_n(s,y),\varphi_k\bigg{)}^2\nonumber\\
& =C\mathbb{E}\sum_{k=1}^{\infty}\lambda_k^{2\widetilde{\gamma}}
\bigg{|}\int_{0}^t\int_{\mathcal{D}}(t-s)^{\alpha-1}e^{-\nu(t-s)}
E_{\alpha,\alpha}\left(-\lambda_k^\beta(t-s)^\alpha\right)\varphi_k(y)
\nonumber\\
&
\relphantom{=}{} \cdot h(s,u_n(s,y))d\mathbb{W}^H_n(s,y)\bigg{|}^2 &
\end{flalign}

\begin{flalign*}
&=C\mathbb{E}\sum_{k=1}^{\infty}\lambda_k^{2\widetilde{\gamma}}
\bigg(\int_{0}^{t}\int_{\mathcal{D}}(t-s)^{\alpha-1}e^{-\nu(t-s)}E_{\alpha,\alpha}\left(-\lambda_k^\beta(t-s)^\alpha\right)
\varphi_k(y) \\
&\relphantom{=}{}\cdot  \sum_{j,l=1}^{\infty}(h(s,u_n(s))\cdot e_l,\varphi_j) \varphi_j(y)\varrho_l^n(s)\left(\sum_{i=1}^{N_t}\frac{1}{\tau^{1-H}}
\xi_{li}^H\chi_i(s)\right)dyds\bigg)^2\nonumber\\
&=C\mathbb{E}\sum_{k=1}^{\infty}\lambda_k^{2\widetilde{\gamma}}
\bigg(\int_{0}^{t}
(t-s)^{\alpha-1}e^{-\nu(t-s)}E_{\alpha,\alpha}\left(-\lambda_k^\beta(t-s)^\alpha\right)
\sum_{l=1}^{\infty}(h(s,u_n(s))\cdot e_l,\varphi_k)\nonumber\\
&\relphantom{=}{}\cdot\varrho_l^n(s)\frac{1}{\tau}\int_{0}^t
\sum_{i=1}^{N_t}\chi_i(r)d\xi_l^H(r)\chi_i(s)ds\bigg)^2\nonumber &
\end{flalign*}

\begin{flalign*}
&= \frac{C}{\tau^2}\mathbb{E}\sum_{k,l=1}^{\infty}\lambda_k^{2\widetilde{\gamma}}
\bigg(
\int_{0}^{t}\int_{0}^{t}(t-s)^{\alpha-1}
e^{-\nu(t-s)}E_{\alpha,\alpha}\left(-\lambda_k^\beta(t-s)^\alpha\right)\\
&\relphantom{=}{}\cdot(h(s,u_n(s))\cdot e_l,\varphi_k)\varrho_l^n(s)\sum_{i=1}^{N_t}\chi_i(r)\chi_i(s)dsd\xi_l^H(r)\bigg)^2\nonumber\\
&\leq\frac{C}{\tau^2}t^{2H-1}\mathbb{E}\sum_{k,l=1}^{\infty}\lambda_k^{2\widetilde{\gamma}} \int_{0}^{t}\bigg(\int_{0}^{t}(t-s)^{\alpha-1}
e^{-\nu(t-s)}E_{\alpha,\alpha}\left(-\lambda_k^\beta(t-s)^\alpha\right)\nonumber\\
&\relphantom{=}{}\cdot(h(s,u_n(s))\cdot e_l,\varphi_k)\varrho_l^n(s)\sum_{i=1}^{N_t}\chi_i(r)\chi_i(s)ds\bigg)^2dr &
\end{flalign*}

\begin{flalign*}
&= \frac{C}{\tau^2}\mathbb{E}\sum_{k,l=1}^{\infty}\lambda_k^{2\widetilde{\gamma}} \int_{0}^{t}\sum_{i=1}^{N_t}\chi_i(r)^2\bigg(\int_{0}^{t}
(t-s)^{\alpha-1}
e^{-\nu(t-s)}E_{\alpha,\alpha}\left(-\lambda_k^\beta(t-s)^\alpha\right)\nonumber\\
&\relphantom{=}{}\cdot(h(s,u_n(s))\cdot e_l,\varphi_k)\varrho_l^n(s)\chi_i(s)ds\bigg)^2dr\nonumber\\
&\leq \frac{C}{\tau}\mathbb{E}\sum_{k,l=1}^{\infty}\lambda_k^{2\widetilde{\gamma}} \sum_{i=1}^{N_t}\int_{t_i}^{t_{i+1}}
\int_{t_i}^{t_{i+1}}(t-s)^{2\alpha-2}
e^{-2\nu(t-s)}\left|E_{\alpha,\alpha}\left(-\lambda_k^\beta(t-s)^\alpha\right)\right|^2\nonumber\\
&\relphantom{=}{}\cdot|(h(s,u_n(s))\cdot e_l,\varphi_k)|^2|\varrho_l^n(s)|^2dsdr&
\end{flalign*}

\begin{flalign*}
&\leq  C(\widetilde{\mu}_1^n)^2\mathbb{E}\sum_{k,l=1}^{\infty}
\int_{0}^t (t-s)^{2\alpha-4} e^{-2\nu(t-s)}\frac{(\lambda_k^\beta(t-s)^\alpha)^{\frac{2}{\alpha}}}{(1+\lambda_k^\beta(t-s)^\alpha)^2}
\lambda_k^{2\widetilde{\gamma}-\frac{2\beta}{\alpha}}
\\
&\relphantom{=}{}\cdot(h(s,u_n(s))\cdot e_l,\varphi_k)^2ds \nonumber\\
&\leq  C(\widetilde{\mu}_1^n)^2\mathbb{E}\int_{0}^t (t-s)^{2\alpha-4} e^{-2\nu(t-s)}
\|(-\Delta)^{\widetilde{\gamma}-\frac{\beta}{\alpha}}h(s,u_n(s))\|_{\mathcal{L}_2^0}^2ds \nonumber\\
&\leq  C(\widetilde{\mu}_1^n)^2+C(\widetilde{\mu}_1^n)^2\int_{0}^t (t-s)^{2\alpha-4} e^{-2\nu(t-s)}
\mathbb{E}\|u_n(s)\|_{\mathbb{H}^{2\widetilde{\gamma}}}^2ds. &
\end{flalign*}
Collecting the above estimates, we obtain that for all $t\in[0,T]$,
\begin{align*}
& \mathbb{E}\|u_n(t)\|_{\mathbb{H}^{2\widetilde{\gamma}}}^2 \\
&\leq C(1+\nu^2 t^2)e^{-2\nu t}\mathbb{E}\|a\|_{\mathbb{H}^{2\widetilde{\gamma}}}^2+Ce^{-2\nu t}\mathbb{E}\|b\|_{\mathbb{H}^{2\widetilde{\gamma}-\frac{2\beta}{\alpha}}}^2
+C+C (\mu_1^n)^2+C(\widetilde{\mu}_1^n)^2\nonumber\\
&\relphantom{=}{}+\left(C+C (\mu_1^n)^2+C(\widetilde{\mu}_1^n)^2\right)
\int_{0}^t (t-s)^{2\alpha-4} e^{-2\nu(t-s)}
\mathbb{E}\|u_n(s)\|_{\mathbb{H}^{2\widetilde{\gamma}}}^2ds.
\end{align*}
Then an application of Lemma \ref{lem2.11} gives
\begin{align}\label{eq.4.8}
\mathbb{E}\|u_n(t)\|_{\mathbb{H}^{2\widetilde{\gamma}}}^2
&\leq \left(C+C (\mu_1^n)^2+C(\widetilde{\mu}_1^n)^2\right)(1+\nu^2 t^2)e^{-2\nu t}\mathbb{E}\|a\|_{\mathbb{H}^{2\widetilde{\gamma}}}^2\nonumber\\
&\relphantom{=}{}+\left(C+\left(C+C (\mu_1^n)^2+C(\widetilde{\mu}_1^n)^2\right)t^{2\alpha-3}\right)e^{-2\nu t}\mathbb{E}\|b\|_{\mathbb{H}^{2\widetilde{\gamma}-\frac{2\beta}{\alpha}}}^2\nonumber\\
&\relphantom{=}{}+C+C (\mu_1^n)^2+C(\widetilde{\mu}_1^n)^2+C\left(1+ (\mu_1^n)^2+(\widetilde{\mu}_1^n)^2\right)^2t^{2\alpha-3},
\end{align}
and consequently
\begin{align}\label{eq.4.9}
&\sup_{0\leq t\leq T}\mathbb{E}\|u_n(t)\|_{\mathbb{H}^{2\widetilde{\gamma}}}^2
\nonumber\\
&\leq \left(C+C (\mu_1^n)^2+C(\widetilde{\mu}_1^n)^2\right)\mathbb{E}\|a\|_{\mathbb{H}^{2\widetilde{\gamma}}}^2
+\left(C+C (\mu_1^n)^2+C(\widetilde{\mu}_1^n)^2\right) \mathbb{E}\|b\|_{\mathbb{H}^{2\widetilde{\gamma}-\frac{2\beta}{\alpha}}}^2\nonumber\\
&\relphantom{=}{}+C+C (\mu_1^n)^2+C(\widetilde{\mu}_1^n)^2+C\left(1+ (\mu_1^n)^2+(\widetilde{\mu}_1^n)^2\right)^2.
\end{align}

Step 2. Estimate of $\mathbb{E}\|\partial_t u_n(t)\|_{\mathbb{H}^{2\widetilde{\gamma}-\frac{2\beta}{\alpha}}}$.

By \eqref{eq.3.37}-\eqref{eq.3.41}, we deduce that
\begin{flalign}\label{eq.4.10}
& \partial_tu_n(t,x) \nonumber
\\
&=\int_{\mathcal{D}}(-\nu e^{-\nu t})\sum\limits_{k=1}^{\infty}
\left(E_{\alpha,1}(-\lambda_k^\beta t^\alpha)+\nu tE_{\alpha,2}(-\lambda_k^\beta t^\alpha)\right)
\varphi_k(x)\varphi_k(y)a(y)dy\nonumber\\
&\relphantom{=}{}+\int_{\mathcal{D}}e^{-\nu t}
\sum\limits_{k=1}^{\infty}\left(-\lambda_k^\beta t^{\alpha-1}E_{\alpha,\alpha}(-\lambda_k^\beta t^\alpha)+\nu E_{\alpha,1}(-\lambda_k^\beta t^\alpha)\right)\varphi_k(x)\varphi_k(y)a(y)dy &
\end{flalign}
\begin{flalign*}
&\relphantom{=}{}+\int_{\mathcal{D}}e^{-\nu t}\sum\limits_{k=1}^{\infty}\left(-\nu tE_{\alpha,2}(-\lambda_k^\beta t^\alpha)+E_{\alpha,1}(-\lambda_k^\beta t^\alpha)\right)\varphi_k(x) \varphi_k(y)b(y)dy\nonumber\\
&\relphantom{=}{}+\int^t_0\int_{\mathcal{D}}e^{-\nu(t-s)}\sum\limits_{k=1}^{\infty}
\Big(-\nu (t-s)^{\alpha-1}E_{\alpha,\alpha}(-\lambda_k^\beta (t-s)^\alpha)\nonumber\\
&\relphantom{==}{}+(t-s)^{\alpha-2} E_{\alpha,\alpha-1}(-\lambda_k^\beta(t-s)^\alpha)\Big)
\varphi_k(x)\varphi_k(y)f(s,u_n(s,y))dyds\nonumber &
\end{flalign*}
\begin{flalign*}
&\relphantom{=}{}+\int^t_0\int_{\mathcal{D}}e^{-\nu(t-s)}\sum\limits_{k=1}^{\infty}
\Big(-\nu (t-s)^{\alpha-1}E_{\alpha,\alpha}(-\lambda_k^\beta (t-s)^\alpha)\nonumber\\
&\relphantom{==}{}
+(t-s)^{\alpha-2} E_{\alpha,\alpha-1}(-\lambda_k^\beta(t-s)^\alpha)\Big)
\varphi_k(x)\varphi_k(y)g(s,u_n(s,y))d\mathbb{W}_n(s,y)\nonumber\\
&\relphantom{=}{}+\int^t_0\int_{\mathcal{D}}e^{-\nu(t-s)}\sum\limits_{k=1}^{\infty}
\Big(-\nu (t-s)^{\alpha-1}E_{\alpha,\alpha}(-\lambda_k^\beta (t-s)^\alpha)\nonumber\\
&\relphantom{==}{}
+(t-s)^{\alpha-2} E_{\alpha,\alpha-1}(-\lambda_k^\beta(t-s)^\alpha)\Big)
\varphi_k(x)\varphi_k(y)h(s,u_n(s,y))d\mathbb{W}_n^H(s,y)\nonumber &
\end{flalign*}
\begin{flalign*}
&:=\sum_{i=1}^6\widehat{\mathcal{J}}_i, & \nonumber
\end{flalign*}
\begin{align}\label{eq.4.11}
\mathbb{E}\|\widehat{\mathcal{J}}_1\|_{\mathbb{H}^{2\widetilde{\gamma}-\frac{2\beta}{\alpha}}}^2
+\mathbb{E}\|\widehat{\mathcal{J}}_2\|_{\mathbb{H}^{2\widetilde{\gamma}-\frac{2\beta}{\alpha}}}^2
\leq Ce^{-2\nu t}\mathbb{E}\|a\|_{\mathbb{H}^{2\widetilde{\gamma}}}^2,
\end{align}
\begin{align}\label{eq.4.12}
\mathbb{E}\|\widehat{\mathcal{J}}_3\|_{\mathbb{H}^{2\widetilde{\gamma}-\frac{2\beta}{\alpha}}}^2
\leq Ce^{-2\nu t}\mathbb{E}\|b\|_{\mathbb{H}^{2\widetilde{\gamma}-\frac{2\beta}{\alpha}}}^2,
\end{align}
and
\begin{align}\label{eq.4.13}
\mathbb{E}\|\widehat{\mathcal{J}}_4\|_{\mathbb{H}^{2\widetilde{\gamma}-\frac{2\beta}{\alpha}}}^2
\leq C(t^{2\alpha-1}+t^{2\alpha-3})\sup_{0\leq t\leq T}\left(1+\mathbb{E}\|u_n(t)\|_{\mathbb{H}^{2\widetilde{\gamma}}}^2\right).
\end{align}
By slightly modifying the arguments in \eqref{eq.4.6} and \eqref{eq.4.7}, we conclude that
\begin{flalign} \label{eq.4.14}
&\mathbb{E}\|\widehat{\mathcal{J}}_5\|_{\mathbb{H}^{2\widetilde{\gamma}-\frac{2\beta}{\alpha}}}^2\nonumber\\
&%\relphantom{=}{}
=\mathbb{E}\sum_{k=1}^{\infty}\lambda_k^{2\widetilde{\gamma}-\frac{2\beta}{\alpha}}
\bigg{(}\int_{0}^t\int_{\mathcal{D}}e^{-\nu(t-s)}
\sum_{l=1}^{\infty}\Big(-\nu(t-s)^{\alpha-1}E_{\alpha,\alpha}\left(-\lambda_l^\beta(t-s)^\alpha\right)\nonumber\\
&\relphantom{=}{}
+(t-s)^{\alpha-2}E_{\alpha,\alpha-1}\left(-\lambda_l^\beta(t-s)^\alpha\right)\Big)
\varphi_l(\cdot)\varphi_l(y)
g(s,u_n(s,y))d\mathbb{W}_n(s,y),\varphi_k\bigg{)}^2
&%\relphantom{=}{}
\end{flalign}
\begin{flalign*}
&=\mathbb{E}\sum_{k=1}^{\infty}\lambda_k^{2\widetilde{\gamma}-\frac{2\beta}{\alpha}}
\bigg(\int_{0}^t\int_{\mathcal{D}}e^{-\nu(t-s)}
\Big(-\nu(t-s)^{\alpha-1}E_{\alpha,\alpha}\left(-\lambda_k^\beta(t-s)^\alpha\right) \\
&\relphantom{=}{}
+(t-s)^{\alpha-2}E_{\alpha,\alpha-1}\left(-\lambda_k^\beta(t-s)^\alpha\right)\Big)\varphi_k(y)
\sum_{j,l=1}^{\infty}(g(s,u_n(s))\cdot e_l,\varphi_j)\varphi_j(y) \\
&\relphantom{=}{}
\cdot\varsigma_l^n(s)\left(\sum_{i=1}^{N_t}\frac{1}{\sqrt{\tau}}\xi_{li}\chi_i(s)\right)dyds\bigg)^2 &%\relphantom{=}{}
\end{flalign*}
\begin{flalign*}
&%\relphantom{=}{}
 =\frac{1}{\tau^2}\mathbb{E}\sum_{k=1}^{\infty}
\lambda_k^{2\widetilde{\gamma}-\frac{2\beta}{\alpha}}
\bigg(\sum_{i=1}^{N_t}
\int_{t_i}^{t_{i+1}}\int_{t_i}^{t_{i+1}}e^{-\nu(t-s)}
\Big(-\nu(t-s)^{\alpha-1}E_{\alpha,\alpha}\left(-\lambda_k^\beta(t-s)^\alpha\right) \\
&\relphantom{=}{}
+(t-s)^{\alpha-2}E_{\alpha,\alpha-1}\left(-\lambda_k^\beta(t-s)^\alpha\right)\Big)
\sum_{l=1}^{\infty}(g(s,u_n(s))\cdot e_l,\varphi_k)\varsigma_l^n(s)dsd\xi_l(r)\bigg)^2 \\
&%\relphantom{=}{}
\leq \frac{1}{\tau}
\mathbb{E}\sum_{k,l=1}^{\infty}\lambda_k^{2\widetilde{\gamma}-\frac{2\beta}{\alpha}} \sum_{i=1}^{N_t}\bigg|\int_{t_i}^{t_{i+1}}e^{-\nu(t-s)}
\Big(-\nu(t-s)^{\alpha-1}E_{\alpha,\alpha}\left(-\lambda_k^\beta(t-s)^\alpha\right) \\
&\relphantom{=}{}
+(t-s)^{\alpha-2}E_{\alpha,\alpha-1}\left(-\lambda_k^\beta(t-s)^\alpha\right)\Big)
(g(s,u_n(s))\cdot e_l,\varphi_k)\varsigma_l^n(s)ds\bigg|^2 
&
\end{flalign*}
\begin{flalign*}
&
\leq C\mathbb{E}\sum_{k,l=1}^{\infty}\int_{0}^{t}
e^{-2\nu(t-s)}\left((t-s)^{2\alpha-2}+(t-s)^{2\alpha-4}\right)
\frac{1}{(1+\lambda_k^\beta(t-s)^\alpha)^2} \\
&\relphantom{=}{}\cdot|\lambda_k^{\widetilde{\gamma}-\frac{\beta}{\alpha}}(g(s,u_n(s))\cdot e_l,\varphi_k)\varsigma_l^n(s)|^2ds\\
&%\relphantom{=}{}
\leq  C(\mu_1^n)^2\int_{0}^t e^{-2\nu(t-s)}
\left((t-s)^{2\alpha-2}+(t-s)^{2\alpha-4}\right)
\mathbb{E}\|(-\Delta)^{\widetilde{\gamma}-\frac{\beta}{\alpha}}g(s,u_n(s))\|_{\mathcal{L}_2^0}^2ds \\
&%\relphantom{=}{}
\leq  C(\mu_1^n)^2\left(t^{2\alpha-1}+t^{2\alpha-3}\right)
\sup_{0\leq t\leq T}\left(1+\mathbb{E}\|u_n(t)\|_{\mathbb{H}^{2\widetilde{\gamma}}}^2\right),  &
\end{flalign*}
and
\begin{flalign*}
&\mathbb{E}\|\widehat{\mathcal{J}}_6\|_{\mathbb{H}^{2\widetilde{\gamma}-\frac{2\beta}{\alpha}}}^2\nonumber\\
& %\relphantom{=}{} 
= \mathbb{E}\sum_{k=1}^{\infty}\lambda_k^{2\widetilde{\gamma}-\frac{2\beta}{\alpha}}
\bigg{(}\int_{0}^t\int_{\mathcal{D}}e^{-\nu(t-s)}
\sum_{l=1}^{\infty}\Big(-\nu(t-s)^{\alpha-1}E_{\alpha,\alpha}\left(-\lambda_l^\beta(t-s)^\alpha\right)\nonumber\\
&\relphantom{=}{}
+(t-s)^{\alpha-2}E_{\alpha,\alpha-1}\left(-\lambda_l^\beta(t-s)^\alpha\right)\Big)
\varphi_l(\cdot)\varphi_l(y)
h(s,u_n(s,y))d\mathbb{W}^H_n(s,y),\varphi_k\bigg{)}^2\nonumber &
\end{flalign*}
\begin{flalign} \label{eq.4.15}
& %\relphantom{=}{} 
=\mathbb{E}\sum_{k=1}^{\infty}
\lambda_k^{2\widetilde{\gamma}-\frac{2\beta}{\alpha}}
\bigg{(}\int_{0}^t\int_{\mathcal{D}}e^{-\nu(t-s)}
\Big(-\nu(t-s)^{\alpha-1}E_{\alpha,\alpha}\left(-\lambda_k^\beta(t-s)^\alpha\right)\nonumber\\
&\relphantom{=}{}
+(t-s)^{\alpha-2}E_{\alpha,\alpha-1}\left(-\lambda_k^\beta(t-s)^\alpha\right)\Big)\varphi_k(y)
\sum_{j,l=1}^{\infty}(h(s,u_n(s))\cdot e_l,\varphi_j)\varphi_j(y)\nonumber\\
&\relphantom{=}{}\cdot\varrho_l^n(s)\left(\sum_{i=1}^{N_t}\frac{1}{\tau^{1-H}}\xi_{li}^H\chi_i(s)\right)dyds\bigg)^2
&%\relphantom{=}{} 
\end{flalign}
\begin{flalign*}
&= \frac{1}{\tau^2}
\mathbb{E}\sum_{k,l=1}^{\infty}\lambda_k^{2\widetilde{\gamma}-\frac{2\beta}{\alpha}}
\bigg(\int_{0}^{t}\int_{0}^{t}e^{-\nu(t-s)}
\Big(-\nu(t-s)^{\alpha-1}E_{\alpha,\alpha}\left(-\lambda_k^\beta(t-s)^\alpha\right)\nonumber\\
&\relphantom{=}{}
+(t-s)^{\alpha-2}E_{\alpha,\alpha-1}\left(-\lambda_k^\beta(t-s)^\alpha\right)\Big)
(h(s,u_n(s))\cdot e_l,\varphi_k)\varrho_l^n(s)\nonumber\\
&\relphantom{=}{}\cdot\sum_{i=1}^{N_t}\chi_i(r)\chi_i(s)dsd\xi_l^H(r)\bigg)^2&
\end{flalign*}
\begin{flalign*}
& %\relphantom{=}{}
\leq\frac{t^{2H-1}}{\tau^2}
\mathbb{E}\sum_{k,l=1}^{\infty}\lambda_k^{2\widetilde{\gamma}-\frac{2\beta}{\alpha}} \int_{0}^{t}\bigg(\int_{0}^{t}e^{-\nu(t-s)}
\Big(-\nu(t-s)^{\alpha-1}E_{\alpha,\alpha}\left(-\lambda_k^\beta(t-s)^\alpha\right)\nonumber\\
&\relphantom{=}{}
+(t-s)^{\alpha-2}E_{\alpha,\alpha-1}\left(-\lambda_k^\beta(t-s)^\alpha\right)\Big)
(h(s,u_n(s))\cdot e_l,\varphi_k)\varrho_l^n(s)\sum_{i=1}^{N_t}\chi_i(r)\chi_i(s)ds\bigg)^2dr&
\end{flalign*}
\begin{flalign*}
&
\leq \frac{t^{2H-1}}{\tau^2}\mathbb{E}\sum_{k,l=1}^{\infty}\lambda_k^{2\widetilde{\gamma}-\frac{2\beta}{\alpha}} \int_{0}^{t}\sum_{i=1}^{N_t}\chi_i(r)^2\bigg(\int_{0}^{t}e^{-\nu(t-s)}
\Big(-\nu(t-s)^{\alpha-1}E_{\alpha,\alpha}\left(-\lambda_k^\beta(t-s)^\alpha\right)\nonumber\\
&\relphantom{=}{}
+(t-s)^{\alpha-2}E_{\alpha,\alpha-1}\left(-\lambda_k^\beta(t-s)^\alpha\right)\Big)
(h(s,u_n(s))\cdot e_l,\varphi_k)\varrho_l^n(s)\chi_i(s)ds\bigg)^2dr&
\end{flalign*}
\begin{flalign*}
&
\leq \frac{Ct^{2H-1}}{\tau}\mathbb{E}\sum_{k,l=1}^{\infty}\sum_{i=1}^{N_t}\int_{t_i}^{t_{i+1}}\int_{t_i}^{t_{i+1}}
e^{-2\nu(t-s)}\left((t-s)^{2\alpha-2}+(t-s)^{2\alpha-4}\right)\nonumber\\
&\relphantom{=}{}\cdot\frac{1}{(1+\lambda_k^\beta(t-s)^\alpha)^2}
|\lambda_k^{\widetilde{\gamma}-\frac{\beta}{\alpha}} (h(s,u_n(s))\cdot e_l,\varphi_k)\varrho_l^n(s)|^2dsdr&
\end{flalign*}
\begin{flalign*}
&
\leq  Ct^{2H-1}(\widetilde{\mu}_1^n)^2
\int_{0}^te^{-2\nu(t-s)}\left((t-s)^{2\alpha-2}+(t-s)^{2\alpha-4}\right)
\mathbb{E}\|(-\Delta)^{\widetilde{\gamma}-\frac{\beta}{\alpha}}h(s,u_n(s))\|_{\mathcal{L}_2^0}^2ds \nonumber\\
& %\relphantom{=}{}
\leq  C(\widetilde{\mu}_1^n)^2\left(t^{2\alpha+2H-2}+(t-s)^{2\alpha+2H-4}\right)
\sup_{0\leq t\leq T}\left(1+\mathbb{E}\|u_n(t)\|_{\mathbb{H}^{2\widetilde{\gamma}}}^2\right). &
\end{flalign*}
Then \eqref{eq.4.10}-\eqref{eq.4.15} implies that
\begin{flalign*}
 & \mathbb{E}\|\partial_tu_n(t)\|_{\mathbb{H}^{2\widetilde{\gamma}-\frac{2\beta}{\alpha}}}^2 \\
 &
\leq Ce^{-2\nu t}\mathbb{E}\|a\|_{\mathbb{H}^{2\widetilde{\gamma}}}^2
+Ce^{-2\nu t}\mathbb{E}\|b\|_{\mathbb{H}^{2\widetilde{\gamma}-\frac{2\beta}{\alpha}}}^2\nonumber\\
& \relphantom{=}{} +C(t^{2\alpha-1}+t^{2\alpha-3})\sup_{0\leq t\leq T}\left(1+\mathbb{E}\|u_n(t)\|_{\mathbb{H}^{2\widetilde{\gamma}}}^2\right)\nonumber\\
& \relphantom{=}{} +C(\mu_1^n)^2\left(t^{2\alpha-1}+t^{2\alpha-3}\right)
\sup_{0\leq t\leq T}\left(1+\mathbb{E}\|u_n(t)\|_{\mathbb{H}^{2\widetilde{\gamma}}}^2\right)\nonumber\\
& \relphantom{=}{} +C(\widetilde{\mu}_1^n)^2\left(t^{2\alpha+2H-2}+(t-s)^{2\alpha+2H-4}\right)
\sup_{0\leq t\leq T}\left(1+\mathbb{E}\|u_n(t)\|_{\mathbb{H}^{2\widetilde{\gamma}}}^2\right);
\end{flalign*}
and thus
\begin{flalign}\label{eq.4.16}
& \sup_{0\leq t\leq T}\mathbb{E}\|\partial_tu_n(t)\|_{\mathbb{H}^{2\widetilde{\gamma}-\frac{2\beta}{\alpha}}}^2 \nonumber
\\
& \leq  C\mathbb{E}\|a\|_{\mathbb{H}^{2\widetilde{\gamma}}}^2
+C\mathbb{E}\|b\|_{\mathbb{H}^{2\widetilde{\gamma}-\frac{2\beta}{\alpha}}}^2\nonumber\\
& \relphantom{=}{}+\left(C+C(\mu_1^n)^2+C(\widetilde{\mu}_1^n)^2\right)
\left(1+\sup_{0\leq t\leq T}\mathbb{E}\|u_n(t)\|_{\mathbb{H}^{2\widetilde{\gamma}}}^2\right).
\end{flalign}

Step 3. Estimate of $\mathbb{E}\|{_0^c\partial}_t^{\alpha,\nu}u_n(t)\|_{\mathbb{H}^{2\widetilde{\gamma}-2\beta}}^2$.

It follows from \eqref{eq0.1} that
\begin{flalign}\label{eq.4.17}
&\mathbb{E}\|{_0^c\partial}_t^{\alpha,\nu}u_n(t)\|_{\mathbb{H}^{2\widetilde{\gamma}-2\beta}}^2 \nonumber \\
&=\mathbb{E}\sum_{k=1}^{\infty}\lambda_k^{2\widetilde{\gamma}-2\beta}
({_0^c\partial}_t^{\alpha,\nu}u_n(t),\varphi_k)^2\nonumber\\
&\leq C\mathbb{E}\sum_{k=1}^{\infty}\lambda_k^{2\widetilde{\gamma}-2\beta}
\left((-\Delta)^\beta u_n(t),\varphi_k\right)^2+C\mathbb{E}\sum_{k=1}^{\infty}\lambda_k^{2\widetilde{\gamma}-2\beta}
(f(t,u_n(t)),\varphi_k)^2\nonumber\\
&\relphantom{=}{}+C\mathbb{E}\sum_{k=1}^{\infty}\lambda_k^{2\widetilde{\gamma}-2\beta}
\left(g(t,u_n(t))\frac{\partial^2\mathbb{W}_n(t,\cdot)}{\partial t\partial\cdot},\varphi_k\right)^2\nonumber\\
&\relphantom{=}{}+C\mathbb{E}\sum_{k=1}^{\infty}\lambda_k^{2\widetilde{\gamma}-2\beta}
\left(h(t,u_n(t))\frac{\partial^2\mathbb{W}^H_n(t,\cdot)}{\partial t\partial\cdot},\varphi_k\right)^2\nonumber\\
&:=\widetilde{\mathcal{J}}_1+\widetilde{\mathcal{J}}_2+\widetilde{\mathcal{J}}_3+\widetilde{\mathcal{J}}_4.
\end{flalign}
Note that $\{\varphi_k\}_{k=1}^{\infty}$ is an orthonormal basis in $L^2(\mathcal{D})$. By the definition of fractional Laplacian, there exists
\begin{align}\label{eq.4.18}
\widetilde{\mathcal{J}}_1&=C\mathbb{E}\sum_{k=1}^{\infty}\lambda_k^{2\widetilde{\gamma}-2\beta}
\left(\sum_{j=1}^\infty \lambda_j^\beta(u_n(t),\varphi_j)\varphi_j,\varphi_k\right)^2=C\mathbb{E}\sum_{k=1}^{\infty}\lambda_k^{2\widetilde{\gamma}}
(u_n(t),\varphi_k)^2\nonumber\\
&=C\mathbb{E}\|u_n(t)\|_{\mathbb{H}^{2\widetilde{\gamma}}}^2.
\end{align}
For $\widetilde{\mathcal{J}}_2$, we use $(\bf{A}_1)$, $\lambda_k^{\frac{2\beta}{\alpha}-2\beta}\leq C$ and $\mathbb{E}\|u_n(t)\|_{\mathbb{H}^{2\widetilde{\gamma}-\frac{2\beta}{\alpha}}}^2
\leq C \mathbb{E}\|u_n(t)\|_{\mathbb{H}^{2\widetilde{\gamma}}}^2$ to get
\begin{align}\label{eq.4.19}
\widetilde{\mathcal{J}}_2&=C\mathbb{E}\sum_{k=1}^{\infty}\lambda_k^{\frac{2\beta}{\alpha}-2\beta}
\lambda_k^{2\widetilde{\gamma}-\frac{2\beta}{\alpha}}
(f(t,u_n(t)),\varphi_k)^2\leq C\mathbb{E}\sum_{k=1}^{\infty}\lambda_k^{2\widetilde{\gamma}-\frac{2\beta}{\alpha}}
(f(t,u_n(t)),\varphi_k)^2\nonumber\\
&\leq C\mathbb{E}\left(1+\|u_n(t)\|_{\mathbb{H}^{2\widetilde{\gamma}-\frac{2\beta}{\alpha}}}^2\right) \leq C\mathbb{E}\left(1+\|u_n(t)\|_{\mathbb{H}^{2\widetilde{\gamma}}}^2\right).
\end{align}
Note that $\{\varphi_k\}_{k=1}^{\infty}$ is an orthonormal basis in $L^2(\mathcal{D})$, the Brownian motion has independent increments, and
$\{\xi_l\}_{l=1}^{\infty}$ is a family of mutally independent one-dimensional standard Brownain motions. Hence we deduce from \eqref{eq.2.8}, $(\bf{A}_1)$, the It\^{o} isometry and the boundedness assumption on $\varsigma_k^n$ that
\begin{flalign}\label{eq.4.20}
&\widetilde{\mathcal{J}}_3
\nonumber\\
&=C\mathbb{E}\sum_{k=1}^{\infty}\lambda_k^{2\widetilde{\gamma}-2\beta}
\left(\sum_{j,l=1}^{\infty}(g(t,u_n(t))\cdot e_l,\varphi_j)\varphi_j\varsigma_l^n(t)\left(\sum_{i=1}^{N_t}\frac{1}{\sqrt{\tau}}\xi_{li}\chi_i(t)\right),
\varphi_k\right)^2\nonumber\\
&= C\mathbb{E}\sum_{k=1}^{\infty}\lambda_k^{2\widetilde{\gamma}-2\beta}
\left(\sum_{l=1}^{\infty}(g(t,u_n(t))\cdot e_l,\varphi_k)\varsigma_l^n(t)\left(\sum_{i=1}^{N_t}\frac{\xi_{l}(t_{i+1})-\xi_{l}(t_{i})}{\tau}\chi_i(t)\right)
\right)^2\nonumber\\
&\leq \frac{C}{\tau^2}\mathbb{E}\sum_{k,l=1}^{\infty}\lambda_k^{2\widetilde{\gamma}-2\beta}
\sum_{i=1}^{N_t}\bigg(\int_{t_i}^{t_{i+1}}(g(t,u_n(t))\cdot e_l,\varphi_k)\varsigma_l^n(t)\chi_i(t)d\xi_l(r)\bigg)^2\nonumber&
\end{flalign}
\begin{flalign}
& \leq \frac{C(\mu_1^n)^2}{\tau^2}\mathbb{E}\sum_{k,l=1}^{\infty}\lambda_k^{2\widetilde{\gamma}-2\beta}
\sum_{i=1}^{N_t}\int_{t_i}^{t_{i+1}}\chi_i(t)(g(t,u_n(t))\cdot e_l,\varphi_k)^2dr\nonumber\\
&\leq \frac{Ct(\mu_1^n)^2}{\tau}\mathbb{E}\sum_{k,l=1}^{\infty}\lambda_k^{2\widetilde{\gamma}-\frac{2\beta}{\alpha}}
(g(t,u_n(t))\cdot e_l,\varphi_k)^2\nonumber\\
&=\frac{Ct(\mu_1^n)^2}{\tau}\mathbb{E}
\|(-\Delta)^{\widetilde{\gamma}-\frac{\beta}{\alpha}}g(t,u_n(t))\|_{\mathcal{L}_2^0}^2\nonumber\\
&\leq \frac{Ct(\mu_1^n)^2}{\tau} \left(1+\mathbb{E}\|u_n(t)\|_{\mathbb{H}^{2\widetilde{\gamma}}}^2\right). &
\end{flalign}
Since $\{\xi_l^H\}_{l=1}^{\infty}$ is a family of mutally independent one-dimensional fractional Brownain motions, using \eqref{eq.2.9}, $(\bf{A}_1)$, Lemma \ref{le2.10} and the boundedness assumption on $\varrho_k^n$, we get
\begin{flalign}\label{eq.4.21}
& \widetilde{\mathcal{J}}_4 \nonumber \\
&=C\mathbb{E}\sum_{k=1}^{\infty}\lambda_k^{2\widetilde{\gamma}-2\beta}
\left(\sum_{j,l=1}^{\infty}(h(t,u_n(t))\cdot e_l,\varphi_j)\varphi_j\varrho_l^n(t)\left(\sum_{i=1}^{N_t}\frac{1}{\tau^{1-H}}\xi^H_{li}\chi_i(t)\right),
\varphi_k\right)^2\nonumber\\
&= C\mathbb{E}\sum_{k=1}^{\infty}\lambda_k^{2\widetilde{\gamma}-2\beta}
\left(\sum_{l=1}^{\infty}(h(t,u_n(t))\cdot e_l,\varphi_k)\varrho_l^n(t)\left(\sum_{i=1}^{N_t}\frac{\xi_{l}^H(t_{i+1})-\xi_{l}^H(t_{i})}{\tau}\chi_i(t)\right)
\right)^2\nonumber\\
&= \frac{C}{\tau^2}\mathbb{E}\sum_{k=1}^{\infty}\lambda_k^{2\widetilde{\gamma}-2\beta}
\bigg(\sum_{i=1}^{N_t}\int_{t_i}^{t_{i+1}}\sum_{l=1}^{\infty}(h(t,u_n(t))\cdot e_l,\varphi_k)\varrho_l^n(t)\chi_i(t)d\xi_l^H(r)\bigg)^2\nonumber\\ &= \frac{C}{\tau^2}\mathbb{E}\sum_{k,l=1}^{\infty}\lambda_k^{2\widetilde{\gamma}-2\beta}
\sum_{i=1}^{N_t}\bigg(\int_{t_i}^{t_{i+1}}(h(t,u_n(t))\cdot e_l,\varphi_k)\varrho_l^n(t)\chi_i(t)d\xi_l^H(r)\bigg)^2\nonumber&
\end{flalign}
\begin{flalign} &\leq \frac{C(\widetilde{\mu}_1^n)^2\tau^{2H-1}}{\tau^2}\mathbb{E}\sum_{k,l=1}^{\infty}\lambda_k^{2\widetilde{\gamma}-\frac{2\beta}{\alpha}}
\sum_{i=1}^{N_t}\int_{t_i}^{t_{i+1}}\chi_i(t)(h(t,u_n(t))\cdot e_l,\varphi_k)^2dr\nonumber\\
&\leq Ct(\widetilde{\mu}_1^n)^2\tau^{2H-2}\mathbb{E}\sum_{k,l=1}^{\infty}\lambda_k^{2\widetilde{\gamma}-\frac{2\beta}{\alpha}}
(h(t,u_n(t))\cdot e_l,\varphi_k)^2\nonumber\\
&=Ct(\widetilde{\mu}_1^n)^2\tau^{2H-2}\mathbb{E}
\|(-\Delta)^{\widetilde{\gamma}-\frac{\beta}{\alpha}}h(t,u_n(t))\|_{\mathcal{L}_2^0}^2\nonumber\\
&\leq Ct(\widetilde{\mu}_1^n)^2\tau^{2H-2} \left(1+\mathbb{E}\|u_n(t)\|_{\mathbb{H}^{2\widetilde{\gamma}}}^2\right). &
\end{flalign}
Therefore,
\begin{align*}
\mathbb{E}\|{_0^c\partial}_t^{\alpha,\nu}u_n(t)\|_{\mathbb{H}^{2\widetilde{\gamma}-2\beta}}^2
\leq &C\mathbb{E} \left(1+\|u_n(t)\|_{\mathbb{H}^{2\widetilde{\gamma}}}^2\right)
+Ct(\mu_1^n)^2\tau^{-1} \mathbb{E}\left(1+\|u_n(t)\|_{\mathbb{H}^{2\widetilde{\gamma}}}^2\right)\\
&+Ct(\widetilde{\mu}_1^n)^2\tau^{2H-2} \mathbb{E}\left(1+\|u_n(t)\|_{\mathbb{H}^{2\widetilde{\gamma}}}^2\right),
\end{align*}
and consequently,
\begin{align}\label{eq.4.22}
&\sup_{0\leq t\leq T}\mathbb{E}\|{_0^c\partial}_t^{\alpha,\nu}u_n(t)\|_{\mathbb{H}^{2\widetilde{\gamma}-2\beta}}^2 \nonumber
\\
&\leq \left(C+C(\mu_1^n)^2\tau^{-1}
+C(\widetilde{\mu}_1^n)^2\tau^{2H-2}\right) \left(1+\sup_{0\leq t\leq T}\mathbb{E} \|u_n(t)\|_{\mathbb{H}^{2\widetilde{\gamma}}}^2\right).
\end{align}

Step 4. We prove a H\"{o}lder regularity property of the solution $u_n$.

In a similar way as in \eqref{eq.3.29}-\eqref{eq.3.33}, we have
for $0\leq \theta_1\leq \theta_2\leq T$,
\begin{flalign}\label{eq.4.23}
&\mathbb{E}\|u_n(\theta_2)-u_n(\theta_1)\|_{\mathbb{H}^{2\widetilde{\gamma}-\frac{2\beta}{\alpha}}}^2\nonumber\\
&%\relphantom{=}{}
\leq C
\mathbb{E}\bigg{\|}\int_{\mathcal{D}}\mathcal{T}^\nu_{\alpha,\beta}(\theta_2,\cdot,y)a(y)dy
-\int_{\mathcal{D}}\mathcal{T}^\nu_{\alpha,\beta}(\theta_1,x,y)a(y)dy\bigg{\|}_{\mathbb{H}^{2\widetilde{\gamma}-\frac{2\beta}{\alpha}}}^2\nonumber\\
&\relphantom{=}{}+C\mathbb{E}\bigg{\|}\int_{\mathcal{D}}\mathcal{R}^\nu_{\alpha,\beta}(\theta_2,\cdot,y)b(y)dy
-\int_{\mathcal{D}}\mathcal{R}^\nu_{\alpha,\beta}(\theta_1,\cdot,y)b(y)dy\bigg{\|}_{\mathbb{H}^{2\widetilde{\gamma}-\frac{2\beta}{\alpha}}}^2\nonumber\\
&\relphantom{=}{}+C\mathbb{E}\bigg{\|}\int_{0}^{\theta_2}\int_{\mathcal{D}}\mathcal{S}^\nu_{\alpha,\beta}
(\theta_2-s,\cdot,y)f(s,u_n(s,y))dyds\nonumber\\
&\relphantom{=}{}-\int_{0}^{\theta_1}\int_{\mathcal{D}}\mathcal{S}^\nu_{\alpha,\beta}
(\theta_1-s,\cdot,y)f(s,u_n(s,y))dyds\bigg{\|}_{\mathbb{H}^{2\widetilde{\gamma}-\frac{2\beta}{\alpha}}}^2\nonumber&
\end{flalign}
\begin{flalign}
&\relphantom{=}{}+C\mathbb{E}\bigg{\|}\int_{0}^{\theta_2}\int_{\mathcal{D}}\mathcal{S}^\nu_{\alpha,\beta}
(\theta_2-s,\cdot,y)g(s,u_n(s,y))d{\mathbb{W}_n}(s,y)\nonumber\\
&\relphantom{=}{}-\int_{0}^{\theta_1}\int_{\mathcal{D}}
\mathcal{S}^\nu_{\alpha,\beta}(\theta_1-s,\cdot,y)g(s,u_n(s,y))d{\mathbb{W}_n}(s,y)
\bigg{\|}_{\mathbb{H}^{2\widetilde{\gamma}-\frac{2\beta}{\alpha}}}^2\nonumber\\
&\relphantom{=}{}+C\mathbb{E}\bigg{\|}\int_{0}^{\theta_2}\int_{\mathcal{D}}\mathcal{S}^\nu_{\alpha,\beta}
(\theta_2-s,\cdot,y)h(s,u_n(s,y))d{\mathbb{W}^H_n}(s,y)\nonumber\\
&\relphantom{=}{}-\int_{0}^{\theta_1}\int_{\mathcal{D}}
\mathcal{S}^\nu_{\alpha,\beta}(\theta_1-s,\cdot,y)h(s,u_n(s,y))d{\mathbb{W}^H_n}(s,y)
\bigg{\|}_{\mathbb{H}^{2\widetilde{\gamma}-\frac{2\beta}{\alpha}}}^2\nonumber\\
&%\relphantom{=}{}
:=\widetilde{\mathcal{E}}_1+\widetilde{\mathcal{E}}_2+\widetilde{\mathcal{E}}_3
+\widetilde{\mathcal{E}}_4+\widetilde{\mathcal{E}}_5, &
\end{flalign}
\begin{align}\label{eq.4.24}
\widetilde{\mathcal{E}}_1 & \leq C|\theta_2-\theta_1|^2
\mathbb{E}\|a\|_{\mathbb{H}^{2\widetilde{\gamma}}}^2,
\end{align}
\begin{align}\label{eq.4.25}
\widetilde{\mathcal{E}}_2 & \leq C |\theta_2-\theta_1|^2
\mathbb{E}\|b\|_{\mathbb{H}^{2\widetilde{\gamma}-\frac{2\beta}{\alpha}}}^2,
\end{align}
and
\begin{align}\label{eq.4.26}
\widetilde{\mathcal{E}}_3 & \leq C \left(|\theta_2-\theta_1|^2+|\theta_2-\theta_1|^{2\alpha}\right)
\left(1+\sup_{0\leq s \leq T}\mathbb{E}\|u_n(s)\|_{\mathbb{H}^{2\widetilde{\gamma}}}^2\right).
\end{align}
Without loss of generality, we assume that there exist positive integers $N_{\theta_1}$, $N_{\theta_2}$ such that $\theta_1=t_{N_{\theta_1}+1}$, $\theta_2=t_{N_{\theta_2}+1}$.
Recall that $\{\varphi_k\}_{k=1}^{\infty}$ is an orthonormal basis in $L^2(\mathcal{D})$, $\{\xi_l\}_{l=1}^{\infty}$ and $\{\xi_l^H\}_{l=1}^{\infty}$, respectively, are the sequences of mutually independent one-dimensional standard Brownian motions and fractional Brownian motions, and that the Brownian motion has independent increments. By slightly modifying the arguments in \eqref{eq.3.34} and \eqref{eq.3.35}, in view of \eqref{eq.2.8}-\eqref{eq.2.9}, $(\bf{A}_1)$, Holder's inequality, the It\^{o} isometry, Lemma \ref{le2.10}, the boundedness assumption on $\varsigma_k^n(t)$ and $\varrho_k^n(t)$, we deduce that
\begin{flalign}\label{eq.4.27}
& \widetilde{\mathcal{E}}_4 \nonumber
\\
& =
C\mathbb{E}\sum\limits_{k=1}^{\infty}\lambda_k^{2\widetilde{\gamma}-\frac{2\beta}{\alpha}}
\bigg(\int_{0}^{\theta_2}\int_{\mathcal{D}}
(\theta_2-s)^{\alpha-1}e^{-\nu(\theta_2-s)}\sum\limits_{l=1}^{\infty}
E_{\alpha,\alpha}\left(-\lambda_l^\beta(\theta_2-s)^\alpha\right)\varphi_l(\cdot)\varphi_l(y)\nonumber\\
&\relphantom{=}{}\times
g(s,u_n(s,y))d\mathbb{W}_n(s,y)-\int_{0}^{\theta_1}\int_{\mathcal{D}}
(\theta_1-s)^{\alpha-1}e^{-\nu(\theta_1-s)}\sum\limits_{l=1}^{\infty}
E_{\alpha,\alpha}\left(-\lambda_l^\beta(\theta_1-s)^\alpha\right)\nonumber\\
&\relphantom{=}{}\times\varphi_l(\cdot)\varphi_l(y)
g(s,u_n(s,y))d\mathbb{W}_n(s,y),\varphi_k\bigg)^2\nonumber &
\end{flalign}
\begin{flalign}
&\leq C\mathbb{E}\sum_{k=1}^{\infty}\lambda_k^{2\widetilde{\gamma}-\frac{2\beta}{\alpha}}
\bigg|\int_{0}^{\theta_1}\int_{\mathcal{D}}
\bigg((\theta_2-s)^{\alpha-1}e^{-\nu(\theta_2-s)}
E_{\alpha,\alpha}\left(-\lambda_k^\beta(\theta_2-s)^\alpha\right)\nonumber\\
&\relphantom{=}{}-(\theta_1-s)^{\alpha-1}e^{-\nu(\theta_1-s)}
E_{\alpha,\alpha}\left(-\lambda_k^\beta(\theta_1-s)^\alpha\right)\bigg)
\varphi_k(y)\nonumber\\
&\relphantom{=}{}\times\sum_{j,l=1}^{\infty}(g(s,u_n(s))\cdot e_l,\varphi_j)\varphi_j(y)\varsigma_l^n(s)\left(\sum_{i=1}^{N_{\theta_1}}\frac{1}{\sqrt{\tau}}
\xi_{li}\chi_i(s)\right)dyds\bigg|^2\nonumber\\
&\relphantom{=}{}+C\mathbb{E}\sum_{k=1}^{\infty}\lambda_k^{2\widetilde{\gamma}-\frac{2\beta}{\alpha}}
\bigg|\int_{\theta_1}^{\theta_2}\int_{\mathcal{D}}
(\theta_2-s)^{\alpha-1}e^{-\nu(\theta_2-s)}
E_{\alpha,\alpha}\left(-\lambda_k^\beta(\theta_2-s)^\alpha\right)\varphi_k(y)\nonumber\\
&\relphantom{=}{}\times\sum_{j,l=1}^{\infty}(g(s,u_n(s))\cdot e_l,\varphi_j)\varphi_j(y)\varsigma_l^n(s)\left(\sum_{i=N_{\theta_1}+1}^{N_{\theta_2}}\frac{1}{\sqrt{\tau}}
\xi_{li}\chi_i(s)\right)dyds\bigg|^2\nonumber&
\end{flalign}
\begin{flalign}
&= C\mathbb{E}\sum_{k=1}^{\infty}\lambda_k^{2\widetilde{\gamma}-\frac{2\beta}{\alpha}}
\bigg|\sum_{i=1}^{N_{\theta_1}}\int_{t_i}^{t_{i+1}}
\bigg((\theta_2-s)^{\alpha-1}e^{-\nu(\theta_2-s)}
E_{\alpha,\alpha}\left(-\lambda_k^\beta(\theta_2-s)^\alpha\right)\nonumber\\
&\relphantom{=}{}-(\theta_1-s)^{\alpha-1}
e^{-\nu(\theta_1-s)} E_{\alpha,\alpha}\left(-\lambda_k^\beta(\theta_1-s)^\alpha\right)\bigg)
\sum_{l=1}^{\infty}(g(s,u_n(s))\cdot e_l,\varphi_k)\nonumber\\
&\relphantom{=}{}
\times\varsigma_l^n(s)\frac{\xi_l(t_{i+1})-\xi_l(t_{i})}{\tau}ds\bigg|^2
+C\mathbb{E}\sum_{k=1}^{\infty}\lambda_k^{2\widetilde{\gamma}-\frac{2\beta}{\alpha}}
\bigg|\sum_{i=N_{\theta_1}+1}^{N_{\theta_2}}
\int_{t_i}^{t_{i+1}}(\theta_2-s)^{\alpha-1}e^{-\nu(\theta_2-s)}\nonumber\\
&\relphantom{=}{}\times
E_{\alpha,\alpha}\left(-\lambda_k^\beta(\theta_2-s)^\alpha\right)
\sum_{l=1}^{\infty}(g(s,u_n(s))\cdot e_l,\varphi_k)\varsigma_l^n(s)\frac{\xi_l(t_{i+1})-\xi_l(t_{i})}{\tau}ds\bigg|^2\nonumber&
\end{flalign}
\begin{flalign}
&=\frac{C}{\tau^2}\mathbb{E}\sum_{k,l=1}^{\infty}\lambda_k^{2\widetilde{\gamma}-\frac{2\beta}{\alpha}}
\sum_{i=1}^{N_{\theta_1}}\bigg|\int_{t_i}^{t_{i+1}}
\int_{t_i}^{t_{i+1}}\bigg((\theta_2-s)^{\alpha-1}e^{-\nu(\theta_2-s)}
E_{\alpha,\alpha}\left(-\lambda_k^\beta(\theta_2-s)^\alpha\right)\nonumber\\
&\relphantom{=}{}-(\theta_1-s)^{\alpha-1}e^{-\nu(\theta_1-s)} E_{\alpha,\alpha}\left(-\lambda_k^\beta(\theta_1-s)^\alpha\right)\bigg)
(g(s,u_n(s))\cdot e_l,\varphi_k)\varsigma_l^n(s)dsd\xi_l(r)\bigg|^2\nonumber\\
&\relphantom{=}{}+\frac{C}{\tau^2}\mathbb{E}\sum_{k,l=1}^{\infty}\lambda_k^{2\widetilde{\gamma}-\frac{2\beta}{\alpha}}
\sum_{i=N_{\theta_1}+1}^{N_{\theta_2}}
\bigg|\int_{t_i}^{t_{i+1}}\int_{t_i}^{t_{i+1}}(\theta_2-s)^{\alpha-1}e^{-\nu(\theta_2-s)}
\nonumber\\
&\relphantom{=}{}\times E_{\alpha,\alpha}\left(-\lambda_k^\beta(\theta_2-s)^\alpha\right)(g(s,u_n(s))\cdot e_l,\varphi_k)\varsigma_l^n(s)dsd\xi_l(r)\bigg|^2\nonumber & 
\end{flalign}
\begin{flalign}
&\leq C \mathbb{E} \sum_{k,l=1}^{\infty}\lambda_k^{2\widetilde{\gamma}-\frac{2\beta}{\alpha}}
\sum_{i=1}^{N_{\theta_1}}\int_{t_i}^{t_{i+1}}\bigg((\theta_2-s)^{\alpha-1}e^{-\nu(\theta_2-s)}
E_{\alpha,\alpha}\left(-\lambda_k^\beta(\theta_2-s)^\alpha\right)\nonumber\\
&\relphantom{=}{}-(\theta_1-s)^{\alpha-1}
e^{-\nu(\theta_1-s)} E_{\alpha,\alpha}\left(-\lambda_k^\beta(\theta_1-s)^\alpha\right)\bigg)^2
\left|(g(s,u_n(s))\cdot e_l,\varphi_k)\right|^2
|\varsigma_l^n(s)|^2 ds\nonumber\\
&\relphantom{=}{}+C\mathbb{E}\sum_{k,l=1}^{\infty}\lambda_k^{2\widetilde{\gamma}-\frac{2\beta}{\alpha}}
\sum_{i=N_{\theta_1}+1}^{N_{\theta_2}}
\int_{t_i}^{t_{i+1}}(\theta_2-s)^{2\alpha-2}e^{-2\nu(\theta_2-s)}
\left|E_{\alpha,\alpha}\left(-\lambda_k^\beta(\theta_2-s)^\alpha\right)\right|^2\nonumber\\
&\relphantom{=}{}\times\left|(g(s,u_n(s))\cdot e_l,\varphi_k)\right|^2|\varsigma_l^n(s)|^2ds\nonumber &
\end{flalign}
\begin{flalign}
&
\leq C \mathbb{E}\sum_{k,l=1}^{\infty}
\int_{0}^{\theta_1}\bigg(\bigg|\int_{\theta_1}^{\theta_2}(\tau-s)^{\alpha-2}
E_{\alpha,\alpha-1}\left(-\lambda_k^\beta(\tau-s)^\alpha\right)d\tau\bigg|^2
+(\theta_1-s)^{2\alpha-2}|\theta_2-\theta_1|^2\nonumber\\
&\relphantom{=}{}\times \frac{1}{(1+\lambda_k^\beta(\theta_1-s)^\alpha)^2}\bigg)
\lambda_k^{2\widetilde{\gamma}-\frac{2\beta}{\alpha}}
\left|(g(s,u_n(s))\cdot e_l,\varphi_k)\right|^2|\varsigma_l^n(s)|^2ds\nonumber\\
&\relphantom{=}{}+C\mathbb{E}\sum_{k,l=1}^{\infty}
\int_{\theta_1}^{\theta_2}(\theta_2-s)^{2\alpha-2}
\frac{1}{(1+\lambda_k^\beta(\theta_2-s)^\alpha)^2}
\lambda_k^{2\widetilde{\gamma}-\frac{2\beta}{\alpha}}
\nonumber\\
&
\relphantom{=}{} \times\left|(g(s,u_n(s))\cdot e_l,\varphi_k)\right|^2|\varsigma_l^n(s)|^2ds\nonumber
&
\end{flalign}
\begin{flalign}
&
\leq C (\mu_1^n)^2|\theta_2-\theta_1|^2\int_{0}^{\theta_1}\left(
(\theta_1-s)^{2\alpha-4}+(\theta_1-s)^{2\alpha-2}\right)
\mathbb{E}\|(-\Delta)^{\widetilde{\gamma}-\frac{\beta}{\alpha}}g(s,u_n(s))\|_{\mathcal{L}_2^0}^2ds\nonumber\\
&\relphantom{=}{}+C(\mu_1^n)^2\int_{\theta_1}^{\theta_2}
(\theta_2-s)^{2\alpha-2}
\mathbb{E}\|(-\Delta)^{\widetilde{\gamma}-\frac{\beta}{\alpha}}g(s,u_n(s))\|_{\mathcal{L}_2^0}^2ds\nonumber\\
&\leq C (\mu_1^n)^2\left(|\theta_2-\theta_1|^2+|\theta_2-\theta_1|^{2\alpha-1}\right)
\left(1+\sup_{0\leq s\leq T}\mathbb{E}\|u_n(s)\|_{\mathbb{H}^{2\widetilde{\gamma}}}^2\right),
\end{flalign}
and
\begin{flalign}\label{eq.4.28}
&\widetilde{\mathcal{E}}_5  \nonumber
\\
&= C\mathbb{E}\sum\limits_{k=1}^{\infty}\lambda_k^{2\widetilde{\gamma}-\frac{2\beta}{\alpha}}
\bigg(\int_{0}^{\theta_2}\int_{\mathcal{D}}
(\theta_2-s)^{\alpha-1}e^{-\nu(\theta_2-s)}\sum\limits_{l=1}^{\infty}
E_{\alpha,\alpha}\left(-\lambda_l^\beta(\theta_2-s)^\alpha\right)\varphi_l(\cdot)\varphi_l(y)\nonumber\\
&\relphantom{=}{}\times
h(s,u_n(s,y))d\mathbb{W}^H_n(s,y)-\int_{0}^{\theta_1}\int_{\mathcal{D}}
(\theta_1-s)^{\alpha-1}e^{-\nu(\theta_1-s)}\sum\limits_{l=1}^{\infty}
E_{\alpha,\alpha}\left(-\lambda_l^\beta(\theta_1-s)^\alpha\right)\nonumber\\
&\relphantom{=}{}\times\varphi_l(\cdot)\varphi_k(y)
h(s,u_n(s,y))d\mathbb{W}^H_n(s,y),\varphi_k\bigg)^2\nonumber
\nonumber &
\end{flalign}
\begin{flalign}
&\leq C\mathbb{E}\sum_{k=1}^{\infty}\lambda_k^{2\widetilde{\gamma}-\frac{2\beta}{\alpha}}
\bigg|\int_{0}^{\theta_1}\int_{\mathcal{D}}
\bigg((\theta_2-s)^{\alpha-1}e^{-\nu(\theta_2-s)}
E_{\alpha,\alpha}\left(-\lambda_k^\beta(\theta_2-s)^\alpha\right)\nonumber\\
&\relphantom{=}{}-(\theta_1-s)^{\alpha-1}e^{-\nu(\theta_1-s)}
E_{\alpha,\alpha}\left(-\lambda_k^\beta(\theta_1-s)^\alpha\right)\bigg)
\varphi_k(y)\nonumber\\
&\relphantom{=}{}\times\sum_{j,l=1}^{\infty}(h(s,u_n(s))\cdot e_l,\varphi_j)\varphi_j(y)\varrho_l^n(s)\left(\sum_{i=1}^{N_{\theta_1}}\frac{1}{\tau^{1-H}}
\xi^H_{li}\chi_i(s)\right)dyds\bigg|^2\nonumber\\
&\relphantom{=}{}+C\mathbb{E}\sum_{k=1}^{\infty}\lambda_k^{2\widetilde{\gamma}-\frac{2\beta}{\alpha}}
\bigg|\int_{\theta_1}^{\theta_2}\int_{\mathcal{D}}
(\theta_2-s)^{\alpha-1}e^{-\nu(\theta_2-s)}
E_{\alpha,\alpha}\left(-\lambda_k^\beta(\theta_2-s)^\alpha\right)\varphi_k(y)\nonumber\\
&\relphantom{=}{}\times\sum_{j,l=1}^{\infty}(h(s,u_n(s))\cdot e_l,\varphi_j)\varphi_j(y)\varrho_l^n(s)\left(\sum_{i=N_{\theta_1}+1}^{N_{\theta_2}}\frac{1}{\tau^{1-H}}
\xi^H_{li}\chi_i(s)\right)dyds\bigg|^2\nonumber&
\end{flalign}
\begin{flalign}
&= C\mathbb{E}\sum_{k=1}^{\infty}\lambda_k^{2\widetilde{\gamma}-\frac{2\beta}{\alpha}}
\bigg|\int_{0}^{\theta_1}
\bigg((\theta_2-s)^{\alpha-1}e^{-\nu(\theta_2-s)}
E_{\alpha,\alpha}\left(-\lambda_k^\beta(\theta_2-s)^\alpha\right)\nonumber\\
&\relphantom{=}{}-(\theta_1-s)^{\alpha-1}
e^{-\nu(\theta_1-s)} E_{\alpha,\alpha}\left(-\lambda_k^\beta(\theta_1-s)^\alpha\right)\bigg)
\nonumber\\
&
\relphantom{=}{}\times\sum_{l=1}^{\infty}(h(s,u_n(s))\cdot e_l,\varphi_k)\varrho_l^n(s)\frac{1}{\tau}\int_{0}^{\theta_1}
\sum_{i=1}^{N_{\theta_1}}\chi_i(r)d\xi^H_{l}(r)\chi_i(s)ds\bigg|^2\nonumber\\
&\relphantom{=}{}+C\mathbb{E}\sum_{k=1}^{\infty}\lambda_k^{2\widetilde{\gamma}-\frac{2\beta}{\alpha}}
\bigg|\int_{\theta_1}^{\theta_2}(\theta_2-s)^{\alpha-1}e^{-\nu(\theta_2-s)}
E_{\alpha,\alpha}\left(-\lambda_k^\beta(\theta_2-s)^\alpha\right)\nonumber\\
&\relphantom{=}{}\times\sum_{l=1}^{\infty}(h(s,u_n(s))\cdot e_l,\varphi_k)\varrho_l^n(s)\frac{1}{\tau}\int_{\theta_1}^{\theta_2}
\sum_{i=N_{\theta_1}+1}^{N_{\theta_2}}\chi_i(r)d\xi^H_{l}(r)\chi_i(s)ds\bigg|^2\nonumber&
\end{flalign}
\begin{flalign}
&=\frac{C}{\tau^2}\mathbb{E}\sum_{k,l=1}^{\infty}\lambda_k^{2\widetilde{\gamma}-\frac{2\beta}{\alpha}}
\bigg|\int_{0}^{\theta_1}\int_{0}^{\theta_1}
\bigg((\theta_2-s)^{\alpha-1}e^{-\nu(\theta_2-s)}
E_{\alpha,\alpha}\left(-\lambda_k^\beta(\theta_2-s)^\alpha\right)\nonumber\\
&\relphantom{=}{}-(\theta_1-s)^{\alpha-1}
e^{-\nu(\theta_1-s)} E_{\alpha,\alpha}\left(-\lambda_k^\beta(\theta_1-s)^\alpha\right)\bigg)
(h(s,u_n(s))\cdot e_l,\varphi_k)\varrho_l^n(s)\nonumber\\
&\relphantom{=}{}
\times\sum_{i=1}^{N_{\theta_1}}\chi_i(r)\chi_i(s)dsd\xi^H_l(r)\bigg|^2
+\frac{C}{\tau^2}\mathbb{E}\sum_{k,l=1}^{\infty}\lambda_k^{2\widetilde{\gamma}-\frac{2\beta}{\alpha}}
\bigg|\int_{\theta_1}^{\theta_2}\int_{\theta_1}^{\theta_2}(\theta_2-s)^{\alpha-1}e^{-\nu(\theta_2-s)}\nonumber\\
&\relphantom{=}{}\times E_{\alpha,\alpha}\left(-\lambda_k^\beta(\theta_2-s)^\alpha\right)
(h(s,u_n(s))\cdot e_l,\varphi_k)\varrho_l^n(s)\sum_{i=N_{\theta_1}+1}^{N_{\theta_2}}\chi_i(r)\chi_i(s)dsd\xi^H_l(r)\bigg|^2\nonumber&
\end{flalign}
\begin{flalign}
&\leq \frac{C}{\tau^2}T^{2H-1} \mathbb{E} \sum_{k,l=1}^{\infty}\lambda_k^{2\widetilde{\gamma}-\frac{2\beta}{\alpha}}
\int_{0}^{\theta_1}\bigg|\int_{0}^{\theta_1}\bigg((\theta_2-s)^{\alpha-1}e^{-\nu(\theta_2-s)}
E_{\alpha,\alpha}\left(-\lambda_k^\beta(\theta_2-s)^\alpha\right)\nonumber\\
&\relphantom{=}{}-(\theta_1-s)^{\alpha-1}e^{-\nu(\theta_1-s)} E_{\alpha,\alpha}\left(-\lambda_k^\beta(\theta_1-s)^\alpha\right)\bigg)^2
(h(s,u_n(s))\cdot e_l,\varphi_k)\varrho_l^n(s)\nonumber\\
&\relphantom{=}{}\times \sum_{i=1}^{N_{\theta_1}}\chi_i(r)\chi_i(s)ds\bigg|^2dr
+\frac{C}{\tau^2}T^{2H-1}\mathbb{E}\sum_{k,l=1}^{\infty}\lambda_k^{2\widetilde{\gamma}-\frac{2\beta}{\alpha}}
\int_{\theta_1}^{\theta_2}\bigg|\int_{\theta_1}^{\theta_2}(\theta_2-s)^{\alpha-1}e^{-\nu(\theta_2-s)}\nonumber\\
&\relphantom{==}{}\times E_{\alpha,\alpha}\left(-\lambda_k^\beta(\theta_2-s)^\alpha\right)
(h(s,u_n(s))\cdot e_l,\varphi_k)\varrho_l^n(s)\sum_{i=N_{\theta_1}+1}^{N_{\theta_2}}\chi_i(r)\chi_i(s)ds\bigg|^2dr\nonumber&
\end{flalign}
\begin{flalign}
&=\frac{C}{\tau^2}\mathbb{E} \sum_{k,l=1}^{\infty}\lambda_k^{2\widetilde{\gamma}-\frac{2\beta}{\alpha}}
\int_{0}^{\theta_1}\sum_{i=1}^{N_{\theta_1}}\chi_i(r)^2\bigg|\int_{0}^{\theta_1}\bigg((\theta_2-s)^{\alpha-1}e^{-\nu(\theta_2-s)}
E_{\alpha,\alpha}\left(-\lambda_k^\beta(\theta_2-s)^\alpha\right)\nonumber\\
&\relphantom{=}{}-(\theta_1-s)^{\alpha-1}
e^{-\nu(\theta_1-s)} E_{\alpha,\alpha}\left(-\lambda_k^\beta(\theta_1-s)^\alpha\right)\bigg)^2
(h(s,u_n(s))\cdot e_l,\varphi_k)
\varrho_l^n(s)\chi_i(s)ds\bigg|^2dr\nonumber\\
&\relphantom{=}{}+\frac{C}{\tau^2}\mathbb{E}\sum_{k,l=1}^{\infty}\lambda_k^{2\widetilde{\gamma}-\frac{2\beta}{\alpha}}
\int_{\theta_1}^{\theta_2}\sum_{i=N_{\theta_1}+1}^{N_{\theta_2}}\chi_i(r)^2\bigg|\int_{\theta_1}^{\theta_2}(\theta_2-s)^{\alpha-1}e^{-\nu(\theta_2-s)}
\nonumber\\
&\relphantom{=}{}\times E_{\alpha,\alpha}\left(-\lambda_k^\beta(\theta_2-s)^\alpha\right)(h(s,u_n(s))\cdot e_l,\varphi_k)\varrho_l^n(s)\chi_i(s)ds\bigg|^2dr\nonumber&
\end{flalign}
\begin{flalign}
&\leq \frac{C}{\tau}\mathbb{E} \sum_{k,l=1}^{\infty}\lambda_k^{2\widetilde{\gamma}-\frac{2\beta}{\alpha}}
\sum_{i=1}^{N_{\theta_1}}\int_{t_i}^{t_{i+1}}\int_{t_i}^{t_{i+1}}
\bigg((\theta_2-s)^{\alpha-1}e^{-\nu(\theta_2-s)}
E_{\alpha,\alpha}\left(-\lambda_k^\beta(\theta_2-s)^\alpha\right)\nonumber\\
&\relphantom{=}{}-(\theta_1-s)^{\alpha-1}
e^{-\nu(\theta_1-s)} E_{\alpha,\alpha}\left(-\lambda_k^\beta(\theta_1-s)^\alpha\right)\bigg)^2
|(h(s,u_n(s))\cdot e_l,\varphi_k)|^2
|\varrho_l^n(s)|^2dsdr\nonumber\\
&\relphantom{=}{}+\frac{C}{\tau}\mathbb{E}\sum_{k,l=1}^{\infty}\lambda_k^{2\widetilde{\gamma}-\frac{2\beta}{\alpha}}
\sum_{i=N_{\theta_1}+1}^{N_{\theta_2}}\int_{t_i}^{t_{i+1}}\int_{t_i}^{t_{i+1}}
(\theta_2-s)^{2\alpha-2}e^{-2\nu(\theta_2-s)}
\nonumber\\
&\relphantom{=}{}\times\left|E_{\alpha,\alpha}\left(-\lambda_k^\beta(\theta_2-s)^\alpha\right)\right|^2|(h(s,u_n(s))\cdot e_l,\varphi_k)|^2|\varrho_l^n(s)|^2dsdr\nonumber&
\end{flalign}
\begin{flalign}
&\leq C \mathbb{E}\sum_{k,l=1}^{\infty}
\int_{0}^{\theta_1}\bigg(\bigg|\int_{\theta_1}^{\theta_2}(\tau-s)^{\alpha-2}
E_{\alpha,\alpha-1}\left(-\lambda_k^\beta(\tau-s)^\alpha\right)d\tau\bigg|^2
+(\theta_1-s)^{2\alpha-2}|\theta_2-\theta_1|^2\nonumber\\
&\relphantom{=}{}\times \frac{1}{(1+\lambda_k^\beta(\theta_1-s)^\alpha)^2}\bigg)
\lambda_k^{2\widetilde{\gamma}-\frac{2\beta}{\alpha}}
\left|(h(s,u_n(s))\cdot e_l,\varphi_k)\right|^2|\varrho_l^n(s)|^2ds\nonumber\\
&\relphantom{=}{}+C\mathbb{E}\sum_{k,l=1}^{\infty}
\int_{\theta_1}^{\theta_2}(\theta_2-s)^{2\alpha-2}
\frac{1}{(1+\lambda_k^\beta(\theta_2-s)^\alpha)^2}
\lambda_k^{2\widetilde{\gamma}-\frac{2\beta}{\alpha}}
\nonumber\\
&
\relphantom{=}{}\times\left|(h(s,u_n(s))\cdot e_l,\varphi_k)\right|^2|\varrho_l^n(s)|^2ds\nonumber
&
\end{flalign}
\begin{flalign}
&\leq C (\widetilde{\mu}_1^n)^2|\theta_2-\theta_1|^2\int_{0}^{\theta_1}\left(
(\theta_1-s)^{2\alpha-4}+(\theta_1-s)^{2\alpha-2}\right)
\mathbb{E}\|(-\Delta)^{\widetilde{\gamma}-\frac{\beta}{\alpha}}h(s,u_n(s))\|_{\mathcal{L}_2^0}^2ds\nonumber\\
&\relphantom{=}{}+C(\widetilde{\mu}_1^n)^2\int_{\theta_1}^{\theta_2}
(\theta_2-s)^{2\alpha-2}
\mathbb{E}\|(-\Delta)^{\widetilde{\gamma}-\frac{\beta}{\alpha}}h(s,u_n(s))\|_{\mathcal{L}_2^0}^2ds\nonumber\\
&\leq C (\widetilde{\mu}_1^n)^2\left(|\theta_2-\theta_1|^2+|\theta_2-\theta_1|^{2\alpha-1}\right)
\left(1+\sup_{0\leq s\leq T}\mathbb{E}\|u_n(s)\|_{\mathbb{H}^{2\widetilde{\gamma}}}^2\right).
\end{flalign}
We collect all of the above estimates and arrive at
\begin{align}\label{eq.4.29}
&\mathbb{E}\|u_n(\theta_2)-u_n(\theta_1)\|_{\mathbb{H}^{2\widetilde{\gamma}-\frac{2\beta}{\alpha}}}^2
\leq C|\theta_2-\theta_1|^2\Big(\mathbb{E}\|a\|_{\mathbb{H}^{2\widetilde{\gamma}}}^2
+\mathbb{E}\|b\|_{\mathbb{H}^{2\widetilde{\gamma}-\frac{2\beta}{\alpha}}}^2+1+
(\mu_1^n)^2+(\widetilde{\mu}_1^n)^2\nonumber\\
&\relphantom{==}{}+\sup_{0\leq s\leq T}\mathbb{E}\|u_n(s)\|_{\mathbb{H}^{2\widetilde{\gamma}}}^2+
(\mu_1^n)^2\sup_{0\leq s\leq T}\mathbb{E}\|u_n(s)\|_{\mathbb{H}^{2\widetilde{\gamma}}}^2
+(\widetilde{\mu}_1^n)^2\sup_{0\leq s\leq T}\mathbb{E}\|u_n(s)\|_{\mathbb{H}^{2\widetilde{\gamma}}}^2\Big).
\end{align}
The proof is therefore complete.
\end{proof}
%-------------------------------------------------------------------------------
\section{Proof of Theorem \ref{thm2.1}.}\label{B.2}
\begin{proof}
Without loss of generality, we assume that there exists a positive integer ${N_t}$
such that $t=t_{N_t+1}$. Subtracting \eqref{eq2.15} from \eqref{eq2.1}, we have
\begin{align}\label{eq2.18}
\displaystyle &u(t,x)-u_n(t,x)\nonumber\\
&=\int^t_0\int_{\mathcal{D}}(t-s)^{\alpha-1}e^{-\nu(t-s)}
\sum^{\infty}_{k=1}{E}_{\alpha,\alpha}\left(-\lambda^{\beta}_k(t-s)^
{\alpha}\right)\varphi_k(x)\varphi_k(y)f(s,u(s,y))dyds\nonumber\\
&%\relphantom{=}{}
-\int^t_0\int_{\mathcal{D}}(t-s)^{\alpha-1}e^{-\nu(t-s)}
\sum^{\infty}_{k=1}{E}_{\alpha,\alpha}\left(-\lambda^{\beta}_k(t-s)^
{\alpha}\right)\varphi_k(x)\varphi_k(y)f(s,u_n(s,y))dyds\nonumber\\
&%\relphantom{=}{}
+\int^t_0\int_{\mathcal{D}}(t-s)^{\alpha-1}e^{-\nu(t-s)}\sum^{\infty}_{k=1}
{E}_{\alpha,\alpha}\left(-\lambda^{\beta}_k(t-s)^{\alpha}\right)\varphi_k(x)\varphi_k(y)
g(s,u(s,y))d{\mathbb{W}}(s,y)\nonumber\\
&%\relphantom{=}{}
-\int^t_0\int_{\mathcal{D}}(t-s)^{\alpha-1}e^{-\nu(t-s)}\sum^{\infty}_{k=1}
{E}_{\alpha,\alpha}\left(-\lambda^{\beta}_k(t-s)^{\alpha}\right)\varphi_k(x)\varphi_k(y)
g(s,u_n(s,y))d{\mathbb{W}_n}(s,y)\nonumber\\
&%\relphantom{=}{}
+\int^t_0\int_{\mathcal{D}}(t-s)^{\alpha-1}e^{-\nu(t-s)}\sum^{\infty}_{k=1}
{E}_{\alpha,\alpha}\left(-\lambda^{\beta}_k(t-s)^{\alpha}\right)\varphi_k(x)\varphi_k(y)
h(s,u(s,y))d{\mathbb{W}^H}(s,y)\nonumber\\
&%\relphantom{=}{}
-\int^t_0\int_{\mathcal{D}}(t-s)^{\alpha-1}e^{-\nu(t-s)}\sum^{\infty}_{k=1}
{E}_{\alpha,\alpha}\left(-\lambda^{\beta}_k(t-s)^{\alpha}\right)\varphi_k(x)\varphi_k(y)
h(s,u_n(s,y))d{\mathbb{W}^H_n}(s,y),
\end{align}
where
\begin{equation*}\label{eq2.19}
\displaystyle d\mathbb{W}(s,y)=\frac{\partial^2\mathbb{W}}{\partial s\partial y}dyds
=\sum^{\infty}_{k=1}\varsigma_k(s)e_k(y)dyd{\xi}_k(s),
\end{equation*}
\begin{equation*}\label{eq2.20}
\displaystyle d\mathbb{W}_n(s,y)=\frac{\partial^2{\mathbb{W}_n}}{\partial s
\partial y}dyds=\sum^{\infty}_{k=1}\varsigma^n_k(s)e_k(y)\left(\sum\limits_{i=1}^{N_t}
\frac{1}{\sqrt{\tau}}\xi_{ki}\chi_i(s)\right)dyds,
\end{equation*}
\begin{equation*}\label{eq2.19}
\displaystyle d\mathbb{W}^H(s,y)=\frac{\partial^2\mathbb{W}^H}{\partial s\partial y}dyds
=\sum^{\infty}_{k=1}\varrho_k(s)e_k(y)dyd{\xi}^H_k(s),
\end{equation*}
and
\begin{equation*}\label{eq2.20}
\displaystyle d\mathbb{W}^H_n(s,y)=\frac{\partial^2{\mathbb{W}^H_n}}{\partial s
\partial y}dyds=\sum^{\infty}_{k=1}\varrho^n_k(s)e_k(y)\left(\sum\limits_{i=1}^{N_t}
\frac{1}{\tau^{1-H}}\xi^H_{ki}\chi_i(s)\right)dyds.
\end{equation*}
Let
\begin{align*}
\displaystyle \mathscr{A}_1 &=\int^{t}_{0}\int_{\mathcal{D}}(t-s)^{\alpha-1}e^{-\nu(t-s)}
\sum^{\infty}_{k=1}{E}_{\alpha,\alpha}\left(-\lambda
^{\beta}_k(t-s)^{\alpha}\right)\varphi_k(x)\varphi_k(y)\\
&\relphantom{=}{}\times(f(s,u(s,y))-f(s,u_n(s,y)))dyds,
\end{align*}
\begin{align*}
\displaystyle \mathscr{A}_2 &=\int^{t}_{0}\int_{\mathcal{D}}(t-s)^{\alpha-1}e^{-\nu(t-s)}
\sum^{\infty}_{k=1}{E}_{\alpha,\alpha}\left(-\lambda
^{\beta}_k(t-s)^{\alpha}\right)\varphi_k(x)\varphi_k(y)\\
&\relphantom{=}{}\times\sum_{l=1}^{\infty}(g(s,u(s,y))-g(s,u_n(s,y)))\cdot e_l(y)\varsigma_l(s)dyd\xi_l(s),
\end{align*}
\begin{align*}
\displaystyle \mathscr{A}_3 &=\sum_{i=1}^{N_t}\int^{t_{i+1}}_{t_i}\int_{\mathcal{D}}
(t-s)^{\alpha-1}e^{-\nu(t-s)}\sum^{\infty}_{k=1}{E}_{\alpha,\alpha}\left(-\lambda^{\beta}_k
(t-s)^{\alpha}\right)\varphi_k(x)\varphi_k(y)\\
&\relphantom{==}{}\times\sum_{l=1}^{\infty}(g(s,u_n(s,y))\cdot e_l(y))\varsigma_l(s)dyd\xi_l(s)\\
&\relphantom{=}{}-\sum_{i=1}^{N_t}\int^{t_{i+1}}_{t_i}\int_{\mathcal{D}}(t-t_i)
^{\alpha-1}e^{-\nu(t-t_i)}\sum^{\infty}_{k=1}{E}_{\alpha,\alpha}
\left(-\lambda^{\beta}_k(t-t_i)^{\alpha}\right)\varphi_k(x)\varphi_k(y)\\
&\relphantom{==}{}\times\sum_{l=1}^{\infty}(g(t_i,u_n(t_i,y))\cdot e_l(y))\varsigma_l(t_i)dyd\xi_l(s),
\end{align*}
\begin{align*}
\displaystyle \mathscr{A}_4 &=\sum_{i=1}^{N_t}\int^{t_{i+1}}_{t_i}\int_{\mathcal{D}}
(t-t_i)^{\alpha-1}e^{-\nu(t-t_i)}\sum^{\infty}_{k=1}{E}_{\alpha,\alpha}
\left(-\lambda^{\beta}_k(t-t_i)^{\alpha}\right)\varphi_k(x)\varphi_k(y)\\
&\relphantom{==}{}\times\sum_{l=1}^{\infty}
(g(t_i,u_n(t_i,y))\cdot e_l(y))(\varsigma_l(t_i)-\varsigma_l^n(t_i))dyd\xi_l(s),
\end{align*}
\begin{align*}%\label{eq2.45}\begin{split}
\displaystyle \mathscr{A}_5 &=\sum_{i=1}^{N_t}\int^{t_{i+1}}_{t_i}\int_{\mathcal{D}}
(t-t_i)^{\alpha-1}e^{-\nu(t-t_i)}\sum^{\infty}_{k=1}{E}_{\alpha,\alpha}
\left(-\lambda^{\beta}_k(t-t_i)^{\alpha}\right)\varphi_k(x)\varphi_k(y)\\
&\relphantom{==}{}\times\sum_{l=1}^{\infty}
(g(t_i,u_n(t_i,y))\cdot e_l(y))\varsigma_l^n(t_i)dyd\xi_l(s)\\
&\relphantom{=}{}-\sum_{i=1}^{N_t}\int^{t_{i+1}}_{t_i}\int_{\mathcal{D}}
(t-s)^{\alpha-1}e^{-\nu(t-s)}\sum^{\infty}_{k=1}{E}_{\alpha,\alpha}
\left(-\lambda^{\beta}_k(t-s)^{\alpha}\right)\varphi_k(x)\varphi_k(y)\\
&\relphantom{==}{}\times\sum_{l=1}^{\infty}
(g(s,u_n(s,y))\cdot e_l(y))\varsigma_l^n(s)\frac{\xi_l(t_{i+1})-\xi_l(t_i)}{\tau}dyds,
\end{align*}
\begin{align*}
\displaystyle \mathscr{A}_6 &=\int^{t}_{0}\int_{\mathcal{D}}(t-s)^{\alpha-1}e^{-\nu(t-s)}
\sum^{\infty}_{k=1}{E}_{\alpha,\alpha}\left(-\lambda
^{\beta}_k(t-s)^{\alpha}\right)\varphi_k(x)\varphi_k(y)\\
&\relphantom{=}{}\times\sum_{l=1}^{\infty}(h(s,u(s,y))-h(s,u_n(s,y)))\cdot e_l(y)\varrho_l(s)dyd\xi_l^H(s),
\end{align*}
\begin{align*}
\displaystyle \mathscr{A}_7 &=\sum_{i=1}^{N_t}\int^{t_{i+1}}_{t_i}\int_{\mathcal{D}}
(t-s)^{\alpha-1}e^{-\nu(t-s)}\sum^{\infty}_{k=1}{E}_{\alpha,\alpha}\left(-\lambda^{\beta}_k
(t-s)^{\alpha}\right)\varphi_k(x)\varphi_k(y)\\
&\relphantom{==}{}\times\sum_{l=1}^{\infty}(h(s,u_n(s,y))\cdot e_l(y))\varrho_l(s)dyd\xi_l^H(s)\\
&\relphantom{=}{}-\sum_{i=1}^{N_t}\int^{t_{i+1}}_{t_i}\int_{\mathcal{D}}(t-t_i)
^{\alpha-1}e^{-\nu(t-t_i)}\sum^{\infty}_{k=1}{E}_{\alpha,\alpha}
\left(-\lambda^{\beta}_k(t-t_i)^{\alpha}\right)\varphi_k(x)\varphi_k(y)\\
&\relphantom{==}{}\times\sum_{l=1}^{\infty}(h(t_i,u_n(t_i,y))\cdot e_l(y))\varrho_l(t_i)dyd\xi^H_l(s),
\end{align*}
\begin{align*}
\displaystyle \mathscr{A}_8 &=\sum_{i=1}^{N_t}\int^{t_{i+1}}_{t_i}\int_{\mathcal{D}}
(t-t_i)^{\alpha-1}e^{-\nu(t-t_i)}\sum^{\infty}_{k=1}{E}_{\alpha,\alpha}
\left(-\lambda^{\beta}_k(t-t_i)^{\alpha}\right)\varphi_k(x)\varphi_k(y)\\
&\relphantom{==}{}\times\sum_{l=1}^{\infty}(h(t_i,u_n(t_i,y))\cdot e_l(y))
(\varrho_l(t_i)-\varrho_l^n(t_i))dyd\xi_l^H(s),
\end{align*}
and
\begin{align*}
\displaystyle \mathscr{A}_9 &=\sum_{i=1}^{N_t}\int^{t_{i+1}}_{t_i}\int_{\mathcal{D}}
(t-t_i)^{\alpha-1}e^{-\nu(t-t_i)}\sum^{\infty}_{k=1}{E}_{\alpha,\alpha}
\left(-\lambda^{\beta}_k(t-t_i)^{\alpha}\right)\varphi_k(x)\varphi_k(y)\\
&\relphantom{==}{}\times\sum_{l=1}^{\infty}
(h(t_i,u_n(t_i,y))\cdot e_l(y))\varrho_l^n(t_i)dyd\xi_l^H(s)\\
&\relphantom{=}{}-\sum_{i=1}^{N_t}\int^{t_{i+1}}_{t_i}\int_{\mathcal{D}}
(t-s)^{\alpha-1}e^{-\nu(t-s)}\sum^{\infty}_{k=1}{E}_{\alpha,\alpha}
\left(-\lambda^{\beta}_k(t-s)^{\alpha}\right)\varphi_k(x)\varphi_k(y)\\
&\relphantom{==}{}\times\sum_{l=1}^{\infty}
(h(s,u_n(s,y))\cdot e_l(y))\varrho_l^n(s)\frac{\xi_l^H(t_{i+1})-\xi_l^H(t_i)}{\tau}dyds.
\end{align*}
Thus
\begin{equation*}\label{eq2.46}
u(t,x)-u_n(t,x)=\mathscr{A}_1+\mathscr{A}_2+\mathscr{A}_3+\mathscr{A}_4+\mathscr{A}_5+\mathscr{A}_6+
\mathscr{A}_7+\mathscr{A}_8+\mathscr{A}_9.
\end{equation*}

For $\mathscr{A}_1$, since $\{\varphi_k\}_{j=1}^\infty$ is an orthonormal basis in $L^2(\mathcal{D})$, by the assumption on $f$ given in $(\bf{A}_1)$, Lemma \ref{le2.1} and H\"{o}lder's inequality, we have
\begin{align}\label{eq2.49}
&\displaystyle \mathbb{E}\|\mathscr{A}_1\|_{\mathbb{H}^{2\widetilde{\gamma}-\frac{2\beta}{\alpha}}}^2
\nonumber\\
&=\mathbb{E}\sum^{\infty}_{k=1}\lambda_k^{2\widetilde{\gamma}-\frac{2\beta}{\alpha}}
\bigg(\int^{t}_{0}\int_{\mathcal{D}}(t-s)^{\alpha-1}e^{-\nu(t-s)}
\sum^{\infty}_{l=1}{E}_{\alpha,\alpha}\left(-\lambda
^{\beta}_l(t-s)^{\alpha}\right)\varphi_l(\cdot)\varphi_l(y)\nonumber\\
&\relphantom{=}{}\times(f(s,u(s,y))-f(s,u_n(s,y)))dyds,\varphi_k\bigg)^2\nonumber\\
&=\mathbb{E}\sum^{\infty}_{k=1}\lambda_k^{2\widetilde{\gamma}-\frac{2\beta}{\alpha}}
\bigg|\int^{t}_{0}\int_{\mathcal{D}}(t-s)^{\alpha-1}e^{-\nu(t-s)}
{E}_{\alpha,\alpha}\left(-\lambda
^{\beta}_k(t-s)^{\alpha}\right)\varphi_k(y)\nonumber\\
&\relphantom{=}{}\times(f(s,u(s,y))-f(s,u_n(s,y)))dyds\bigg|^2\nonumber\\
&\leq CT\mathbb{E}\sum^{\infty}_{k=1}\int^{t}_{0}(t-s)^{2\alpha-2}e^{-2\nu(t-s)}
\frac{1}{(1+\lambda
^{\beta}_k(t-s)^{\alpha})^2}\nonumber\\
&\relphantom{=}{}\times\lambda_k^{2\widetilde{\gamma}-\frac{2\beta}{\alpha}}
(f(s,u(s,y))-f(s,u_n(s,y)),\varphi_k)^2ds\nonumber\\
&\leq C\int^{t}_{0}(t-s)^{2\alpha-2}e^{-2\nu(t-s)}\mathbb{E}\|u(s)-u_n(s)\|_{\mathbb{H}^{2\widetilde{\gamma}-\frac{2\beta}{\alpha}}}^2ds\nonumber\\
\end{align}
where we have used $\frac{1}{(1+\lambda
^{\beta}_k(t-s)^{\alpha})^2}\leq C$.

Since $\{\varphi_k\}_{j=1}^\infty$ is an orthonormal basis in $L^2(\mathcal{D})$ and $\{\xi_l\}_{l=1}^{\infty}$ is a family of mutually independent one-dimensional standard Brownian motions, then by \eqref{eq.2.8}, $(\bf{A}_1)$, $(\bf{A}_3)$ and the It\^{o} isometry, we have
\begin{flalign}\label{eq2.49c}
\displaystyle &\mathbb{E}\|\mathscr{A}_2\|_{\mathbb{H}^{2\widetilde{\gamma}-\frac{2\beta}{\alpha}}}^2\nonumber\\
&=\mathbb{E}\sum^{\infty}_{k=1}\lambda_k^{2\widetilde{\gamma}-\frac{2\beta}{\alpha}}
\bigg(\int^{t}_{0}\int_{\mathcal{D}}
(t-s)^{\alpha-1}e^{-\nu(t-s)}\sum^{\infty}_{j=1}{E}_{\alpha,\alpha}\left(-\lambda
^{\beta}_j(t-s)^{\alpha}\right)\varphi_j(\cdot)\varphi_j(y)\nonumber\\
&\relphantom{=}{}\times\sum_{l=1}^{\infty}(g(s,u(s,y))-g(s,u_n(s,y)))\cdot e_l(y)\varsigma_l(s)dyd\xi_l(s),\varphi_k\bigg)^2\nonumber\\
&=\mathbb{E}\sum^{\infty}_{k=1}\lambda_k^{2\widetilde{\gamma}-\frac{2\beta}{\alpha}}
\bigg{|}\int^{t}_{0}\int_{\mathcal{D}}
(t-s)^{\alpha-1}e^{-\nu(t-s)}{E}_{\alpha,\alpha}\left(-\lambda
^{\beta}_k(t-s)^{\alpha}\right)\varphi_k(y)\nonumber\\
&\relphantom{=}{}\times\sum_{j,l=1}^{\infty}((g(s,u(s))-g(s,u_n(s)))\cdot e_l,\varphi_j)\varphi_j(y)\varsigma_l(s)dyd\xi_l(s)\bigg{|}^2\nonumber&
\end{flalign}
\begin{flalign}
& =\mathbb{E}\sum^{\infty}_{k=1}\lambda_k^{2\widetilde{\gamma}-\frac{2\beta}{\alpha}}
\bigg{|}\int^{t}_{0}
(t-s)^{\alpha-1}e^{-\nu(t-s)}{E}_{\alpha,\alpha}\left(-\lambda
^{\beta}_k(t-s)^{\alpha}\right)\nonumber\\
&\relphantom{=}{}\times\sum_{l=1}^{\infty}((g(s,u(s))-g(s,u_n(s)))\cdot e_l,\varphi_k)\varsigma_l(s)d\xi_l(s)\bigg{|}^2\nonumber\\
&\leq C\mathbb{E}\sum^{\infty}_{k,l=1}
\int^{t}_{0}(t-s)^{2\alpha-2}e^{-2\nu(t-s)}
\frac{1}{(1+\lambda
^{\beta}_k(t-s)^{\alpha})^2}\nonumber\\
&\relphantom{=}{}\times\left|\lambda_k^{\widetilde{\gamma}-\frac{\beta}{\alpha}}
((g(s,u(s))-g(s,u_n(s)))\cdot e_l,\varphi_k)\varsigma_l(s)\right|^2ds\nonumber\\
&\leq C\int^{t}_{0}(t-s)^{2\alpha-2}e^{-2\nu(t-s)}
\mathbb{E}\|u(s)-u_n(s)\|_{\mathbb{H}^{2\widetilde{\gamma}-\frac{2\beta}{\alpha}}}^2ds, &
\end{flalign}
and
\begin{flalign}\label{eq2.47'}
&\displaystyle \mathbb{E}\|\mathscr{A}_3\|_{\mathbb{H}^{2\widetilde{\gamma}-\frac{2\beta}{\alpha}}}^2\nonumber\\ &=\mathbb{E}\sum^{\infty}_{k=1}\lambda_k^{2\widetilde{\gamma}-\frac{2\beta}{\alpha}}
\bigg{|}\sum_{i=1}^{N_t}\int^{t_{i+1}}_{t_i}
(t-s)^{\alpha-1}e^{-\nu(t-s)}{E}_{\alpha,\alpha}\left(-\lambda^{\beta}_k
(t-s)^{\alpha}\right)\sum_{l=1}^{\infty}(g(s,u_n(s))\cdot e_l,\varphi_k)\nonumber\\
&\relphantom{==}{}\times\varsigma_l(s)d\xi_l(s)-\sum_{i=1}^{N_t}\int^{t_{i+1}}_{t_i}(t-t_i)
^{\alpha-1}e^{-\nu(t-t_i)}{E}_{\alpha,\alpha}
\left(-\lambda^{\beta}_k(t-t_i)^{\alpha}\right)\nonumber\\
&\relphantom{==}{}\times
\sum_{l=1}^{\infty}(g(t_i,u_n(t_i))\cdot e_l,\varphi_k)\varsigma_l(t_i)d\xi_l(s)\bigg{|}^2\nonumber&
\end{flalign}
\begin{flalign}
&\leq C\mathbb{E}\sum^{\infty}_{k,l=1}\sum_{i=1}^{N_t}\int^{t_{i+1}}_{t_i}
\bigg((t-s)^{\alpha-1}
e^{-\nu(t-s)}{E}_{\alpha,\alpha}\left(-\lambda^{\beta}_k(t-s)^{\alpha}\right)\varsigma_l(s)-
(t-t_i)^{\alpha-1}e^{-\nu(t-t_i)}\nonumber\\
&\relphantom{==}{}\times{E}_{\alpha,\alpha}\left(-\lambda^{\beta}_k(t-t_i)^{\alpha}\right)
\varsigma_l(t_i)\bigg)^2\lambda_k^{2\widetilde{\gamma}-\frac{2\beta}{\alpha}}
(g(t_i,u_n(t_i))\cdot e_l,\varphi_k)^2ds\nonumber\\
&\relphantom{==}{}+C\mathbb{E}\sum^{\infty}_{k,l=1}\sum_{i=1}^{N_t}\int^{t_{i+1}}_{t_i}
(t-s)^{2\alpha-2}e^{-2\nu(t-s)}\left|{E}_{\alpha,\alpha}
\left(-\lambda^{\beta}_k(t-s)^{\alpha}\right)\right|^2|\varsigma_l(s)|^2\nonumber\\
&\relphantom{==}{}\times\lambda_k^{2\widetilde{\gamma}-\frac{2\beta}{\alpha}}
\left((g(s,u_n(s))-g(t_i,u_n(t_i)))\cdot e_l,\varphi_k\right)^2ds. &
\end{flalign}
Estimating the last term in \eqref{eq2.47'}, in view of $(\bf{A}_1)$, \eqref{eq.4.9} and \eqref{eq.4.29}, we obtain
\begin{flalign}\label{B.15}
&C\mathbb{E}\sum^{\infty}_{k,l=1}\sum_{i=1}^{N_t}\int^{t_{i+1}}_{t_i}
(t-s)^{2\alpha-2}e^{-2\nu(t-s)}\left|{E}_{\alpha,\alpha}
\left(-\lambda^{\beta}_k(t-s)^{\alpha}\right)\right|^2|\varsigma_l(s)|^2\nonumber\\
&\relphantom{==}{}\times\lambda_k^{2\widetilde{\gamma}-\frac{2\beta}{\alpha}}
\left((g(s,u_n(s))-g(t_i,u_n(t_i)))\cdot e_l,\varphi_k\right)^2ds\nonumber\\
&\leq C \mathbb{E}\sum^{\infty}_{k,l=1}\sum_{i=1}^{N_t}\int^{t_{i+1}}_{t_i}
(t-s)^{2\alpha-2}e^{-2\nu(t-s)}\frac{1}{(1+\lambda^{\beta}_k(t-s)^{\alpha})^2}
\lambda_k^{2\widetilde{\gamma}-\frac{2\beta}{\alpha}}\nonumber\\
&\relphantom{==}{}\times\left((g(s,u_n(s))-g(t_i,u_n(t_i)))\cdot e_l,\varphi_k\right)^2ds\nonumber&
\end{flalign}
\begin{flalign}
&\leq C\mathbb{E}\sum_{i=1}^{N_t}\int^{t_{i+1}}_{t_i}
(t-s)^{2\alpha-2}e^{-2\nu(t-s)}
\|(-\Delta)^{\widetilde{\gamma}-\frac{\beta}{\alpha}}(g(s,u_n(s))-g(t_i,u_n(t_i)))\|_{\mathcal{L}_2^0}^2ds\nonumber\\
&\leq C\sum_{i=1}^{N_t}\int^{t_{i+1}}_{t_i}
(t-s)^{2\alpha-2}e^{-2\nu(t-s)}\left(|s-t_i|^2+
\mathbb{E}\|u_n(s)-u_n(t_i)\|_{\mathbb{H}^{2\widetilde{\gamma}-\frac{2\beta}{\alpha}}}^2\right)ds\nonumber\\
&\leq C\tau^2. &
\end{flalign}
On the other hand, we observe that for $s\in[t_i,t_{i+1}]$, by \eqref{eq.4.9}, Lemma \ref{le2.1}, Lagrange's mean value theorem, and the equality (see \cite[eq.(1.82)]{Podlubny})
$$\frac{d}{d\tau}(t-\tau)^{\alpha-1}{E}_{\alpha,\alpha}
\left(-\lambda^{\beta}_k(t-\tau)^{\alpha}\right)=-(t-\tau)^{\alpha-2}
{E}_{\alpha,\alpha-1}
\left(-\lambda^{\beta}_k(t-\tau)^{\alpha}\right),$$
we have
\begin{align}\label{C.5}
&\Big|(t-s)^{\alpha-1}e^{-\nu(t-s)}E_{\alpha,\alpha}\left(-\lambda_k^\beta
(t-s)^\alpha\right)\varsigma_l(s)
\nonumber\\
&
\relphantom{=}{}-(t-t_i)^{\alpha-1}e^{-\nu(t-t_i)}E_{\alpha,\alpha}\left(-\lambda_k^\beta (t-t_i)^\alpha\right)\varsigma_l(t_i)\Big|\nonumber\\
&\leq e^{-\nu(t-t_i)}|\varsigma_l(t_i)|\left|(t-t_i)^{\alpha-1}E_{\alpha,\alpha}
\left(-\lambda_k^\beta (t-t_i)^\alpha\right)-(t-s)^{\alpha-1}E_{\alpha,\alpha}
\left(-\lambda_k^\beta (t-s)^\alpha\right)\right|\nonumber\\
&\relphantom{=}{}+(t-s)^{\alpha-1}\left|E_{\alpha,\alpha}
\left(-\lambda_k^\beta (t-s)^\alpha\right)\right||\varsigma_l(t_i)|
\left|e^{-\nu(t-t_i)}-e^{-\nu(t-s)}\right|\nonumber\\
&\relphantom{=}{}+(t-s)^{\alpha-1}e^{-\nu(t-s)}\left|E_{\alpha,\alpha}
\left(-\lambda_k^\beta (t-s)^\alpha\right)\right||\varsigma_l(t_i)-\varsigma_l(s)|\nonumber\\
&\leq \mu_1\int_{t_i}^s(t-r)^{\alpha-2}\left|E_{\alpha,\alpha-1}
\left(-\lambda_k^\beta (t-r)^\alpha\right)\right|dr\nonumber\\
&\relphantom{=}{}+C\mu_1\tau\frac{1}{1+\lambda_k^\beta(t-s)^\alpha}(t-s)^{\alpha-1}
+C\gamma_1\tau\frac{1}{1+\lambda_k^\beta(t-s)^\alpha}(t-s)^{\alpha-1}\nonumber\\
&\leq C\int_{t_i}^s(t-r)^{\alpha-2}\frac{1}{1+\lambda_k^\beta(t-r)^\alpha}dr+C \tau(t-s)^{\alpha-1}\nonumber\\
&\leq C\int_{t_i}^s(t-r)^{\alpha-2}dr+C \tau(t-s)^{\alpha-1}\leq C\tau(t-s)^{\alpha-2}+C \tau(t-s)^{\alpha-1}.
\end{align}
Thus we get from \eqref{eq2.47'}-\eqref{C.5} that
\begin{align}\label{C.6}
&\mathbb{E}\|\mathscr{A}_3\|_{\mathbb{H}^{2\widetilde{\gamma}-\frac{2\beta}{\alpha}}}^2\nonumber\\
&%\relphantom{=}{}
 \leq C\mathbb{E}\sum^{\infty}_{k,l=1}\sum_{i=1}^{N_t}\int^{t_{i+1}}_{t_i}
\tau^2\left((t-s)^{2\alpha-4}+(t-s)^{2\alpha-2}\right)\lambda_k^{2\widetilde{\gamma}-\frac{2\beta}{\alpha}}
\nonumber\\
&
\relphantom{=}{}\times (g(t_i,u_n(t_i))\cdot e_l,\varphi_k)^2ds+C\tau^2\nonumber\\
&%\relphantom{=}{}
\leq C\tau^2\mathbb{E}\sum_{i=1}^{N_t}\int^{t_{i+1}}_{t_i}
\left((t-s)^{2\alpha-4}+(t-s)^{2\alpha-2}\right)\|(-\Delta)^{\widetilde{\gamma}-\frac{\beta}{\alpha}}
g(t_i,u_n(t_i))\|_{\mathcal{L}_2^0}^2ds+C\tau^2\nonumber\\
&%\relphantom{=}{}
\leq C\tau^2\int^{t}_{0}\left((t-s)^{2\alpha-4}+(t-s)^{2\alpha-2}\right)
\left(1+\sup_{0\leq s\leq T}\mathbb{E}\|u_n(s)\|_{\mathbb{H}^{2\widetilde{\gamma}}}^2\right)ds+C\tau^2
\leq C\tau^2.
\end{align}
For $\mathscr{A}_4$, in a similar way as above, we get from \eqref{eq.4.9} and \eqref{eq.4.29} that
\begin{flalign}\label{B.21}
\displaystyle &\mathbb{E}\|\mathscr{A}_4\|_{\mathbb{H}^{2\widetilde{\gamma}-\frac{2\beta}{\alpha}}}^2\nonumber\\ &\relphantom{=}{}=\mathbb{E}\sum^{\infty}_{k=1}\lambda_k^{2\widetilde{\gamma}-\frac{2\beta}{\alpha}}
\bigg{|}\sum_{i=1}^{N_t}\int^{t_{i+1}}_{t_i}
(t-t_i)^{\alpha-1}e^{-\nu(t-t_i)}{E}_{\alpha,\alpha}
\left(-\lambda^{\beta}_k(t-t_i)^{\alpha}\right)\nonumber\\
&\relphantom{==}{}\times\sum_{l=1}^{\infty}(g(t_i,u_n(t_i))\cdot e_l,\varphi_k)
(\varsigma_l(t_i)-\varsigma_l^n(t_i))d\xi_l(s)\bigg{|}^2\nonumber\\
&\relphantom{=}{}\leq \mathbb{E}\sum^{\infty}_{k,l=1}
\sum_{i=1}^{N_t}\int^{t_{i+1}}_{t_i}
(t-t_i)^{2\alpha-2}e^{-2\nu(t-t_i)}\left|{E}_{\alpha,\alpha}
\left(-\lambda^{\beta}_k(t-t_i)^{\alpha}\right)\right|^2\nonumber\\
&\relphantom{==}{}\times|\varsigma_l(t_i)-\varsigma_l^n(t_i)|^2
\lambda_k^{2\widetilde{\gamma}-\frac{2\beta}{\alpha}}
(g(t_i,u_n(t_i))\cdot e_l,\varphi_k)^2ds\nonumber&
\end{flalign}
\begin{flalign}
&\relphantom{=}{}\leq C\mathbb{E}\sum^{\infty}_{k,l=1} \sum_{i=1}^{N_t}\int^{t_{i+1}}_{t_i}(t-t_i)^{2\alpha-2}e^{-2\nu(t-t_i)}
\frac{1}{(1+\lambda^{\beta}_k(t-t_i)^{\alpha})^2}(\eta_l^n)^2\nonumber\\
&\relphantom{==}{}\times \lambda_k^{2\widetilde{\gamma}-\frac{2\beta}{\alpha}}
(g(t_i,u_n(t_i))\cdot e_l,\varphi_k)^2ds\nonumber\nonumber \\
&\relphantom{=}{}\leq C\mathbb{E}\sum_{i=1}^{N_t}\int^{t_{i+1}}_{t_i}(t-t_i)^{2\alpha-2}e^{-2\nu(t-t_i)}
\sum^{\infty}_{l=1}(\eta_l^n)^2
\|(-\Delta)^{\widetilde{\gamma}-\frac{\beta}{\alpha}}g(t_i,u_n(t_i))\|^2_{\mathcal{L}_2^0}ds\nonumber\\
&\relphantom{=}{}\leq C\sum_{i=1}^{N_t}\int^{t_{i+1}}_{t_i}\sum^{\infty}_{l=1}(\eta_l^n)^2 \left(1+\mathbb{E}\|u_n(t_i)\|_{\mathbb{H}^{2\widetilde{\gamma}}}^2\right)ds
\leq C  \sum_{l=1}^{\infty}(\eta_l^n)^2. &
\end{flalign}
For $\mathscr{A}_5$, similar to the arguments in \eqref{eq2.47'}-\eqref{C.6} , we find that
\begin{flalign}\label{C.8}
\displaystyle &\mathbb{E}\|\mathscr{A}_5\|_{\mathbb{H}^{2\widetilde{\gamma}-\frac{2\beta}{\alpha}}}^2 \nonumber\\ &=\mathbb{E}\sum^{\infty}_{k=1}\lambda_k^{2\widetilde{\gamma}-\frac{2\beta}{\alpha}}
\bigg(\sum_{i=1}^{N_t}\frac{1}{\tau}\int^{t_{i+1}}_{t_i}\int^{t_{i+1}}_{t_i}\int_{\mathcal{D}}
(t-t_i)^{\alpha-1}e^{-\nu(t-t_i)}\sum^{\infty}_{j=1}{E}_{\alpha,\alpha}
\left(-\lambda^{\beta}_j(t-t_i)^{\alpha}\right)\nonumber\\
&\relphantom{=}{}\times\varphi_j(\cdot)\varphi_j(y)\sum_{l=1}^{\infty}
(g(t_i,u_n(t_i,y))\cdot e_l(y))\varsigma_l^n(t_i)dydsd\xi_l(r)\nonumber\\
&\relphantom{=}{}-\sum_{i=1}^{N_t}\frac{1}{\tau}\int^{t_{i+1}}_{t_i}\int^{t_{i+1}}_{t_i}\int_{\mathcal{D}}
(t-s)^{\alpha-1}e^{-\nu(t-s)}\sum^{\infty}_{j=1}{E}_{\alpha,\alpha}
\left(-\lambda^{\beta}_j(t-s)^{\alpha}\right)\varphi_j(\cdot)\varphi_j(y)\nonumber\\
&\relphantom{=}{}\times\sum_{l=1}^{\infty}(g(s,u_n(s,y))\cdot e_l(y))\varsigma_l^n(s)dydsd\xi_l(r),\varphi_k\bigg)^2\nonumber&
\end{flalign}
\begin{flalign}
&=\mathbb{E}\sum^{\infty}_{k=1}\lambda_k^{2\widetilde{\gamma}-\frac{2\beta}{\alpha}}
\bigg{|}\sum_{i=1}^{N_t}
\frac{1}{\tau}\int^{t_{i+1}}_{t_i}\int^{t_{i+1}}_{t_i}
(t-t_i)^{\alpha-1}e^{-\nu(t-t_i)}{E}_{\alpha,\alpha}
\left(-\lambda^{\beta}_k(t-t_i)^{\alpha}\right)\nonumber\\
&\relphantom{=}{}\times\sum^{\infty}_{l=1}(g(t_i,u_n(t_i))\cdot e_l,\varphi_k)\varsigma_l^n(t_i)dsd\xi_l(r)-\sum_{i=1}^{N_t}
\frac{1}{\tau}\int^{t_{i+1}}_{t_i}\int^{t_{i+1}}_{t_i}
(t-s)^{\alpha-1}e^{-\nu(t-s)}\nonumber\\
&\relphantom{=}{}\times{E}_{\alpha,\alpha}
\left(-\lambda^{\beta}_k(t-s)^{\alpha}\right)\sum^{\infty}_{l=1}(g(s,u_n(s))\cdot e_l,\varphi_k)\varsigma_l^n(s)dsd\xi_l(r)\bigg{|}^2\nonumber&
\end{flalign}
\begin{flalign}
&\leq C\mathbb{E} \sum^{\infty}_{k,l=1}\sum_{i=1}^{N_t}\int^{t_{i+1}}_{t_i}
\bigg((t-t_i)^{\alpha-1}e^{-\nu(t-t_i)}{E}_{\alpha,\alpha}
\left(-\lambda^{\beta}_k(t-t_i)^{\alpha}\right)\varsigma_l^n(t_i)\nonumber\\
&\relphantom{=}{}-(t-s)^{\alpha-1}e^{-\nu(t-s)}{E}_{\alpha,\alpha}
\left(-\lambda^{\beta}_k(t-s)^{\alpha}\right)\varsigma_l^n(s)\bigg)^2
\lambda_k^{2\widetilde{\gamma}-\frac{2\beta}{\alpha}}
(g(t_i,u_n(t_i))\cdot e_l,\varphi_k)^2ds\nonumber\\
&\relphantom{=}{}+C\mathbb{E} \sum^{\infty}_{k,l=1}\sum_{i=1}^{N_t}\int^{t_{i+1}}_{t_i}
(t-s)^{2\alpha-2}e^{-2\nu(t-s)}\left|{E}_{\alpha,\alpha}
\left(-\lambda^{\beta}_k(t-s)^{\alpha}\right)\right|^2|\varsigma_l^n(s)|^2\nonumber\\
&\relphantom{=}{}\times
\lambda_k^{2\widetilde{\gamma}-\frac{2\beta}{\alpha}}
((g(t_i,u_n(t_i))-g(s,u_n(s)))\cdot e_l,\varphi_k)^2ds\nonumber&
\end{flalign}
\begin{flalign}
&\leq C\mathbb{E}
\sum^{\infty}_{k,l=1}\sum_{i=1}^{N_t}\int^{t_{i+1}}_{t_i}\tau^2\left((t-s)^{2\alpha-4}+(t-s)^{2\alpha-2}\right)
\lambda_k^{2\widetilde{\gamma}-\frac{2\beta}{\alpha}}
(g(t_i,u_n(t_i))\cdot e_l,\varphi_k)^2ds\nonumber\\
&\relphantom{=}{}+C\mathbb{E}
\sum^{\infty}_{k,l=1}\sum_{i=1}^{N_t}\int^{t_{i+1}}_{t_i}
(t-s)^{2\alpha-2}e^{-2\nu(t-s)}\lambda_k^{2\widetilde{\gamma}-\frac{2\beta}{\alpha}}
\nonumber \\
&\relphantom{=}{}
\times ((g(t_i,u_n(t_i))-g(s,u_n(s)))\cdot e_l,\varphi_k)^2ds\nonumber &
\end{flalign}
\begin{flalign}
&\leq C\tau^2\mathbb{E}\sum_{i=1}^{N_t}\int^{t_{i+1}}_{t_i}\left((t-s)^{2\alpha-4}+(t-s)^{2\alpha-2}\right)
\|(-\Delta)^{\widetilde{\gamma}-\frac{\beta}{\alpha}}g(t_i,u_n(t_i))\|_{\mathcal{L}_2^0}^2ds\nonumber\\
&\relphantom{=}{}
+C\mathbb{E}\sum_{i=1}^{N_t}\int^{t_{i+1}}_{t_i}(t-s)^{2\alpha-2}
\|(-\Delta)^{\widetilde{\gamma}-\frac{\beta}{\alpha}}(g(t_i,u_n(t_i))-g(s,u_n(s)))\|_{\mathcal{L}_2^0}^2ds\nonumber &
\end{flalign}
\begin{flalign}
&\leq C\tau^2\int^{t}_{0}\left((t-s)^{2\alpha-4}+(t-s)^{2\alpha-2}\right)
\left(1+\sup_{0\leq s\leq T}\|u_n(s)\|_{\mathbb{H}^{2\widetilde{\gamma}}}^2\right)ds\nonumber\\
&\relphantom{=}{}
+C\sum_{i=1}^{N_t}\int^{t_{i+1}}_{t_i}(t-s)^{2\alpha-2}
\left(|s-t_i|^2+
\mathbb{E}\|u_n(t_i)-u_n(s)\|_{\mathbb{H}^{2\widetilde{\gamma}-\frac{2\beta}{\alpha}}}^2\right)ds\nonumber\\
&\leq C\tau^2. &
\end{flalign}
Since $\{\varphi_k\}_{j=1}^\infty$ is an orthonormal basis in $L^2(\mathcal{D})$ and $\{\xi_l^H\}_{l=1}^{\infty}$ is a family of mutually independent one-dimensional fractional Brownian motions, in view of \eqref{eq.2.9}, $(\bf{A}_1)$, $(\bf{A}_3)$ and Lemma \ref{le2.10}, we deduce that
\begin{flalign}\label{C.9}
&\mathbb{E}\|\mathscr{A}_6\|_{\mathbb{H}^{2\widetilde{\gamma}-\frac{2\beta}{\alpha}}}^2 \nonumber\\
&=\mathbb{E}\sum^{\infty}_{k=1}\lambda_k^{2\widetilde{\gamma}-\frac{2\beta}{\alpha}} \bigg(\int^{t}_{0}\int_{\mathcal{D}}(t-s)^{\alpha-1}e^{-\nu(t-s)}
\sum^{\infty}_{j=1}{E}_{\alpha,\alpha}\left(-\lambda^{\beta}_j(t-s)^{\alpha}\right)
\varphi_j(\cdot)\varphi_j(y)\nonumber\\
&\relphantom{=}{}\times\sum_{l=1}^{\infty}(h(s,u(s,y))-h(s,u_n(s,y)))\cdot e_l(y)\varrho_l(s)dyd\xi_l^H(s),\varphi_k\bigg)^2\nonumber\\
&=\mathbb{E}\sum^{\infty}_{k=1}\lambda_k^{2\widetilde{\gamma}-\frac{2\beta}{\alpha}} \bigg{|}\int^{t}_{0}\int_{\mathcal{D}}(t-s)^{\alpha-1}e^{-\nu(t-s)}
{E}_{\alpha,\alpha}\left(-\lambda^{\beta}_k(t-s)^{\alpha}\right)
\varphi_k(y)\nonumber\\
&\relphantom{=}{}\times\sum_{j,l=1}^{\infty}((h(s,u(s))-h(s,u_n(s)))\cdot e_l,\varphi_j)\varphi_j(y)\varrho_l(s)dyd\xi_l^H(s)\bigg{|}^2\nonumber &
\end{flalign}
\begin{flalign}
&=\mathbb{E}\sum^{\infty}_{k=1}\lambda_k^{2\widetilde{\gamma}-\frac{2\beta}{\alpha}} \bigg{|}\int^{t}_{0}(t-s)^{\alpha-1}e^{-\nu(t-s)}
{E}_{\alpha,\alpha}\left(-\lambda^{\beta}_k(t-s)^{\alpha}\right)
\nonumber\\
&\relphantom{=}{}\times\sum_{l=1}^{\infty}((h(s,u(s))-h(s,u_n(s)))\cdot e_l,\varphi_k)\varrho_l(s)d\xi_l^H(s)\bigg{|}^2\nonumber\\
&\leq 2HT^{2H-1}\mathbb{E}\sum_{k,l=1}^{\infty}\int^{t}_{0}
(t-s)^{2\alpha-2}e^{-2\nu(t-s)}\frac{1}{(1+\lambda^{\beta}_k(t-s)^{\alpha})^2}\nonumber\\
&\relphantom{=}{}\times\left|\lambda_k^{\widetilde{\gamma}-\frac{\beta}{\alpha}} ((h(s,u(s))-h(s,u_n(s)))\cdot e_l,\varphi_k)\varrho_l(s)\right|^2ds\nonumber\\
&\leq C\int^{t}_{0}(t-s)^{2\alpha-2}e^{-2\nu(t-s)}
\mathbb{E}\|u(s)-u_n(s)\|_{\mathbb{H}^{2\widetilde{\gamma}-\frac{2\beta}{\alpha}}}^2ds. &
\end{flalign}
By $(\bf{A}_1)$, \eqref{eq.4.9}, \eqref{eq.4.29} and the similar argument as in \eqref{C.5}, we obtain
\begin{flalign}\label{C.10}
&\mathbb{E}\|\mathscr{A}_7\|_{\mathbb{H}^{2\widetilde{\gamma}-\frac{2\beta}{\alpha}}}^2\nonumber\\ 
 & =\mathbb{E}\sum^{\infty}_{k=1}\lambda_k^{2\widetilde{\gamma}-\frac{2\beta}{\alpha}} \bigg{|}\sum_{i=1}^{N_t}\int^{t_{i+1}}_{t_i}
(t-s)^{\alpha-1}e^{-\nu(t-s)}{E}_{\alpha,\alpha}\left(-\lambda^{\beta}_k
(t-s)^{\alpha}\right)\nonumber\\
&\relphantom{=}{}\times\sum_{l=1}^{\infty}(h(s,u_n(s))\cdot e_l,\varphi_k)\varrho_l(s)d\xi_l^H(s)
-\sum_{i=1}^{N_t}\int^{t_{i+1}}_{t_i}(t-t_i)
^{\alpha-1}e^{-\nu(t-t_i)}\nonumber\\
&\relphantom{=}{}\times{E}_{\alpha,\alpha}
\left(-\lambda^{\beta}_k(t-t_i)^{\alpha}\right)\sum_{l=1}^{\infty}(h(t_i,u_n(t_i))\cdot e_l,\varphi_k)\varrho_l(t_i)d\xi^H_l(s)\bigg{|}^2\nonumber &
\end{flalign}
\begin{flalign}
&= \mathbb{E}\sum^{\infty}_{k=1}\lambda_k^{2\widetilde{\gamma}-\frac{2\beta}{\alpha}} \bigg{|}\int^{t}_{0}\sum_{i=1}^{N_t}\chi_i(s)\bigg(
(t-s)^{\alpha-1}e^{-\nu(t-s)}{E}_{\alpha,\alpha}\left(-\lambda^{\beta}_k
(t-s)^{\alpha}\right)\nonumber\\
&\relphantom{=}{}\times\sum_{l=1}^{\infty}(h(s,u_n(s))\cdot e_l,\varphi_k)\varrho_l(s)
-(t-t_i)^{\alpha-1}e^{-\nu(t-t_i)}{E}_{\alpha,\alpha}
\left(-\lambda^{\beta}_k(t-t_i)^{\alpha}\right)
\nonumber\\
&\relphantom{=}{}\times
\sum_{l=1}^{\infty}(h(t_i,u_n(t_i))\cdot e_l,\varphi_k)\varrho_l(t_i)\bigg)d\xi^H_l(s)\bigg{|}^2\nonumber&
\end{flalign}
\begin{flalign}
& \leq CHT^{2H-1} \mathbb{E}\sum^{\infty}_{k,l=1}\lambda_k^{2\widetilde{\gamma}-\frac{2\beta}{\alpha}} \int^{t}_{0}\bigg{|}\sum_{i=1}^{N_t}\chi_i(s)\bigg(
(t-s)^{\alpha-1}e^{-\nu(t-s)}{E}_{\alpha,\alpha}\left(-\lambda^{\beta}_k
(t-s)^{\alpha}\right)\nonumber\\
&\relphantom{=}{}\times(h(s,u_n(s))\cdot e_l,\varphi_k)
\varrho_l(s)
-(t-t_i)^{\alpha-1}e^{-\nu(t-t_i)}{E}_{\alpha,\alpha}
\left(-\lambda^{\beta}_k(t-t_i)^{\alpha}\right)\nonumber\\
&\relphantom{=}{}\times(h(t_i,u_n(t_i))\cdot e_l,\varphi_k)\varrho_l(t_i)\bigg)\bigg{|}^2 ds\nonumber&
\end{flalign}
\begin{flalign}
&\leq CT^{2H-1} \mathbb{E}\sum^{\infty}_{k,l=1}\lambda_k^{2\widetilde{\gamma}-\frac{2\beta}{\alpha}} \int^{t}_{0}\sum_{i=1}^{N_t}\chi_i(s)^2\bigg(
(t-s)^{\alpha-1}e^{-\nu(t-s)}{E}_{\alpha,\alpha}\left(-\lambda^{\beta}_k
(t-s)^{\alpha}\right)\nonumber\\
&\relphantom{=}{}\times(h(s,u_n(s))\cdot e_l,\varphi_k)
\varrho_l(s)
-(t-t_i)^{\alpha-1}e^{-\nu(t-t_i)}{E}_{\alpha,\alpha}
\left(-\lambda^{\beta}_k(t-t_i)^{\alpha}\right)\nonumber\\
&\relphantom{=}{}\times(h(t_i,u_n(t_i))\cdot e_l,\varphi_k)\varrho_l(t_i)\bigg)^2 ds\nonumber&
\end{flalign}
\begin{flalign}
&= CT^{2H-1} \mathbb{E}\sum^{\infty}_{k,l=1}\lambda_k^{2\widetilde{\gamma}-\frac{2\beta}{\alpha}} \sum_{i=1}^{N_t}\int^{t_{i+1}}_{t_i}\bigg(
(t-s)^{\alpha-1}e^{-\nu(t-s)}{E}_{\alpha,\alpha}\left(-\lambda^{\beta}_k
(t-s)^{\alpha}\right)\nonumber\\
&\relphantom{=}{}\times(h(s,u_n(s))\cdot e_l,\varphi_k)
\varrho_l(s)-(t-t_i)^{\alpha-1}e^{-\nu(t-t_i)}{E}_{\alpha,\alpha}
\left(-\lambda^{\beta}_k(t-t_i)^{\alpha}\right)\nonumber\\
&\relphantom{=}{}\times(h(t_i,u_n(t_i))\cdot e_l,\varphi_k)\varrho_l(t_i)\bigg)^2 ds\nonumber&
\end{flalign}
\begin{flalign}
&\leq CT^{2H-1}\mathbb{E}
\sum_{k,l=1}^{\infty}\sum_{i=1}^{N_t}\int_{t_i}^{t_{i+1}}
\bigg((t-t_i)^{\alpha-1}e^{-\nu(t-t_i)}
{E}_{\alpha,\alpha}\left(-\lambda^{\beta}_k(t-t_i)^{\alpha}\right)
\varrho_l(t_i)\nonumber\\
&\relphantom{=}{}-(t-s)^{\alpha-1}e^{-\nu(t-s)}
E_{\alpha,\alpha}\left(-\lambda^{\beta}_k(t-s)^{\alpha}\right)\varrho_l(s)\bigg)^2
\lambda_k^{2\widetilde{\gamma}-\frac{2\beta}{\alpha}} (h(t_i,u_n(t_i))\cdot e_l,\varphi_k)^2ds\nonumber\\
&\relphantom{=}{}+CT^{2H-1}\mathbb{E}\sum_{k,l=1}^{\infty}\sum_{i=1}^{N_t}\int_{t_i}^{t_{i+1}}
(t-s)^{2\alpha-2}e^{-2\nu(t-s)}\left|E_{\alpha,\alpha}\left(-\lambda^{\beta}_k(t-s)^{\alpha}\right)\right|^2
|\varrho_l(s)|^2\nonumber\\
&\relphantom{=}{}\times \lambda_k^{2\widetilde{\gamma}-\frac{2\beta}{\alpha}} ((h(t_i,u_n(t_i))-h(s,u_n(s)))\cdot e_l,\varphi_k)^2ds\nonumber&
\end{flalign}
\begin{flalign}
&\leq CT^{2H-1}\mathbb{E}
\sum_{k,l=1}^{\infty}\sum_{i=1}^{N_t}\int_{t_i}^{t_{i+1}}
\tau^2\left((t-s)^{2\alpha-4}+(t-s)^{2\alpha-2}\right)
\lambda_k^{2\widetilde{\gamma}-\frac{2\beta}{\alpha}} \nonumber\\
&\relphantom{=}{} \times (h(t_i,u_n(t_i))\cdot e_l,\varphi_k)^2ds+CT^{2H-1}\mathbb{E}
\sum_{k,l=1}^{\infty}\sum_{i=1}^{N_t}\int_{t_i}^{t_{i+1}}
(t-s)^{2\alpha-2}e^{-2\nu(t-s)}\nonumber\\
&\relphantom{=}{}\times \frac{1}{(1+\lambda^{\beta}_k(t-s)^{\alpha})^2}
\lambda_k^{2\widetilde{\gamma}-\frac{2\beta}{\alpha}} ((h(t_i,u_n(t_i))-h(s,u_n(s)))\cdot e_l,\varphi_k)^2ds\nonumber&
\end{flalign}
\begin{flalign}
&\leq CT^{2H-1}\tau^2
\int_{0}^{t}(t-s)^{2\alpha-4}
\mathbb{E}\|(-\Delta)^{\widetilde{\gamma}-\frac{\beta}{\alpha}}h(t_i,u_n(t_i))\|_{\mathcal{L}_2^0}^2ds\nonumber\\
&\relphantom{=}{}+CT^{2H-1}\mathbb{E}
\sum_{i=1}^{N_t}\int_{t_i}^{t_{i+1}}
(t-s)^{2\alpha-2}e^{-2\nu(t-s)}
\nonumber\\
&\relphantom{=}{}
\times \|(-\Delta)^{\widetilde{\gamma}-\frac{\beta}{\alpha}}(h(t_i,u_n(t_i))-h(s,u_n(s)))\|_{\mathcal{L}_2^0}^2ds\nonumber&
\end{flalign}
\begin{flalign}
&\leq CT^{2H-1}\tau^2
\int_{0}^{t}(t-s)^{2\alpha-4}
\left(1+\sup_{0\leq s\leq T}\mathbb{E}\|u_n(s)\|_{\mathbb{H}^{2\widetilde{\gamma}}}^2\right)ds\nonumber\\
&\relphantom{=}{}+CT^{2H-1}
\sum_{i=1}^{N_t}\int_{t_i}^{t_{i+1}}
(t-s)^{2\alpha-2}e^{-2\nu(t-s)} \nonumber
\\
&\relphantom{=}{}
\times
\left(|s-t_i|^2+
\mathbb{E}\|u_n(s)-u_n(t_i)\|_{\mathbb{H}^{2\widetilde{\gamma}-\frac{2\beta}{\alpha}}}^2\right)ds\nonumber\\
&\leq C\tau^2. &
\end{flalign}
For $\mathscr{A}_8$ and $\mathscr{A}_9$, arguing as in \eqref{C.10}, we have
\begin{flalign}\label{B.40}
\displaystyle &\mathbb{E}\|\mathscr{A}_8\|_{\mathbb{H}^{2\widetilde{\gamma}-\frac{2\beta}{\alpha}}}^2\nonumber\\  &%\relphantom{=}{} 
= \mathbb{E}\sum^{\infty}_{k=1}\lambda_k^{2\widetilde{\gamma}-\frac{2\beta}{\alpha}}  \bigg{|}\sum_{i=1}^{N_t}\int^{t_{i+1}}_{t_i}
(t-t_i)^{\alpha-1}e^{-\nu(t-t_i)}{E}_{\alpha,\alpha}
\left(-\lambda^{\beta}_k(t-t_i)^{\alpha}\right)\nonumber\\
&\relphantom{=}{}\times \sum_{l=1}^{\infty}(h(t_i,u_n(t_i,y))\cdot e_l,\varphi_k)
(\varrho_l(t_i)-\varrho_l^n(t_i))d\xi_l^H(s)\bigg{|}^2\nonumber&
\end{flalign}
\begin{flalign}
& \leq CHT^{2H-1}\mathbb{E}\sum^{\infty}_{k,l=1}
\sum_{i=1}^{N_t}\int_{t_i}^{t_{i+1}}(t-t_i)^{2\alpha-2}e^{-2\nu(t-t_i)}\left|{E}_{\alpha,\alpha}
\left(-\lambda^{\beta}_k(t-t_i)^{\alpha}\right)\right|^2\nonumber\\
&\relphantom{=}{}\times|\varrho_l(t_i)-\varrho_l^n(t_i)|^2\lambda_k^{2\widetilde{\gamma}-\frac{2\beta}{\alpha}}  (h(t_i,u_n(t_i,y))\cdot e_l,\varphi_k)^2ds\nonumber\\
&%\relphantom{=}{}
\leq CT^{2H-1}\mathbb{E}\sum^{\infty}_{k,l=1}
\sum_{i=1}^{N_t}\int_{t_i}^{t_{i+1}}(t-t_i)^{2\alpha-2}e^{-2\nu(t-t_i)}
\frac{1}{(1+\lambda^{\beta}_k(t-t_i)^{\alpha})^2}(\widetilde{\eta}_l^n)^2\nonumber\\
&\relphantom{=}{}\times
\lambda_k^{2\widetilde{\gamma}-\frac{2\beta}{\alpha}}  (h(t_i,u_n(t_i,y))\cdot e_l,\varphi_k)^2ds\nonumber&
\end{flalign}
\begin{flalign}
&%\relphantom{=}{}
\leq CT^{2H-1}\mathbb{E} \sum_{i=1}^{N_t}\int_{t_i}^{t_{i+1}}(t-t_i)^{2\alpha-2}e^{-2\nu(t-t_i)}
\sum_{l=1}^{\infty}(\widetilde{\eta}_l^n)^2\|(-\Delta)^{\widetilde{\gamma}-\frac{\beta}{\alpha}}h(t_i,u_n(t_i)) \|_{\mathcal{L}_2^0}^2ds\nonumber\\
&%\relphantom{=}{}
\leq C\sum_{i=1}^{N_t}\int_{t_i}^{t_{i+1}}
\sum_{l=1}^{\infty}(\widetilde{\eta}_l^n)^2
\left(1+\mathbb{E}\|u_n(t_i)\|_{\mathbb{H}^{2\widetilde{\gamma}}}^2\right)
ds\leq C\sum_{l=1}^{\infty}(\widetilde{\eta}_l^n)^2. &
\end{flalign}
and
\begin{flalign}
&\mathbb{E}\|\mathscr{A}_9\|_{\mathbb{H}^{2\widetilde{\gamma}-\frac{2\beta}{\alpha}}}^2\nonumber\\   & = \mathbb{E}\sum^{\infty}_{k=1}\lambda_k^{2\widetilde{\gamma}-\frac{2\beta}{\alpha}}  \bigg(\sum_{i=1}^{N_t}\frac{1}{\tau}\int^{t_{i+1}}_{t_i}\int^{t_{i+1}}_{t_i}\int_{\mathcal{D}}
(t-t_i)^{\alpha-1}e^{-\nu(t-t_i)}\sum^{\infty}_{j=1}{E}_{\alpha,\alpha}
\left(-\lambda^{\beta}_j(t-t_i)^{\alpha}\right)\nonumber\\
&\relphantom{=}{}\times\varphi_j(\cdot)\varphi_j(y)\sum_{l=1}^{\infty}
(h(t_i,u_n(t_i,y))\cdot e_l(y))\varrho_l^n(t_i)dydsd\xi_l^H(r)\nonumber\\
&\relphantom{=}{}
-\sum_{i=1}^{N_t}\frac{1}{\tau}\int^{t_{i+1}}_{t_i}
\int^{t_{i+1}}_{t_i}\int_{\mathcal{D}}(t-s)^{\alpha-1}e^{-\nu(t-s)}\sum^{\infty}_{j=1}
{E}_{\alpha,\alpha}\left(-\lambda^{\beta}_j(t-s)^{\alpha}\right)\varphi_j(\cdot)\varphi_j(y)\nonumber\\
&\relphantom{=}{}\times\sum_{l=1}^{\infty}
(h(s,u_n(s,y))\cdot e_l(y))\varrho_l^n(s)dydsd\xi_l^H(r),\varphi_k\bigg)^2\nonumber&
\end{flalign}
\begin{flalign}
& = \mathbb{E}\sum^{\infty}_{k=1}\lambda_k^{2\widetilde{\gamma}-\frac{2\beta}{\alpha}}
\bigg{|}\sum_{i=1}^{N_t}\frac{1}{\tau}\int^{t_{i+1}}_{t_i}\int^{t_{i+1}}_{t_i}
(t-t_i)^{\alpha-1}e^{-\nu(t-t_i)}{E}_{\alpha,\alpha}
\left(-\lambda^{\beta}_k(t-t_i)^{\alpha}\right)\nonumber\\
&\relphantom{=}{}\times\sum_{l=1}^{\infty}
(h(t_i,u_n(t_i))\cdot e_l,\varphi_k)\varrho_l^n(t_i)dsd\xi_l^H(r)-\sum_{i=1}^{N_t}
\frac{1}{\tau}\int^{t_{i+1}}_{t_i}\int^{t_{i+1}}_{t_i}
(t-s)^{\alpha-1}e^{-\nu(t-s)}\nonumber\\
&\relphantom{=}{}\times{E}_{\alpha,\alpha}
\left(-\lambda^{\beta}_k(t-s)^{\alpha}\right)\sum_{l=1}^{\infty}
(h(s,u_n(s))\cdot e_l,\varphi_k)\varrho_l^n(s)dsd\xi_l^H(r)\bigg{|}^2\nonumber&
\end{flalign}
\begin{flalign}
&\leq CHT^{2H-1}\mathbb{E}\sum^{\infty}_{k,l=1}
\sum_{i=1}^{N_t}\int_{t_i}^{t_{i+1}}\Big((t-t_i)^{\alpha-1}e^{-\nu(t-t_i)}{E}_{\alpha,\alpha}
\left(-\lambda^{\beta}_k(t-t_i)^{\alpha}\right)\varrho_l^n(t_i)\nonumber\\
&\relphantom{=}{}-(t-s)^{\alpha-1}e^{-\nu(t-s)}{E}_{\alpha,\alpha}
\left(-\lambda^{\beta}_k(t-s)^{\alpha}\right)\varrho_l^n(s)\Big)^2
\lambda_k^{2\widetilde{\gamma}-\frac{2\beta}{\alpha}}
(h(t_i,u_n(t_i))\cdot e_l,\varphi_k)^2ds\nonumber\\
&\relphantom{=}{}+CHT^{2H-1}\mathbb{E}\sum^{\infty}_{k,l=1}
\sum_{i=1}^{N_t}\int_{t_i}^{t_{i+1}}(t-s)^{2\alpha-2}e^{-2\nu(t-s)}\left|{E}_{\alpha,\alpha}
\left(-\lambda^{\beta}_k(t-s)^{\alpha}\right)\right|^2|\varrho_l^n(s)|^2\nonumber\\
&\relphantom{=}{}\times
\lambda_k^{2\widetilde{\gamma}-\frac{2\beta}{\alpha}}
((h(t_i,u_n(t_i))-h(s,u_n(s)))\cdot e_l,\varphi_k)^2ds\nonumber&
\end{flalign}
\begin{flalign}
&\leq CT^{2H-1}\mathbb{E}\sum^{\infty}_{k,l=1}
\sum_{i=1}^{N_t}\int_{t_i}^{t_{i+1}}\tau^2\Big((t-s)^{2\alpha-4}+(t-s)^{2\alpha-2}\Big)
\lambda_k^{2\widetilde{\gamma}-\frac{2\beta}{\alpha}}
\nonumber\\
&\relphantom{=}{} \times (h(t_i,u_n(t_i))\cdot e_l,\varphi_k)^2ds+CT^{2H-1}\mathbb{E}\sum^{\infty}_{k,l=1}
\sum_{i=1}^{N_t}\int_{t_i}^{t_{i+1}}(t-s)^{2\alpha-2}e^{-2\nu(t-s)}
\nonumber\\
&\relphantom{=}{}\times \lambda_k^{2\widetilde{\gamma}-\frac{2\beta}{\alpha}}
((h(t_i,u_n(t_i))-h(s,u_n(s)))\cdot e_l,\varphi_k)^2ds\nonumber&
\end{flalign}
\begin{flalign}
&\leq C\tau^2\mathbb{E}
\sum_{i=1}^{N_t}\int_{t_i}^{t_{i+1}}\Big((t-s)^{2\alpha-4}+(t-s)^{2\alpha-2}\Big)
\|(-\Delta)^{\widetilde{\gamma}-\frac{\beta}{\alpha}}h(t_i,u_n(t_i))\|_{\mathcal{L}_2^0}^2ds\nonumber\\
&\relphantom{=}{}+CT^{2H-1}\mathbb{E}
\sum_{i=1}^{N_t}\int_{t_i}^{t_{i+1}}(t-s)^{2\alpha-2}e^{-2\nu(t-s)}
\nonumber\\
&\relphantom{=}{}
\times\|(-\Delta)^{\widetilde{\gamma}-\frac{\beta}{\alpha}}(h(t_i,u_n(t_i))-h(s,u_n(s)))\|_{\mathcal{L}_2^0}^2ds\nonumber&
\end{flalign}
\begin{flalign}
&\leq C\tau^2
\sum_{i=1}^{N_t}\int_{t_i}^{t_{i+1}}\Big((t-s)^{2\alpha-4}+(t-s)^{2\alpha-2}\Big)
\left(1+\sup_{0\leq s\leq T}\mathbb{E}\|u_n(s)\|_{\mathbb{H}^{2\widetilde{\gamma}}}^2\right)ds\nonumber\\
&\relphantom{=}{}+C
\sum_{i=1}^{N_t}\int_{t_i}^{t_{i+1}}(t-s)^{2\alpha-2}
\left(|s-t_i|^2+\mathbb{E}\|u_n(t_i)-u_n(s)\|_{\mathbb{H}^{2\widetilde{\gamma}-\frac{2\beta}{\alpha}}}^2\right)ds\nonumber\\
&\leq C\tau^2. &
\end{flalign}

Collecting the above estimates, we find that
\begin{align*}
e^{2\nu t}\mathbb{E}\|u(t)-u_n(t)\|_{\mathbb{H}^{2\widetilde{\gamma}-\frac{2\beta}{\alpha}}}^2
&\leq Ce^{2\nu t}\tau^2+Ce^{2\nu t}\sum_{l=1}^{\infty}(\eta_l^n)^2+Ce^{2\nu t}\sum_{l=1}^{\infty}(\widetilde{\eta}_l^n)^2\\
&\relphantom{=}{}+C
\int^{t}_{0}(t-s)^{2\alpha-2}e^{2\nu s}\mathbb{E}\|u(s)-u_n(s)\|_{\mathbb{H}^{2\widetilde{\gamma}-\frac{2\beta}{\alpha}}}^2ds.
\end{align*}
Thus Lemma \ref{lem2.11} leads to
\begin{align}\label{B.42}
\mathbb{E}\|u(t)-u_n(t)\|_{\mathbb{H}^{2\widetilde{\gamma}-\frac{2\beta}{\alpha}}}^2
\leq C\tau^2+C\sum_{l=1}^{\infty}(\eta_l^n)^2+C \sum_{l=1}^{\infty}(\widetilde{\eta}_l^n)^2,
\quad t>0.
\end{align}
The proof is completed.
\end{proof}

\end{document}